\documentclass{amsart}
\usepackage{amsmath}
\usepackage{amsthm}
\usepackage{amsfonts}
\usepackage{amssymb}
\usepackage{amscd}
\usepackage{mathabx}
\usepackage{psfrag}
\usepackage[dvips]{graphicx}
\usepackage[dvipsnames]{xcolor} %used for font color
\usepackage[english]{babel}
\usepackage[ansinew]{inputenc}

\def\EE{{\mathbb E}}
\def\NN{{\mathbb N}}
\def\ZZ{{\mathbb Z}}

\def\RR{{\mathbb R}}

\def\TT{{\mathbb T}}
\def\EE{{\mathbb E}}
\def\PP{{\mathbb P}}
\def\proj{{\mathbb P}}

\def\prodspa1{E_1(\vep_1)^2}
\def\prodspar{\fE_r(\vep_r)^2}

\def\gene{\operatorname{span}}
\def\sign{{\operatorname{sign}}}

\def\id{{\operatorname{id}\,}}

\def\supp{\operatorname{supp}}
\def\grass{\operatorname{Gr}}
\def\GL{\operatorname{GL}}
\def\SL{\operatorname{SL}}
\def\OO{\operatorname{O}}

\def\const{\operatorname{const}}
\def\VA{\operatorname{VA}}
\def\VP{\operatorname{VP}}
\def\SVA{\operatorname{SVA}}
\def\SVP{\operatorname{SVP}}
\def\Diag{\operatorname{Diag}}

\def\ev{\mathfrak{E}}
\def\li{\mathfrak{L}}
\def\ug{{\mathfrak{g}}}
\def\vg{{\mathfrak{g_\pm}}}

\def\Ext{\Lambda}
\def\cMc{{\cP_c(G)}}
\def\vep{\varepsilon}

\def\uomega{{\overset \to \omega}}
\def\quand{\text{ and }}

\def\cA{{\mathcal A}}
\def\cB{{\mathcal B}}

\def\cD{{\mathcal D}}
\def\cE{{\mathcal E}}
\def\cF{{\mathcal F}}

\def\cH{{\mathcal H}}

\def\cP{{\mathcal P}}
\def\cQ{{\mathcal Q}}

\def\cT{{\mathcal T}}

\def\cW{{\mathcal W}}

\def\cY{{\mathcal Y}}

\def\Chi{{\mathcal X}}
\def\Bd{\mathcal{B}}
\def\Cd{\mathcal{C}}

\def\tq{\tilde{q}}
\def\teta{{\tilde\eta}}
\def\ttheta{{\tilde\theta}}
\def\tsigma{{\tilde\sigma}}

\def\tvarphi{{\tilde\varphi}}
\def\tpsi{{\tilde\psi}}
\def\tomega{{\tilde\omega}}
\def\tT{{\widetilde\cT}}
\def\tQ{{\widetilde\cQ}}

\def\tE{{\widetilde\cE}}

\def\heta{{\hat\eta}}
\def\htheta{{\hat\theta}}
\def\hsigma{{\hat\sigma}}

\def\hpsi{{\hat\psi}}

\def\hT{{\widehat\cT}}

\def\hE{{\widehat\cE}}

\def\chq{\check{q}}
\def\cheta{{\check\eta}}
\def\chtheta{{\check\theta}}
\def\chsigma{{\check\sigma}}

\def\chnu{{\check\nu}}

\def\chomega{{\check\omega}}
\def\chT{{\widecheck\cT}}

\def\fpi{\pi^{\diamond}}
\def\fE{E^{\diamond}}
\def\fcT{\cT^{\diamond}}
\def\ftT{\tT^{\diamond}}
\def\fhT{\hT^{\diamond}}

\def\fsigma{\sigma^{\diamond}}
\def\ftsigma{\tsigma^{\diamond}}
\def\fhsigma{\hsigma^{\diamond}}

\def\feta{\eta^{\diamond}}

\def\fheta{\heta^{\diamond}}
\def\fcheta{\cheta^{\diamond}}

\def\qhT{\hT^{\cQ}}
\def\qchT{\chT^{\cQ}}

\def\qchsigma{\chsigma^{\cQ}}

\def\qheta{\heta^{\cQ}}
\def\qcheta{\cheta^{\cQ}}

\theoremstyle{plain}
\newtheorem{main}{Theorem}

\newtheorem{theorem}{Theorem}[section]
\newtheorem{lemma}[theorem]{Lemma}
\newtheorem{proposition}[theorem]{Proposition}
\newtheorem{corollary}[theorem]{Corollary}
\theoremstyle{definition}

\newtheorem{remark}[theorem]{Remark}
\newtheorem{example}[theorem]{Example}

\numberwithin{equation}{section}

\begin{document}

\title{Continuity of the Lyapunov exponents of random matrix products}
\author{Artur Avila, Alex Eskin, and Marcelo Viana}

\address{IMPA, Est. D. Castorina 110, Jardim Bot\^{a}nico, 22460-320 Rio de Janeiro, Brazil and
Institut f{\"u}r Mathematik, Universit{\"a}t Z{\"u}rich, Winterthurerstrasse 190, 8057 Z{\"u}rich,
Switzerland}
\email{avila@impa.br}

\address{Department of Mathematics, University of Chicago, Chicago, IL 60637, USA}
\email{eskin@math.uchicago.edu}

\address{IMPA, Est. D. Castorina 110, Jardim Bot\^{a}nico, 22460-320 Rio de Janeiro, Brazil}
\email{viana@impa.br}

\thanks{M.V. was partially supported by CNPq, FAPERJ, and  Fondation
  Louis D--Institut de France.}
\thanks{A.E. was partially supported by NSF grants DMS-1201422,
  DMS-1500702, DMS-1800646 and the Simons Foundation}
\date{\today}

\begin{abstract}
We prove that the Lyapunov exponents of random products in a (real or
complex) matrix group
depends continuously on the matrix coefficients
and probability weights. More generally, the Lyapunov exponents of the
random product defined by any compactly supported probability distribution
on $\GL(d)$ vary continuously with the distribution,
in a natural topology corresponding to weak$^*$-closeness of the
distributions and Hausdorff-closeness of their supports.
\end{abstract}

\maketitle

\tableofcontents

\part{Lyapunov exponents and random walks}\label{p_one}

\section{Introduction}\label{s_introduction}

\subsection*{Lyapunov exponents}
The notion of Lyapunov exponents is rooted in the stability theory of
differential equations created
by Lyapunov~\cite{Lya92} at the end of the 19th century.
Consider a differential equation
\begin{equation}\label{eq.quasi-linear}
x'= L(t)x + R(t,x),
\end{equation}
where $L(t):\RR^d\to\RR^d$ is linear and $R(t,\cdot)$ is a perturbation
of order greater than $1$.
The \emph{Lyapunov exponent function} $v \mapsto \lambda(v)$ is defined by
\begin{equation}\label{eq.limsup}
\lambda(v) = \limsup_{t\to\infty} \frac 1t \log \|x_v(t)\|
\end{equation}
where $x_v$ is the solution of the linear equation $x'=L(t)x$
with initial condition $v$.
When $\lambda<0$ the constant solution $x_0(t) \equiv 0$ is exponentially
stable for this linear equation.
The stability theorem of Lyapunov asserts that, under
a  technical \emph{Lyapunov regularity} condition,
it remains exponentially stable for \eqref{eq.quasi-linear}.

In 1960, Furstenberg, Kesten~\cite{FK60} proved that the
limit in \eqref{eq.limsup} exists for almost
every $x$, relative to any probability measure invariant under the flow.
A few years later, Oseledets~\cite{Ose68} showed that
Lyapunov regularity also holds for almost every point.
Such results brought the subject of Lyapunov exponents to the
realm of ergodic theory, where it has
prospered since. Three main problems have a central role in the theory.

The first one is \emph{non-triviality} of the
Lyapunov spectrum: \emph{when is it the case that not all
Lyapunov exponents are equal?} This was founded by
Furstenberg~\cite{Fur63} in the 1960s and has been
much studied since, especially in the last couple of decades
or so. See Viana~\cite{LLE} and references
therein. A product of this theory much exploited recently is the
Invariance Principle~\cite{Led86,BGV03,Extremal,ASV13}, a general
statement to the effect that systems
with trivial Lyapunov spectra are very rigid.

A related issue is that of \emph{simplicity} of the Lyapunov spectrum:
\emph{when are all the Lyapunov exponents distinct, with multiplicity $1$?}
This was initiated by Guivarc'h and Raugi~\cite{GuR86} and by
Gol'dsheid and Margulis~\cite{GoM89},
and has also been the object of considerable interest in recent years.
See~\cite{LLE} for references and a detailed discussion. An application
was the proof of the Zorich--Kontsevich conjecture on the
Lyapunov spectrum of the Teichm\"uller flow on the moduli space of
Abelian differentials~\cite{AvV07b}.

\subsection*{Continuity theorem}
In the present paper we are mostly concerned with
the \emph{dependence} problem: \emph{how do the Lyapunov
exponents depend on their underlying system?} Several references to the
literature on this problem will be
given in a while. Right now, the following special case of our main result
illustrates the kind of goals we pursue here.

Let $(A_1, \dots, A_m)$ be an $m$-uple of matrices in
$G = GL(\RR^d)$ and $(p_1, \dots, p_m)$ be an
element of the open simplex $\Delta_m$ of dimension $m-1$, that is,
an $m$-uple of numbers $p_i \in (0,1)$
such that $\sum_{i=1}^m p_i = 1$. Let $\nu$ be the probability measure
on $G$ given by
\begin{displaymath}
\nu = \sum_{i=1}^m p_i \delta_{A_i}
\end{displaymath}
where $\delta_A$ denotes the Dirac mass at any $A\in G$.
Let $\lambda_1 \ge \dots \ge \lambda_d$ be the Lyapunov exponents
of the random matrix product
induced by $\nu$ (definitions will appear in
Section~\ref{s_Lyapunov_exponents}). We prove:

\begin{main}\label{theorem:main}
For each $1 \le j \le n$, the number $\lambda_j$ depends continuously
on the $A_i$ and the $p_i$ at every
point of the domain $G^m \times \Delta_m$.
\end{main}

The 2-dimensional case of Theorem~\ref{theorem:main} has
been proved by Bocker and Viana~\cite{BoV17}.
A different proof of that case that introduces a few of
the ideas in this paper appeared in Chapter~10
of the book~\cite{LLE}. Kifer~\cite{Ki82b} observed the
Lyapunov exponents may actually jump when some
weight $p_i$ goes to zero, and that is why $\Delta_m$ is taken to
be an \emph{open} simplex.

A crucial point in  Theorem~\ref{theorem:main} is that the conclusion
holds even at \emph{reducible} points,
that is, when the matrices $A_1, \dots, A_n$ share one or more
invariant proper subspaces.
Indeed, continuity of the Lyapunov exponents at irreducible points had
already been proved in the 1980s,
independently, by Furstenberg and Kifer~\cite{FK83}
and by Hennion~\cite{Hen84}.
As often happens in this field, the reducible case
is a lot more subtle, requiring a whole different set of ideas.

Our approach relies on a quantitative analysis of the random walk
$$
x \mapsto gx, \quad\text{$g\in G$ a random variable with distribution $\nu$,}
$$
defined on the projective space $P=\PP(\RR^d)$ by the probability measure $\nu$.
By Furstenberg and Kifer~\cite{FK83}, discontinuity of
the Lyapunov exponents can only occur if there is
some proper subspace invariant under all the matrices
and containing all the ``most contracting'' directions for
the cocycle (see Section~\ref{s_stationary_measures_and_equator} for the
precise statement).

In a nutshell, we prove that if such an
invariant subspace $E$ (the ``equator'') does exist,
typical trajectories of nearby generic
random walks spend very little time in
its vicinity, rendering the presence
of the equator effectively harmless.
A bit more precisely, we consider generic distributions $\nu_k$
converging to $\nu$ as $k\to\infty$,
and we show that the stationary measures $\eta_k$ for the
corresponding random walks cannot accumulate on
the equator: any limit point $\eta_\infty$ as $k\to\infty$
(which is automatically a stationary measure
for $\nu$) must satisfy $\eta_\infty(E)=0$. These notions
and their basic properties will also be recalled
in Section~\ref{s_stationary_measures_and_equator}.

\subsection*{Margulis functions}
The key technical tool to do this is the concept of Margulis function.
Such functions have been introduced to the dynamics literature
by Margulis in \cite{EMM98}. (In the probability setting,
a Margulis function is also called a Foster-Lyapunov (or drift) function,
and has been used extensively. See the book \cite{MT09} for further references.)

In a few words, $\Phi:X\to[0,\infty]$ is
a \emph{(multiplicative) Margulis function}
for a Markov operator $\cT$ on some space $X$, relative to
a set $Y \subset X$, if $\Phi \equiv \infty$
on $Y$ and there exist constants $c < 1$ and $b<\infty$ satisfying
\begin{equation}\label{eq.multiplicative}
\cT\Phi(x) \le c \Phi(x) + b
\text{ for all $x\in X$.}
\end{equation}
The distinctive feature implied by this inequality is
that $\cT\Phi(x)$ is much smaller than $\Phi(x)$ near $Y$,
even if it may be somewhat bigger on other parts of $X$.
Such functions have been used, for example, in
\cite{EsM01,EsM04,EMM05,Ath06,
AvF07,BGTM22,EsM11,BeQ12,EK12,AvG13,KKLM17,MaV15,BBB18,GKM22,GLT22}
and \cite[Chapter~10]{LLE}. For a fairly recent survey
  on this topic see \cite{EsMo22}.

Nevertheless, our application of Margulis functions in the present
setting comes with a number of novelties.
To begin with, we need a different kind of Margulis function, which we
introduce here:
given a partition $(A,B)$ of the space $X$, $\Phi$ is an \emph{(additive)
Margulis function} if there exist positive
constants $\kappa_A$ and $\kappa_B$ such that
\begin{equation}\label{eq.additive}
\begin{aligned}
\cT\Psi(x) & \le \Psi(x) - \kappa_A \text{ for every } x\in A\\
\cT\Psi(x) & \le \Psi(x) + \kappa_B \text{ for every } x\in B.
\end{aligned}
\end{equation}

If $\Phi$ is a multiplicative Margulis function then $\log\Phi$
is an additive Margulis function relative to a
suitable partition $(A,B)$ (see Remark~\ref{r.additive_multiplicative}).
On the other hand, it is not true that if $\Psi$ is an additive Margulis
function then $\exp\Psi$ is a multiplicative one.
That is because the inequality \eqref{eq.multiplicative} is very sensitive
to the ``worst case'' behavior of $\Psi$,
whereas \eqref{eq.additive} depend more on the ``average case'' behavior.
For this reason, it is often much easier to
construct an additive Margulis function than a multiplicative one. In
fact, we do not know how to construct a useful
multiplicative Margulis function in our setting beyond the case $d=2$
(see \cite[Chapter~10]{LLE} and \cite{MaV15}).

Another point worth emphasizing is that in all the previous
constructions (apart from \cite{MaV15,LLE,BBB18}),
the dynamical system is fixed. Instead, in the present paper
we try to make the same Margulis function work for
a whole family of dynamical systems, which have very different
behaviors near the equator. This introduces quite
a lot of new issues. Furthermore, the dynamical behavior near
the equator is totally different from the behavior
in other parts of phase-space. Thus, we need to carry
out a localized analysis of the random walk,
which is another important source of difficulties.

\subsection*{Further perspectives}

Random products of matrices may be represented as a special kind of
\emph{linear cocycle}
$$
F: M \times \RR^d \to M \times \RR^d, \quad F(x,v) = (f(x),A(x)v)
$$
where the base dynamics $f:M\to M$ is a shift map $f((x_n)_n) = (x_{n+1})_n$
endowed with a Bernoulli measure,
and the cocycle function $A$ depends only on the coordinate $x_0$.
The dependence problem extends naturally to this general setting
of linear cocyles: usually one takes the base
dynamics $f$ and the corresponding invariant probability measure
$\mu$ to be fixed, and one is interested in
understanding how the Lyapunov exponents depend on $A$.

A natural step is to try and allow for much more general
cocycle functions $A$.
For reasons that we will soon discuss, it is convenient
to assume some regularity, like H\"older continuity.
Moreover, essentially all known results assume the
cocycle to satisfy a kind of quasi-conformality condition
called \emph{fiber-bunching} (see \cite{BGV03,Extremal})
which also involves H\"older continuity.

Another natural way to broaden the scope of the theory is to
weaken the assumptions on the base dynamics,
to consider general dynamical systems more general than shift maps,
as well as invariant measures satisfying
much milder independence conditions. In this latter
direction, Theorem~\ref{theorem:main} has been extended to
Markov products of 2-dimensional matrices by Malheiro and Viana~\cite{MaV15}.

For H\"older cocycles, Backes, Brown and Butler~\cite{BBB18}
extended the 2-dimensional case of Theorem~\ref{theorem:main}
to general fiber-bunched cocycles whose base transformation $f$ is a
hyperbolic homeomorphism on a compact
metric space (in the sense of \cite{Almost}) and whose invariant probability
measure has \emph{local product structure}
(see \cite{BGV03,Almost}), a mild requirement meaning, roughly speaking, that
the future depends only weakly on the past. In fact, their statement
extends to the class of linear cocycles with invariant holonomies.
Both versions had been conjectured in \cite[Section~10.6]{LLE}.

Another interesting path to possibly generate
further progress is to consider linear cocycles over partially
hyperbolic diffeomorphisms, volume-preserving or not.
Groundwork in this direction has been laid in~\cite{ASV13}
and some continuity results have been derived in~\cite{Extremal}.
See also Avila, Viana and Wilkinson~\cite{AVW15,AVW22}, and Poletti
and Viana~\cite{PoV19}.

The need for some regularity of the cocycle function is highlighted by the
following result of Bochi~\cite{Boc02}:
if the system $(f,\mu)$ is aperiodic, $d=2$,
and the linear cocycle is continuous and not uniformly hyperbolic,
then it may be $C^0$-approximated by linear cocycles with trivial
Lyapunov spectra (the two Lyapunov exponents
are equal). Thus, continuity can only hold at cocycles which either
are uniformly hyperbolic or have trivial spectra.

In fact, the same is true restricted to the class
of derivative cocycles of area-preserving surface diffemorphisms,
a much harder fact which was discovered by Ma\~n\'e~\cite{Man84}
and whose proof was completed by Bochi~\cite{Boc02}.
These results have been extended to arbitrary dimensions by Bochi
and Viana~\cite{BcV05,Boc09}.
They are generally not true for cocycles over non-invertible maps,
even in the $\SL(\RR^2)$ case, according to
Viana and Yang~\cite{ViY19}.

On the other hand, the actual relevance of the fiber-bunching
condition in this context is presently not entirely
clear, indeed this remains one of the outstanding open questions
in this area.
Examples of discontinuity of the Lyapunov exponents for
H\"older continuous linear cocycles which are not
fiber-bunched have been found in \cite[Section~9.3]{LLE}
and Butler~\cite{But18}.

We have restricted our attention to matrix groups, for good reason.
While the basic concepts discussed here (such as Lyapunov exponents,
Oseledets regularity, etc.) extend to the more general setting of all
(not necessarily invertible) matrices, there is no hope to obtain any
general regularity result for Lyapunov exponents in this more general
setting.

To explain why, let us consider the Lyapunov exponents of the random
product of two real matrices $A_1$ and $A_2$, with probability weights
$p_1=p_2=1/2$.  The Lyapunov exponents are well defined, but the bottom
one is equal to $-\infty$ if one of the matrices is not invertible.
Moreover, if some finite matrix product involving $A_1$ and $A_2$ is
zero then the top Lyapunov exponent is $-\infty$ as well.

For instance, let $L(\theta)$ denote the top Lyapunov exponent for
$$
A_1=\begin{pmatrix} 1&0\\0&0\end{pmatrix}
\text{ and }
A_2=\begin{pmatrix} \cos 2 \pi \theta&-\sin 2 \pi \theta\\
\sin 2 \pi \theta&\cos 2 \pi \theta \end{pmatrix},
$$
viewed as a function of $\theta \in \RR$.
For $\theta=p/4 q$ with $p$ odd and $q$ a non-zero integer, we have that $A_1 A_2^q A_1=0$,
and so $L(\theta)=-\infty$.  By an upper semi-continuity argument, it follows that
$L(\theta)=-\infty$ for Baire-generic $\theta$.
In fact, it is not difficult to give an explicit generic quantitative condition ensuring
that $L(\theta)=-\infty$, and even a sharp one, using the easily checked formula
$$
L(\theta)=\sum_{k=0}^\infty 2^{-k-2} \log |\cos 2 \pi k \theta|,
$$
which also shows that the Lyapunov exponent $L(\theta)$ is finite (and discontinuous)
at a full Lebesgue measure set of $\theta$.

What happens in this sort of situation is that, while one can still analyze the Lyapunov
exponents of cocycles such as this one using a stationary measure $\eta$ on the projective space,
just as we do  in the present paper for the invertible case (but taking care of issues such as
indeterminacy), in the present setting the measure $\eta$ becomes atomic, being the sum of
Dirac masses with weights $2^{-k-1}$ on the lines through $(\cos 2 \pi k \theta, \sin 2 \pi k \theta)$.

\subsection*{Significance and applications of continuity}
Knowledge that the Lyapunov exponents are continuous at some $f$ can in itself
give information about the dynamics of $f$, as pointed out in Bochi, Viana~\cite{BcV05},
and abstract facts about the existence of many (in the Baire sense) continuity points can
be leveraged to a fine understanding of the dynamics from the generic point of view.

For instance, continuity of Lyapunov exponents ensures that the Oseledets decomposition
varies continuously in a suitable sense: see Backes, Poletti~\cite{BaP17}.
Moreover, the convergence in the Oseledets theorem is locally uniform on the cocycle.
This sort of uniformity is useful in various situations in dynamical systems.
An example is the following relevant question in the ergodic theory of volume-preserving
diffeomorphisms $f:M \to M$. See~\cite{FHY83} or~\cite{PSh89} for background.

Let $K\subset M$ be a Pesin block, that is, a compact (non-invariant) set where the Lyapunov
exponents are all bounded away from zero and the estimates in the Oseledets theorem hold uniformly.
Over such a set, the Pesin stable and unstable manifolds are well defined and depend continuously
on the point. In particular they have a definite size, and so nearby points in the Pesin block
must belong to the same ergodic component. One may thus ask about the stability of Pesin blocks:
is it the case that a smooth perturbation of $g$ must also possess a Pesin block $K_g$ nearby
(in the sense that the symmetric difference $K \Delta K_g$ has small measure) such that its
Pesin manifolds are close to the unperturbed ones? This can be shown to follow from a suitable
control of the dependence of the Lyapunov exponents, and in particular if the averaged Lyapunov
exponents depend continuously on the diffeomorphism at the point $f$.

One setting where knowledge about the continuity of the Lyapunov exponents has been used as an
essential ingredient in the understanding of the dynamics is in the study of quasiperiodic
Schr\"odinger operators.
Here $f=f_\alpha:x \mapsto x+\alpha$ is a translation on a finite dimensional torus $\TT^d$,
and the cocycle function takes the form
$$
A(x) = \begin{pmatrix}
E-v(x)&-1\\1&0 \end{pmatrix} \in \SL(2,\RR).
$$
In this case the continuity of Lyapunov exponents with respect to both $\alpha$ and $A$ has been
proved for analytic $A$ and totally irrational $\alpha$ (meaning that $f_\alpha$ is minimal),
by Bourgain and Jitomirskaya~\cite{BJ02b} when $d=1$ and by Bourgain~\cite{Bou05} in the general case.

It was used, for instance, in the solution by Avila and Jitomirskaya~\cite{AvJ09} of the
Ten Martini Problem, which asked whether the Almost Mathieu Operator has a Cantor spectrum
(this can be rephrased as density of uniform hyperbolicity within certain one-parameter families of cocycles).
It also appears prominently in the proof of the quantization of the acceleration for $d=1$,
which is the starting point of the so-called global theory of one-frequency Schr\"odinger operators
(Avila~\cite{Avi15}).

Continuity is very subtle in this context: for instance,
the aforementioned result of Bourgain and Jitomirskaya does not hold
when $A$ is merely $C^\infty$ (a result of
Wang-You~\cite {WY}).  Continuity as a function of $\alpha$ also may fail at
rational $\alpha$ even when $A$ is
analytic, see \cite{BJ02b} for a discussion.
%Examples of discontinuous dependence of the Lyapunov exponents
%on the energy $E$ for Schr\"odinger cocycles over quasi-periodic
%systems have been exhibited by Johnson~\cite{Joh84}.

Another setting which connects with the ideas discussed in this paper is that of
$\SL(2,\RR)$-actions on moduli spaces of Abelian or quadratic differentials.
Let $\eta$ be an $\SL(2,\RR)$-invariant probability measure.
By Eskin-Mirzakhani~\cite {EsM18}, $\eta$ is equivalent to Lebesgue measure on some submanifold.
The Kontsevich-Zorich cocycle over the Teichm\"uller flow, which plays a fundamental role
in the ergodic theory of translation surfaces, has a nice behavior with respect to $\eta$:
it is basically a random matrix product (involving countably many matrices),
except that the products are merely ``quasi-independent''.
See \cite {AvV07b} for an application of this idea to the issue of simplicity of the Lyapunov
spectrum.

Let $(\nu_k)_k$ be a sequence of $\SL(2,\RR)$-invariant probability measures converging
in the weak$^*$ sense to some probability measure $\nu$.
It would be tempting to use the techniques of our paper to address the issue of continuity of
the Lyapunov exponents in this context. However, it turns out that the difficulties addressed
in our paper do not show up, and hence simpler techniques can be applied, as was done in
Bonatti, Eskin, and Wilkinson~\cite{BEW20}. Indeed, by Theorem~2.3 in Eskin, Mirzakhani,
and Mohammadi~\cite{EMM15} (see also Theorem~2.6 in \cite{BEW20}), the support of $\nu_k$ is
contained in $\supp\nu$ for every large $k$.  So, in terms of the random matrix models,
if there is an invariant subspace for $\eta$ then it is also invariant for the $\eta_k$ for
large $k$.  This allows one to quotient out bad invariant spaces, and establish continuity
by the usual Furstenberg--Kifer argument~\cite {FK83}.

%Thus, weak$^*$
%convergence implies that $(\nu_k)_k$ converges to $\nu$
%in the sense considered in the present paper
%(see Section~\ref{s_Lyapunov_exponents}).
%Our main result (Theorem~\ref{theorem:main} can then be used to conclude that their Lyapunov exponents converge as well,
%as we are going to explain. ARTUR COMPLETES

\subsection*{Quantitative regularity}
Another natural question is how much can the regularity of Lyapunov be upgraded from mere continuity.
An old result of Ruelle~\cite{Rue79a} asserts that if all the matrix coefficients are positive then
the largest Lyapunov exponent is a real-analytic function of those coefficients.
For locally constant cocycles over Markov shifts, Peres~\cite{Pe91} has shown that if the Lyapunov
exponents are simple then they depend real-analytically on the transition data, assuming the cocycle
function itself is fixed.

For parametrized random matrix products satisfying strong irreducibility and the contraction property,
Le Page~\cite{LP89} has proved that the largest Lyapunov exponent is a H\"older continuous function of
the parameter. This function is even $C^\infty$ if the probability distributions are absolutely continuous.
In the opposite direction, a construction of Halperin (see Simon and Taylor~\cite[Appendix~3]{ST85})
shows that for every $\alpha>0$ one can find random Schr\"odinger cocycles near which the Lyapunov
exponents fail to be $\alpha$-H\"older continuous.

These results have been sharpened by Duarte and Klein, who developed a unified approach to proving
generic moduli of continuity of the Lyapunov exponents for different classes of linear cocycles,
both random and quasi-periodic, especially in the $2$-dimensional case. See~\cite{DuK19a,DuK19b}
and also~\cite{DuK16b} for an account of their approach and many applications.

Still in the 2-dimensional case, Tal and Viana~\cite{TVi20} have shown that H\"older continuity holds
at every point where the Lyapunov spectrum is simple. This is an application of estimates obtained from
the methods we develop here, namely a uniform bound
$$
\eta(E(r)) \le C r^\beta
$$
for the weight of the neighborhood $E(r)$ of the equator relative to stationary measures $\eta$ of
nearby random walks. Tal and Viana~\cite{TVi20} have also shown that, while H\"older continuity may fail
when the two Lyapunov exponents coincide, a weaker log-H\"older modulus of continuity does hold at every
point. It would be interesting to extend these results to arbitrary dimension.

Still regarding products of finitely many matrices, one problem that has proved to be very resistent to
all techniques so far is whether the dependence of the Lyapunov exponent can be much better than H\"older
in some non-trivial region of the parameter space. For instance, let us consider random matrix products
of two $\SL(2,\RR)$-matrices $A_1$ and $A_2$, with probability weights $p_1=p_2=1/2$.
Over the open set $\mathbb {UH}$ of uniformly hyperbolic pairs $(A_1,A_2)$, the top Lyapunov exponent is
clearly a real analytic function. Fix $1 \leq k<\infty$. \emph{Is there a non-empty open subset in the
complement of $\mathbb{UH}$ over which the top Lyapunov exponent is $C^k$?}

It is tempting to try to answer (affirmatively) this question by establishing a suitable spectral gap.
Unfortunately the current approaches to the spectral gap use algebraic properties of the matrix coefficients,
and thus do not apply over any open set, see \cite {Bou14}.

\section{Statement of main result}\label{s_Lyapunov_exponents}

We state our main result, Theorem~\ref{theorem:second} below, of which Theorem~\ref{theorem:main}
is an easy consequence. Initially, we recall the notion of Lyapunov exponents and

Given any compactly supported probability measure $\nu$ on $G=\GL(\RR^d)$,
let $\nu^\NN$ and $\nu^\ZZ$ denote the
corresponding Bernoulli measures on $G^\NN$ and $G^\ZZ$, respectively. Consider the shift maps
$\sigma:G^\NN \to G^\NN$ and $\sigma:G^\ZZ \to G^\ZZ$ given by
$$
\sigma\left((g_n)_n\right) = (g_{n+1})_n.
$$

By the Oseledets multiplicative ergodic theorem (see~\cite[Theorems~4.1 and~4.2]{LLE}),
there exist $k\in\{1, \dots, d\}$ and real numbers
\begin{equation}\label{eq_sec2_exponents}
\chi_1(\nu) > \cdots > \chi_k(\nu)
\end{equation}
such that for $\nu^\NN$-almost every $\ug = (g_0, \dots, g_n, \dots) \in G^\NN$
there exists a decreasing family of vector subspaces
\begin{equation}\label{eq_sec2_filtration}
\RR^d = V^1(\ug) > \cdots > V^k(\ug) > V^{k+1}(\ug)=\{0\}
\end{equation}
and for $\nu^\ZZ$-almost every $\vg = (\dots, g_{-n}, \dots, g_0, \dots, g_n, \dots) \in G^\ZZ$
there exists a direct sum decomposition
\begin{equation}\label{eq_sec2_splitting}
\RR^d = E^1(\vg) \oplus \cdots \oplus E^k(\vg)
\end{equation}
such that, for every $i=1, \dots, k$,
\begin{itemize}
\item $g_0 V^i(\ug) = V^i(\sigma(\ug))$ and $g_0 E^i(\vg) = E^i(\sigma(\vg))$ for $\nu$-almost every $g_0\in G$;
\item $V^i(\ug) = E^i(\vg) \oplus V^{i+1}(\ug)$ for $\ug=\pi(\vg)$, where $\pi:G^\ZZ\to G^\NN$
denotes the canonical projection.

\item for every non-zero $v_i \in V^i(\ug)\setminus V^{i+1}(\ug)$ and $\nu^\NN$-almost every $\ug\in G^\NN$
\begin{equation}\label{eq_sec2_filtration_limit}
\lim_n \frac 1n \log\|g_{n-1}\cdots g_0 v_i\| = \chi_i(\nu);
\end{equation}

\item for every non-zero $v_i\in E^i(\vg)$ and $\nu^\ZZ$-almost every $\vg\in G^\ZZ$
\begin{equation}\label{eq_sec2_splitting_limit}
\lim_n \frac 1n \log\|g_{n-1}\cdots g_0 v_i\| = \chi_i(\nu) = \lim_n \frac 1{-n} \log\|g_{-n}^{-1} \cdots g_{-1}^{-1} v_i\|.
\end{equation}
\end{itemize}

The maps $\ug \mapsto V^i(\ug)$ and $\vg \mapsto E^i(\vg)$ with values in the Grassmannian of $\RR^d$
are measurable and the dimensions $\dim V^i(\ug)$ and $\dim E^i(\vg)$ are constant on full measure sets.
The number $m_i = \dim V^i - \dim V^{i+1}=\dim E^i$ is called the \emph{multiplicity} of the \emph{Lyapunov
exponent} $\chi_i(\nu)$. Denote by $\lambda_1(\nu) \ge \cdots \ge \lambda_d(\nu)$ the Lyapunov exponents
\emph{counted with multiplicity}.

Let $P=\PP(\RR^d)$.  The \emph{random walk} defined by $\nu$ is described by the pair $(F,\nu^\NN)$,
where
\begin{equation}\label{eq_sec2_randomwalk}
F:G^\NN\times P \to G^\NN\times P,
\quad (\ug,v) \mapsto (\sigma(\ug), g_0v).
\end{equation}
The Lyapunov exponents and the Oseledets filtration may also be obtained from it, as follows.
Define
\begin{equation}\label{eq_sec2_Phi}
\Phi:G \times P \to \RR, \quad \Phi(g,v) = \log \frac{\|g v\|}{\|v\|}.
\end{equation}
(For notational simplicity, we use the same symbol ($v$, say) to denote both a non-zero vector in $\RR^d$
and the corresponding element of $P$; analogously, we use the same notation ($L$, say) for a vector
subspace of $\RR^d$ and the subset of $P$ associated to it.)
A result of Ledrappier (see~\cite[Theorem~6.1]{LLE}) asserts that:
\begin{itemize}
\item Given any $F$-invariant ergodic probability measure $m$ on $G^\NN\times P$ that projects
to $\nu^\NN$, there exists $j\in\{1, \dots, k\}$ such that
\begin{equation}\label{eq_sec2_Ledrappier}
\int_{G^\NN\times P} \Phi \,dm = \chi_j(\nu) \quand m\left(\{(\ug,v): v \in V^j(\ug) \setminus V^{j+1}(\ug)\}\right)=1.
\end{equation}
\item Given any $j\in\{1, \dots, k\}$ there is an ergodic $F$-invariant probability measure $m$ projecting
to $\nu^\NN$ and satisfying \eqref{eq_sec2_Ledrappier}.
\end{itemize}

Let $(A_{1,k}, \dots, A_{m,k})$, $k\in\NN$ be a sequence of $m$-uples of matrices converging to some
$(A_{1,\infty}, \dots, A_{m,\infty})\in G^m$ and $(p_{1,k}, \dots, p_{m,k})$, $k\in\NN$ be a sequence
of probability $m$-vectors real numbers converging to some $(p_{1,\infty}, \dots, p_{m,\infty})\in \Delta_m$.
Let $\nu_k$ and $\nu_\infty$ be the probability measures in $G$ given by
\begin{equation}\label{eq_sec2_muk_muinf}
\nu_k = \sum_{i=1}^m p_{i,k} \delta_{A_{i,k}}
\quand
\nu_\infty = \sum_{i=1}^m p_{i,\infty} \delta_{A_{i,\infty}}.
\end{equation}
We want to prove that $\lambda_j(\nu_k)\to\lambda_j(\nu_\infty)$ when $k\to\infty$,
for every $j=1, \dots, d$.

In fact, we prove a stronger statement, involving probability measures whose supports need not be
finite. Let $\cMc$ be the space of compactly supported probability measures on $G$, with the
smallest topology $\cT$ that contains both:
\begin{itemize}
\item $\cW=$ the restriction of the weak$^*$ topology in the space of probability measures on $G$
\item $\supp^*\cH=$ the pull-back under $\nu\mapsto\supp\nu$ of the Hausdorff topology $\cH$ in the space
of compact subsets of $G$.
\end{itemize}
This topology $\cT$ is metrizable, because both $\cW$ and $\cH$ are. A sequence $(\nu_k)$ converges to
$\nu_\infty$ in $\cMc$ if and only if
\begin{itemize}
\item[(i)] $(\nu_k)_k \to \nu_\infty$ in the weak$^*$ topology and
\item[(ii)] $(\supp\nu_k)_k \to \supp\nu_\infty$ in the Hausdorff topology.
\end{itemize}
That is the case for the measures in \eqref{eq_sec2_muk_muinf} if $A_{i,k} \to A_{i,\infty}$ and
$p_{i,k} \to p_{i,\infty}$ for every $i=1, \dots, m$; here, the assumption that $p_{i,\infty}>0$
is important to ensure continuity of the supports.

Related to this, the example of Kifer~\cite{Ki82b} shows that part (i) alone is not enough to ensure
continuity of the Lyapunov exponents: for our results to hold one cannot omit part (ii) of the
definition of the topology.

\begin{remark}
If $(\nu_k)_k\to\nu_\infty$ in the weak$^*$ topology then, given any $\vep>0$, the support of $\nu_\infty$
is contained in the $\vep$-neighborhood of $\supp\nu_k$ for all large $k$.
Thus, the condition (i) in the definition contains half of the condition (ii).
The other half is that, given any $\vep>0$, the support of $\nu_k$ is contained in the $\vep$-neighborhood
of $\supp\nu_\infty$ for all large $k$. This will be used repeatedly.
\end{remark}

Also, let us point out that for proving Theorem~\ref{theorem:main} it suffices to consider the case $j=1$.
That is because of the following construction.
Let $E=\RR^d$ and $1 \le l \le d$. The \emph{exterior $l$-power} $\Ext^l E$ of $E$ is the
vector space of alternating $l$-linear forms $\omega:E^*\times \cdots \times E^*\to\RR$ on the
dual space $E^*$. The \emph{exterior product} of vectors $v_1, \dots, v_l \in E$, is the
alternating $l$-linear form $v_1 \wedge \cdots \wedge v_l :E^*\times \cdots \times E^*\to\RR$
defined by
$$
\left(v_1 \wedge \cdots \wedge v_l\right) (\phi_1, \ldots, \phi_l)
= \sum_\sigma \sign(\sigma) \, \phi_{\sigma(1)}(v_1) \cdots \phi_{\sigma(l)}(v_l),
$$
where the sum is over all permutations of $\{1, \dots, l\}$. If $\{e_j: j=1, \dots, d\}$ is a basis of $E$
then $\{e_{j_1} \wedge \cdots \wedge e_{j_l}: 1 \le j_1 < \cdots < j_l \le d\}$ is a basis of $\Ext^l E$.
So,
$$
\dim\Ext^l E = \begin{pmatrix} d \\ l \end{pmatrix}.
$$

Every $g\in G$ induces an invertible linear map $\Ext^l g:\Ext^l E \to \Ext^l E$, defined by
$$
\Ext^l g(\omega): (\phi_1, \ldots, \phi_l) \mapsto \omega(\phi_1 \circ g, \ldots, \phi_l \circ g),
$$
for $\omega\in\Ext^l E$ and $\phi_1, \dots, \phi_l \in E^*$. Thus, any measure $\nu$ in $G$
induces a measure $\Ext^l\nu$ in $\GL(\Ext^l E)$, by push-forward under $g\mapsto \Ext^l g$.
Moreover, the maps $\nu \mapsto\Ext^l\nu$ are continuous.
One can check (see~\cite[Proposition~4.17]{LLE}) that the Lyapunov exponents of the random
walk defined by $\Ext^l\nu$, counted with multiplicity, are the sums
$$
\lambda_{i(1)}(\nu) + \cdots + \lambda_{i(l)}(x)
\text{ with $1 \le i_1 < \cdots < i_l \le d$}.
$$
In particular,
$$
\lambda_1(\Ext^l\nu)=\lambda_1(\nu) + \cdots + \lambda_l(\nu).
$$
Thus, proving  that $\nu \mapsto \lambda_1(\Ext^l\nu)$ is continuous, for every $1 \le l \le d$,
will entail that $\nu\mapsto\lambda_j(\nu)$ is continuous, for every $1 \le j \le d$.

In view of these observations, Theorem~\ref{theorem:main} will follow immediately from:

\begin{main}\label{theorem:second}
The function $\lambda_1:\cMc\to\RR$, $\nu \mapsto \lambda_1(\nu)$ is continuous, in any dimension $d\ge 2$.
\end{main}

The rest of the paper is devoted to proving Theorem~\ref{theorem:second}.
In Sections~\ref{s_stationary_measures_and_equator} and~\ref{s_uniform_estimate}
we present a useful large deviations principle for Lyapunov exponents (Theorem~\ref{t_sec4_uniform_estimate}).
In Sections~\ref{s_random_repeller} and~\ref{s_tool_box} we introduce several useful tools.
In Section~\ref{s_outline_of_the_proof} we reduce the proof of Theorem~\ref{theorem:second}
to a main technical result, Theorem~\ref{theorem:inductive}.
The proof of the latter result is  by induction on the dimension $r$ of the equator,
as outlined in Section~\ref{s_outline_of_the_proof}.
The case $r=1$ is carried out in detail in Sections~\ref{s_step1_preparing_a_Margulis_function}
through~\ref{s_step_1_recoupling_and_conclusion}.
The inductive step is dealt with in Sections~\ref{s_preparing_a_Margulis_function}
through~\ref{s_recoupling_conclusion}.

Before stating  Theorem~\ref{theorem:inductive} and outlining its proof, we must introduce several
general notions and a number of auxiliary results. On the other hand, the proofs of those results,
given in Sections~\ref{s_stationary_measures_and_equator} to~\ref{s_tool_box}, are in themselves
not used for establishing Theorem~\ref{theorem:inductive}. Thus the reader is encouraged to skip
them at first reading, proceeding as directly as possible to Section~\ref{s_outline_of_the_proof}.

\section{Invariant subspaces}\label{s_stationary_measures_and_equator}

In this section we introduce some background material, due mostly to Furstenberg and Kifer~\cite{Fur63,FK83}.
This also allows us to introduce some notations that will be useful in the following.
Proofs and more information can also be found in Chapters 4 through 6 of \cite{LLE}.

\subsection{Stationary measures}\label{ss_stationary_measures}

Fix $\nu\in\cMc$. We say that a probability measure $\eta$ on $P$ is \emph{$\nu$-stationary} if
$$
\int_P \psi(x) \, d\eta(x) = \int_{G \times P} \psi(gy) \, d\nu(g) \, d\eta(y)
$$
for every bounded measurable function $\psi:P\to\RR$. In other words, $\eta$ is $\nu$-stationary if and
only if $\cP_\nu^*\eta=\eta$, where $\cP_\nu^*$ is the operator defined by
\begin{equation}\label{eq_sec3_Markov_operator1}
\cP_\nu^*\eta = \int_G \left(g_*\eta\right) \, d\nu(g)
\end{equation}
in the space of probability measures. Moreover (see \cite[Proposition~5.5]{LLE}), $\eta$ is $\nu$-stationary
if and only if the probability measure $\nu^\NN\times \eta$ is invariant under the projective cocycle
$F:G^\NN \times P \to G^\NN \times P$ defined in \eqref{eq_sec2_randomwalk}.
Stationary measures always exist (see \cite[Proposition~5.6]{LLE}).

We also consider the operator $\cP_\nu:\Bd(P)\to\Bd(P)$ defined in the space $\Bd(P)$ of
measurable bounded functions $\psi:P\to\RR$ by
\begin{equation}\label{eq_sec3_Markov_operator2}
\cP_\nu\psi(v) = \int_G \psi(gv) \, d\nu(g).
\end{equation}
A function $\psi$ is said to be \emph{$\nu$-stationary} if $\cP_\nu\psi=\psi$. A $\nu$-stationary measure $\eta$
is \emph{ergodic} if every $\nu$-stationary function is constant on some full $\eta$-measure set.
This happens (see \cite[Proposition~5.13]{LLE}), if and only if the $F$-invariant measure
$\nu^\NN\times\eta$ is ergodic for $F$. The ergodic decomposition theorem (see \cite[Theorem~5.14]{LLE})
asserts that every $\nu$-stationary measure is a convex combination of ergodic $\nu$-stationary measures.

Let $\Phi$ be as in \eqref{eq_sec2_Phi}.
%(on the left-hand side $v$ denotes any projective class, that is, any element of the projective
%space $P$, whereas on the right-hand side it denotes any vector $v\in\RR^d$ in that projective class).
Then (see \cite[Proposition~6.7]{LLE}),
\begin{equation}\label{eq_sec3_Furstenberg_formula}
\lambda_1(\nu) = \max\big\{\int_{G \times P} \Phi \, d(\nu\times\eta): \text{$\eta$ is a $\nu$-stationary measure}\big\}.
\end{equation}
Denote $\alpha(\eta)= \int_{G \times P} \Phi \, d(\nu\times\eta)$ for each $\nu$-stationary measure $\eta$.
The ergodic decomposition theorem implies that the maximum does not change if we restrict to ergodic
$\nu$-stationary measures.

We may also view $\Phi$ as a function on $G^\NN\times P$ that depends only on $g_0$ an $v$:
$$
\Phi(\ug,v) = \log \frac{\|g_0 v\|}{\|v\|}.
$$
Then
$$
\begin{aligned}
%\lim_n \frac 1n \log \|g_{n-1}\cdots g_0 v\|
% =
\lim_n \frac 1n \log\frac{\|g_{n-1}\cdots g_0 v\|}{\|v\|}
= \lim_n \frac 1n \sum_{j=0}^{n-1} \log\frac{\|g_jg_{j-1}\cdots g_0 v\|}{\|g_{j-1}\cdots g_0 v\|}
= \lim_n \frac 1n \sum_{j=0}^{n-1} \Phi(F^j(\ug, v)).
\end{aligned}
$$
So, for any ergodic $\nu$-stationary measure $\eta$, we have
\begin{equation}\label{eq_sec3_Birkhoff_exponent}
\lim_n \frac 1n \log \frac{\|g_{n-1}\cdots g_0 v\|}{\|v\|}
= \int_{G \times P} \Phi \, d(\nu\times\eta)
\end{equation}
$\nu^\NN\times\eta$-almost everywhere in $G^\NN\times P$.

Furstenberg and Kifer have shown (see \cite[Theorem~2.1]{FK83}) that if $\alpha(\eta)=\lambda_1(\nu)$
for every (ergodic) $\nu$-stationary measure $\eta$ then for every $v\in P$
\begin{equation}\label{eq_sec3_FK}
\lim_n \frac 1n \log \frac{\|g_{n-1}\cdots g_0 v\|}{\|v\|} = \lambda_1(\nu)
\text{ for $\nu^\NN$-almost every $\ug$.}
\end{equation}
This fact is also contained in Theorem~\ref{t_sec4_uniform_estimate} below.
In the next section we analyze what happens when the hypothesis of \eqref{eq_sec3_FK}
is not fulfilled.

\subsection{The equator}\label{ss_equator}

A vector subspace $L$ of $\RR^d$ is said to be $\nu$-invariant if $gL=L$ for $\nu$-almost every $g$ or,
equivalently, for every $g\in\supp\nu$. Observe that if $L$ is a $\nu$-invariant subspace then
$$
\lim_n \frac 1n \log \|g_{n-1}\cdots g_0\mid L\|
$$
exists for every $\ug$ in a full $\nu^\NN$-measure subset of $G^\NN$ and is constant on
this subset. This direct consequence of Kingman's subadditive ergodic theorem
(see~\cite[Theorem~3.3]{LLE}), together with the fact that the Bernoulli shift $(\sigma,\nu^\NN)$
is ergodic, will be used repeatedly. It is part of the proof of the Oseledets theorem
(see~\cite[Proposition~4.11]{LLE}) that for $L=\RR^d$ the limit coincides with the largest Lyapunov
$\lambda_1(\nu)$.

For any ergodic $\nu$-stationary measure $\eta$, define $L(\eta)$ to be
the smallest vector subspace such that $\eta(L(\eta))=1$. Equivalently, $L(\eta)$ is the vector subspace
spanned by the support of $\eta$.
Then $L(\eta)$ is $\nu$-invariant:
$$
\int_{G \times P} \chi_{L(\eta)}(gv) \, d\nu(g) d\eta(v)
= \int_P \chi_{L(\eta)}(v) \, d\eta(v) = \eta(L) =1
$$
and this implies
$$
\eta(g^{-1}L(\eta)) = \int_P \chi_{L(\eta)}(gv) \, d\eta(v) =1 \text{ for $\nu$-almost every $g$.}
$$
So, it follows from the definition that $g^{-1}L(\eta)=L(\eta)$ for $\nu$-almost every $g$.

Since the support of $\eta$ spans $L(\eta)$, it follows from \eqref{eq_sec3_Birkhoff_exponent} that we may find a basis
$v_1, \dots, v_l$ of $L(\eta)$ and a full $\nu^\NN$-measure subset of $\ug$ such that
$$
\lim_n \frac 1n \log \frac{\|g_{n-1}\cdots g_0 v_i\|}{\|v_i\|}
= \alpha(\eta)
\text{ for $i=1, \dots, l$.}
$$
Then
\begin{equation}\label{eq_sec3_FK1}
\lim_n \frac 1n \log \|g_{n-1}\cdots g_0 \mid L(\eta)\| = \alpha(\eta).
\end{equation}
It also follows that if $\eta$ is such that $\alpha(\eta) < \lambda_1(\nu)$ then $L(\eta)$ is a proper subspace.

We call the \emph{equator} of $\nu$ a maximal $\nu$-invariant subspace $E$ such that
\begin{equation}\label{eq_sec3_FK2}
\lim_n \frac 1n \log \|g_{n-1} \cdots g_0 \mid E\| <\lambda_1(\nu).
\end{equation}
Such a subspace is necessarily proper,  but it may not exist. It follows from the previous paragraph that the
equator does exist if $\alpha(\eta) < \lambda_1(\nu)$ for some (ergodic) $\nu$-stationary measure;
the converse is also true.
Moreover, the equator is unique and contains $L(\eta)$ for every ergodic $\nu$-stationary measure $\eta$:
both claims follow immediately from the observation that if two subspaces satisfy \eqref{eq_sec3_FK2} then
so does their sum.

If the equator does exist, every $g$ in (the group generated by) the support of $\nu$ may be written as
\begin{equation}\label{eq_sec3_triangular_matrix}
g = \left(\begin{array}{cc} g^E & h \\ 0 & g^\perp\end{array}\right)
\text{ with $g^E\in \GL(E)$ and $g^\perp \in \GL(E^\perp)$. }
\end{equation}
Let $\nu^E$ and $\nu^\perp$ be the push-forwards of $\nu$ under the maps $g \mapsto g^E$ and
$g\mapsto g^\perp$.
Furstenberg and Kifer~\cite[Lemma~3.6]{FK83} observed that
$$
\lambda_1(\nu) = \max\{\lambda_1(\nu^E), \lambda_1(\nu^\perp)\big\}.
$$
The property \eqref{eq_sec3_FK2} means that $\lambda_1(\nu^E)  < \lambda_1(\nu)$. Hence,
$\lambda_1(\nu^\perp) = \lambda_1(\nu)$.

\begin{remark}\label{r_sec3_perp}
Any $g\in G$ that preserves $E$ may written in the form \eqref{eq_sec3_triangular_matrix}.
Then $(gu)^\perp=g^\perp u^\perp$ for every $u$, where $\perp$ denotes the component
orthogonal to $E$. Do not mistake this for $gu^\perp=hu^\perp + g^\perp u^\perp$.
It is equally clear that $\|g^\perp\| \le \|g\|$.
These simple facts will be used several times.
\end{remark}

Suppose that there exists some ergodic $\nu^\perp$-stationary measure $\eta^\perp$ such that
$$
\alpha^\perp(\eta^\perp) = \int_{G\times P} \Phi \, d(\nu^\perp \times \eta^\perp)
$$
is strictly less than $\lambda_1(\nu^\perp)=\lambda_1(\nu)$. Then $L(\eta^\perp)$ is a proper
$\nu^\perp$-invariant subspace of $E^\perp$ and, according to \eqref{eq_sec3_FK1},
$$
\lim_n \frac 1n \log \|g^\perp_{n-1}\cdots g_0^\perp \mid L(\eta^\perp)\| = \alpha^\perp(\eta^\perp).
$$
Then $E'=E\oplus L(\eta^\perp)$ is a proper $\nu$-invariant subspace of $\RR^d$ and it satisfies the
equator property \eqref{eq_sec3_FK2}, because
$$
\begin{aligned}
\lim_n \frac 1n & \log \|g_{n-1} \cdots g_0 \mid E\oplus L(\eta^\perp)\| \\
& \le \max \big\{ \lim_n \frac 1n \log \|g_{n-1} \cdots g_0 \mid E\|,
                      \lim_n \frac 1n \log \|g_{n-1} \cdots g_0 \mid L(\eta^\perp)\|\big\}
\end{aligned}
$$
Since $E'$ contains $E$ strictly, this contradicts the definition of the equator. This contradiction
proves that $\alpha^\perp(\eta^\perp) = \lambda_1(\nu)$ for every $\nu^\perp$-stationary measure $\eta^\perp$.

So, by \eqref{eq_sec3_FK}, for every $v^\perp \in \proj(E^\perp)$ there exists a full
$\nu^\NN$-measure subset of $\ug$ for which
\begin{equation}\label{eq_sec3_FK3}
\lim_n \frac 1n \log \frac{\|g_{n-1}^\perp \cdots g_0^\perp v^\perp\|}{\|v^\perp\|} = \lambda_1(\nu).
\end{equation}
This also implies (see~\cite[Proposition~4.14]{LLE}) that for every $v \in\RR^d \setminus E$ there
exists a full $\nu^\NN$-measure subset of $\ug$ for which
\begin{equation}\label{eq_sec3_FK4}
\lim_n \frac 1n \log \|g_{n-1} \cdots g_0 v\|=
\lim_n \frac 1n \log \frac{\|g_{n-1} \cdots g_0 v\|}{\|v\|} = \lambda_1(\nu).
\end{equation}

The following example shows that the equator need not be any of the subspaces $V^i$ in the
Oseledets flag \eqref{eq_sec2_filtration}:

\begin{example}
Let $d=3$ and $m=2$ and $\nu=p_1\delta_{A_1}+p_2\delta_{A_2}$ where the matrices $A_1$ and $A_2$
are given by
\begin{align*}
& A_1 = \left(\begin{array}{cc}B & 0 \\ 0 & 1\end{array}\right)
\text{ with }
B=\left(\begin{array}{cc}\sigma& 0 \\ 0 & \sigma^{-1}\end{array}\right)
\text{ and $\sigma>1$} \\
& A_2=\left(\begin{array}{cc}R_\theta B & 0 \\ 0 & 1\end{array}\right)
\text{ with }
R_\theta=\left(\begin{array}{cc}\cos\theta& \sin\theta \\ -\sin\theta & \cos\theta\end{array}\right)
\text{ and $\theta\neq 0$ small.}
\end{align*}
The subspace $E_3=\{(0,0)\}\times\RR$ is $\nu$-invariant and corresponds to a zero Lyapunov exponent.
The only other $\nu$-invariant subspace is $E_{12}=\RR^2\times\{0\}$. Given any $\vep>0$, the cone
$$
C^u=\{(x,y,0)\in\RR^3: |y| \le \vep |x|\}\subset E_{12}
$$
is forward invariant under both $A_1$ and $A_2$, as long as $\theta$ is close enough to zero.
This implies that some Lyapunov exponent is close to $\log\sigma$. Analogously, the cone
$$
C^s=\{(x,y,0)\in\RR^3: |x| \le \vep |y|\}\subset E_{12}
$$
is backward invariant under both $A_1$ and $A_2$, which implies that some Lyapunov exponent is close
to $-\log\sigma$. So, the Lyapunov exponents are
$$
\lambda_1(\nu) \approx \log\sigma
\quand
\lambda_2(\nu)=0
\quand
\lambda_3(\nu) \approx - \log\sigma.
$$
Thus, $E_3$ is the equator of $\nu$ and corresponds to the middle eigenvalue $\lambda_2(\nu)=0$.
Note that both $A_1^\perp = B$ and $A_2^\perp=R_\theta B$ have determinant $1$.
\end{example}

\section{Uniform convergence in measure}\label{s_uniform_estimate}

We need to prove that the limit in \eqref{eq_sec3_FK3} is \emph{uniform in measure} with respect to
$v\in\proj(E^\perp)$. This follows directly from a corresponding fact for the limit in \eqref{eq_sec3_FK},
that we state precisely as follows:

\begin{theorem}\label{t_sec4_uniform_estimate}
Assume that $\alpha(\eta)=\lambda_1(\nu)$ for every $\nu$-stationary measure $\eta$.
Then for any $\vep>0$ there exist constants $C=C(\nu,\vep)>0$ and $c=c(\nu,\vep)>0$ such that
for any $v\in P$ and $N\in\NN$, there exists a measurable set $\cE=\cE(\nu,\vep,v,N)\subset G^\NN$
satisfying:
\begin{enumerate}
\item $\nu^\NN(\cE^c) \le C e^{-cN}$ and
\item
$\displaystyle{\frac 1n \log \frac{\|g_{n-1}\cdots g_0v\|}{\|v\|} \in (\lambda_1(\nu)-\vep,\lambda_1(\nu)+\vep)}$
for every $\ug \in \cE$ and $n\ge N$.
\end{enumerate}
\end{theorem}

In what follows we prove Theorem~\ref{t_sec4_uniform_estimate}.
The proof will not be used in the rest of the paper, so the reader may choose to skip the remainder
of this section at first reading.

Recall that $S=\supp\nu$ is taken to be compact.
Let $C(S\times P)$ be the Banach space of continuous functions $\psi:S \times P \to \RR$ with the norm
$\|\psi\|=\sup|\psi|$. Consider the operator $\cQ_\nu$ defined on $C(S\times P)$ by
$$
\cQ_\nu\psi(g',v) = \int_S \psi(g,gv) \, d\nu(g).
$$
Note that $\cQ_\nu\psi(g',v)$ does not depend on $g'$. The dual operator $\cQ_\nu^*$ acts in the space of finite signed measures
on $S\times P$ by
$$
\int_{S\times P} \psi \, d\cQ_\nu^*\lambda = \int_{S\times P} \cQ_\nu\psi \, d\lambda
\text{ for every $\psi\in C(S\times P)$.}
$$
A signed measure $\lambda$ is said to be \emph{$\cQ_\nu$-invariant} if $\cQ_\nu^*\lambda=\lambda$. Given
$\varphi\in C(S\times P)$, denote
$$
\Sigma(\nu,\varphi)=\sup \left\{\int_{S \times P} \varphi \, d\lambda:
\lambda \text{ is a $\cQ_\nu$-invariant probability measure}\right\}.
$$

\begin{lemma}\label{l_sec4_cohomology}
Given any non-negative $\varphi \in C(S \times P)$ and any $\vep>0$ there exist
$\psi, \xi \in C(S \times P)$ such that $\varphi=\cQ_\nu\psi-\psi+\xi$ and
$\|\xi\| \le \Sigma(\nu,\varphi) +\vep$.
\end{lemma}

\begin{proof}
Let $\cW\subset C(S\times P)$ be the subspace of functions of the form $\varphi=\cQ_\nu\psi-\psi$ for some $\psi\in C(S\times P)$.
By Hahn--Banach, given any $\varphi\in C(S\times P)$ there exists a continuous linear functional
$L:C(S\times P) \to \RR$ such that $\|L\|=1$ and
$L \mid \cW \equiv 0$ and $L(\varphi) = d(\varphi,\cW)$.
By Riez--Markov, there exists some signed measure $\lambda$ on
$S \times P$ such that $\|\lambda\|=1$ and $L(\psi) = \int_{S\times P} \psi \, d\lambda$ for every
$\psi\in C(S\times P)$. The fact that $L$ vanishes on $W$ ensures that $\cQ_\nu^*\lambda=\lambda$:
$$
\int_{S\times P} \left(\cQ_\nu\psi-\psi\right) \, d\lambda = L(\cQ_\nu\psi-\psi) = 0
\text{ for every $\psi\in C(S\times P)$.}
$$
Let $\lambda=\lambda^+-\lambda^-$ be the Hahn decomposition of $\lambda$. Then
$\|\lambda^+\|+\|\lambda^-\|=\|\lambda\|=1$ and
$\lambda=\cQ_\nu^*\lambda=\cQ_\nu^*\lambda^+ - \cQ_\nu^*\lambda^-$.
Since $\cQ_\nu^*\lambda^\pm$ are non-negative measures,
the latter implies that $\cQ_\nu^*\lambda^\pm \ge \lambda^\pm$.
Since $\cQ_\nu 1 = 1$ we have
$$
\|\cQ_\nu^*\lambda^\pm\| = \int_{S\times P} 1 \, d\cQ_\nu^*\lambda^\pm
= \int_{S\times P} \cQ_\nu 1 \, d\lambda^\pm = \int_{S\times P} 1 \, d\lambda^\pm =\|\lambda^\pm\|.
$$
Hence, $\cQ_\nu^*\lambda^{\pm} = \lambda^{\pm}$. Let $\lambda_0=\lambda^+/\|\lambda^+\|$.
Then $\cQ_\nu^*\lambda_0=\lambda_0$ and $\lambda_0 \ge \lambda^+ \ge \lambda$.
Since $\varphi$ is assumed to be non-negative, it follows that
\begin{equation}\label{eq_sec4_uniform_estimate}
\int_{S\times P} \varphi \, d\lambda_0 \ge \int_{S\times P} \varphi \, d\lambda = L(\varphi) = d(\varphi,\cW)
\end{equation}
Take $\psi \in C(S \times P)$ such that $\|\varphi - (\cQ_\nu\psi - \psi)\| \le d(\varphi,\cW) + \vep$  and then define
$\xi = \varphi - (\cQ_\nu\psi-\psi)$. Then, by \eqref{eq_sec4_uniform_estimate},
$$
\|\xi\| \le d(\varphi,W) + \vep \le  \int_{S\times P} \varphi \, d\lambda_0 + \vep
$$
and this implies that $\|\xi\| \le \Sigma(\nu,\varphi) +\vep$.
\end{proof}

For $n\ge 0$, let $\cB_n$ be the $\sigma$-algebra of $S^\NN$ generated by the family of cylinders
$[0;A_0, \dots, A_n]$, where the $A_j$ are measurable subsets of $S$.
For a measurable function $X:S^\NN\to\RR$, let $\EE(X)=\int_{G^\NN} X \, d\nu^\NN$ and
$\EE(X \mid \cB_n)$ denote the \emph{expectation} of $X$ conditioned to $\cB_n$, that is,
the (essentially unique) $\cB_n$-measurable function such that
$$
\int_{G^\NN} \EE(X \mid \cB_n) \, d\nu^\NN = \int_{G^\NN} X \, d\nu^\NN.
$$

The next lemma is a particular instance of the Azuma--Hoeffding inequality (Azuma~\cite{Azu67},
Hoeffding~\cite{Hoe63}) for sums of bounded random variables:

\begin{lemma}\label{l_sec3_Azuma-Hoeffding}
Let $Y_n: G^\NN \to \RR$, $n\in\NN$ be such that $A = \sup_n \|Y_n\|$ is finite
and
\begin{equation}\label{eq_sec4_Azuma-Hoeffding}
\EE(Y_{i_1} \cdots Y_{i_k})=0
\text{ for every $1 \le i_1 < \cdots < i_k$.}
\end{equation}
Then, for any $s>0$ and $n\in\NN$,
$$
\nu^\NN\left(\{\ug:|Y_1 + \cdots + Y_n|(\ug) \ge s\}\right) \le 2\exp\left(-\frac{s^2}{2nA^2}\right).
$$
\end{lemma}

\begin{proof}
Since the exponential function is convex,
$$
\begin{aligned}
e^{ax}
= \exp\left(a\left(\frac{1+x}{2}\right) - a \left(\frac{1-x}{2}\right)\right)
\le \frac{1+x}{2}e^a + \frac{1-x}{2}e^{-a}
= \cosh a + x \sinh a
\end{aligned}
$$
for every $x\in [-1,1]$ and $a\in\RR$. Taking $x=Y_i/\|Y_i\|$ and $a=t\|Y_i\|$, we get that
$$
e^{tY_i} % = e^{aX} \le \cosh a + x \sinh a  =
\le \cosh (t\|Y_i\|) + \frac{Y_i}{\|Y_i\|}\sinh(t\|Y_i\|)
$$
for any $t\in\RR$. The hypothesis \eqref{eq_sec4_Azuma-Hoeffding} implies that
$$
\EE\left(\prod_{i=1}^n (a_i + b_i Y_i)\right)= \prod_{i=1}^n a_i
$$
for  any real numbers $a_1, \dots, a_n$ and $b_1, \dots, b_n$. Hence,
$$
\begin{aligned}
\EE\left(e^{t\sum_{i=1}^n Y_i}\right)
& %= \EE\left(\prod_{i=1}^n e^{t Y_i}\right)
\le \EE\left(\prod_{i=1}^n \cosh (t\|Y_i\|) + \frac{Y_i}{\|Y_i\|}\sinh(t\|Y_i\|)\right)
= \prod_{i=1}^n \cosh (t\|Y_i\|).
\end{aligned}
$$
Using the fact that, for every $x\in\RR$,
$$
\cosh x
= \sum_{k=0}^\infty \frac{x^{2k}}{(2k)!}
\le \sum_{k=0}^\infty \frac{x^{2k}}{2^k k!}
= e^{x^2/2}
$$
we conclude that
$$
\EE\left(e^{t\sum_{i=1}^n Y_i}\right)
\le \prod_{i=1}^n \exp\left(\frac{t^2}{2} \|Y_i||^2\right)
= \exp\left(\frac{t^2}{2}\sum_{i=1}^n \|Y_i\|^2\right)
$$
for any $t\in\RR$. Then, by the Chebyshev inequality,
$$
\begin{aligned}
\nu^\NN\left(\{\ug:\sum_{j=1}^n Y_j (\ug)\ge s\}\right)
& = \nu^\NN\left(\{\ug: e^{t \sum_{j=1}^n Y_j(\ug)} \ge e^{ts} \}\right)\\
& \le e^{-ts} \EE\left(e^{t\sum_{i=1}^nY_j}\right)
\le \exp\left(-ts + \frac{t^2}{2}\sum_{i=1}^n \|Y_i\|^2\right)
\end{aligned}
$$
for any $t \in \RR$.  Taking $t=s/\sum_{i=1}^n \|Y_i\|^2$, we conclude that
$$
\nu^\NN\left(\{\ug:\sum_{i=1}^n Y_j(\ug) \ge s\}\right)
\le \exp\left(-\frac{s^2}{2\sum_{i=1}^n \|Y_i\|^2}\right)
\le \exp\left(-\frac{s^2}{2nA^2}\right)
$$
Analogously, replacing each $Y_j$ with $-Y_j$,
$$
\nu^\NN\left(\{\ug:\sum_{i=1}^n Y_j (\ug) \le - s\}\right)
%\le \exp\left(-\frac{s^2}{2\sum_{i=1}^n \sum_i=1}^n \|\|Y_i\|^2}\right),
\le \exp\left(-\frac{s^2}{2nA^2}\right).
$$
Adding these two inequalities, we get the conclusion of the lemma.
\end{proof}

\begin{lemma}\label{l_sec4_FK1}
For any $\psi\in C(S\times P)$ and $\vep>0$ there exist $C_1=C_1(\psi,\vep)>0$ and $c_1=c_1(\psi,\vep)>0$
such that for any $v\in\RR^d$ and $N\in\NN$ there exists a measurable set $\cE_1=\cE_1(\nu,\psi,\vep,v,N)$
satisfying:
\begin{enumerate}
\item $\nu^\NN(\cE_1^c) \le C_1 e^{-c_1N}$ and
\item $\displaystyle{\frac 1n \sum_{j=0}^{n-1} \cQ_\nu\psi(g_{j-1}, g_{j-1}\cdots g_0v)-\psi(g_j,g_j\cdots g_0v) \in (-\vep,\vep)}$
      for any $\ug\in\cE_1$ and $n\ge N$.
\end{enumerate}
\end{lemma}

\begin{proof}
We may suppose that $\|\psi\|>0$ for otherwise the statement is trivial. Given $v\in\RR^d$ and $n\in\NN$, define
$$
Y_n(\ug) = \cQ_\nu\psi(g_{n-1}, g_{n-1}\cdots g_0v)-\psi(g_n,g_n\cdots g_0v).
$$
Clearly, $\|Y_n\| \le 2 \|\psi\|$. Moreover, $Y_n$ is $\cB_n$-measurable and, by the definition of $\cQ_\nu$,
$$
\EE(Y_n\mid \cB_{n-1}) = \cQ_\nu\psi(g_{n-1}, g_{n-1}\cdots g_0v)-\int_G\psi(g,gg_{n-1}\cdots g_0v)\, d\nu(g) \equiv 0.
$$
Thus, given any $k$-uple $(i_1, \dots, i_{k-1},i_k)$ with $i_k>i_j$ for every $j=1, \dots, k-1$,
$$
\EE\left(Y_{i_1}\cdots Y_{i_{k-1}}Y_{i_k}\right) =
\EE\left(Y_{i_1}\cdots Y_{i_{k-1}}\EE(Y_{i_k}\mid B_{i_k-1})\right)=0.
$$
So, $(Y_n)_n$ satisfies the hypotheses of Lemma~\ref{l_sec4_FK1}. Thus,
$$
\nu^\NN\left(\{\ug\in G^n: \frac 1n |Y_0 (\ug) + \cdots + Y_{n-1}(\ug)| \ge \vep\}\right)
\le 2 \exp\left(-\frac{n\vep^2}{8\|\psi\|^2}\right)
$$
for every $\vep>0$ and every $n\in\NN$. The conclusion follows by letting $\cE_1$ be the set of
all $\ug$ such that the claim in part (2) holds for every $n\ge N$ and taking $c_1=\vep^2/(8\|\psi\|^2)$ and
$C_1=2\sum_{j=0}^\infty e^{-cj}$.
\end{proof}

\begin{lemma}\label{l_sec4_FK2}
For every $\varphi\in C(S\times P)$ and $\vep>0$ there exist $C_2=C_2(\nu, \varphi,\vep)>0$ and
$c_2=c_2(\nu, \varphi,\vep)>0$ such that, for every $v\in P$ and $N\in\NN$ there exists a measurable set
$\cE_2=\cE_2(\nu, \varphi,\vep,v,N)\subset G^\NN$ satisfying:
\begin{enumerate}
\item $\nu^\NN(\cE_2^c) \le C_2 e^{-c_2N}$ and
\item$\displaystyle{\frac 1n \sum_{j=0}^{n-1} \varphi(g_j,g_j\dots g_0v) < \Sigma(\nu,\varphi)+\vep}$ for every $\ug\in\cE_2$ and $n\ge N$.
\end{enumerate}
\end{lemma}

\begin{proof}
Neither the hypothesis nor the conclusion are affected if one replaces $\varphi$ with $\varphi+\const$
(clearly, $\Sigma_{\varphi+\const}=\Sigma(\nu,\varphi)+\const$). Thus, it is no restriction to suppose that
$\varphi$ is non-negative. Then, by Lemma~\ref{l_sec4_cohomology}, for any $\vep>0$ we may find
$\psi, \xi \in C(S\times P)$ such that $\varphi=\cQ_\nu\psi-\psi + \xi$ and $\|\xi\| < \Sigma(\nu,\varphi)+\vep/4$.
Then
$$
\begin{aligned}
\frac 1n \sum_{j=0}^{n-1} \varphi(g_j,g_j\dots g_0v)
& = \frac 1n \sum_{j=1}^n \cQ_\nu\psi(g_{j-1},g_{j-1}\dots g_0v) - \psi(g_j,g_j\dots g_0v)\\
& + \frac 1n \sum_{j=0}^{n-1} \xi(g_j,g_j\dots g_0v) + \frac 1n \psi(g_n,g_n\cdots g_0v) - \frac 1n \psi(g_0,g_0v).
\end{aligned}
$$
Let $C_1$, $c_1$ and $\cE_1$ be as in Lemma~\ref{l_sec3_Azuma-Hoeffding}, with $\vep$ replaced with $\vep/4$.
The  first term in the previous sum is less than $\vep/4$ for every $\ug\in\cE_1$.
The second term is bounded above by $\|\xi\| < \Sigma(\nu,\varphi) + \vep/4$.
Let $L=[\|\psi\|/(4\vep)]$.  The third and fourth terms are less than $\vep/4$ if $n > L$.
Thus, for every $\ug\in\cE_1$ and $n\ge N > L$,
$$
\frac 1n \sum_{j=0}^{n-1} \varphi(g_j,g_j\dots g_0v) < \Sigma(\nu,\varphi)+\vep.
$$
Take $c_2=c_1$ and $C_2=\max\{C_1, e^{c_2L}\}$.
Define $\cE_2 = \cE_1$ if $N > L$ and $\cE_2 = G^\NN$ otherwise.
\end{proof}

\begin{corollary}\label{c_sec4_FK}
Let $\varphi\in C(S\times P)$ be such that $\int_{S\times P} \varphi \, d\lambda=\Sigma(\nu,\varphi)$
for every $\cQ_\nu$-invariant probability measure $\lambda$. For $\vep>0$ there exist $C_3=C_3(\nu, \varphi,\vep)>0$
and $c_3=c_3(\nu, \varphi,\vep)>0$ such that, for every $v\in P$ and $N\in\NN$
there exists a measurable set $\cE_3=\cE_3(\nu, \varphi,\vep,v,N)\subset G^\NN$ satisfying:
\begin{enumerate}
\item $\nu^\NN(\cE_3^c) \le C_3 e^{-c_2N}$ and
\item $\displaystyle{\frac 1n \sum_{j=0}^{n-1} \varphi(g_j,g_j\dots g_0v) \in(\Sigma(\nu,\varphi)-\vep, \Sigma(\nu,\varphi)+\vep)}$
      for $\ug\in\cE_3$ and $n\ge N$.
\end{enumerate}
\end{corollary}

\begin{proof}
The assumption implies that $\Sigma(\nu,-\varphi)=-\Sigma(\nu,\varphi)$. So, applying Lemma~\ref{l_sec4_FK2}  also to
the function $-\varphi$ we get $C_2(-\varphi,\vep)$ and $c_2(-\varphi,\vep)$ and
$\cE_2(-\varphi,\vep,v,N)$ such that
$$
\frac 1n \sum_{j=0}^{n-1} \varphi(g_j,g_j\dots g_0v) > \Sigma(\nu,\varphi)-\vep.
$$
for any $\ug\in\cE_2(-\varphi,\vep,v,N)$ and $n\ge N$.
To conclude, take $C_3$ to be the sum of $C_2(\nu,\pm\varphi,\vep)$ and $c_3$ to be the minimum of
$c_2(\nu,\pm\varphi,\vep)$ and $\cE_3$ to be the intersection of $\cE_2(\nu,\pm\varphi,\vep,v,N)$.
\end{proof}

We also need to interpret $\Sigma(\nu,\varphi)$ in terms of the $\nu$-stationary measures. That is the purpose
of the next lemma:

\begin{lemma}\label{l_sec4_FK3}
Let $\iota:S \times P \to S \times P$, $\iota(g,v) = (g,gv)$. Then $\eta \mapsto \iota_*(\nu\times\eta)$ maps
the set of $\nu$-stationary measures bijectively to the set of $\cQ_\nu$-invariant probability measures.
Its inverse is the push-forward $\pi_*$ of the canonical projection $\pi:S \times P \to P$.
\end{lemma}

\begin{proof}
Given any probability measure $\lambda$ with $\cQ_\nu^*\lambda=\lambda$, let $\eta$ be its projection to $P$.
For any bounded measurable function $\varphi:S \times P \to \RR$,
$$
\begin{aligned}
\int_{S \times P} \varphi(g',v) & \, d\lambda(g',v)
= \int_{S \times P} \cQ_\nu\varphi (g',v) \, d\lambda(g',v)\\
& = \int_{S \times S \times P} \varphi(g,gv) \,d\nu(g) \, d\lambda(g',v)
= \int_{S \times P} \varphi(g,g v) \,d\nu(g)\,d\eta(v)
\end{aligned}
$$
(because the integrand does not depend on $g'$). Moreover, the special case when $\varphi$ does not depend on $g$
means that
$$
\begin{aligned}
\int_P \phi(v) \, d\eta(v)
& = \int_{S \times P} \phi(v) \, d\lambda(g', v)
= \int_{S \times P} \cQ_\nu\phi(g',v) \, d\lambda(g', v)\\
& = \int_{S \times S \times P} \phi(gv) \, d\nu(g) \, d\lambda(g',v)
= \int_{S \times P} \phi(gv) \, d\nu(g) \, d\eta(v)
\end{aligned}
$$
for any bounded measurable function $\phi:P\to\RR$. In other words, $\eta$ is $\nu$-stationary.

Conversely, given any $\nu$-stationary measure $\eta$, define $\lambda$ on $S \times P$ by
$$
\int_{S \times P} \varphi(g',v) \, d\lambda(g',v)
= \int_{S \times P} \varphi(g,g v) \,d\nu(g)\,d\eta(v).
$$
for any bounded measurable function $\varphi:S \times P \to \RR$. Then
$$
\begin{aligned}
\int_{S \times P} \cQ_\nu\varphi(g',v) \, d\lambda(g',v)
& = \int_{S \times P} \cQ_\nu\varphi(g',g'v) \, d\nu(g') \, d\eta(v)\\
& = \int_{S \times S \times P} \varphi(g,gg'v) \, d\nu(g) \, d\nu(g') \, d\eta(v).
\end{aligned}
$$
Since $\eta$ is $\nu$-stationary, the right-hand side may be rewritten as
$$
\int_{S \times P} \varphi(g,gv) \, d\nu(g) \, d\eta(v)
= \int_{S \times P} \varphi(g,v) \, d\lambda(g,v).
$$
Combining these two identities, one sees that $\cQ_\nu^*\lambda=\lambda$.
\end{proof}

\begin{proof}[Proof of Theorem~\ref{t_sec4_uniform_estimate}]
Consider $\varphi=\Phi\circ\iota^{-1}$, that is,
$$
\varphi:S \times P \to\RR, \quad \varphi(g,v) = \log\frac{\|v\|}{\|g^{-1}v\|}.
$$
Lemma~\ref{l_sec4_FK3} implies that
$$
\begin{aligned}
\Sigma(\nu,\varphi)
& = \sup \big\{\int_{S \times P} \varphi \, d\lambda: \cQ_\nu^*\lambda=\lambda\big\}\\
%= \sup \big\{\int_{S \times P} \varphi \circ \iota \, d(\nu\times\eta): \text{ $\eta$ is a $\nu$-stationary measure}\big\}
& = \sup \big\{\int_{S \times P} \Phi \, d(\nu\times\eta): \text{ $\eta$ is a $\nu$-stationary measure}\big\}
= \lambda_1(\nu)
\end{aligned}
$$
and the hypothesis of Corollary~\ref{c_sec4_FK} corresponds precisely to the hypothesis of Theorem~\ref{t_sec4_uniform_estimate}.
Take $C=C_3(\nu, \varphi,\vep)$ and $c=c_3(\nu, \varphi,\vep)$ and
$\cE=\cE_3(\nu, \varphi,\vep,v,N)$. Then $\nu^\NN(\cE^c) \le C e^{-cn}$ and for any $\ug\in\cE$,
$$
\begin{aligned}
\frac 1n \log \frac{\|g_{n-1}\cdots g_0 v\|}{\|v\|}
& = \frac 1n \sum_{j=0}^{n-1} \Phi\left(g_j, g_{j-1} \cdots g_0 v\right)
= \frac 1n \sum_{j=0}^{n-1} \varphi\left(g_j, g_j g_{j-1} \cdots g_0 v\right) \\
& \in (\Sigma(\nu,\varphi)-\vep,\Sigma(\nu,\varphi)+\vep)
= (\lambda_1(\nu)-\vep,\lambda_1(\nu)+\vep).
\end{aligned}
$$
This completes the proof of the theorem.
\end{proof}

\part{Preliminaries and outline of the proof}\label{p_two}

\section{The equator is a repeller}\label{s_random_repeller}

Let $(\nu_k)_k$ be a sequence converging to some $\nu_\infty$ in the space $\cMc$.
By \eqref{eq_sec3_Furstenberg_formula},
for every $k$ there exists a $\nu_k$-stationary measure $\eta_k$ on $P$ such that
$$
\lambda_1(\nu_k) = \int_{G \times P} \Phi \, d(\nu_k\times\eta_k).
$$
Since the space of probability measures on $P$ is weak$^*$-compact,
to prove that $(\lambda_1(\nu_k))$ converges to $\lambda_1(\nu_\infty)$ it is no restriction
to suppose that the sequence $(\eta_k)_k$ converges to some probability measure $\eta_\infty$.
Then
$$
\lambda_1(\nu_k) \to \int_{G \times P} \Phi \, d(\nu_\infty\times\eta_\infty)
\text{ when $k\to\infty$.}
$$
The measure $\eta_\infty$ is necessarily $\nu_\infty$-stationary (see \cite[Proposition~5.9]{LLE}).

Now there are two alternatives. If $\int_{G \times P} \Phi \, d(\nu_\infty\times\eta_\infty)=\lambda_1(\nu_\infty)$
then
$$
\lambda_1(\nu_k) \to \lambda_1(\nu_\infty),
$$
as we wanted to prove. Otherwise, we are in the setting of Section~\ref{s_uniform_estimate}:
there is a proper subspace $E$ of $\RR^d$ such that
\begin{itemize}
\item[(i)] $E$ is $\nu_\infty$-invariant and
$$
\lim_n \frac 1n \log \|g_{n-1}\cdots g_0 \mid E\| < \lambda_1(\nu_\infty).
$$
\item[(ii)] For any $\vep>0$ and $\delta>0$ there exists $N=N(\nu_\infty,\vep, \delta)\in\NN$ and for every
$v^\perp\in\proj(E^\perp)$ there exists $\cE=\cE(\nu_\infty,\vep,\delta,v^\perp)$ such that
$\nu_\infty^\NN(\cE^c) < \delta$ and
$$
\frac 1n \log \frac{\|g_{n-1}^\perp \cdots g_0^\perp v^\perp\|}{\|v^\perp\|} \in (\lambda_1(\nu_\infty)-\vep,\lambda_1(\nu_\infty)+\vep)
$$
for any $\ug\in\cE$ and $n \ge N$.
\item[(iii)] $\eta_\infty(E)>0$.
\end{itemize}

We are going to see that the properties (i) - (iii) are incompatible with the fact that $\eta_\infty$
is the limit of stationary measures for nearby random walks.
Indeed, if $E$ satisfies (i) and (ii) then it is a kind of repeller for the random walk on $P$ associated
to $\nu_\infty$ (the precise statements are in Section~\ref{ss_repelling_behavior}).
As we are going to see, that implies that the $\eta_k$-measure of any neighborhood of $E$
is small when $k$ is large, so that the limit $\eta_\infty$ cannot satisfy (iii);
the proof will use some general tools that we introduce in
Sections~\ref{ss_couplings} through~\ref{ss_Margulis_functions}.
Hence the second alternative above cannot actually occur,
and thus Theorem~\ref{theorem:second} will follow.

\subsection{Generic measures}\label{ss_generic_measures}

Actually, it suffices to carry these arguments in a special case, as we are going to explain.
For each $1 \le r \le d$, let $\grass(r,d)$ be the Grassmannian manifold of $r$-dimensional
subspaces of $\RR^d$. Moreover, let $\cF(r,d)$ be the space of flags
$$
F_1 \subset F_2 \subset \cdots \subset F_{r-1} \subset F_r \subset \RR^d,
$$
where each $F_i$ has dimension $i$.
Note that $\grass(1,d)=\cF(1,d)$ coincides with the projective space $P$.

The natural action of $G$ on the projective space extends to group actions on every $\grass(r,d)$ and $\cF(r,d)$.
Thus, in particular, to each probability measure $\nu$ on $G$ we may associate operators $\cP_\nu$ and $\cP_\nu^*$
acting, respectively, on bounded measurable functions and on measures of $\grass(r,d)$ or $\cF(r,d)$,
just as we did for $P$ in \eqref{eq_sec3_Markov_operator1} and \eqref{eq_sec3_Markov_operator2}:
\begin{equation}\label{eq_sec5_Markov_operator12}
\cP_\nu\psi(v) = \int_G \psi(gv) \, d\nu(g)
\text{ and }
\cP_\nu^*\eta = \int_G \left(g_*\eta\right) \, d\nu(g).
\end{equation}
We continue to say that a function $\psi$ is \emph{$\nu$-stationary} if $\cP_\nu\psi=\psi$
and a probability measure $\zeta$ is \emph{$\nu$-stationary} if $\cP_\nu^*\zeta=\zeta$.

A subset of an algebraic variety $X$  is \emph{Zariski-dense} if it is not contained in any proper
algebraic subvariety of $X$. The cases we are interested in are $X = \grass(r,d)$, $\cF(r,d)$, or $G$.
We call a measure $\eta$ on the algebraic variety \emph{generic} if $\eta(M)=0$ for any proper algebraic subvariety $M$. Then, in particular, $\eta$ is \emph{non-atomic},
meaning that $\eta(\{p\}) = 0$ for every point $p$ in the domain.

\begin{remark}\label{r_sec5_generic_elementary}
The restrictions $\eta \mid U$ of generic measures are generic, and so are the products
$\eta_1 \times \eta_2$ of generic measures. Moreover, if $\{\eta_t: t \in T\}$ is a family of
generic measures, and $\xi$ is a measure on $T$ then $\eta = \int_T \eta_t \, d\xi(t)$ is
a generic measure. In particular, the generic measures form a vector subspace.
\end{remark}

\begin{proposition}\label{p_sec5_GM}
Let $X = \grass(r,d)$ or $X=\cF(r,d)$. If $\nu$ is a probability measure whose support is Zariski-dense
in $G$ then $\nu$ admits a unique stationary measure $\eta$ on $X$, and this measure is generic.
\end{proposition}

\begin{proof}
Lemmas~4.2 and~4.5 in Gol'dsheid, Margulis~\cite{GoM89}.
\end{proof}

\begin{proposition}\label{p_sec5_Zariski1}
Every $\nu\in\cMc$ is approximated by generic probability measures whose supports
are Zariski-dense.
\end{proposition}

\begin{proof}
For each $j\in\NN$, let $\xi_j\in\cMc$ be the normalized restriction of the Haar measure of $G$ to the ball of radius
$1/j$ around the identity $I$ relative to some left-invariant distance on $G$.
It is clear that $\supp\xi_j$ coincides with that ball, and $\xi_j$ vanishes on any proper subvariety of $G$.
Moreover, $(\xi_j)_j$ converges to the Dirac mass $\delta_I$ in the topology of $\cMc$.
For each $j\in\NN$, let $\nu_j=\nu*\xi_j$ be the probability measure on $G$ such that,
for any limited function $\phi:G\to\RR$,
$$
\int_G \phi \, d\nu_j = \int_G \phi(gh) \, d\nu(g) \, d\xi_j(h).
$$
The assumption on $\xi_j$ ensures that $\nu_j$ is generic.
Moreover, $(\nu_j)_j$ converges to $\nu$ in the weak$^*$ topology.
Furthermore, $\supp\nu_j$ coincides with the $(1/j)$-neighborhood of the support of $\nu$,
and so $(\supp\nu_j)_j \to \supp\nu$ in the Hausdorff topology.
This proves that $(\nu_j)_j\to\nu$ in $\cMc$.
The fact that $\supp\nu_j$ has non-empty interior implies that it is Zariski-dense, for every $j\in\NN$.
\end{proof}

\begin{corollary}\label{c_sec5_Zariski2}
Suppose that $(\lambda_1(\nu'_k))_k \to \lambda_1(\nu_\infty)$ for any sequence $(\nu'_k)_k$
converging to $\nu_\infty$ in $\cMc$ such that $\nu'_k$ is generic and $\supp\nu'_k$
is Zariski-dense in $G$ for every $k$.
Then $(\lambda_1(\nu_k))_k \to \lambda_1(\nu_\infty)$ for every sequence $(\nu_k)_k$
converging to $\nu_\infty$ in $\cMc$.
\end{corollary}

\begin{proof}
Let $(\nu_k)_k$  be any sequence converging to $\nu_\infty$ in $\cMc$,
and $d$ be any distance generating the topology of $\cMc$.
By Proposition~\ref{p_sec5_Zariski1}, for each $k$ we may find a generic probability measure
$\nu_k'$ arbitrarily close to $\nu_k$ and whose support is Zariski-dense.
Take $\nu_k'$ such that $d(\nu_k,\nu_k') < 1/k$.
Using the well-known fact that the function $\zeta \mapsto \lambda_1(\zeta)$ is upper
semi-continuous (this is a consequence of \eqref{eq_sec3_Furstenberg_formula} below),
we may also suppose that $\lambda_1(\nu_k') \le \lambda_1(\nu_k) + 1/k$.
Then $(\nu_k')_k \to \nu_\infty$ and so $\lim_k \lambda_1(\nu_k') = \lambda_1(\nu_\infty)$.
Moreover,
$$
\liminf_k \lambda_1(\nu_k) \ge \lim_k \lambda_1(\nu_k')= \lambda_1(\nu_\infty).
$$
Using semi-continuity once more, this implies $\lim_k \lambda_1(\nu_k) = \lambda_1(\nu_\infty)$.
\end{proof}

Thus, to prove Theorem~\ref{theorem:second} it suffices to consider sequences
$(\nu_k)_k \to \nu_\infty$ of generic measures such that every $\supp\nu_k$ is Zariski-dense in $G$.
We do so in all that follows.
Then, by Proposition~\ref{p_sec5_GM}, the $\nu_k$-stationary measure $\eta_k$ is unique and generic.

\subsection{Repelling behavior}\label{ss_repelling_behavior}

Given any subspace $W\subset\RR^d$, let $\Pi_W:\RR^d\to W^\perp$ be the orthogonal projection along $W$.
When $W$ is $1$-dimensional we also write $\Pi_W = \Pi_w$ for any non-zero vector $w\in W$.

For any $g\in G$ and $v\in P$, let $Dg_v:T_vP \to T_{gv}P$ denote the derivative of $g:G\to G$ at the point
$v\in P$. The tangent space $T_vP$ is naturally identified with the orthogonal space $\{v\}^\perp$. Then
\begin{equation}\label{eq_sec5_derivative1}
Dg_v\dot{v} = \Pi_{gv} g\dot{v} \frac{\|v\|}{\|gv\|}.
\end{equation}
For $v\in E$, let
\begin{equation}\label{eq_sec5_derivative2}
Dg^\perp_v= \Pi_E \circ Dg_v\mid E^\perp:E^\perp \to E^\perp,
\quad  Dg^\perp_vv^\perp %= \Pi_{E} gv^\perp \frac{\|v\|}{\|gv\|}
                                = g^\perp v^\perp \frac{\|v\|}{\|gv\|}.
\end{equation}

For $n\in\NN$ and any  probability measure $\nu$ on $G$, let $\nu^{(n)}$ denote the
$n$-\emph{convolution}, that is, the push-forward of $\nu^\NN$ under the map
$G^\NN \to G$ defined by $\ug \mapsto g_{n-1} \cdots g_0$:
\begin{equation}\label{eq_sec5_convolution}
\nu^{(n)} = \nu^\NN(\{\ug\in G^\NN: g_{n-1}\cdots g_0 \in B\})
\end{equation}
for any measurable set $B \subset G$.
%Equivalently, $\nu_k^{(n)}$ is the push-forward of $\nu_k^\ZZ$ under the map $G^\ZZ \to G$ defined by
%$\vg \mapsto g_{n-1} \cdots g_0$.
Note that the map $\nu \mapsto \nu^{(n)}$ is continuous relative to the weak$^*$ topology.

\begin{proposition}\label{p_sec5_derivative_estimate}
There exists $\kappa_0=\kappa_0(\nu_\infty)>0$ and for each $\delta>0$ there exist $N_0=N_0(\nu_\infty,\delta)\in\NN$
and $\tau_0=\tau_0(\nu_\infty,\delta)>0$ such that for every $n\ge N$ and $v^\perp \in \proj(E^\perp)$
there exists $\cE_0=\cE_0(\nu_\infty,\delta,n,v^\perp)\subset \supp\nu_\infty^{(n)}$ with $\nu_\infty^{(n)}(\cE_0^c) < \delta$ and
\begin{enumerate}
\item $\displaystyle{\log \frac{\|Dg^\perp_vv^\perp\|}{\|v^\perp\|} > \kappa_0 n}$ for any $g\in\cE_0$ and $v \in E$;
\item $\displaystyle{\frac{\|g^\perp v^\perp\|}{\|gv^\perp\|} > \tau_0}
$ for any $g\in\cE_0$.
\end{enumerate}
\end{proposition}

\begin{proof}
Fix numbers $\alpha=\alpha(\nu_\infty)$ and $\beta=\beta(\nu_\infty)$ such that
$$
\lim_n \frac 1n \log \|g_{n-1}\cdots g_0 \mid E\| < \alpha < \beta < \lambda_1(\nu_\infty)
$$
and then choose $0 < \kappa_0 < \beta-\alpha$. Let $\delta>0$. By property (i), there exist $N'\in\NN$ and
$\cE'\subset G^\NN$ such that $\nu_\infty^\NN((\cE')^c) < \delta/4$ and
\begin{equation}\label{eq_sec5_derivative3}
\log\frac{\|g_{n-1} \cdots g_0 v\|}{\|v\|} \le \alpha n
\text{ for any $\ug\in\cE'$ and $v\in E$ and $n\ge N'$.}
\end{equation}
By property (ii), there exists $N''\in\NN$ and for each $v^\perp\in E^\perp$ there exists $\cE''\subset G^\NN$
such that $\nu_\infty^\NN((\cE'')^c) < \delta/4$
and
\begin{equation}\label{eq_sec5_derivative4}
\log\frac{\|g^\perp_{n-1} \cdots g_0^\perp v^\perp\|}{\|v^\perp\|}\ge \beta n
\text{ for any $\ug\in\cE''$ and $n\ge N''$.}
\end{equation}
Let $N'''=\max\{N',N''\}$ and for each $n\ge N'''$ and $v^\perp \in \proj(E^\perp)$ define
$$
\cE'''=\{g_{n-1} \cdots g_0: \ug \in \cE' \cap \cE''\}.
$$
The definition \eqref{eq_sec5_convolution} gives that
$\nu_\infty^{(n)}((\cE''')^c) \le \nu_\infty^\NN((\cE')^c \cup (\cE'')^c) < \delta/2$.
Moreover, if $g\in\cE'''$ then \eqref{eq_sec5_derivative3} and \eqref{eq_sec5_derivative4} give that
\begin{equation}\label{eq_sec5_derivative5}
\log\frac{\|Dg^\perp_vv^\perp\|}{\|v^\perp\|}
= \log\frac{\|g^\perp v^\perp\|}{\|v^\perp\|} - \log\frac{\|gv\|}{\|v\|}
\ge (\beta-\alpha) n > \kappa_0 n
\end{equation}
for any $v \in E$ and $n\ge N'''$. This gives claim (1), as long as we choose $N_0\ge N'''$ and $\cE_0\subset\cE'''$,
which we will do in the next paragraph.

Now let us explain how to obtain claim (2). The following elementary inequality will be used a couple of times:
\begin{equation}\label{eq_sec5_sine_distance_estimate}
|\sin\angle(v_1,v_2)| \le \frac{\|v_1 \pm v_2\|}{\|v_2\|}
\text{ for any non-zero } v_1, v_2 \in\RR^d.
\end{equation}
Let $a = (\chi_1 -\chi_2)/8$ where $\chi_1>\chi_2$ are the two largest Lyapunov exponents in \eqref{eq_sec2_exponents}.
Clearly, we may assume that $\beta$ has been chosen greater than
$\chi_1-a$ (keep in mind that $\chi_1=\lambda_1$). Consider the Oseledets splitting in \eqref{eq_sec2_splitting} and denote
$E^*=E^2\oplus\cdots\oplus E^k$. By the Oseledets theorem (see \cite[Theorem~4.2]{LLE}), there exists a measurable
function $\vg \mapsto c(\vg)$ with values in $(0,1)$ such that
\begin{align}
\label{eq_sec5_Oseledets_estimate1}
c(\vg) \, e^{n(\chi_1-a)} ||v_1\|\le \|g_{n-1}\cdots g_0 v_1\| & \le c(\vg)^{-1} e^{n(\chi_1+a)}\|v_1\|\\
\label{eq_sec5_Oseledets_estimate2}
\|g_{n-1}\cdots g_0 v_*\| & \le c(\vg)^{-1} e^{n(\chi_2+a)}\|v_*\|
\end{align}
for any $v_1\in E^1(\vg)$, $v_*\in E^*(\vg)$ and $n\in\NN$, and for $\nu_\infty^\ZZ$-almost every $\vg$.
In particular, $E^1(\vg) \cap E^*(\vg) = \{0\}$.
Since the growth rate of every vector $v\in E$ is strictly less than $\lambda_1=\chi_1$ (property (i) above),
we also have that $E^1(\vg) \cap E = \{0\}$. Thus, up to reducing the function $c(\vg)$, we may suppose that
\begin{equation}\label{eq_sec5_Oseledets_estimate3}
|\sin\angle(E^1(\vg),E^*(\vg))|\ge c(\vg) \quand |\sin\angle(E^1(\vg),E)|\ge c(\vg)
\end{equation}
for $\nu_\infty^\ZZ$-almost every $\vg$. Fix $b=b(\nu_\infty,\delta)>0$ small enough that the set
$$
\cA=\{\vg\in G^\ZZ: c(\vg) \ge b\}
$$
has $\nu_\infty^\ZZ(\cA^c) < \delta/4$. Then fix $M=M(\nu_\infty,\delta)\in\NN$ such that
\begin{equation}\label{eq_sec5_angle_estimate1}
e^{- a n} < \frac{b^2}{2}
\text{ for every $n\ge M$.}
\end{equation}
Let $N_0=\max\{N',N'',M\}$ and for each $n\ge N_0$ and $v^\perp \in \proj(E^\perp)$ define
$$
\cE_0=\{g_{n-1} \cdots g_0: \vg \in \pi^{-1}\cE' \cap \pi^{-1}\cE'' \cap \cA \cap \sigma^{-n}\cA\}
\cap \supp\nu_\infty^{(n)}
$$
($\pi:G^\ZZ\to G^\NN$ is the canonical projection). The choices of $\cE'$, $\cE''$ and $\cA$ ensure that
$\nu_\infty^{(n)}(\cE_0^c)<\delta$. For each $g\in\cE_0$, take $\vg\in \cE'\cap\cE''\cap\cA\cap\sigma^{-n}\cA$
such that $g=g_{n-1}\cdots g_0$ and then let $v^\perp = v_1 + v_*$ be the decomposition of $v^\perp$ with respect
to the splitting $\RR^d=E^1(\vg)\oplus E^*(\vg)$. From property \eqref{eq_sec5_Oseledets_estimate3}
and the fact that $\vg\in\cA$ we get
\begin{equation}\label{eq_sec5_angle_estimate2}
\|v_*\|
\le \frac{\|v^\perp\|}{|\sin\angle(E^1(\vg),E^*(\vg))|}
\le c(\vg)^{-1} \|v^\perp\| \le b^{-1}\|v^\perp\|.
\end{equation}
Since $\beta>\chi_1-a$,
the inequality \eqref{eq_sec5_derivative4} gives
\begin{equation}\label{eq_sec5_angle_estimate3}
\|g v^\perp\| \ge \|g^\perp v^\perp\| \ge e^{(\chi_1-a)n}\|v^\perp\|.
\end{equation}
Properties \eqref{eq_sec5_Oseledets_estimate1} and \eqref{eq_sec5_Oseledets_estimate2} give
\begin{equation}\label{eq_sec5_angle_estimate4}
\|g v^\perp\|
%\le \|g v_1\| + \|g v_*\|
\le b^{-1}e^{(\chi_1+a)n}\|v_1\| + b^{-1} e^{(\chi_2+a)n}\|v_*\|.
\end{equation}
Putting the relations \eqref{eq_sec5_angle_estimate2} through \eqref{eq_sec5_angle_estimate4} together,
and using \eqref{eq_sec5_angle_estimate1}, we obtain
$$
\begin{aligned}
\|v^\perp\|
& \le b^{-1} e^{2a n} \|v_1\| + b^{-1} e^{(\chi_2-\chi_1+2a)n}\|v_*\|\\
& \le b^{-1} e^{2a n} \|v_1\| + b^{-2} e^{-6 an}\|v^\perp\|
\le b^{-1} e^{2a n} \|v_1\| + \frac 12 \|v^\perp\|
\end{aligned}
$$
This proves that $\|v_1\| \ge (b/2) e^{-2a n} \|v^\perp\|$. Combining this inequality with \eqref{eq_sec5_angle_estimate2}
and properties \eqref{eq_sec5_Oseledets_estimate1} and \eqref{eq_sec5_Oseledets_estimate2},
$$
\begin{aligned}
\|gv_1\| & \ge b e^{(\chi_1-a) n} \|v_1\| \ge (b^2/2) e^{(\chi_1-3a) n} \|v^\perp\| \\
\|gv_*\| & \le b^{-1} e^{(\chi_2+a) n} \|v_*\| \le b^{-2} e^{(\chi_2+a) n}\|v^\perp\|
\end{aligned}
$$
In view of our choice of $a$ and the relation \eqref{eq_sec5_angle_estimate2}, this implies that
$$
\frac{\|gv_*\|}{\|gv_1\|}
\le 2b^{-4} e^{(\chi_2-\chi_1+4a)n}
\le 2b^{-4} e^{-4 a n}
< \frac b 2 e^{-2 a n}
\le \frac b2.
$$
Then it follows that
$$
\begin{aligned}
|\sin\angle(gv^\perp,E^1(\sigma^n\vg))|
& \le |\sin\angle(gv^\perp,gv_1)|
%\le |\tan\angle(gv^\perp,gv_1)|\\
\le \frac{\|gv_*\|}{\|gv_1\|}
< b/2.
\end{aligned}
$$
Now,  property \eqref{eq_sec5_Oseledets_estimate3} implies that $|\sin\angle(E^1(\sigma^n\vg),E)| \ge b$,
because we have taken $\vg$ such that $\sigma^n\vg\in\cA$.
Since $|\sin|$ is a subadditive function, these two inequalities
imply that $|\sin\angle(gv^\perp,E))| > b/2$, which means that
$\|g^\perp v^\perp\| > (b/2) \|gv^\perp\|$. This proves (2) with $\tau_0=b/2$.
\end{proof}

Let $d(\cdot,\cdot)$ be the distance defined on the projective space $P$ by
\begin{equation}\label{eq_sec5_distance_P}
d(u,v)=\left\| \Pi_{u} \frac{v}{\|v\|}\right\| = |\sin\angle(u,v)|.
\end{equation}
Note that $0 \le d(u,v) \le 1$ for every $u, v \in P$.
Next, we formulate the infinitesimal estimate in part (1) of Proposition~\ref{p_sec5_derivative_estimate}
in terms of the distance to the equator:

\begin{corollary}\label{c_sec5_distance_contraction1}
For each $n\ge N$ there exists $\rho_0=\rho_0(\nu_\infty, n)>0$ and for each $x\in P$ with $d(x,E)\le\rho_0$
there exists $\cD_\infty(x)=\cD_\infty(\nu_\infty,\delta,n,x)\subset\supp\nu_\infty^{(n)}$ with
$\nu_\infty^{(n)}(\cD_\infty(x)^c) < \delta$
and
\begin{equation}\label{eq_sec5_distance_contraction1}
- \log d(g x,E) \le - \log d(x,E) - \frac{3\kappa_0}{4} n \quad \text{for any $g\in\cD_\infty(x)$.}
\end{equation}
\end{corollary}

\begin{proof}
Let $\exp$ denote the exponential map of the Riemannian manifold $P$. For each $x\in P$ close to the equator
we may write $x=\exp_v v^\perp$ for a (unique) $v^\perp\in E^\perp$ with $\|v^\perp\| = d(x,v) = d(x,E)$. Then
$$
\lim_{x\to E} \frac{d(gx,E)}{\|Dg^\perp_v v^\perp\|}
= \lim_{x\to E} \frac{d(\exp_{gv}Dg_v v^\perp, E)}{\|Dg^\perp_v v^\perp\|}
= 1
$$
and the limits are uniform in $g\in\supp\nu^{(n)}_\infty$. In particular, there exists
$\rho_0=\rho_0(\nu_\infty,n)>0$ such that
\begin{equation}\label{eq_sec5_rho0_continuity}
d(x,E) \le \rho_0 \quad \Rightarrow \quad
\Big|\log\frac{d(gx,E)}{d(x,E)} - \log\frac{\|Dg_v^\perp v^\perp\|}{\|v^\perp\|}\Big| < \frac{\kappa_0}{4} n
\end{equation}
for every $g\in\supp\nu^{(n)}_\infty$. Define $\cD_\infty(x)$ to be the set $\cE_0(\nu_\infty,\delta,n,v^\perp)$
given by Proposition~\ref{p_sec5_derivative_estimate}.
Then \eqref{eq_sec5_distance_contraction1} follows from \eqref{eq_sec5_rho0_continuity} and part (1)
of Proposition~\ref{p_sec5_derivative_estimate}.
\end{proof}

We also need to extend these estimates from $k=\infty$ to every large $k\in\NN$:

\begin{corollary}\label{c_sec5_distance_contraction2}
For any $n\ge N$ and $\rho \in (0,\rho_0)$ there is $k_0=k_0(\nu_\infty, \delta, n, \rho)\in\NN$ and for any
$x\in P$  with $\rho\le d(x,E)\le\rho_0$ there is $\cD_k(x)=\cD_k(\nu_\infty,\delta,n,\rho,x)\subset \supp\nu_k^{(n)}$
such that $\nu_k^{(n)}(\cD_k(x)^c) < \delta$ and
\begin{equation}\label{eq_sec5_distance_contraction2}
- \log d(g x,E) \le -\log d(x,E) - \frac{\kappa_0}{2} n
\quad \text{for any $g\in\cD_k(x)$ and $k\ge k_0$.}
\end{equation}
\end{corollary}

\begin{proof}
Let $n\ge N$ and $\rho\in(0,\rho_0)$ be fixed. The set $K=\{x \in P: \rho \le d(x,E) \le \rho_0\}$ is compact.
By continuity, there exists $\theta=\theta(\nu_\infty, n, \rho)>0$ (keep in mind that $\kappa_0=\kappa_0(\nu_\infty)$
and $\rho_0=\rho_0(\nu_\infty,n)$), such that
\begin{equation}\label{eq_sec5_theta_continuity}
\Big|\log \frac{d(gx,E)}{d(x,E)}-\log \frac{d(hy,E)}{d(y,E)}\Big| < \frac{\kappa_0}{4} n
\end{equation}
for any $g\in B(h,\theta)$ and $x\in B(y,\theta)$ and $h\in\supp\nu_\infty^{(n)}$ and $y\in K$.
Choose a finite set $\{x_j: j=1, \dots, l\}\subset K$ such that the balls $B(x_j,\theta)$, $j=1, \dots, l$ cover $K$.
For each $x\in K$ choose $j\in\{1, \dots, l\}$ such that $x\in B(x_j,\theta)$ and define
\begin{equation}\label{eq_sec5_definition_cDk}
\cD_k(x) = \big[\text{ $\theta$-neighborhood of $\cD_\infty(x_j)$}\big] \cap \supp\nu_k^{(n)}.
\end{equation}
Let $x\in K$ and $g\in\cD_k(x)$. By definition, there exist $x_j\in K$ and $h\in\cD_\infty(x_j)$ such
that $d(x,x_j)<\theta$ and $d(g,h)<\theta$. Then \eqref{eq_sec5_theta_continuity} gives that
$$
\begin{aligned}
- \log d(gx,E)
& \le - \log d(x,E) - \log d(hx_j,E) + \log d(x_j,E) + \frac{\kappa_0}{4}n,
\end{aligned}
$$
whereas \eqref{eq_sec5_distance_contraction1} gives that $- \log d(hx_j,E) \le - \log d(x_j,E) - (3/4)\kappa_0 n$.
Substituting the latter in the former one obtains \eqref{eq_sec5_distance_contraction2}.
Since $\nu_k^{(n)}$ converges to $\nu_\infty^{(n)}$ in the weak$^*$ topology, the limit inferior of the
$\nu^{(n)}_k$-measures of \eqref{eq_sec5_definition_cDk} is greater than or equal to
$$
\nu^{(n)}_\infty\left(\cD_\infty(x_j)\right) > 1 - \delta
$$
for any $j=1, \dots, l$. In particular, there is $k_0=k_0(\nu_\infty,\delta, n,\rho)\in\NN$ such that
$$
\nu_k^{(n)}(\cD_k(x)) > 1 - \delta
\text{ for every $k\ge k_0$ and $x\in K$.}
$$
This completes the argument.
\end{proof}

\section{A toolbox}\label{s_tool_box}

Here we collect several fairly general ideas and facts that are required for the continuation of our arguments.
The proofs themselves will not be used in what follows, so the reader is advised to skip them at this stage,
and to return to this section for reference along the way, as needed. In our applications, the metric spaces
$X$ and $X'$ will be suitable subsets of Grassmannian manifolds or flag varieties.

\subsection{Couplings}\label{ss_couplings}

Let $\eta$ and $\eta'$ be measures on measurable spaces $X$ and $X'$, respectively,
with $\eta(X)=\eta'(X')$.
A measure $\teta$ on $X \times X'$ is a \emph{coupling} of $\eta$ and $\eta'$ if it projects to $\eta$
on the first factor and to $\eta'$ on the second factor, that is, if
$$
\teta(A\times X') = \eta(A) \quand \teta(X\times A')=\eta'(A')
$$
for any measurable sets $A\subset X$ and $A'\subset X'$. When $\eta=\eta'$ (and $X=X'$) we
call $\teta$ a \emph{self-coupling} of $\eta$. A self-coupling is \emph{symmetric} if it is invariant
under the involution $\iota:(x,x') \mapsto (x',x)$.

For example, the \emph{diagonal embedding} of a measure $\eta$ on $X$ is the symmetric
self-coupling $\tilde\eta$ of $\eta$ defined by
\begin{equation}\label{eq_sec6_diagonal_embedding}
\tilde\eta(B) = \eta(\{x\in X: (x,x) \in B\}).
\end{equation}
Another simple example of a coupling is the \emph{scaled product}
\begin{equation}\label{eq_sec6_scaled_product}
\teta = \frac{1}{c} (\eta \times \eta'), \text{ where $c=\eta(X)=\eta'(X')$.}
\end{equation}

Couplings are far from being unique, which turns out to be a very convenient feature in our context.
Especially, we will use the following elementary construction:

\begin{lemma}\label{l_sec6_coupling_noatoms1}
Suppose that $A\subset X$ and $A'\subset X'$ are such that $\eta(A) < \eta(X\setminus A)$ and
$\eta'(A') < \eta'(X'\setminus A')$. Then there exists a coupling $\teta$ of $\eta$ and $\eta'$
such that $\teta(A \times A') =0$.
\end{lemma}

\begin{proof}
Let $B=X \setminus A$ and $B'=X' \setminus A'$ and then take
\begin{equation}\label{eq_sec6_crop}
\begin{aligned}
\teta
& = \frac{1}{\eta'(B')} (\eta \mid A) \times (\eta' \mid B') + \frac{1}{\eta(B)} (\eta \mid B) \times (\eta' \mid A')\\
& \qquad\qquad+ \left(\frac{1}{\eta(B)} + \frac{1}{\eta'(B')} - \frac{c}{\eta(B)\eta'(B')}\right)(\eta \mid B) \times (\eta'\mid B'),
\end{aligned}
\end{equation}
where $c=\eta(X)=\eta'(X')$.
\end{proof}

\begin{lemma}\label{l_sec6_coupling_noatoms2}
Let $A_j \subset C_j\subset X$ and $A_j'\subset C_j' \subset X'$, $j=1, \dots, k$ be such that
\begin{enumerate}
\item $\eta(C_j) < \eta(X\setminus C_j)$ and $\eta'(C_j') < \eta'(X'\setminus C_j')$ for $1 \le j \le k$;
\item $A_j \times (X' \setminus C'_j)$ and $(X \setminus C_j) \times A'_j$ are disjoint from $A_i \times A'_i$ for $1 \le i < j \le k$.

\end{enumerate}
Then there exists a coupling $\teta$ of $\eta$ and $\eta'$ such that $\teta(A_j \times A_j')=0$ for $j=1, \dots, k$.
\end{lemma}

\begin{proof}
We are going to construct couplings $\teta_j$, $j=1, \dots, k$ of $\eta$
and $\eta'$ such that
\begin{equation}\label{eq_sec6_diagonal_cover}
\teta_j(A_i \times A'_i)=0
\text{ for any $1 \le i \le j$.}
\end{equation}
Then it suffices to take $\teta=\teta_k$.

The case $j=1$ of \eqref{eq_sec6_diagonal_cover} is contained in Lemma~\ref{l_sec6_coupling_noatoms1}.
We proceed by induction.
Let $j\in\{2, \dots, k\}$ and assume that we have constructed a coupling $\teta_{j-1}$ of $\eta$
and $\eta'$ such that $A_i \times A_i$ has zero measure for every $i=1, \dots, {j-1}$.
We claim that
\begin{equation}\label{eq_sec6_not_balanced}
\teta_{j-1}(C_j \times C'_j) % \le \teta_{j-1}(C_j \times C'_j)
< \teta_{j-1}(C_j^c \times (C'_j)^c),
\end{equation}
where $C_j^c=X \setminus C_j$ and $(C'_j)^{c}=X' \setminus C'_j$.
Indeed, suppose that $\eta'(C'_j) \le \eta(C_j)$. Recall that $\eta(C_j) < \eta(C_j^c)$, by assumption.
Moreover,
\begin{itemize}
\item $\teta_{j-1}(C_j \times C'_j) + \teta_{j-1}(C_j^c \times C'_j) = \eta'(C'_j)$,
\item $\teta_{j-1}(C_j^c \times C'_j) + \teta_{j-1}(C_j^c \times (C'_j)^c) = \eta(C_j^c)$.
\end{itemize}
Combining these relations we get the inequality in \eqref{eq_sec6_not_balanced}.
The case $\eta(C_j) \le \eta'(C'_j)$ is analogous, and so the claim is proved.
In particular, the following number is smaller than $1$:
$$
\theta_j=\frac{\teta_{j-1}(A_j \times A'_j)}{\teta_{j-1}(C_j^c \times (C'_j)^c)}.
$$

Let $\pi:X \times X' \to X$ and $\pi':X \times X' \to X'$ be the canonical projections.
Let $\zeta_j$ be the scaled product of $\pi_*\left(\teta_{j-1} \mid A_j \times A'_j\right)$
and $\theta_j \pi'_*\left(\teta_{j-1} \mid C_j^c \times (C'_j)^c\right)$,
and $\zeta_j'$ be the scaled product of
$\theta_j \pi_*\left(\teta_{j-1} \mid C_j^c \times (C'_j)^c\right)$ and
$\pi'_*\left(\teta_{j-1} \mid A_j \times A'_j\right)$. Then define
\begin{equation}\label{eq_sec6_etaij}
\teta_j = \teta_{j-1} - \left(\teta_{j-1} \mid A_j   \times A'_j\right)
             - \theta_j \left(\teta_{j-1} \mid C_j^c \times (C'_j)^c\right) + \zeta_j + \zeta_j'.
\end{equation}
It is clear that $\teta_j$ is a (positive) measure, because $\theta_j<1$.
It is also clear that $\teta_j$ is a coupling of $\eta$ and $\eta'$:
$$
\begin{aligned}
\pi_* \teta_j
& = \pi_* \teta_{j-1} - \pi_* \left(\teta_{j-1} \mid A_j \times A'_j\right)
 - \theta_j \pi_* \left(\teta_{j-1} \mid C_j^c \times (C'_j)^c\right) \\
& \hspace*{1.4cm} + \pi_* \left(\teta_{j-1} \mid A_j \times A'_j\right)
+ \theta_j \pi_* \left(\eta_{j-1} \mid C_j^c \times (C'_j)^c\right)
 = \pi_* \teta_{j-1} = \eta
\end{aligned}
$$
and, analogously, $\pi'_*\teta_j = \eta'$. Note also that $\teta_j(A_j \times A'_j)=0$.
Moreover, $\zeta_j$ is concentrated in $A_j \times (C'_j)^c$
and $\zeta_j'$ is concentrated in $C_j^c \times A'_j$.
Since both sets are assumed to be disjoint from $A_i \times A'_i$, we get that
$$
\teta_j\left(A_i\times A'_i\right) \le \teta_{j-1}\left(A_i \times A'_i\right) = 0
\text{ for $i=1, \ldots, j-1$.}
$$
This completes the induction.
\end{proof}

\begin{remark}\label{r_sec6_continuous_dependence}
By construction, the coupling $\teta$ varies continuously with $\eta$ and $\eta'$
in the weak$^*$-topology at all points such that the boundaries of all $A_j$ and
$C_j$ have zero $\eta$-measure and the boundaries of all $A'_j$ and $C'_j$ have
zero $\eta'$-measure.
\end{remark}

\begin{remark}\label{r_sec6_preserve_generic}
The constructions in Lemmas~\ref{l_sec6_coupling_noatoms1} and~\ref{l_sec6_coupling_noatoms2}
preserve the class of non-atomic measures and, when $X$ is an algebraic variety, also the class
of generic measures: if $\eta$ and $\eta'$ are generic then so is the coupling $\teta$.
That follows directly from Remark~\ref{r_sec5_generic_elementary} and the expressions
\eqref{eq_sec6_crop} and \eqref{eq_sec6_etaij}.
\end{remark}

\begin{remark}\label{r_sec6_involution_invariant1}
When $X=X'$, these constructions are \emph{involution-invariant} in the following sense.
First,  in Lemma~\ref{l_sec6_coupling_noatoms1} the coupling $\teta$ given by \eqref{eq_sec6_crop} is
replaced with $\iota_*\teta$ when one  exchanges the  roles of $\eta, A$ and $\eta', A'$.
In particular, if $\eta=\eta'$ and $A=A'$ then the self-coupling $\teta$ is symmetric.
In Lemma~\ref{l_sec6_coupling_noatoms2} we may take $\zeta'_j=\iota_*\zeta_j$ and then
the coupling $\teta_j$ is replaced with $\iota_*\teta_j$ when the roles of $\eta, A_i, C_i$
and $\eta', A'_i, C'_i$ are exchanged. In particular, if $\eta=\eta'$, $A_i=A'_i$,
and $C_i=C'_i$ then the self-coupling $\teta=\teta_k$ is symmetric.
\end{remark}

\begin{proposition}\label{p_sec6_coupling_Kset}%{l_sec6_coupling_diagonal}
Let $K$ be a compact subset of the product $X \times X'$ such that
$$
K(x') = \{y \in X: (y,x') \in K\} \text{ and } K'(x) = \{y'\in X': (x,y') \in K\}.
$$
satisfy
\begin{equation}\label{eq_nonatomicK}
\eta\left(K(x')\right) < \eta\left(X \setminus K(x')\right)
\text{ and }
\eta'\left(K'(x)\right) < \eta'\left(X' \setminus K'(x)\right)
\end{equation}
for every $(x, x') \in X \times X'$. Then there exists a coupling $\teta$ of $\eta$ and $\eta'$
that vanishes on a neighborhood of $K$.
\end{proposition}

\begin{proof}
We are going to find open sets $A_j \subset C_j \subset X$ and $A'_j \subset C'_j \subset X'$,
$j=1, \dots, k$, such that
\begin{itemize}
\item[(a)] $\eta(C_j) < \eta(X \setminus C_j)$ and $\eta(C'_j) < \eta(X \setminus C'_j)$;
\item[(b)] $A_j \times A'_j$ is disjoint from $(X \setminus C_i) \times A'_i$ and $A_i \times (X \setminus C'_i)$;
\item[(c)] and the union of the products $A_j \times A'_j$ contains $K$.
\end{itemize}
Then, by Lemma~\ref{l_sec6_coupling_noatoms2}, there exists a coupling $\teta$ of $\eta$ and $\eta'$
vanishing on the union of $A_j \times A'_j$, $j=1, \dots, k$, which gives the claim of the
present lemma. So let us explain how to construct such sets.

By compactness, there exists $\vep>0$ such that \eqref{eq_nonatomicK} remains valid when $K(x')$
and $K'(x)$ are replaced with their $10\vep$-neighborhoods. Let $\vep$ be fixed.
Also by compactness, the maps $x \mapsto K'(x)$ and $x' \mapsto K(x')$ are upper semicontinuous.
In particular, for any $x \in X$ and $x' \in X'$ there exist $\delta(x)>0$ and $\delta'(x')>0$ such that
\begin{equation}\label{eq_valid_in_closure}
\begin{aligned}
& K'(y) \subset B_\vep\left(K'(x)\right) \text{ if } d(x,y) < 4\delta(x)
\text{ and }\\
& K(y') \subset B_\vep\left(K(x')\right) \text{ if } d(x',y') < 4\delta(x').
\end{aligned}
\end{equation}
It is no restriction to assume that $\delta(x)$ and $\delta'(x')$ are bounded by $\vep$, and we do so.
Let $(x_1,x'_1), \dots, (x_k,x'_k) \in K$ be such that $B(x_j,\delta(x_j)) \times B(x'_j,\delta'(x'_j))$,
$j =1, \dots, k$ cover $K$. Initially, define
\begin{equation}\label{eq_sec6_A_C_A_C}
\begin{aligned}
A_j = B(x_j,3\delta(x_j)),& \quad A'_j=B(x'_j,3\delta'(x'_j))\\
C_j = B\left(K(x'_j),10\vep\right),& \quad C'_j = B\left(K'(x_j),10\vep\right).
\end{aligned}
\end{equation}
Note that $A_j \subset C_j$ and $A'_j \subset C'_j$ since $(x_j,x'_j) \in K$ and $\delta, \delta' \le \vep$.
Conditions (a) and (c) are clearly satisfied, but not necessarily (b).
In the following we replace the sets $A_j$ and $A'_j$ with suitable subsets,
in such a way as to achieve (b) while preserving (c). The condition (a) is clearly not affected.

\begin{figure}[ht]
\begin{center}
\psfrag{xj}{$x_j$}\psfrag{xil}{$x'_i$}\psfrag{xjl}{$x'_j$}
\psfrag{Aj}{$A_j$}\psfrag{Ail}{$A'_i$}\psfrag{Ajl}{$A'_j$}
\psfrag{Ci}{$X \setminus C_i$}
\psfrag{y}{$y$}\psfrag{yl}{$y'$}\psfrag{z}{$z$}
\includegraphics[height=2in]{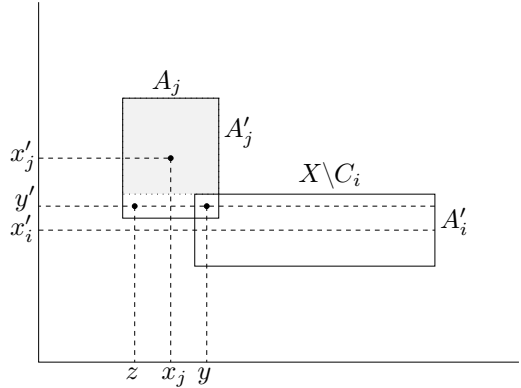}
\caption{\label{f_disjoint}
Trimming the sets $A_j$ and $A'_j$: to achieve the disjointness condition (c), in the situation described
in the figure $A_j \times A'_j$ is replaced with $A_j \times (A'_j \setminus \bar{A'_i})$, which corresponds
to the shaded region. A dual operation is applied to $A_j$, and the whole procedure is repeated for every $i$ and $j$.}
\end{center}
\end{figure}

Consider any $i, j = 1, \dots, k$. If $A_j\times A'_j$ is disjoint from $A_i \times (X' \setminus C'_i)$
and $(X \setminus C_i) \times A'_i$ there is nothing to do. Next, let us consider the case when
there exists $(y,y')$ in $\left(A_j\times A'_j\right) \cap \left((X \setminus C_i) \times A'_i\right)$.
See Figure~\ref{f_disjoint}. We claim that $A_j \times B(x'_i,4\delta(x'_i)$ is disjoint from $K$.
Indeed, suppose there existed $z \in A_j$ and $y'' \in B(x'_i,4\delta(x'_i))$ such that $(z,y')\in K$.
Since $d(z,y) < 6\delta(x_j)\le 6\vep$, it would follow that $y\in B(K(y''),6\vep)$.
On the other hand, $K(y'') \subset B(K(x'_i),\vep)$ because $d(y'',x'_i) < 4\delta(x'_i)$.
Hence, we would have $y \in B(K(x'_i),7\vep)$, which would contradict the fact that
$y \in X \setminus C_i$. This contradiction proves our claim.
Now, this ensures that $A_j \times \bar{A'_i}$ is disjoint and at a definite distance from $K$.
Thus, we may replace $A'_j$ with $A'_j \setminus \bar{A'_i}$ in our construction without affecting condition (d) and, by doing it, we get that $A_j \times A'_j$ becomes disjoint from
$(X \setminus C_i) \times A'_i$.
The case when there exists $(y,y')$ in $(A_j\times A'_j) \cap \left(A_i \times (X' \setminus C'_i)\right)$
is treated in the same way, trimming $A'_j$ instead.
Repeating this procedure for each $i$ and $j$, we get all three conditions (a) to (c).
\end{proof}

\begin{remark}\label{r_sec6_boundaries}
The union of the boundaries $\partial A_j$ over all $j=1, \dots, k$ does not increase under the trimming operation,
and the same holds for the union of the boundaries $\partial A'_j$ over $j=1, \dots, k$.
\end{remark}

Next we want to state and prove a parametrized version of Proposition~\ref{p_sec6_coupling_Kset}.
The following elementary fact will be useful at other places as well:

\begin{lemma}\label{l_sec6_elementary2}
Let $g:Z \times T \to \RR$ be a bounded measurable function, where $Z$ is a metric space and $(T,\mu)$ is a probability space.
Let $z_0\in Z$ be such that the set $D(z_0)$ of values of $t \in T$ such that $z \mapsto g(z,t)$ is discontinuous at $z_0$
has zero $\mu$-measure. Then $z \mapsto \int_T g(z,t) \, d\mu(t)$ is continuous at $z=z_0$.
\end{lemma}

\begin{proof}
Fix any $\vep>0$.
For each $k\in\NN$, denote by $T_k$ the set of values of $t \in T$ such that $|g(z_0,t) - g(z,t)|\le \vep$
for any $z$ in the $(1/k)$-neighborhood of $z_0$.
The sequence $T_k$ is non-decreasing and the assumption ensures that $\cup_k T_k$ has full $\mu$-measure.
Fix $k$ such that $\mu(T_k^c) < \vep$. Then for any $z$ in the $(1/k)$-neighborhood of $z_0$,
$$
\left| \int_T g(z_0,\cdot) \, d\mu - \int_T g(z,\cdot) \, d\mu \right|
\le \vep + \left| \int_{T_k^c} g(z_0,\cdot) \, d\mu \right| + \left| \int_{T_k^c} g(z,\cdot) \, d\mu \right|
\le (1 + 2\|g\|) \vep.
$$
Since $\vep>0$ is arbitrary, this proves that $z_0$ is a continuity point.
\end{proof}

\begin{proposition}\label{p_sec6_coupling_Kset_parametrized}
Let $X$, $X'$, $Y$, and $Y'$ be compact metric spaces, $K$ be a compact subset of $X \times X'$,
and $\eta=\{\eta_y:y\in Y\}$ and $\eta'=\{\eta'_{y'}:y' \in Y'\}$ be continuous families of probability
measures on $X$ and $X'$, respectively, such that
\begin{equation}\label{eq_nonatomicKY}
\eta_{y}\left(K(x')\right) < \eta_{y}\left(X \setminus K(x')\right)
\text{ and }
\eta'_{y'}\left(K'(x)\right) < \eta'_{y'}\left(X' \setminus K'(x)\right).
\end{equation}
for every $(x, x') \in X \times X'$ and $(y, y') \in Y \times Y'$.
Then there exists a continuous family $\teta=\{\teta_{y,y'}: (y, y') \in Y \times Y'\}$
of probability measures on $X \times X'$ such that each $\teta_{y,y'}$ is a coupling of $\eta_y$ and
$\eta'_{y'}$ vanishing on a uniform neighborhood of $K$.
\end{proposition}

\begin{proof}
By compactness, the assumption \eqref{eq_nonatomicKY} implies that there exists $\vep>0$ such that
\begin{equation*}
\begin{aligned}
& \eta_{y}(B(K(x'),10\vep))<\eta_{y}(X \setminus B(K(x'),10\vep))
\text{ and } \\
& \eta'_{y'}(B(K'(x),10\vep))<\eta'_{y'}(X' \setminus B(K'(x),10\vep))
\end{aligned}
\end{equation*}
for every $(x,x') \in X \times X'$ and $(y,y') \in Y \times Y'$. Fix $\vep$ and let
$\delta(x), \delta'(x') \in(0,\vep)$ be as in \eqref{eq_valid_in_closure}.
Let $(x_1,x'_1), \dots, (x_k,x'_k) \in K$ be such that $B(x_j,\delta(x_j)) \times B(x'_j,\delta'(x'_j))$,
$j =1, \dots, k$ cover $K$. For each $s\in[0,1]$ and $j =1, \dots, k$, define
\begin{equation}\label{eq_sec6_As_Cs_As_Cs}
\begin{aligned}
A_{j,s} = B(x_j,(3-s)\delta(x_j)),& \quad A'_{j,s}=B(x'_j,(3-s)\delta'(x'_j))\\
C_{j,s} = B\left(K(x'_j),(10-s)\vep\right),& \quad C'_{j,s} = B\left(K'(x_j),(10-s)\vep\right).
\end{aligned}
\end{equation}

Applying to the measures $\eta_y$ and $\eta'_{y'}$, and the sets $A_{j,s}$, $C_{j,s}$, $A'_{j,s}$,
and $C'_{j,s}$ the same construction in Lemma~\ref{l_sec6_coupling_noatoms2},
we find for each $s\in [0,1]$ and $(y,y')\in Y \times Y'$ a coupling $\teta_{y,y',s}$ of $\eta_y$
and $\eta'_{y'}$ which vanishes on a neighborhood of $K$ independent of both $s$ and $(y,y')$.
(recall that the trimming is always done at a definite distance from $K$).

By Remark~\ref{r_sec6_continuous_dependence}, for each fixed $s \in [0,1]$ the map
$(y,y') \mapsto \teta_{y, y',s}$ is discontinuous at a given point $(z,z')\in Y \times Y'$ only if
the union of the boundaries of $A_{j,s}$ and $C_{j,s}$ has positive measure for $\eta_{z}$
or the union of the boundaries of $A'_{j,s}$ and $C'_{j,s}$ has positive measure for $\eta'_{z'}$.
In the setting of \eqref{eq_sec6_As_Cs_As_Cs}, the boundaries of the $A_{j,s}$ are pairwise disjoint for each
fixed $j$, and the same is true for the $C_{j,s}$, $A'_{j,s}$, and $C'_{j,s}$.
Thus, positive measure may occur only for a countable subset of values of $s$.
Remark~\ref{r_sec6_boundaries} ensures that the latter conclusion remains valid after the trimming.
In conclusion, every $(z,z')\in Y \times Y'$ is a continuity point of $(y,y') \mapsto \teta_{y, y',s}$
for all but countably many values of $s\in [0,1]$. Then, using Lemma~\ref{l_sec6_elementary2},
\begin{equation*}
\teta_{y,y'} = \int_0^1 \teta_{y,y',s} \, ds
\end{equation*}
is a coupling of $\eta_{y}$ and $\eta'_{y'}$ depending continuously on $y$ and $y'$ and vanishing on
a uniform neighborhood of $K$.
\end{proof}

\begin{remark}\label{r_sec6_involution_invariant2}
When $X=X'$ and the set $K \subset X \times X$ is symmetric, one may exchange
$A_j, C_j$ with $A'_j, C'_j$ in \eqref{eq_sec6_A_C_A_C}, and $A_{j,s}, C_{j,s}$ with
$A'_{j,s}, C'_{j,s}$ in \eqref{eq_sec6_As_Cs_As_Cs}.
Thus (recall Remark~\ref{r_sec6_involution_invariant1}) when the roles of $\eta$ and $\eta'$
are exchanged the coupling $\teta$ is replaced with $\iota_*\teta$ in Proposition~\ref{p_sec6_coupling_Kset}
and each $\teta_{y',y}$ is replaced with $\iota_*\teta_{y,y'}$ in Proposition~\ref{p_sec6_coupling_Kset_parametrized}.
In particular, if $\eta=\eta'$ and $Y=Y'$ then the self-coupling $\teta$ in
Proposition~\ref{p_sec6_coupling_Kset} is symmetric, and the family $\teta$
in Proposition~\ref{p_sec6_coupling_Kset_parametrized} is involution-invariant, meaning that
$\teta_{y',y}=\iota_*\teta_{y,y'}$ for every $(y,y') \in Y \times Y$.
\end{remark}

The following special case of Proposition~\ref{p_sec6_coupling_Kset} will be useful:

\begin{corollary}\label{c_sec6_coupling_diagonal}
Let $X=X'$ and assume that $\eta$ and $\eta'$ satisfy $\eta(\{x\})<\eta(X\setminus\{x\})$ and
$\eta'(\{x\})<\eta'(X\setminus\{x\})$  for every $x\in X$.
Then there exists a coupling $\teta$ of $\eta$ and $\eta'$ that vanishes on a neighborhood
of the diagonal of $X \times X$. If $\eta=\eta'$ then the self-coupling $\teta$ may be chosen
to be symmetric.
\end{corollary}

\subsection{Markov operators}\label{ss_Markov_operators}

Let $X$ be a metric space. We denote by  $\Bd(X)$ the Banach space of bounded measurable functions on $X$,
with the norm
$$
\|\psi\|=\sup\{|\psi(x)|: x\in X\}.
$$
A \emph{Markov operator} is a linear operator $\cT:\Bd(X)\to \Bd(X)$ of the form
\begin{equation}\label{eq_sec6_Markov_op1}
\cT\psi(x) = \int_X \psi(y) \, d\sigma_x(y)
\end{equation}
where $\{\sigma_x: x\in X\}$ is a measurable family of probability measures on $X$.
It is clear that $\cT$ is a bounded operator, with $\|\cT\|=1$.
We call it \emph{continuous} if the map $x \mapsto \sigma_x$ is continuous relative to the weak$^*$ topology.
Then $\cT$ preserves the subspace $\Cd(X)$ of bounded continuous functions.

The \emph{dual operator} $\cT^*$ is defined on the space of bounded finitely additive signed
measures $\eta$ on $X$ with the total variation norm (see \cite[IV.4.5]{DSI}) by%\cite[Theorem~5.12]{PuS99}
\begin{equation}\label{eq_sec6_Markov_op2}
\cT^*\eta = \int_X \sigma_x \, d\eta(x).
\end{equation}
The two are related by
\begin{equation}\label{eq_sec6_Markov_dual}
\int_X \psi \, d\left(\cT^*\eta\right) = \int_X \left(\cT\psi\right) \, d\eta
\text{ for every $\psi$ and $\eta$.}
\end{equation}
A measure $\eta$ is \emph{$\cT$-invariant} if $\cT^*\eta=\eta$. If $\cT$ is continuous then $\cT^*$ preserves the
subspace of bounded (countably additive) signed measures on $X$.

\begin{remark}\label{r_sec6_must_be_generic}
If $\{\sigma_x: x \in X\}$ is such that every $\sigma_x$ is a generic measure then,
cf. Remark~\ref{r_sec5_generic_elementary}, the range of the associated dual Markov operator
$\cT^*\eta = \int_X \sigma_x\, d\eta(x)$  is contained in the space of generic measures.
In particular, any $\cT$-invariant measure is a generic measure.
\end{remark}

Suppose that $X$ comes with a transitive $G$--action $(g,x) \mapsto gx$. Grassmannian manifolds $\grass(r,d)$ and
flag varieties $\cF(r,d)$ are the examples we have in mind.
Then, to any probability measure $\nu$ on $G$ we may associate the Markov operators $\cP_\nu$ and $\cP^*_\nu$
in \eqref{eq_sec3_Markov_operator1}, \eqref{eq_sec3_Markov_operator2}, and \eqref{eq_sec5_Markov_operator12}
\begin{equation}\label{eq_sec6_Markov_op3}
\cP_\nu\psi(x) = \int_G \psi(gx) \, d\nu(g)
\quand
\cP^*_\nu\eta = \int_G \left(g_*\eta\right) \, d\nu(g).
\end{equation}
This corresponds to \eqref{eq_sec6_Markov_op1} with $\sigma_x =$ the push-forward of $\nu$ under the map $g \mapsto gx$.
Note that $\cP_\nu$ is continuous if the $G$-action is continuous, and a measure $\eta$ on $Z$ is $\cP_\nu$-invariant
precisely if it is $\nu$-stationary. These are the fundamental examples of Markov operators in our context,
but we will have to deal with other types as well.

One reason is that  the kind of conclusion we are seeking, namely, that stationary measures give small weights to a neighborhood
of the equator is local in nature. That is consistent with the fact that the information on the dynamics we can
extract from Proposition~\ref{p_sec5_derivative_estimate} is clearly local.
In contrast, the property of being a stationary measure is \emph{not} local: the restriction of a $\cP_\nu$-invariant
measure $\eta$ to some set $U \subset X$ is usually not a $\cP_\nu$-invariant measure.
The way we handle this is by finding a ``localized'' Markov operator, related to the original one and to the domain $U$
in an explicit manner, with respect to which the restriction $\eta \mid U$ is indeed an invariant measure.

\begin{remark}\label{r_sec6_measure_lift}
The assumption that the $G$-action on $X$ is transitive means that $G\to X$, $g \mapsto gx$ is surjective for any fixed $x\in X$.
Then every probability measure $\sigma$ on $X$ lifts (non-uniquely) to a probability measure $\mu$ on $G$: use the Hahn--Banach and Riesz--Markov theorems.
In particular, \eqref{eq_sec6_Markov_op1} may be written in the form
$$
\cT\psi(x) = \int_G \psi(gx) \, d\mu_x(g)
$$
for some family $\{\mu_x: x \in X\}$ of probability measures on $G$.
This general statement is not used in the present paper, but an explicit construction
in a special case will appear in Lemma~\ref{l_lifting}.
\end{remark}

\begin{proposition}\label{p_sec6_localized1}
Let $\cT:\Bd(X)\to\Bd(X)$ be a Markov operator, $\eta$ be a $\cT$-invariant measure, and
$U\subset X$ be such that $\eta(U)>0$.
Then there exists a Markov operator $\cT_U:\Bd(U)\to\Bd(U)$ that leaves invariant the
normalized restriction $\eta_U$  of $\eta$ to the subset $U$.
 \end{proposition}

\begin{proof}
We are going to find $\{\sigma_{U,x}: x\in U\}$ such that
$\cT_U\psi(x) = \int_U \psi(y) \, d\sigma_{U,x}(y)$ preserves $\eta_U$.
Let $\chi_U$ denote the characteristic function of $U$. Since $\eta$ is $\cT$-invariant,
$$
\begin{aligned}
0
& = \int_X \left(\cT\chi_U-\chi_U\right) \, d\eta
= \int_X\left[\int_X \chi_U(y) \, d\sigma_x(y) - \chi_U(x)\right] \, d\eta(x) \\
& = \int_U\left[\sigma_x(U) - 1\right] \, d\eta(x)
+ \int_{U^c} \sigma_x(U) \, d\eta(x),
\end{aligned}
$$
that is,
\begin{equation}\label{eq_sec6_localized1}
\int_U \sigma_x(U^c) \, d\eta(x) = \int_{U^c} \sigma_x(U) \, d\eta(x).
\end{equation}
Let $J$ be this number. If $J=0$, there is not much to do: $\eta_U$ turns out to be $\cT$-invariant,
and it suffices to take $\sigma_{U,x}=\sigma_x$ for $x\in U$. If $J>0$, define
\begin{equation}\label{eq_sec6_localized2}
\sigma_{U,x} = (\sigma_x \mid U) + \sigma_x(U^c) \frac{1}{J} \int_{U^c} (\sigma_z \mid U) \, d\eta(z)
\end{equation}
for each $x\in U$. In other words,
$$
\begin{aligned}
\cT_U\psi(x)
& = \int_U \psi(y) \, d\sigma_x(y) + \sigma_x(U^c) \frac{1}{J}  \int_{U^c} \int_U \psi(y) \, d\sigma_z(y) \, d\eta(z) \\
\text{and }\cT_U^* \xi
& = \int_U (\sigma_x \mid U) \, d\xi(x) + \int_U \sigma_x(U^c) \, d\xi(x) \frac{1}{J} \int_{U^c} (\sigma_z \mid U) \, d\eta(z)
\end{aligned}
$$

Observe that $\sigma_{U,x}$ is a probability on $U$:
$$
\sigma_{U,x}(U)
= \sigma_x(U) + \sigma_x(U^c) \frac{1}{I} \int_{U^c} \sigma_x(U) \, d\xi(x)
=  \sigma_x(U) + \sigma_x(U^c)  = 1.
$$
Moreover, by the definition of $J$,
$$
\begin{aligned}
\cT^*\eta_U
& = \int_U (\sigma_x \mid U) \, d\eta_U(x) + \int_U \sigma_x(U^c) \, d\eta_U(x) \frac{1}{J} \int_{U^c} (\sigma_z \mid U) \, d\eta(z)\\
& = \frac{1}{\eta(U)} \left(\int_U (\sigma_x \mid U) \, d\eta(x) +  \int_{U^c} (\sigma_z \mid U) \, d\eta(z)\right)
= \frac{1}{\eta(U)} \int_X (\sigma_x \mid U) \, d\eta(x).
\end{aligned}
$$
Each $\psi\in\Bd(U)$ may be viewed as an element of $\Bd(X)$ that vanishes outside $U$.
Then
%Thus, $\eta$ and each $\sigma_x$ may also be viewed as measures on $U$, coinciding with $\eta\mid U$ and $\sigma_x \mid U$, respectively.
%With these identifications,
%$$
%\int_X (\sigma_x \mid U) \, d\eta(x)
%= \int_X \sigma_x \, d\eta(x)
%= \cT^*\eta
%= \eta
%= (\eta \mid U).
%$$
$$
\begin{aligned}
\int_U \psi \, d\left(\cT_U^*\eta_U\right)
& = \frac{1}{\eta(U)} \int_X \int_U  \psi(y) \, d\sigma_x(y) \, d\eta(x)
= \frac{1}{\eta(U)} \int_X \int_X  \psi(y) \, d\sigma_x(y) \, d\eta(x)\\
& = \frac{1}{\eta(U)} \int_X \cT \psi(x) \, d\eta(x)
= \frac{1}{\eta(U)} \int_X \psi(x) \, d\eta(x)
= \int_U \psi \, d\eta_U.
\end{aligned}
$$
Thus, $\cT^*\eta_U = \eta_U$, as we wanted to prove.
\end{proof}

The operator $\cT_U$ in Proposition~\ref{p_sec6_localized1} need not be continuous, in general.
In the next proposition we fix that problem.

\begin{proposition}\label{p_sec6_localized2_strong}
Let $\cT:\Bd(X)\to\Bd(X)$ be a Markov operator, $\eta$ be a $\cT$-invariant measure, and
$U\subset X$ be such that $\eta(U)>0$. Assume that $x \mapsto \sigma_x$ is continuous
on $U$. Then there is a continuous Markov operator $\cT_U:\Bd(U)\to\Bd(U)$ such that the
normalized restriction $\eta_U$ is $\cT_U$-invariant.
\end{proposition}

\begin{proof}
The strategy is to consider a monotone family $\{U_t: t\in [0,1]\}$ of subsets of $U$ with
pairwise disjoint boundaries, and to associate to each $U_t$ a Markov operator $\cT_t$ such that
$\cT_t^*\eta_U = \eta_U$. These operators will still be discontinuous, but we can get rid of the
discontinuities by integrating with respect to $t$. The details follow.

As before, let $J$ be the number in \eqref{eq_sec6_localized1}. When $J=0$ there is nothing to do,
because in that case the construction in Proposition~\ref{p_sec6_localized1} does yield a continuous
Markov operator $\cT_U$. From now on, assume that $J>0$. Let $a>0$ be a small number.
For each $t\in [0,1]$, define $U_t=\{x\in U: d(x,U^c)\ge at\}$ and
\begin{equation*}%\label{eq_sec6_localized3}
J(t) = \int_U \sigma_x(U_t^c) \, d\eta(x).
\end{equation*}
Note that $J(t) \ge J > 0$. Then let $\xi_t$ be the probability measure defined on $U$ by
\begin{equation}\label{eq_sec6_localized4}
\int_U (\sigma_x \mid U_t) \, d\eta(x) + J(t) \xi_t = \eta \mid U.
\end{equation}
Observe that $\xi_t$ is well defined (each $\sigma_x \mid U_t$ may be viewed as a measure on $U$,
since $U_t\subset U$, and so all the terms in this identity are measures on $U$) and it is indeed a
probability measure.

Now let $\cT_t:\Bd(U)\to\Bd(U)$ be the Markov operator associated to the family
\begin{equation}\label{eq_sec6_localized8}
\sigma_{x,t} =  (\sigma_x \mid U_t) + \sigma_x(U_t^c) \xi_t
\end{equation}
of probability measures on $U$.
Condition \eqref{eq_sec6_localized4} means that $\eta_U$ is $\cT_t$-invariant:
\begin{equation}\label{eq_sec6_localized5}
\begin{aligned}
\cT_t^*\eta_U
& =  \int_U (\sigma_x \mid U_t) \, d\eta_U(x) + \int_U \sigma_x(U_t^c) \xi_t \, d\eta_U(x)\\
& =  \int_U (\sigma_x \mid U_t) \, d\eta_U(x) + \frac{J(t)}{\eta(U)} \xi_t = \eta_U.
\end{aligned}
\end{equation}
Next, define $\cT_U$ to be the Markov operator associated to the family of probability measures
\begin{equation}\label{eq_sec6_localized8b}
\sigma_{U,x} = \int_0^1 \sigma_{x,t} \, dt.
\end{equation}
It is clear from \eqref{eq_sec6_localized5} that $\eta_U$ is $\cT_U$-invariant.

We are left to show that the map
$x \mapsto \sigma_{U,x}$ is continuous with respect to the weak$^*$ topology, that is, that
\begin{equation}\label{eq_sec6_localized6}
x \mapsto \int \varphi \, d\sigma_{U,x}
= \int_0^1 \left(\int_{U_t} \varphi(y) \, d\sigma_x(y) + \sigma_x(U_t^c) \int_U \varphi(z) \, d\xi_t(z)\right) dt
\end{equation}
is continuous for any bounded continuous function $\varphi:U\to\RR$.
This will be a consequence of the following fact:

\begin{lemma}\label{l_sec6_elementary1}
Let $Z$ be a metric space, $\sigma_0$ be a probability measure on $Z$ and $g:Z\to\RR$ be a
measurable function such that the closure of the set of discontinuity points has zero $\sigma_0$-measure.
Then $\sigma_0$ is a continuity point of the map $\sigma \mapsto \int _Z g \, d\sigma$ in the space of
probability measures on $Z$ with the weak$^*$ topology.
\end{lemma}

\begin{proof}
Denote by $R$ the closure of the set of discontinuity points. Given $\vep>0$, let $V$ be an open neighborhood
of $R$ whose closure $\overline{V}$ satisfies $\sigma_0(\overline{V})<\vep$.
Then $\sigma(\overline{V})<\vep$ for any $\sigma$ in a weak$^*$ neighborhood of $\sigma_0$.
By the Tietze extension theorem, there exists a continuous function $h:Z\to\RR$ coinciding with $g$ on the
complement of $V$ and satisfying $\|h\|\le\|g\|$. Then
$$
\left|\int_Z g \, d\sigma - \int_Z g \, d\sigma_0\right|
\le \left|\int_Z h \, d\sigma - \int_Z h \, d\sigma_0\right| + 4 \|g\| \vep
\le \left(1+4\|g\|\right) \vep
$$
for any $\sigma$  in a weak$^*$ neighborhood of $\sigma_0$. Thus, $\sigma_0$ is a continuity point.
\end{proof}

Going back to proving the proposition, fix any $z\in U$ and consider $Z = X$ and $\sigma_0 = \sigma_{z}$.
Keep in mind that $x\mapsto \sigma_x$ is assumed to be continuous on $U$.
Thus, applying Lemma~\ref{l_sec6_elementary1} twice, to $g=\varphi \chi_{U_t}$
and to $g=\chi_{U_t^c}$, we see that $z$ is a point of discontinuity of
\begin{equation}\label{eq_sec6_localized7}
x \mapsto \int_{U_t} \varphi(y) \, d\sigma_x(y) + \sigma_x(U_t^c) \int_U \varphi(z) \, d\xi_t(z)
\end{equation}
only if the boundary $\partial U_t = \partial U_t^c$ has positive measure for $\sigma_{z}$.
Since these boundaries are pairwise disjoint when $t$ varies, the latter can only happen for countably many values of $t$.
Thus, we may apply Lemma~\ref{l_sec6_elementary2} to $Z=X$ and the function $g(x,t)$ given by the right-hand
side of \eqref{eq_sec6_localized7}, to conclude that \eqref{eq_sec6_localized6} is continuous.
\end{proof}

\begin{remark}\label{r_sec6_nonatomic}
The localization procedure in Propositions~\ref{p_sec6_localized1}
and~\ref{p_sec6_localized2_strong} preserves the class of non-atomic measures and,
when $X$ is an algebraic variety, also the class of generic measures.
That is a direct consequence of \eqref{eq_sec6_localized2},
 \eqref{eq_sec6_localized4}, \eqref{eq_sec6_localized8}, and
Remark~\ref{r_sec5_generic_elementary}.
\end{remark}

\subsection{Invariant couplings}\label{ss_invariant_couplings}

Let $X$ be a metric space $X$ and $\cT:\Bd(X)\to \Bd(X)$ be a Markov operator, given by
$$
\cT\psi(x) = \int_{X} \psi(y) \, d\sigma_x(y).
$$
A \emph{self-coupling} of $\cT$ is a Markov operator $\tT:\Bd(X \times X) \to \Bd(X \times X)$
of the form
\begin{equation}\label{eq_sec6_Markov_operator4}
\tT\tpsi(x,x') = \int_{X \times X'} \tpsi(y,y') \, d\tsigma_{x,x'}(y,y')
\end{equation}
where each $\tsigma_{x,x'}$ is a coupling of $\sigma_x$ and $\sigma'_{x'}$.
The self-coupling $\tT$ is \emph{continuous} if the map $(x,x') \mapsto \sigma_{x,x'}$ is continuous
on $X \times X$.

\begin{lemma}\label{l_sec6_coupling_iterate}
If $\teta$ is a coupling of $\eta$ and $\eta'$ and $\tT$ is a self-coupling of $\cT$ then
$\tT^*\teta$ is a coupling of $\cT^*\eta$ and $\cT^*\eta'$.
\end{lemma}

\begin{proof}
Let $\tpsi:X \times X \to \RR$ be any bounded measurable function that depends only on the first
variable: $\tpsi(x,x')=\psi(x)$ for some $\psi\in \Bd(X)$. By definition,
$$
\begin{aligned}
\int_{X \times X} \tpsi \, d(\tT^*\teta)
&= \int_{X \times X} \int_{X \times X} \tpsi(y,y') \, d\tsigma_{x,x'}(y,y')  \, d\teta(x,x')\\
& = \int_{X \times X} \int_{X \times X} \psi(y) \, d\tsigma_{x,x'}(y,y')  \, d\teta(x,x').
\end{aligned}
$$
Since $\tsigma_{x,x'}$ projects to $\sigma_x$ and $\teta$ projects to $\eta$ on the first factor,
this last expression may be written as
$$
\int_{X \times X} \int_X \psi(y) \, d\sigma_x(y)  \, d\teta(x,x')
= \int_X \int_X \psi(y) \, d\sigma_x(y)  \, d\eta(x)
 = \int_X \psi \, d(\cT^*\eta).
$$
This proves that $\tT^*\teta$ projects to $\cT^*\eta$ on the first factor.
Analogously, it projects to $\cT^*\eta'$ on the second factor.
\end{proof}

\begin{lemma}\label{l_sec6_coupling_invariant_exists}
Assume that $X$ is compact and $\tT$ is continuous, and let $\eta$ and $\eta'$ be
$\cT$-invariant probability measures on $X$.
Given any coupling $\teta_0$ of $\eta$ and $\eta'$, every accumulation point $\teta$ of the sequence
\begin{equation*}%\label{eq_sec6_tetan}
\teta_n = \frac{1}{n} \sum_{j=0}^{n-1} \tT^{n*}\teta_0
\end{equation*}
is a $\tT$-invariant coupling of $\eta$ and $\eta'$. In particular, $\tT$-invariant couplings do exist.
\end{lemma}

\begin{proof}
By Lemma~\ref{l_sec6_coupling_iterate}, every $\teta_n$ is a coupling of $\eta$ and $\eta'$.
By compactness, there exists $(n_i)_i\to\infty$ such that $(\teta_{n_i})_i$ converges to some
$\teta$ in the weak$^*$ topology. Clearly, $\teta$ is still a coupling of $\eta$ and $\eta'$.
Let $\tpsi: X \times X \to \RR$ be any bounded continuous function.
The assumption ensures that $\tT\tpsi$ is also continuous. Thus,
$$
\begin{aligned}
\int_{X \times X}\left(\tT\tpsi - \tpsi\right)\,d\teta
& = \lim_i \int_{X \times X}\left(\tT\tpsi - \tpsi\right)\,d\teta_{n_i}\\
& = \lim_i \int_{X \times X} \frac{1}{n_i} \left(\tT^{n_i}\tpsi - \tpsi\right)\,d\teta_0 = 0
\end{aligned}
$$
(recall that $\|\tT^{n}\tpsi\| \le \|\tpsi\|$ for every $n$). Thus, $\tT^*\teta=\teta$,
as we wanted to prove.
\end{proof}

\begin{remark}\label{r_sec6_symmetric_preserved}
A self-coupling $\tT$ is \emph{symmetric} if $\sigma_{x',x} = \iota_*\sigma_{x,x'}$ for all $x,x' \in X$.
If $\tT$ is symmetric and $\teta$ is a symmetric self-coupling of $\eta'$ then
$\tT^*\teta$ is a symmetric self-coupling of $\tT^*\eta$. Moreover, the $\tT$-invariant self-coupling
$\teta$ in Lemma~\ref{l_sec6_coupling_invariant_exists} may be taken to be symmetric.
\end{remark}

\subsection{Margulis functions}\label{ss_Margulis_functions}

As before, let
$$
\cT:\Bd(X)\to \Bd(X),\quad \cT\psi(x) = \int_X\psi(y) \, d\sigma_x(y)
$$
be a Markov operator on a metric space $X$. Let $X=A \cup B$ be a partition of $X$ into disjoint
sets $A$ and $B$. An \emph{(additive) Margulis function} for $\cT$ relative to $(A,B)$ is a
measurable function $\Psi:X\to [0,\infty]$ such that there exist $\kappa_A>0$ and $\kappa_B>0$
such that
\begin{align}
\cT\Psi(x) & \le \Psi(x) - \kappa_A \text{ for every } x\in A \label{eq_sec6_Margulis2}\\
\cT\Psi(x) & \le \Psi(x) + \kappa_B \text{ for every } x\in B. \label{eq_sec6_Margulis3}
\end{align}
($\Psi$ is usually not bounded, but its image under $\cT$ is easily defined using monotone convergence: let $\cT\Psi = \lim_n \cT(\min\{\Psi, n\})$.
Then $\int_X \cT\Psi \, d\zeta = \int \Psi \, d\left(\cT^*\zeta\right)$ for any probability measure
$\zeta$ on $X$.)
We make following technical assumption, which is used in the context of \eqref{eq_sec6_semi-continuous}:
there exists $L>0$ such that $\Psi$ is lower semi-continuous on $\Psi^{-1}([L,\infty])$.

Margulis functions are a very effective tool for estimating the spatial distribution of $\cT$-invariant
measures. The simple lemma that follows illustrates this idea:

 \begin{lemma}\label{l_sec6_Margulis_at_work1}
Let $\Psi:X\to[0,\infty]$ be a Margulis function for a Markov operator $\cT$
relative to a partition $(A,B)$.
Let $\zeta$ be any measure on $X$ such that $\int\Psi \, d\zeta < \infty$ and
$\int_X \cT\Psi \, d\zeta \ge \int_X \Psi \, d\zeta$. Then
\begin{equation}\label{eq_sec6_Margulis4}
\zeta(B) \ge \frac{\kappa_A}{\kappa_A+\kappa_B}\zeta(X).
\end{equation}
In particular, this holds if $\zeta$ is $\cT$-invariant and satisfies $\int\Psi \, d\zeta < \infty$.
\end{lemma}

\begin{proof}
We have
$$
\int_X \Psi \, d\zeta \le \int_X \cT\Psi \, d\zeta \le \int_X \Psi \, d\zeta -\kappa_A\zeta(A) + \kappa_B\zeta(B).
$$
Thus, $\kappa_A\zeta(A) - \kappa_B\zeta(B) \le 0$, which is the same as \eqref{eq_sec6_Margulis4}.
To get the last claim, just note that $\int_X \cT\Psi \, d\zeta = \int_X \Psi \, d\zeta$ if $\zeta$ is
$\cT$-invariant.
\end{proof}

\begin{remark}\label{r.additive_multiplicative}
Given a set $Y\subset X$, we call \emph{multiplicative Margulis function} for $\cT$ relative to
$(X,Y)$ any measurable function $\Phi:X\to[1,\infty]$ such that $\log\Phi$ is uniformly continuous,
$\Phi(x) = \infty$ if and only if $x\in Y$, $\Phi$ is a proper function on $X\setminus Y$,
and there exist constants $c < 1$ and $b<\infty$ such that
\begin{equation}\label{eq_sec6_Margulis1}
\cT\Phi(x) \le c \Phi(x) + b
\text{ for all $x\in X$.}
\end{equation}
If $\Phi$ is a multiplicative Margulis function then $\log\Phi$ is an additive  Margulis function
relative to the partition $(A,B)$ defined by
$$
A = \{x\in X: \Phi(x) > \alpha\} \quand B = \{x\in X: \Phi(x) \le \alpha\}
$$
for any $\alpha>b/(1-c)$. Indeed, the Jensen inequality implies that
$$
\cT\log\Phi(x) \le \log \cT\Phi(x) \le \log(c\Phi(x)+b)
\text{ for every $x$.}
$$
Moreover,
$$
\log(cy+b) \le \left\{\begin{array}{ll}\log y + \log(c+b) & \text{for every $y\ge 1$}\\
                                       \log y + \log(c+b/\alpha) & \text{if $y>\alpha$.}\end{array}\right.
$$
Thus, we may take $\kappa_A = - \log(c+b/\alpha)>0$ and any $\kappa_B\ge\log(c+b)$.
On the other hand, as was already pointed out in the Introduction, it is not true that if $\Psi$ is an
additive Margulis function then $\exp\Psi$ is a multiplicative one.
%In fact, the inequality \eqref{eq_sec6_Margulis1} is very sensitive to the ``worst case'' behavior of $\Psi$
%in the support of the measure $\tau_x$, whereas \eqref{eq_sec6_Margulis2} and \eqref{eq_sec6_Margulis3}
%depend more on the ``average case'' behavior. For this reason, it is often much easier to construct an
%additive Margulis function than a multiplicative one. In fact, we do not know how to construct a useful
%multiplicative Margulis function in our setting beyond the case $d=2$ (see \cite[Chapter~10]{LLE} and \cite{MaV15}).
\end{remark}

\begin{proposition}\label{p_sec6_Margulis_finite_energy}
Assume that $X$ is compact and let $\tT$ be a continuous self-coupling of $\cT$.
Let $\Psi:X \times X\to[0,+\infty]$ be a
Margulis function for $\tT$ which is bounded outside every neighborhood of the diagonal,
and let $\eta$ be a non-atomic $\cT$-invariant measure on $X$.
Then one can find a $\tT$-invariant self-coupling $\teta$ of $\eta$ and a sequence
$(\teta_j)_j$ of self-couplings of $\eta$ converging to $\teta$ in the weak$^*$ topology
and such that
\begin{equation}\label{eq_sec6_Margulis_finite_energy}
\int_{X \times X} \Psi \, d\teta_j <\infty \quand
\int_{X \times X} \tT\Psi \, d\teta_j \ge \int_{X \times X}  \Psi \, d\teta_j
\text{ for every $j$.}
\end{equation}
\end{proposition}

\begin{proof}
By Corollary~\ref{c_sec6_coupling_diagonal}, there exists some self-coupling $\heta_0$
of $\eta$ that vanishes on a neighborhood of the diagonal. Then $\int_{X \times X} \Psi \, d\heta_0$
is finite. Conditions \eqref{eq_sec6_Margulis2} and \eqref{eq_sec6_Margulis3} imply that
$\tT^j\Psi (x) \le \Psi(x) + j \kappa_B$ for every $x\in X$, and so
$$
\int_{X \times X} \tT^j\Psi \, d\heta_0 \le  \int_{X \times X} \Psi \, d\heta_0 +  j\kappa_B < \infty
$$
for every $j$. Let $\teta = \lim_i \heta_{n_i}$ be any weak$^*$ accumulation point of the sequence
\begin{equation*}%\label{eq_sec6_tetan}
\heta_n = \frac{1}{n} \sum_{j=0}^{n-1} \tT^{n*}\heta_0.
\end{equation*}
As noted in Lemma~\ref{l_sec6_coupling_invariant_exists},
every $\heta_n$ is a self-coupling of $\eta$ with $\int_{X \times X} \Psi \, d\heta_n < \infty$,
and $\teta$ is a $\tT$-invariant self-coupling of $\eta$.

If $\int_X \Psi \, d\teta$ is finite then the claim follows by taking $\teta_j=\teta$ for every $j$.
In this case the equality holds on the second part of \eqref{eq_sec6_Margulis_finite_energy}.
Now suppose that $\int_X \Psi \, d\teta$ is infinite.
By the lower semi-continuity assumption on $\Psi$, this implies that
\begin{equation}\label{eq_sec6_semi-continuous}
\int_X \Psi \, d\heta_{n_i} \to \infty \text{ as } i\to\infty.
\end{equation}
Then $\int_X \Psi \, d\left(\tT^{*n}\heta_0\right)$ must be unbounded.
In particular, one can find $(m_j)_j\to\infty$ such that
$$
\int_{X \times X} \tT\Psi \, d\heta_{m_j} - \int_{X \times X} \Psi \, d\heta_{m_j}
= \frac{1}{m_j} \left(\int_{X \times X} \Psi \, d\left(\tT^{*m_j}\heta_0\right) - \int_{X \times X}\Psi \, d\heta_0\right)
\ge 0.
$$
Thus, it suffices to take $\teta_j = \heta_{m_j}$ for every $j$.
\end{proof}

\subsection{Adapted operators}\label{ss_adapted_operators}

Let $Z$ be an algebraic variety endowed with a continuous $G$--action $(g,z) \mapsto gz$.
Let $\nu$ be a compactly supported probability measure on $G$,
and $\eta$ be a $\nu$-stationary probability measure on $Z$.
For each $z\in Z$, let $\nu_z$ denote the push-forward of $\nu$ under the map $g \mapsto gz$.

Let $X\subset Z$ and $a_X>0$ be some small number.
Consider the family of subsets $X_t=\{x\in X: d(x,Z \setminus X)\ge a_X t\}$, $t\in [0,2]$.
Note that $t \mapsto X_t$ is monotone decreasing. Define the \emph{$\nu$--core} of $X$ to be
\begin{equation}\label{eq_sec6_core}
\Chi_\nu X = \left\{x \in X_2: gx \in X_2 \text{ for all } g\in\supp\nu\right\}.
\end{equation}
The complement $\partial_\nu X = X \setminus\Chi_\nu X$ is called the \emph{$\nu$--border of $X$}.
Define also
\begin{equation}\label{eq_sec6_core2}
\Chi^\#_\nu X = \left\{x \in \Chi_\nu X: g^{-1}x \in \Chi_\nu X \text{ for all } g\in\supp\nu \right\}.
\end{equation}
It is clear that $\Chi^\#_\nu X \subset \Chi_\nu X \subset X_2 \subset X$. Moreover,
\begin{align}
\nu_x\left(X \setminus X_2\right)&=0 \text{ for all $x\in\Chi_\nu X$, and }
\label{eq_def_adapted_core}\\
\nu_x\left(\Chi^\#_\nu X\right)&=0 \text{ for all $x \in \partial_\nu X$.}
\label{eq_def_adapted_border}
\end{align}

\begin{example}\label{ex_sec6_core}
In our initial applications, $Z$ will be a Grassmannian manifold $\grass(r,d)$ and $X$ will be the
closed $\vep$-neighborhood $E_r(\vep)$ of the subset of $r$-dimension\-al subspaces of $\RR^d$
contained in the equator $E$. Later we will also take $Z$ to be a flag variety $\cF(r,d)$ and $X$
to be the closed subset $\fE_r(\vep)$ of flags whose $r$-coordinate $F_r$ is in $E_r(\vep)$.
We will always take $a_X=\vep/100$, which means that $X_t=E_r((1-t/100)\vep)$ and
$X_t=\fE_r((1-t/100)\vep)$, respectively, for all $t \in [0,2]$.
\end{example}

A Markov operator $\cT:\Bd(X) \to \Bd(X)$ is said to be \emph{adapted to $(\nu,X)$} if the
associated family of probability measures $\{\sigma_x: x \in X\}$ satisfies
\begin{itemize}
\item[(a)] $\sigma_x=\nu_x$ for every $x$ in a neighborhood of the $\nu$-core $\Chi_\nu X$;
\item[(b)] $\sigma_x(\Chi^\#_\nu X) = 0$ for every $x$ in the $\nu$-border $\partial_\nu X$;
\item[(c)] $\sigma_x$ is generic for every $x\in X$.
\end{itemize}

The assumption of the next proposition implies that $x \mapsto\sigma_x$ is continuous on $U$,
and so Proposition~\ref{p_sec6_localized2_strong} does hold in this setting.
The assumption is satisfied, in particular, if the Markov operator $\cT$ is adapted to $(\nu,X)$
and $U$ is a subset of the $\nu$-core of $X$.

\begin{proposition}\label{p_sec6_adapted_strong}
Let $\cT:\Bd(X)\to\Bd(X)$ be a Markov operator, $\eta$ be a $\cT$-invariant measure, and
$U\subset X$ be such that $\eta(U)>0$. Suppose that
\begin{itemize}
\item[(a)] $\sigma_x = \nu_x$ for every $x \in U$;
\item [(b)] $\sigma_x(\Chi^\#_\nu U)=0$ for every $x \notin U$;
\item[(c)] $\sigma_x$ is a generic measure for every $x\in X$.
\end{itemize}
Then the continuous Markov operator $\cT_U:\Bd(U)\to\Bd(U)$
given by Proposition~\ref{p_sec6_localized2_strong} is adapted to $(\nu,U)$.
\end{proposition}

\begin{proof}
Let $U^c = X \setminus U$ and $U_t^c = X \setminus U_t$. Recall that $\cT_U$
is given by the probability measures $\sigma_{U,x}$ defined in \eqref{eq_sec6_localized8b}.
Since $\eta = \int_X \sigma_y \, d\eta(y)$, because $\eta$ is assumed to be $\cT$-invariant,
the definition \eqref{eq_sec6_localized4} of $\xi_t$ means that
$$
\xi_t = \frac{1}{J(t)} \left(\int_{U^c} (\sigma_y \mid U) \, d\eta(y)
+ \int_U \left(\sigma_y \mid U \cap U^c_t \right) \, d\eta(y)\right),
$$
and so
\begin{equation}\label{eq_sec6_localized9}
\sigma_{x,t} = (\sigma_x \mid U_t) + \frac{\sigma_x(U_t^c)}{J(t)} \left(\int_{U^c} (\sigma_y \mid U) \, d\eta(y) + \int_U \left(\sigma_y \mid U \cap U^c_t \right) \, d\eta(y)\right).
\end{equation}
Condition (c) in the definition of an adapted operator follows directly from
Remark~\ref{r_sec6_nonatomic}. Let us check conditions (a) and (b).

If $x \in \Chi_\nu U$ then $gx \in U_2$ for every $g\in\supp\nu$.
Since $U_1$ is a neighborhood of $U_2$, and the support of $\nu$ is compact,
it follows that there exists a neighborhood $V$ of $\Chi_\nu U$ such that $gx \in U_1$
for every $x \in V$ and $g\in\supp\nu$. Thus $\sigma_x(U_1^c)=\nu_x(U_1^c)=0$,
and so,  for all $t\in [0,1]$, the second term on the right hand side of
\eqref{eq_sec6_localized9} vanishes, which means that $\sigma_{x,t} = \sigma_x = \nu_x$.
Integrating with respect to $t$, we find that $\sigma_{U,x} = \sigma_x  = \nu_x$ for all
$x\in V$, which proves condition (a).

Now consider $x \in \partial_\nu U$. We claim that all three terms on the right-hand side of
\eqref{eq_sec6_localized9} vanish on $\Chi^\#_\nu U$. Initially, \eqref{eq_def_adapted_border}
implies that $\sigma_x(\Chi^\#_\nu U) = \nu_x(\Chi^\#_\nu U) = 0$, which sets the claim for the
first term. The case of the second term is an immediate consequence of the assumption that
$\sigma_y(\Chi^\#_\nu U) = 0$ for every $y \in U^c$. Finally,
$(\sigma_y \mid U \cap U^c_t)(\Chi^\#_\nu U)=0$ for every $y\in U$ and $t \in [0,1]$, just because
$\Chi^\#_\nu U$ is contained in $U_t$. Hence, the third term is also zero on $\Chi^\#_\nu U$.
This proves that $\sigma_{x,t}(\Chi^\#_\nu U)=0$ for all $t\in[0,1]$, which implies that
$\sigma_{U,x}$ satisfies (b).
\end{proof}

\begin{remark}\label{r_sec6_notdeep}
By definition $\nu_x(\Chi^\#_\nu U)=0$ for every $x \notin U$.
Compare conditions (a) and (b) in Proposition~\ref{p_sec6_adapted_strong}.
\end{remark}

\section{Outline of the proof}\label{s_outline_of_the_proof}

Given any subspaces $U$ and $V$ of $\RR^d$, define
\begin{equation}\label{eq_sec7_Theta1}
d(U,V) = \sup_{u \in U}\inf_{v\in V} d(u,v) = \sup_{u \in U}\inf_{v\in V} |\sin\angle(u,v)|.
\end{equation}
Then $0 \le d(U,V) \le 1$, with $d(U,V)=0$ if and only if $U\subset V$ and $d(U,V)=1$ if and only if $U \cap V^\perp \neq \{0\}$.
In particular,
\begin{equation}\label{eq_sec7_Theta2}
\dim U > \dim V \quad \Rightarrow\quad d(U,V) = 1.
\end{equation}
It is also clear that $d(U_2,V_1) \ge d(U_1,V_2)$ whenever $U_1\subset U_2$ and $V_1 \subset V_2$.

The function $d(\cdot,\cdot)$ in \eqref{eq_sec7_Theta1} is clearly not symmetric, in general.
However, its restriction to each Grassmannian manifold $\grass(r,d)$ turns out to be a distance,
invariant under the action of the orthogonal group $\OO(d)$ on the
Grassmannian. The case $r=1$ is just \eqref{eq_sec5_distance_P}.

Let $E_r \subset\grass(r,d)$ be the set of $r$-dimensional subspaces contained in the equator $E$.
In particular, $\grass(1,d)=P$ and $E_1=E$.
For $\vep>0$, let $E_r(\vep)\subset\grass(r,d)$ be the closed $\vep$-neighborhood of $E_r$.
We also denote $E_r(\vep,\vep')=E_r(\vep)\setminus E_r(\vep')$ for $0 < \vep' < \vep$.
For each $X=E_r(\vep)$, we always take $a_X=\vep/100$ in the definition \eqref{eq_sec6_core}
of the $\nu$--core $\Chi_\nu X$ and the $\nu$--border $\partial_\nu X$.

%%%
%Positivity. We already know that $d(U,V)\ge 0$ and $d(U,V)=0$ if and only if $U=V$.
%
%Triangle inequality. Given $u\in U$ take $v_u\in V$ (all vectors are unit vectors) such that
%$$
%d(u,v_u)=\inf_v d(u,v).
%$$
%Then for any $w\in W$,
%$$
%d(u,w) \le d(u,v_u)+d(v_u,w) = \inf_v d(u,v)+d(v_u,w).
%$$
%Thus,
%$$
%\inf_w d(u,w) \le \inf_v d(u,v) + \inf_w d(v_u,w)
%$$
%and so
%$$
%\sup_u \inf_w d(u,w)
%\le \sup_u \inf_v d(u,v) + \sup_v \inf_w d(v,w).
%$$
%This means that $d(U,W) \le d(U,V) + d(V,W)$ for any $U$, $V$, $W$.
%
%Symmetry. Note that $U \cap V^\perp \neq \{0\}$ if and only if
%$V \cap U^\perp \neq\{0\}$. Thus, $d(U,V)=\pi/2$ if and only if $d(U,V)=\pi/2$.
%From now on, assume that $U\cap V^\perp =\{0\}$.
%Then $U$ is the graph of a linear map $\phi:V\to V^\perp$: every $u\in U$ may be written (uniquely) as $u=v+\phi(v)$ with $v\in V$. Then,
%$$
%d(u,V) = \angle(u,v) = \arctan\frac{\|\phi(v)\|}{\|v\|}
%\quand
%d(v, U) = \angle(u,v) = \arctan\frac{\|\phi(v)\|}{\|v\|}.
%$$
%Taking the supremum, we get that both $d(U,V)$ and $d(V,U)$ are equal to $\arctan \|\phi\|$.

\subsection{Main inductive statement}\label{ss_reduction_to_an_inductive_statement}

We are going to prove Theorem~\ref{theorem:second} by contradiction:
suppose that there exists a discontinuity point $\nu_\infty$ for the largest Lyapunov exponent
$\lambda_1$. Then, as we saw in Section~\ref{s_random_repeller} (Proposition~\ref{p_sec5_GM}
and Corollary~\ref{c_sec5_Zariski2}), there exists a sequence $(\nu_k)_k$ of generic measures
with Zariski-dense supports converging to some $\nu_\infty$ on $G$ and such that their
(unique) stationary measures $(\eta_k)_k$ on $P$ converge to a $\nu_\infty$-stationary measure
$\eta_\infty$ satisfying $\eta_\infty(E)>0$.

\begin{theorem}\label{theorem:inductive}
Let $1\le r \le d$ and suppose that there exist constants $\vep_r>0$ and $n_r\in\NN$ such that:
\begin{itemize}
\item[(i)] For each $k\in\NN$ there is a continuous Markov operator
$$
\cT_{k,r}:\Bd(E_r(\vep_r)) \to \Bd(E_r(\vep_r)), \quad
\cT_{k,r}\varphi(x)  =\int_{E_r(\vep_r)} \varphi(y) \, d\sigma_{k,r,x}(y)
$$
adapted to $(\nu_k^{(n_r)},E_r(\vep_r))$.
\item[(ii)] For each $k\in\NN$ there is a $\cT_{k,r}$-invariant probability measure $\eta_{k,r}$ on
$E_r(\vep_r)$ such that $\eta_{\infty,r}=\lim_k \eta_{k,r}$ exists and satisfies $\eta_{\infty,r}(E_r)>0$.
\end{itemize}
If $r < \dim E$ then there exist constants $\vep_{r+1}$ and $K_{r+1}\in\NN$ such that (i) and (ii)
hold when $r$ is replaced with $r+1$. If $r=\dim E$ then (i) and (ii) cannot happen.
\end{theorem}

By Remark~\ref{r_sec6_must_be_generic}, the invariant measures $\eta_{k,r}$ are automatically generic measures.

\begin{remark}
Since the support of $\nu_k$ is Zariski-dense, it follows from Proposition~\ref{p_sec5_GM} that
for each $1\le r \le d$ there exists a unique $\nu_k$-stationary probability measure
$\heta_{k,r}$ in $\grass(r,d)$.
However, even if we assume that there exists a subsequence along which $(\heta_{k,1})_k$
converges to a measure $\heta_{\infty,1}$ with $\heta_{\infty,1}(E)>0$, it is not clear that for any
$r>1$ the  sequence $(\heta_{k,r})_k$ admits a subsequence converging to some measure
$\heta_{\infty,r}$ with $\heta_{\infty,r}(E_r)>0$. Thus, Theorem~\ref{theorem:inductive} is not
immediately obvious. Indeed, our measures $\eta_{k,r}$ on $E_r(\vep_r)$ are \emph{not}
the normalized restrictions of the $\heta_{k,r}$.
\end{remark}

To deduce Theorem~\ref{theorem:second}, note that the assumptions of Theorem~\ref{theorem:inductive} hold for $r=1$, with $n_1=1$ and $\vep_1=1$,
so that $E_1(\vep_1)=P$,  and
$$\cT_{k,1}=\cP_{\nu_k}: \Bd(P) \to \Bd(P)
\text{ and } \eta_{k,1} = \eta_k \text{ for $k\in\NN$.}
$$
Indeed, it is clear that $\cP_{\nu_k}$ is adapted to $(\nu_k,P)$ and leaves $\eta_k$ invariant.
Recall that $\nu_k$ is taken to be generic for $k\in\NN$. Moreover, the limit
$\eta_{\infty,1}=\eta_\infty$ satisfies $\eta_\infty(E)>0$. Then we can iterate
Theorem~\ref{theorem:inductive} until we arrive at the case $r=\dim E$,
which leads to a contradiction, which proves Theorem~\ref{theorem:second}.

The proof of Theorem~\ref{theorem:inductive} occupies what is left of this paper.
In the remainder of the present section we outline the main ideas.
Initially, we discuss the case $r=1$, which involves many of the ingredients of the general step,
though not all. Then we hint at how these ideas can be extended to $r>1$.

Throughout, $E_1(\vep)^2 = E_1(\vep) \times E_1(\vep)$ and
$E_1(\vep,\vep')^2 = E_1(\vep,\vep') \times E_1(\vep,\vep')$ for any $0< \vep'<\vep$.
The following simple fact will be used a few times:

\begin{remark}\label{r_sec7_repeatedly}
Given any $\beta>0$ and $n\in\NN$, the neighborhood $E_r(\alpha)$ is contained in the
$\nu_\infty^{(n)}$-core of $E_r(\beta/2)$ for all $1 \le r \le d$ if $\alpha$ is sufficiently small,
depending only on $\nu_\infty$, $n$ and $\beta$.
This is because the equator $E$ is $\nu_\infty$-invariant.
Since $\supp\nu_k^{(n)}$ converges to $\supp\nu_\infty^{(n)}$, it follows that if $\alpha>0$ is
sufficiently small and $k\in\NN$ is large enough, depending only on $\nu_\infty$, $n$ and $\beta$,
then $E_r(\alpha)$ is contained in the $\nu_k^{(n)}$-core of $E_r(\beta)$ for all $1 \le r \le d$.
\end{remark}

\subsection{First step}\label{ss_step1_outline}
Let us consider constants $\vep_1>0$ and $n_1\in\NN$, continuous Markov operators
$$
\cT_{k,1}:\Bd(E_1(\vep_1))\to \Bd(E_1(\vep_1)), \quad
\cT_{k,1}\varphi(x) = \int_G \varphi(y) \, d\sigma_{k,1,x}(y)
$$
and $\cT_{k,1}$-invariant measures $\eta_{k,1}$ as in Theorem~\ref{theorem:inductive}.
We may start from $n_1=1$ and $\vep_1=1$, as in the previous section,
but along the way we replace the operators with convenient iterates, which means that
$n_1$ is increased, and we also localize them to suitable neighborhoods of the equator $E$,
using Propositions~\ref{p_sec6_localized2_strong} and~\ref{p_sec6_adapted_strong} and
Remark~\ref{r_sec7_repeatedly}, which entails reducing $\vep_1$.

The way we set this is by replacing $n_1$ with a variable $n\in\NN$, which we think of
as a free parameter, and by taking $\cT_{k,1}$ and $\vep_1$ as functions of $n$.
The conditions on $\vep_1$ are given in \eqref{eq_sec7_um01} and \eqref{eq_r=1_vep1small},
depending on $\nu_\infty$, $\delta$ and $n$. The condition on $n$ is stated only near the end
of the construction, in \eqref{eq_r=1_defdeltan}, depending on $\nu_\infty$ and $\delta$.

\medskip\noindent{\bf 1.}
Up to reducing $\vep_1$ if necessary, it is no restriction to assume that
\begin{equation}\label{eq_sec7_um01}
\eta_{\infty,1}(E_1(\vep_1)\setminus E) < \frac{1}{10}\eta_{\infty,1}(E).
\end{equation}
That may be seen as follows.
Since $\eta_{\infty,1}(E_1(\alpha))$ converges to $\eta_{\infty,1}(E)$ as $\alpha\to 0$, and the limit
is assumed to be positive, we have
\begin{equation}\label{eq_sec7_um00a}
\eta_{\infty,1}(E_1(\alpha) \setminus E) < \frac{1}{10} \eta_{\infty,1}(E)
\end{equation}
for every small $\alpha > 0$. By Remark~\ref{r_sec7_repeatedly},
\begin{equation}\label{eq_sec7_um00c}
E_1(\alpha) \text{ is contained in the $\nu_k^{(n)}$-core of $E_1(\vep_1)$}
\end{equation}
for every large $k$ and every small $\alpha > 0$.
Clearly,
\begin{equation}\label{eq_sec7_um00b}
\eta_{\infty,1}(\partial E_1(\alpha))=0
\end{equation}
for all but countably many values of $\alpha$.
Fix $\alpha>0$, depending only on $\nu_\infty$ and $\vep_1$,
satisfying \eqref{eq_sec7_um00a}, \eqref{eq_sec7_um00c} and \eqref{eq_sec7_um00b}.

Applying Proposition~\ref{p_sec6_adapted_strong} with $X=E_1(\vep_1)$, $U = E_1(\alpha)$,
$\nu = \nu_k^{(n)}$ and $\cT=\cT_{k,1}$, for $k$ large, we get a continuous Markov operator
$$
\cT'_{k,1}: \Bd(E_1(\alpha)) \to \Bd(E_1(\alpha))
$$
adapted to $(\nu_k^{(n)}, E_1(\alpha))$
and leaving invariant the normalized restriction $\eta'_{k,1}$ of $\eta_{k,1}$ to $E(\alpha)$.
Replace $\vep_1$, $\cT_{k,1}$, and $\eta_{k,1}$ with $\alpha$, $\cT'_{k,1}$, and $\eta'_{k,1}$,
respectively. Conditions (i) and (ii) in Theorem~\ref{theorem:inductive} are not affected by
this: in particular, observe that \eqref{eq_sec7_um00b} ensures that $(\eta'_{k,1})_k \to \eta'_{\infty,1}$.
Then \eqref{eq_sec7_um01} just corresponds to \eqref{eq_sec7_um00a}.

%Moreover, by Proposition~\ref{p_sec6_adjusted_coupling} there exists a self-coupling
%\begin{equation}\label{eq_sec7_um06}
%\begin{aligned}
%& \tT_{k,1}:\Bd(\prodspa1) \to \Bd(\prodspa1) \\
%& \tT_{k,1}\tvarphi(x,x') = \int_{\prodspa1} \tvarphi(y,y') \, d\sigma_{k,1,x,x'}(y,y')
%\end{aligned}
%\end{equation}
%of $\cT'_{k,1}$, continuous and adjusted to
%$(\nu_k^{(n)},E_1(\vep_1))$. In particular,
%\begin{equation}\label{eq_sec7_um07}
%\tT_{k,1}\tvarphi(x,x') = \int_{G} \tvarphi(gx,gx') \, d\nu_k^{(n)}(g)
%\end{equation}
%when $x$ and $x'$ are both in the $\nu_k^{(n)}$-core of $E_1(\vep_1)$.
We also introduce a suitable continuous self-coupling
\begin{equation}\label{eq_sec7_um06}
\begin{aligned}
& \tT_{k,1}:\Bd(\prodspa1) \to \Bd(\prodspa1) \\
& \tT_{k,1}\tvarphi(x,x') = \int_{\prodspa1} \tvarphi(y,y') \, d\tsigma_{k,1,x,x'}(y,y')
\end{aligned}
\end{equation}
of the operator $\cT_{k,1}$ such that
\begin{equation}\label{eq_sec7_um07}
\tT_{k,1}\tvarphi(x,x') = \int_{G} \tvarphi(gx,gx') \, d\nu_k^{(n)}(g)
\end{equation}
when $x$ and $x'$ are both close to the equator $E$.

\medskip\noindent{\bf 2.}
Consider $\vep'''_1\in(0,\vep_1)$ and a sequence $(\omega_{k,1})_k$ decreasing to zero.
Property \eqref{eq_sec7_um01} ensures that
\begin{equation}\label{eq_sec7_um02}
\eta_{k,1}(E_1(\vep_1,\vep'''_1)) < \frac{2}{10} \eta_{k,1}(E_1(\vep_1))
\end{equation}
for every large $k$. Define $A_k$ and $B_k=B' \cup B''_k$ through
\begin{align}
\label{eq_sec7_r=1_A}
A_k & = \{(x,x') \in E_1(\vep'''_1)^2: d(x+x', E) > \omega_{k,1}\}\\
\label{eq_sec7_r=1_Bprime}
B' & = \{(x,x') \in \prodspa1: d(x, E) > \vep'''_1 \text{ or } d(x',E) > \vep'''_1\}\\
\label{eq_sec7_r=1_Bsecond}
B''_k & = \{(x,x')\in \prodspa1: d(x+x', E) \le \omega_{k,1}\}.
\end{align}
It is clear that $A_k$ and $B_k=B' \cup B''_k$ are disjoint, and their union is the whole $\prodspa1$.
For every large $k$, we have $\omega_{k,1} < \vep'''_1$ and, in view of the definition
\eqref{eq_sec7_Theta1}, that ensures that $B'$ and $B''_k$ are also disjoint.
Moreover, \eqref{eq_sec7_Theta2} implies that  $B''_k=\emptyset$ if $\dim E=1$.

\begin{figure}[ht]
\begin{center}
\psfrag{E5}{$E_1(\vep_1)^2$}\psfrag{E3}{$E_1(\vep'''_1)^2$}
\psfrag{A}{$A_k$}\psfrag{B1}{$B'$}\psfrag{B2}{$B''_k$}
\includegraphics[height=2in]{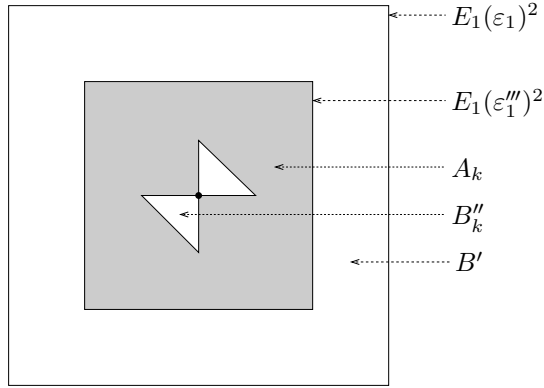}
\caption{\label{f_partition}
A sketch of the partition of $\prodspa1$ into the sets $A_k$, $B'$, and $B''_k$.
The latter converges to the central point $(E,E)$ when $k\to\infty$,
because $\omega_{k,1}\to 0$.}
\end{center}
\end{figure}

\emph{Suppose} that for $n\in\NN$ sufficiently large there exist constants
$\kappa_A, \kappa_B >0$ with $\kappa_A > 9 \kappa_B$,
and for each large $k\in\NN$ there exists a lower semi-continuous function
$\Psi_{k,1}:\prodspa1\to\RR$ such that
\begin{align}
\label{eq_sec7_Margulis21}
\tT_{k,1} \Psi_{k,1} (x,x') \le \Psi_{k,1}(x,x') - \kappa_A \text{ for every $(x,x')\in A_k$}\\
\label{eq_sec7_Margulis22}
\tT_{k,1} \Psi_{k,1}(x,x') \le \Psi_{k,1}(x,x') + \kappa_B \text{ for every $(x,x')\in B_k$.}
\end{align}
Then $\Psi_{k,1}$ is a Margulis function for the operator $\tT_{k,1}$ relative to the partition $(A_k,B_k)$ of $\prodspa1$.
Combining Proposition~\ref{p_sec6_Margulis_finite_energy}
with Lemma~\ref{l_sec6_Margulis_at_work1}, we conclude that there exist self-couplings
$\teta_{k,1,j}$ of $\eta_{k,1}$ such that $\int_{\prodspa1} \Psi_{k,1} \, d\teta_{k,1,j}$ is finite and
\begin{equation}\label{eq_sec7_um14}
\begin{aligned}
\teta_{k,1,j}(B_k)
& \ge \frac{\kappa_A}{\kappa_A+\kappa_B} \teta_{k,1,j}(\prodspa1)\\
& > \frac{9}{10} \teta_{k,1,j}(\prodspa1)
= \frac{9}{10} \eta_{k,1}(E_1(\vep_1)).
\end{aligned}
\end{equation}
Moreover, $(\teta_{k,1,j})_j$ may be taken to converge to a $\tT_{k,1}$-invariant self-coupling
$\teta_{k,1}$ of $\eta_{k,1}$.
%All these couplings are necessarily non-atomic, since $\eta_{k,1}$ is non-atomic.

The fact that $\teta_{k,1,j}$ is a self-coupling of $\eta_{k,1}$ together with the inequality
\eqref{eq_sec7_um02} ensure that
%may be written as
%$$
%\eta_{k,1}(E(\vep_1,\vep'''_1)) < \frac{2}{10} \eta_{k,1}(E(\vep_1)),
%$$
%because $E(r)$ and $E_1(r)$ are the same subset of $\grass(1,d)$ for any $r>0$.
%Thus, since ,
\begin{equation}\label{eq_sec7_um04}
\teta_{k,1,j}(B')
\le 2 \eta_{k,1}(E_1(\vep_1,\vep'''_1))
< \frac{4}{10} \eta_{k,1}(E_1(\vep_1)).
\end{equation}
Subtracting \eqref{eq_sec7_um04} from \eqref{eq_sec7_um14}, we conclude that
\begin{equation}\label{eq_sec7_um05}
\teta_{k,1,j}(B''_k) > \frac{5}{10} \eta_{k,1}(E_1(\vep_1))>0.
\end{equation}
This yields a contradiction when $\dim E=1$ because, as observed previously,
in that case the set $B''_k$ is empty.

\medskip\noindent{\bf 3.}
If $\dim E\ge 2$, consider the map
$$
\Sigma: \prodspa1 \to \grass(2,d), \quad \Sigma(x,x') = x+x'.
$$
The fact that $\Sigma$ is not defined on the diagonal of $\prodspa1$ need not concern us at this stage:
we will deal with it in Section~\ref{ss_step1_closure}.
For each large $k$, let $\eta_{k,2}=\Sigma_*\teta_{k,1}$ and $\{\heta_{k,1,y}: y \in \grass(2,d)\}$
be a disintegration of $\teta_{k,1}$ with respect to the partition $\{\Sigma^{-1}y: y\in\grass(2,d)\}$
of $\prodspa1$. Then define
$$
\cT_{k,2}:\Bd(\Sigma(\prodspa1)) \to \Bd(\Sigma(\prodspa1)),\
\cT_{k,2}\psi(y) = \int_{\Sigma^{-1}y} \tT_{k,1} (\psi \circ\Sigma)  \, d\heta_{k,1,y}.
$$

The measure $\eta_{k,2}$ is $\cT_{k,2}$-invariant. Indeed, since $\teta_{k,1}$ is $\tT_{k,1}$-invariant,
$$
\begin{aligned}
\int_{\Sigma(\prodspa1)} \cT_{k,2} \psi \, d\eta_{k,2}
& = \int_{\Sigma(\prodspa1)} \int_{\Sigma^{-1}y} \tT_{k,1} (\psi \circ\Sigma)  \, d\heta_{k,1,y}  \, d\eta_{k,2}(y) \\
& = \int_{\prodspa1} \tT_{k,1}(\psi \circ\Sigma) \, d\teta_{k,1}
= \int_{\prodspa1} (\psi\circ\Sigma) \, d\teta_{k,1}\\
& = \int_{\Sigma(\prodspa1)} \psi \, d\eta_{k,2}
\end{aligned}
$$
for any $\psi\in \Bd(\Sigma(\prodspa1))$.
Taking the limit as $j\to\infty$ in \eqref{eq_sec7_um05}, we find that
\begin{equation}\label{eq_sec7_um05_lim}
\teta_{k,1}(B''_k) \ge \frac{5}{10} \eta_{k,1}(E_1(\vep_1)).
\end{equation}
Recall that we take $\omega_{k,1}\to 0$ as $k\to\infty$. Then, in view of the definition of $B''_k$ in
\eqref{eq_sec7_r=1_Bsecond}, the sets $\Sigma(B''_k)$ approach $E_2$.
Taking the limit $k\to\infty$ in \eqref{eq_sec7_um05_lim}, we find that
$$
\eta_{\infty,2}(E_2) \ge \frac{5}{10} \eta_{\infty,1}(E_1(\vep_1))>0.
$$

Now define $n_2=n$.
If $y$ is close to $E_2$ in $\grass(2,d)$ then $x$ and $x'$ are close to $E_1$ in $\grass(1,d)$
and, in particular, they are in the $\nu^{(n_2)}_k$-core of $E_1(\vep_1)$ for every large $k$.
Then, using \eqref{eq_sec7_um07},
$$
\begin{aligned}
\cT_{k,2}\psi(y)
%& = \int_{\Sigma^{-1}y} \int_{G\times G} \psi\circ \Sigma(gx,g'x') \, d\tnu^{(N)}(g,g') d\teta_{k,1,y}(x,x')\\
& = \int_{\Sigma^{-1}y} \int_{G} \psi\circ \Sigma(gx,gx') \, d\nu_k^{(n_2)}(g) \, d\teta_{k,1,y}(x,x')\\
& = \int_{\Sigma^{-1}y} \int_{G} \psi(gy) \, d\nu_k^{(n_2)}(g) \, d\teta_{k,1,y}(x,x')
 = \int_{G} \psi(gy) \, d\nu_k^{(n_2)}(g),
\end{aligned}
$$
because each $\teta_{k,1,y}$ is a probability.
Pick $\vep_2>0$ such that this holds for every $y\in E_2(\vep_2)$.
Apply Proposition~\ref{p_sec6_adapted_strong} with $X=\Sigma(\prodspa1)$, $U=E_2(\vep_2)$,
$\nu = \nu_k^{(n_2)}$, and $\cT=\cT_{k,2}$, for $k$ large.
Replace $\cT_{k,2}$ and $\eta_{k,2}$ with this new Markov operator and invariant measure,
respectively.

This would complete the proof of Theorem~\ref{theorem:inductive} for $r=1$.

\medskip\noindent{\bf 4.}
However, in general we cannot construct a Margulis function $\Psi_{k,1}$ as required.
Essentially, the problem is that, since we do not control the measures $\tsigma_{k,1,x,x'}$ in the
border region, the inequality  \eqref{eq_sec7_Margulis22} cannot be proved to hold unless the
Margulis function $\Psi_{k,1}$ is taken to be bounded in the border region.
However, outside the border region $\Psi_{k,1}(x,x')$ must be very large if $x$ and $x'$ are close.
The only way to reconcile these two requirements is to introduce some drastic discontinuities in
$\Psi_{k,1}$ which then cause \eqref{eq_sec7_Margulis22} to fail at some points.

This problem is solved as follows. We do indeed create a discontinuity by cutting $\Psi_{k,1}$ off in such a way
that it is bounded in the border region. The main idea for dealing with the discontinuity, that we call
\emph{recoupling}, involves replacing $\tT_{k,1}$ with another Markov operator $\hT_{k,1}$ that still projects to
$\cT_{k,1}$ on either factor. The recoupling modification is restricted to a region which is disjoint from $B''_k$
and relatively far from the equator, so that the properties of the measures $\eta_{k,2}$ are not affected.
These arguments are detailed in Sections~\ref{s_step1_preparing_a_Margulis_function} through~\ref{s_step_1_recoupling_and_conclusion}.

In Section~\ref{s_step1_preparing_a_Margulis_function} we introduce the notion of \emph{vertical projection} $\VP_1(x,x')$ of a pair of
points $x$ and $x'$ in $P$, and we use it to construct a candidate $-\log \VP_1(x,x')$ to a Margulis function
for $\tT_{\infty,1}$. The problem with this function is that it refers explicitly to $E$ and, since the equator
is not $\nu_k$-invariant for $k\in\NN$, the estimates break down for $\tT_{k,1}$ when $k$ is finite.

This is fixed in Section~\ref{ss_step1_stabilization}, where we replace the vertical projection with a
kind of cut-off that we call the \emph{stabilized vertical projection} and denote as $\SVP_1(x,x')$.
Then, in Section~\ref{ss_step1_cut-off} we give the precise recipe for the other cut-off, that was
mentioned previously.
By the end of Section~\ref{s_step1_turning_the_perturbation_on} point we will have a much more viable candidate
$\Psi_{k,1}$ for Margulis function when $k$ is finite.

The details of the recoupling procedure are described in Section~\ref{ss_step1_recoupling}.
Proposition~\ref{p_r=1_Prop11.3} essentially states that this function $\Psi_{k,1}$ is a Margulis function
for the recoupled Markov operator $\hT_{k,1}$.
In Sections~\ref{ss_step1_contradiction} and~\ref{ss_step1_closure}
we wrap the arguments up to conclude the step $r=1$ of the proof.

\subsection{General step}\label{ss_general_outline}

For $r > 1$ the strategy is similar, except that we need to work with \emph{flag varieties}.
Recall that $\cF(r,d)$ denotes the space of flags
$$
F_1 \subset F_2 \subset \cdots \subset F_{r-1} \subset F_r \subset \RR^d,
$$
where each $F_i$ has dimension $i$. We use $x=(F_1, \dots, F_r)$ and $x'=(F'_1, \dots, F'_r)$
to denote generic elements of $\cF(r,d)$.
Define $\fE_r = \{x \in\cF(r,d): F_r \in E_r\}$ and, for each $\vep > \vep' > 0$,
$$
\fE_r(\vep) = \big\{x \in\cF(r,d): F_r \in E_r(\vep)\big\}
\quand
\fE_r(\vep,\vep') =\fE_r(\vep) \setminus \fE_r(\vep').
$$
By definition, for any $\nu\in\cMc$, the $\nu$-core of $\fE(\vep)$ is the subset of flags
$x$ such that $F_r\in\Chi_{\nu}E_r(\vep_r)$, and the $\nu$-border of $\fE(\vep)$ is
defined in a similar fashion.

We consider $\prodspar = \fE_r(\vep_r) \times \fE_r(\vep_r)$ and we aim to construct a Margulis
function $\Psi_{k,r}$ for a suitable Markov operator $\fhT_{k,r}: \Bd(\prodspar) \to \Bd(\prodspar)$
relative to a convenient generalization $\prodspar = A_k \cup B'_k \cup B''_k$ of the
partition \eqref{eq_sec7_r=1_A}--\eqref{eq_sec7_r=1_Bsecond}.
For $i = 1, 2$, define $\theta_i : \prodspar \to \grass(r, d)$ to be the projection to either factor
followed by the canonical map
\begin{equation}\label{eq_sec7_forgetfulness}
\cF(r,d) \to \grass(r, d), \quad x \mapsto F_r.
\end{equation}
A constraint on $\fhT_{k,r}$ is that it must project to $\cT_{k,r}$ under both $\theta_i$.
Then we can find a $\fhT_{k,r}$-invariant measure $\fheta_{k,r}$ which maps to $\eta_{k,r}$
under both projections.

As in the case $r = 1$, once we have constructed $\fhT_{k,r}$, $\Psi_{k,r}$, and $\fheta_{k,r}$
we can try to get an estimate of the form
$$
\fheta_{k,r}(B''_k) > c \fheta_{k,r}(\prodspar),
$$
where $c>0$ is some absolute constant. Then, if we push $\fheta_{k,r}$ forward by one
of the maps
$$
\begin{aligned}
\Sigma_r: \prodspar \to \grass(r+1,d), \quad
\Sigma_r\left(x,x'\right) = F'_1 + F_r \\
\Sigma_r: \prodspar \to \grass(r+1,d), \quad
\Sigma_r\left(x,x'\right) = F_1 + F'_r
\end{aligned}
$$
we obtain probability measures $\eta_{k,r+1}$ on $\grass(r+1,d)$ satisfying the conditions of
Theorem~\ref{theorem:inductive}.

However, for $r>1$ the simple cut-off procedure we use in the initial step of the induction is
no longer sufficient to ensure that the Margulis function is bounded on the border region.
To fix that, we take advantage of the additional freedom provided by the flag space, which is
that we may vary the projections to $\grass(j, d)$ for $j < r$.
More precisely, we modify the dynamics by averaging ("spreading out") the Markov operator over
the subspaces of dimension less than $r$ in the flag.
Thus we end up with modified Markov operators $\qhT_{k,r}$ that still project to the $\cT_{k,r}$
under both $\theta_i$. Then the kind of argument we sketched in the previous paragraphs can
actually be carried out for $\qhT_{k,r}$-invariant measures $\qheta_{k,r}$ that map to
$\eta_{k,r}$ under both projections.

The detailed arguments are structured as follows in Sections~\ref{s_preparing_a_Margulis_function}
through~\ref{s_recoupling_conclusion}.
In Section~\ref{s_preparing_a_Margulis_function} we extend the notions of \emph{vertical angle}
and \emph{vertical projection} to $r>1$, and we use them to exhibit a candidate $-\log \psi_r(x,x')$
to a Margulis function.
In Section~\ref{ss_stabilization} we move to introduce the $r>1$ version
of the \emph{stabilized vertical projection} $\SVP_r(x,x')$, and in Section~\ref{ss_cut-off} we
describe the corresponding version for the cut-off procedure. In Section~\ref{s_turning_the_perturbation_on}
we extend this analysis to the perturbed random walks, that is, to finite $k$. The Margulis function
$\Psi_{k,r}$ is defined at the end of that section.

The spreading out modification is detailed in Section~\ref{s_spreading_out}. Essentially, we define
$\qhT_{k,r} = \fhT_{q,r} \circ \tQ_r$ where $\tQ_r: Bd(\prodspar) \to \Bd(\prodspar)$ has the form
$$
\tQ_r\psi(x,x') = \int \Psi \, d\tq_{r,x,x'}
$$
where $\tq_{r,x,x'}$ is a smooth probability measure on the subset of pairs of flags whose
$r$-coordinate is $(F_r,F'_r)$. A relevant feature is that $\tQ_r$ maps to the identity under the
canonical map \eqref{eq_sec7_forgetfulness}.
The details of the recoupling procedure are described in Section~\ref{ss_recoupling}.
In Sections~\ref{ss_contradiction} and~\ref{ss_closure} we wrap up the proof.

\part{First step of the induction}\label{p_three}

\section{Preparing a Margulis function}\label{s_step1_preparing_a_Margulis_function}

We are going to construct a positive function $\psi_1$ such that $-\log\psi_1$ has some features of a
Margulis function for large iterates of the operator $\cP_{\nu_\infty}$:

\begin{proposition}\label{p_r=1_Prop5.1}
There exist $\kappa'_1=\kappa'_1(\nu_\infty)>0$ and $C'_1=C'_1(\nu_\infty)>0$
and for any $\delta>0$ there exists $N_1=N_1(\nu_\infty,\delta)\in\NN$ such that
for every $n \ge N_1$ there exists $\rho'_1=\rho'_1(\nu_\infty,\delta,n)>0$ satisfying
\begin{equation}\label{eq_r=1_VPaver}
\int_G - \log\psi_1(gx,gx') \, d\nu_\infty^{(n)}(g) \le -\log\psi_1(x,x') - (\kappa'_1-C'_1\delta)n
\end{equation}
for every $x \neq x'$ in $E_1(\rho'_1)$.
\end{proposition}

The conclusion of Proposition~\ref{p_r=1_Prop5.1} does not quite fit the definition of a Margulis function
because the set $E_1(\rho'_1) \times E_1(\rho'_1)$ where the estimate holds is not invariant under the
$G$-action restricted to $\supp\nu_\infty$ and thus $\cP_{\nu_\infty}$ cannot be considered a
Markov operator on this set. Nevertheless, $-\log\psi_1$ is an important ingredient in the definition
of the actual Margulis function, which will be completed in Section~\ref{s_step_1_recoupling_and_conclusion}.

The proof of Proposition~\ref{p_r=1_Prop5.1} occupies the remainder of this section.
Throughout, we think of the equator as being ``horizontal'' and use the word ``vertical'' to refer to
the orthogonal direction.
The numbers $\delta>0$ and $n\in\NN$ will remain as free parameters for most of our arguments.
Near the end, in \eqref{eq_r=1_defdeltan}, we will fix $\delta>0$ to be small, depending only on
$\nu_\infty$, and $n\in\NN$ large, depending on $\nu_\infty$ and $\delta>0$.

Fix $A=A(\nu_\infty)>0$ and a compact neighborhood $\cW_0=\cW_0(\nu_\infty)\subset G$ of the
support of $\nu_\infty$ such that
\begin{equation*}%\label{eq_def_AAA}
- \log d(U,V) - A \le - \log d(fU,fV) \le - \log d(U,V) + A
\end{equation*}
for any $f\in \cW_0$, any pair $U \neq V$ in $\grass(r,d)$ and any $1\le r\le d$.
Moreover, let $B=B(\nu_\infty)>0$ be defined by
\begin{equation*}%\label{eq_def_BBB}
B = \sup\big\{\log \|f\| + \log \|f^{-1}\|: f\in \cW_0\big\} + 2.
\end{equation*}
Since $(\nu_k)_k$ converges to $\nu_\infty$ in the space $\cMc$, it is no restriction to assume that
$\supp\nu_k \subset \cW_0$ for every $k\in\NN$. In particular,
\begin{align}
\label{eq_AAA}
- \log d(fU,fV) & \le - \log d(U,V) + An\\
\label{eq_BBB}
\log \|f\| \|f^{-1}\| & \le Bn
\end{align}
for every $U \neq V$ in $\grass(r,d)$, $1 \le r \le d$, $f\in\supp\nu_k^{(n)}$, $k\in\NN\cup\{\infty\}$,
and $n\in\NN$.

\subsection{Vertical angle function}\label{ss_step1_vertical_angle}

Given distinct points $x$ and $x'$ in $P=\grass(1,d)$, let $y=y(x,x')\subset P$ denote
the \emph{great circle} through $x$ and $x'$, that is, the subset of $P$ associated with
the element $x+x'$ of $\grass(2,d)$ that contains both $x$ and $x'$.

\begin{figure}[ht]
\begin{center}
\psfrag{x}{$x$}\psfrag{x1}{$x'$}\psfrag{u}{$u$}
\psfrag{up}{$u^\perp$}\psfrag{y}{$y$}\psfrag{E}{$E$}
\includegraphics[height=1.8in]{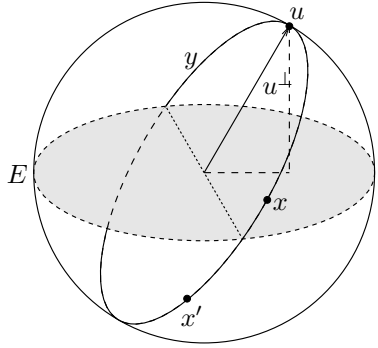}
\caption{\label{f_angle}
Geometric meaning of the vertical angle function $\VA_1$.
By making $x$ and $x'$ go to $E$ along a given great circle, one sees that $\VA_1(x,x')$
need not go to infinity when $x$ and $x'$ approach the equator.}
\end{center}
\end{figure}

The \emph{vertical angle function} $\VA_1$ is the sine of the angle between the great circle $y=y(x,x')$
and the equator, that is (recall \eqref{eq_sec7_Theta1} and check Figure~\ref{f_angle}),
\begin{equation}\label{eq_r=1_VA1_def}
\VA_1(x,x') = d(y,E) = \sup_{u\in y} d(u,E) = \sup_{u\in y} \frac{\|u^\perp\|}{\|u\|}.
\end{equation}
For any $u\in y$ that realizes the supremum,
\begin{equation}\label{eq_r=1_VA23again}
\VA_1(gx,gx')
\ge \frac{\|(g u)^\perp\|}{\|gu\|}
= \VA_1(x,x') \frac{\|(g u)^\perp\|}{\|gu\|}\frac{\|u\|}{\|u^\perp\|}
\end{equation}
for any $g\in G$. By Remark~\ref{r_sec3_perp}, when $g\in\supp\nu_\infty^{(n)}$ this means that
\begin{equation}\label{eq_r=1_VA5again}
\VA_1(gx,gx')
\ge \frac{\|g^\perp u^\perp\|}{\|gu\|}
= \VA_1(x,x') \frac{\|g^\perp u^\perp\|}{\|gu\|}\frac{\|u\|}{\|u^\perp\|}.
\end{equation}
Still for $g\in\supp\nu_\infty^{(n)}$, using Remark~\ref{r_sec3_perp} and \eqref{eq_BBB} we get that
$$
\frac{\|g^\perp u^\perp\|}{\|gu\|}\frac{\|u\|}{\|u^\perp\|}
\ge \frac{1}{\|(g^\perp)^{-1}\|\|g\|} \ge \frac{1}{\|g^{-1}\|\|g\|} \ge e^{-Bn},
$$
and so
\begin{equation}\label{eq_r=1_VA1_def_linear}
- \log\VA_1(gx,gx') \le - \log \VA_1(x,x') + Bn.
\end{equation}

\begin{lemma}\label{l_r=1_Lem5.4}
There exists $\tilde\kappa_1=\tilde\kappa_1(\nu_\infty)>0$ and for each $\delta>0$ there exist $\tilde\theta_1=\tilde\theta_1(\nu_\infty,\delta)>0$
and $\widetilde{N}_1=\widetilde{N}_1(\nu_\infty,\delta)\in\NN$ such that for every $n \ge \widetilde{N}_1$ and $x \neq x'$ in $P$ there exists
$\tE_1=\tE_1(\nu_\infty,\delta,n,x,x')\subset\supp\nu_\infty^{(n)}$ with $\nu_\infty^{(n)}(\tE_1^c)<\delta$ and
\begin{equation}\label{eq_r=1_VA4}
-\log\VA_1(gx,gx') \le \max\{-\log\VA_1(x,x') - \tilde\kappa_1 n, \tilde\theta_1\}
\text{ for every $g\in\tE_1$.}
\end{equation}
\end{lemma}

\begin{proof}
Let $\kappa_0=\kappa_0(\nu_\infty)>0$, $N_0=N_0(\nu_\infty,\delta)\in\NN$,
$\cE_0=\cE_0(\nu_\infty,\delta,n,u^\perp)\subset \supp\nu_\infty^{(n)}$,
and $\tau_0=\tau_0(\nu_\infty,\delta)>0$ be as in Proposition~\ref{p_sec5_derivative_estimate}.
Given $x \neq x'$ in $P$, take $u\in y$ realizing the supremum in \eqref{eq_r=1_VA1_def}.
Write $u=u^E+u^\perp$ with $u^E\in E$ and $u^\perp\in E^\perp$. Take
\begin{equation}\label{eq_r=1_VA5.5}
\begin{aligned}
\tilde\kappa_1={\kappa_0}/{2}, & \quad \tilde\theta_1=-\log({\tau_0}/{4}),\\
\widetilde{N}_1 > \max\left\{N_0, {4}/{\tilde\kappa_1}\right\}, & \quad \text{and }\tE_1=\cE_0(\nu_\infty,\delta,n,u^\perp).
\end{aligned}
\end{equation}
Let $n\ge \widetilde{N}_1$ and $g\in\tE_1\subset\supp\nu_\infty^{(n)}$.
If  ${\|g^\perp u^\perp\|}/{\|gu\|} \ge {\tau_0}/{2}$ then the inequality in \eqref{eq_r=1_VA5again}
implies that
\begin{equation}\label{eq_r=1_VA6}
- \log\VA_1(gx,gx') \le - \log\frac{\tau_0}{2} \le \tilde\theta_1.
\end{equation}
If ${\|g^\perp u^\perp\|}/{\|gu\|} < {\tau_0}/{2}$ then part (2) of Proposition~\ref{p_sec5_derivative_estimate} gives that
$$
\frac{\|g u^\perp\|}{\|gu\|}
%= \frac{\|g u^\perp\|}{\|g^\perp u^\perp\|} \frac{\|g^\perp u^\perp\|}{\|gu\|} < \frac{1}{\tau_0}\frac{\tau_0}{2}
< \frac 12,
\text{ which implies }
\frac{\|g u^E\|}{\|gu\|}
> \frac 12.
$$
Substituting the latter inequality and $\|u\| \ge \|u^E\|$ in \eqref{eq_r=1_VA5again},
we find that
\begin{equation}\label{eq_r=1_VA6.25}
\VA_1(gx,gx')
\ge \frac 12 \VA_1(x,x') \frac{\|g^\perp u^\perp\|}{\|u^\perp\|}\frac{\|u^E\|}{\|g u^E\|}.
\end{equation}
Thus, recalling the definition \eqref{eq_sec5_derivative2} and part (1) of
Proposition~\ref{p_sec5_derivative_estimate},
$$
\VA_1(gx,gx')
\ge \frac 12 \VA_1(x,x')\frac{\|Dg_{u^E}^\perp u^\perp\|}{\|u^\perp\|}
\ge \frac 12 \VA_1(x,x') e^{\kappa_0 n}.
$$
By the choices of $\tilde\kappa_1$ and $\widetilde{N}_1$ in \eqref{eq_r=1_VA5.5}, this implies that
\begin{equation}\label{eq_r=1_VA7}
\begin{aligned}
-\log\VA_1(gx,gx')
& \le - \log \VA_1(x,x') + \log 2 - 2 \tilde\kappa_1 n \\
& \le - \log \VA_1(x,x') - \tilde\kappa_1 n.
\end{aligned}
\end{equation}
The conclusion of the lemma is contained in \eqref{eq_r=1_VA6} and \eqref{eq_r=1_VA7}.
\end{proof}

\subsection{Vertical projection function}\label{ss_step1_vertical_projection}

The function $-\log\VA_1(x,x')$ cannot be used as a Margulis function to detect the equator
because (unless $\dim E=1$) it is possible that $-\log\VA_1(x,x')$ remains bounded even as $d(x,E)$
and $d(x',E)$ go to zero: the great circle through points close to the equator need not be close to the
equator, as illustrated in Figure~\ref{f_angle}.

To (partially) rectify this, we introduce the \emph{vertical projection} function $\VP_1$, defined as follows:
\begin{equation}\label{eq_r=1_VP1_def}
\VP_1(x,x') = \VA_1(x,x')d(x,x')^{\gamma_1}
\end{equation}
where $\gamma_1=\gamma_1(\nu_\infty)$ is a small positive constant to be chosen shortly (Proposition~\ref{p_r=1_Lem5.5}).
Note that if $\gamma_1=1$ then $\VP_1(x,x')$ would indeed be a sort of projection on the orthogonal complement
to the equator. It is clear from the definitions \eqref{eq_r=1_VA1_def} and \eqref{eq_r=1_VP1_def} that
$\VA_1$ and $\VP_1$ are symmetric functions:
\begin{equation}\label{eq_r=1_commute}
\VA_1(x,x')=\VA_1(x',x)
\quand
\VP_1(x,x')=\VP_1(x',x)
\text{ for any $x\neq x'$.}
\end{equation}

%%% The original statement has an additional constant $\eta'_0$
\begin{proposition}\label{p_r=1_Lem5.5}
There exist $\gamma_1=\gamma_1(\nu_\infty)>0$ and $\kappa'_1=\kappa'_1(\nu_\infty)>0$ and for each
$\delta>0$ there exists $N_1=N_1(\nu_\infty,\delta)\in\NN$ such that for every $n\ge N_1$ there exists
$\rho'_1=\rho'_1(\nu_\infty,\delta,n)>0$ such that for any $x \neq x'$ in $E_1(\rho'_1)$ there exists
$\cE'_1=\cE'_1(\nu_\infty,\delta,n,x,x')\subset\supp\nu_\infty^{(n)}$ with $\nu_\infty^{(n)}((\cE'_1)^c)<\delta$ and
\begin{equation}\label{eq_r=1_VPstat}
- \log \VP_1(gx,gx') \le -\log\VP_1(x,x')-\kappa'_1 n \text{ for every $g\in\cE'_1$.}
\end{equation}
\end{proposition}

\begin{proof}
The overall strategy goes as follows. If the vertical angle  $\VA_1(x,x')$ is small then the conclusion of the
present proposition is a consequence of Lemma~\ref{l_r=1_Lem5.4},
provided that we choose the constant $\gamma_1$
small enough to make the variation of $d(x,x')^{\gamma_1}$ negligible.
If $\VA_1(x,x')$ is large, let $v$ and $v'$ be unit vectors in the direction of $x$ and $x'$ and such that the
angle between them is non-obtuse (the latter may always be obtained by replacing $v$ with $-v$ if necessary).
\begin{figure}[ht]
\begin{center}
\psfrag{x}{$x$}\psfrag{x1}{$x'$}\psfrag{u}{$u$}
\psfrag{w}{$w$}\psfrag{y}{$y$}\psfrag{E}{$E$}
\includegraphics[height=1.8in]{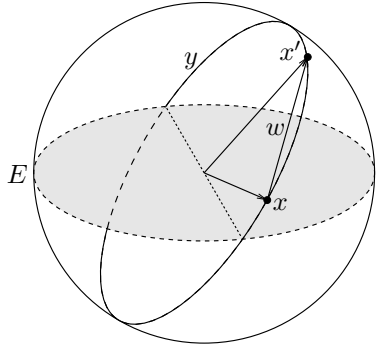}
\caption{\label{f_dilation}
Interpreting Lemma~\ref{l_r=1_Lem5.6}:
if $d(x,E)$ is small compared to $\VA_1(x,x')=d(y,E)$ then $w=v'-v$ is roughly vertical.}
\end{center}
\end{figure}
In Lemma~\ref{l_r=1_Lem5.6} we check that if $x$ is close to $E$ then the difference $w=v'-v$ is roughly vertical,
meaning that the angle between $w$ and the equator is bounded from below.
See Figure~\ref{f_dilation}.
In Lemma~\ref{l_r=1_Lem5.7} we deduce that in this situation $-\log d(x,x')$ decreases under most large iterates:
the reason is that the dynamics increases vertical components and, for $n$ sufficiently large, the vertical
component of the iterate of $w$ dominates. This implies the conclusion of the proposition because in this regime
the variation of $\VA_1(x,x')$ is bounded. Let us fill-in the details.

\begin{lemma}\label{l_r=1_Lem5.6}
Given $x \neq x'$ in $P$, let $w=v'-v$ be the difference between unit vectors in the directions of $x$ and $x'$,
respectively. Then
\begin{equation}\label{eq_r=1_ww}
d(x,E) < \frac 14 \VA_1(x,x') \text{ implies }
\frac{\|w^\perp\|}{\|w\|} > \frac 14 \VA_1(x,x').
\end{equation}
\end{lemma}

\begin{figure}[ht]
\begin{center}
\psfrag{v}{$v$}\psfrag{v1}{$v'$}\psfrag{w}{$w$}
\psfrag{a}{\tiny{$d(x,x')$}}
\includegraphics[height=1.8in]{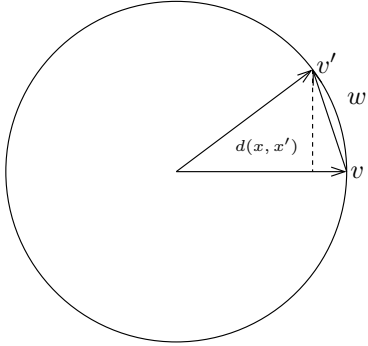}
\end{center}
\caption{\label{f.dist1}
Illustrating the estimates in \eqref{eq_r=1_dista12}: the distance from $x$ to $x'$ is given by the
length of the vertical dashed segment, which is bounded above by $\|w\|$;
the angle between $v$ and $w$ is at least $\pi/4$, as the angle between $v$ and $v'$ is non-obtuse.}
\end{figure}

\begin{proof}
As the angle between $v$ and $v'$ was taken to be non-obtuse (check Figure~\ref{f.dist1})
\begin{equation}\label{eq_r=1_dista12}
d(x,x') = |\sin\angle(v,v')| \le \|w\|
\quand
\angle(w,v) \ge \frac{\pi}{4}.
\end{equation}
Take $u\in y$ realizing the supremum in \eqref{eq_r=1_VA1_def}. Since $y=\gene\{v,v'\}=\gene\{v,w\}$,
we may write $u=av + bw$ with $a, b\in\RR$.
The angle bound in \eqref{eq_r=1_dista12} implies that $\|av\|$ and $\|bw\|$ are both less than
$2\|u\|$. Thus,
$$
\begin{aligned}
\VA_1(x,x')
= \frac{\|u^\perp\|}{\|u\|}
& \le \frac{\|av^\perp\|}{\|u\|} + \frac{\|bw^\perp\|}{\|u\|} \\
& < 2 \frac{\|v^\perp\|}{\|v\|} + 2 \frac{\|w^\perp\|}{\|w\|}
\le 2 d(x,E) + 2 \frac{\|w^\perp\|}{\|w\|}.
\end{aligned}
$$
Thus, $d(x,E)$ and $\|w^\perp\|/\|w\|$ cannot be both less than $\VA_1(x,x')/4$.
\end{proof}

Take $\tilde\kappa_1=\tilde\kappa_1(\nu_\infty)>0$ and $\tilde\theta_1=\tilde\theta_1(\nu_\infty,\delta)>0$ to be as
in Lemma~\ref{l_r=1_Lem5.4}.

\begin{lemma}\label{l_r=1_Lem5.7}
There exists $\hat\kappa_1=\hat\kappa_1(\nu_\infty)>0$ and for each $\delta>0$ there exists
$\widehat{N}_1=\widehat{N}_1(\nu_\infty, \delta)\in\NN$ such that for each $n\ge \widehat{N}_1$ there exists
$\hat\rho_1=\hat\rho_1(\nu_\infty,\delta,n)>0$ such that for any $x\neq x'$ in $E_1(\hat\rho_1)$
with $-\log\VA_1(x,x') \le \tilde\theta_1+\tilde\kappa_1 n$ there exists
$\hE_1=\hE_1(\nu_\infty,\delta,n,x,x')\subset\supp\nu_\infty^{(n)}$ with
$\nu_\infty^{(n)}(\hE_1^c)<\delta$ and
\begin{equation}\label{eq_r=1_5.17}
-\log d(gx,gx') \le -\log d(x,x')-\log\VA_1(x,x')-\hat\kappa_1n\text{ for every $g\in\hE_1$.}
\end{equation}
\end{lemma}

\begin{figure}[ht]
\psfrag{gv}{$gv$}\psfrag{gv1}{$gv'$}
\psfrag{gx}{$gx$}\psfrag{gx1}{$gx'$}
\psfrag{gw}{$gw$}
\psfrag{a}{\tiny{$d(gx,gx')$}}
\psfrag{b}{$\|\pi_{gv} gw \| = \|\pi_{gv} gv'\|$}
\begin{center}
\includegraphics[height=1.8in]{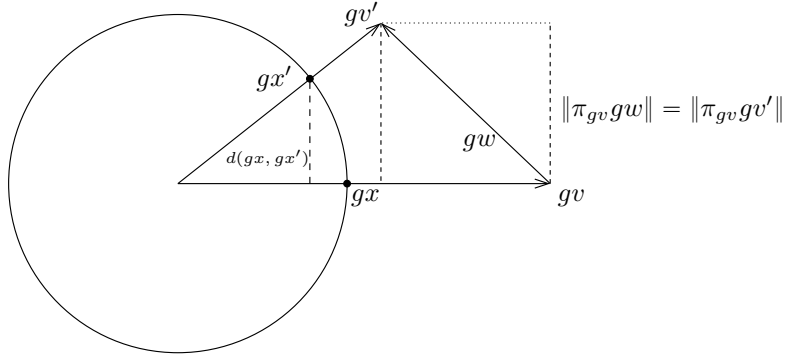}
\end{center}
\caption{\label{f_dist2}Verifying the inequality \eqref{eq_r=1_dista7}, when $\|gv\| \ge \|gv'\|$:
the distance between $gx$ and $gx')$ is given by the length of the vertical dashed segment on the
left, which is greater than the length of the vertical dashed segment on the right divided by $\|gv\|$.}
\end{figure}

\begin{proof}
Let $\kappa_0=\kappa_0(\nu_\infty)>0$, $N_0=N_0(\nu_\infty,\delta)\in\NN$, and
$\cE_0=\cE_0(\nu_\infty,\delta,n,v^\perp)\subset\supp\nu_\infty^{(n)}$ be as in Proposition~\ref{p_sec5_derivative_estimate}.
Given $x \ne x'$ in $P$, let $w=v'-v$ be the difference between unit vectors $v$ and $v'$ in the directions
of $x$ and $x'$, respectively. Take
\begin{equation}\label{eq_r=1_dista22}
\begin{aligned}
& \hat\kappa_1= \kappa_0/2, \quad
\widehat{N}_1 > \max\{N_0, {5}/{\hat\kappa_1}\}, \\
& \hat\rho_1 < e^{- \tilde\theta_1 - \tilde\kappa_1 n}/10,
\quand
\hE_1=\cE_0(\nu_\infty,\delta,n,w^\perp).
\end{aligned}
\end{equation}
Let $n\ge \widehat{N}_1$ and $g\in\hE_1\subset\supp\nu_\infty^{(n)}$. As observed in \eqref{eq_r=1_dista12},
\begin{equation}\label{eq_r=1_dista6}
d(x,x') \le \|w\| %= \frac{\|w\|}{\|v\|} = \frac{\|w\|}{\|v'\|}.
\end{equation}
Let us suppose that $\|gv\|\ge\|gv'\|$; the case $\|gv\|\le\|gv'\|$ is analogous,
reversing the roles of $x$ and $x'$. Then (check Figure~\ref{f_dist2})
\begin{equation}\label{eq_r=1_dista7}
d(gx,gx')
= \left\|\Pi_{gv} \frac{gv'}{\|gv'\|}\right\|
= \left\|\Pi_{gv} \frac{gw}{\|gv'\|}\right\|
\ge \frac{1}{{\|gv\|}} \left\|\Pi_{gv} gw\right\|,
\end{equation}
and so (keep in mind that $\|v\|=\|v'\|=1$),
\begin{equation}\label{eq_r=1_dg1}
-\log d(gx,gx') \le - \log d(x,x') - \log \frac{\|\Pi_{gv}gw\|}{\|gv\|}\frac{\|v\|}{\|w\|}.
\end{equation}

By the condition on $\hat\rho_1$ in \eqref{eq_r=1_dista22}, if $d(x,E)\le\hat\rho_1$ then
$$
4d(x,E) < e^{- \tilde\theta_1 - \tilde\kappa_1 n} < \VA_1(x,x'),
$$
and then Lemma~\ref{l_r=1_Lem5.6} gives that
\begin{equation}\label{eq_r=1_dista75}
\frac{\|w^\perp\|}{\|w\|} > \frac 14 \VA_1(x,x') > \frac 14 e^{- \tilde\theta_1 - \tilde\kappa_1 n}.
\end{equation}
Up to further reducing $\hat\rho_1$, we may also assume that
\begin{equation}\label{eq_r=1_dista9fresh}
\frac{\|hz\|}{\|z\|} \ge \frac 12 \frac{\|hz^E\|}{\|z^E\|}
\quand
\|\Pi_{hz} hw\| \ge \frac 12 \|\Pi_{hz^E} hw\|
\end{equation}
for any non-zero $z=z^E+z^\perp$ in $E\oplus E^\perp$ with $\|z^\perp\|/\|z\| \le \hat\rho_1$
and any $h\in\supp\nu_\infty^{(n)}$.
Indeed, the first part of \eqref{eq_r=1_dista9fresh} is a simple consequence of continuity;
in the second one note also that $w$ is bounded away from the horizontal, by \eqref{eq_r=1_dista75}.
This last part of \eqref{eq_r=1_dista9fresh} implies that
\begin{equation}\label{eq_r=1_dista10}
\|\Pi_{hz} hw\| \ge \frac 12\|\Pi_{hz^E} hw\| \ge \frac 12\|\Pi_E hw\|
= \frac 12\|(hw)^\perp\| = \frac 12\|h^\perp w^\perp\|.
\end{equation}

Noting that $\|v^\perp\|/\|v\|=d(x,E)\le\hat\rho_1$, take $z=v$ and $h=g$ in the previous two relations.
Thus, substituting \eqref{eq_r=1_dista9fresh} and \eqref{eq_r=1_dista10}  in \eqref{eq_r=1_dg1},
\begin{equation*}
- \log d(gx,gx') \le - \log d(x,x') + \log 4 - \log \frac{\|g^\perp w^\perp\|}{\|w\|}\frac{\|v^E\|}{\|gv^E\|}.
\end{equation*}
Then, using also \eqref{eq_r=1_dista75} and \eqref{eq_sec5_derivative2},
\begin{equation*}
\begin{aligned}
- \log d(gx,gx')
& \le -\log d(x,x') - \log\VA_1(x,x') + \log 16 - \log \frac{\|g^\perp w^\perp\|}{\|w^\perp\|}\frac{\|v^E\|}{\|gv^E\|}\\
&= - \log d(x,x') - \log \VA_1(x,x') + \log 16 -  \log\frac{\|Dg^\perp_{v^E} w^\perp\|}{\|w^\perp\|}.
\end{aligned}
\end{equation*}
By part (1) of Proposition~\ref{p_sec5_derivative_estimate}
and the choice of $\hat\kappa_1$ and $\widehat{N}_1$ in \eqref{eq_r=1_dista22}, it follows that
$$
\begin{aligned}
-\log d(gx,gx')
& \le - \log d(x,x') - \log \VA_1(x,x') + 5 - 2\hat\kappa_1 n \\
& \le -\log d(x,x') - \log \VA_1(x,x') - \hat\kappa_1 n,
\end{aligned}
$$
as claimed.
\end{proof}

We are ready to prove Proposition~\ref{p_r=1_Lem5.5}. Take $A=A(\nu_\infty)$ as in \eqref{eq_AAA} and then define
\begin{equation}\label{eq_r=1_constants}
\begin{aligned}
& \gamma_1 = \min\big\{1,{\tilde\kappa_1}/{(2A)}\big\},\\
& \kappa'_1=\min\big\{{\tilde\kappa_1}/{2}, {\gamma_1\hat\kappa_1}/{2}\big\},\\
& \rho'_1=\min\{\hat\rho_1(\nu_\infty,\delta),\hat\rho_1(\nu_\infty,{\delta}/{2})\}\\
& N_1 > \max\big\{\widetilde{N}_1(\nu_\infty,{\delta}/{2}), \widehat{N}_1(\nu_\infty,{\delta}/{2}),
                         {2\tilde\theta_1}/{(\gamma\hat\kappa_1)}\big\} \quand\\
& \cE'_1=\tE_1(\nu_\infty,{\delta}/{2},n,x,x') \cap \hE_1(\nu_\infty,{\delta}/{2},n,x,x').
\end{aligned}
\end{equation}
Observe that $\nu_\infty^{(n)}((\cE'_1)^c)<\delta$. By definition,
\begin{equation}\label{eq_r=1_VP1}
-\log\VP_1(gx,gx')=-\log\VA_1(gx,gx') - \gamma_1\log d(gx,gx').
\end{equation}
%Condition \eqref{eq_def_AAA} implies that
%\begin{equation}\label{eq_r=1_battery1}
%- \log d(gx,gx') \le - \log d(x,x') + A n
%\end{equation}
%for any $g \in \supp\nu_\infty^{(n)}$.

Consider $x \neq x'$ in $E_1(\rho'_1)$, $n\ge N_1$ and $g\in\cE'_1$.
If $-\log\VA_1(x,x')\ge\tilde\theta_1+ \tilde\kappa_1 n$ then, by Lemma~\ref{l_r=1_Lem5.4},
\begin{equation}\label{eq_r=1_VP2}
-\log\VA_1(gx,gx') \le - \log\VA_1(x,x') - \tilde\kappa_1 n
 \end{equation}
 Substituting \eqref{eq_AAA} and \eqref{eq_r=1_VP2} in \eqref{eq_r=1_VP1} we find that
 \begin{equation}\label{eq_r=1_VP3}
 \begin{aligned}
-\log\VP_1(gx,gx')
& \le  - \log \VA_1(x,x') - \tilde\kappa_1 n - \gamma_1 \log d(x,x') + \gamma_1 A n \\
& \le - \log \VP_1(x,x') - \frac{\tilde\kappa_1}{2} n \le - \log \VP_1(x,x') - \kappa'_1 n.
\end{aligned}
\end{equation}
Now assume that $-\log\VA_1(x,x') \le \tilde\theta_1+\tilde\kappa_1 n$.  In this case, Lemma~\ref{l_r=1_Lem5.4} yields
\begin{equation}\label{eq_r=1_VP4}
-\log\VA_1(gx,gx') \le \tilde\theta_1
\end{equation}
and Lemma~\ref{l_r=1_Lem5.7} gives that
\begin{equation}\label{eq_r=1_VP5}
- \log d(gx,gx') \le -\log d(x,x') - \log\VA_1(x,x') - \hat\kappa_1 n.
\end{equation}
Substituting \eqref{eq_r=1_VP4} and \eqref{eq_r=1_VP5} in \eqref{eq_r=1_VP1} we obtain
$$
\begin{aligned}
-\log\VP_1(gx,gx')
& \le \tilde\theta_1 - \gamma_1 \log d(x,x') - \gamma_1 \log\VA_1(x,x') - \gamma_1\hat\kappa_1n \\
& \le -\log\VP_1(x,x')  + \tilde\theta_1 + (1-\gamma_1) \log \VA_1(x,x') - \gamma_1\hat\kappa_1n.
\end{aligned}
$$
Since $\gamma_1 \le 1$, $\VA_1(x,x') \le 1$,
and $n \ge N_1 \ge 2\tilde\theta_1/(\gamma_1\hat\kappa_1)$, this yields
\begin{equation}\label{eq_r=1_VP6}
\begin{aligned}
-\log\VP_1(gx,gx')
& \le -\log\VP_1(x,x') + \tilde\theta_1 - \gamma_1\hat\kappa_1n\\
& \le -\log\VP_1(x,x') - \frac{\gamma_1\hat\kappa_1}{2} n
\le  - \log\VP_1(x,x') - \kappa'_1 n.
\end{aligned}
\end{equation}
The relations \eqref{eq_r=1_VP3} and \eqref{eq_r=1_VP6} contain the conclusion of Proposition~\ref{p_r=1_Lem5.5}.
\end{proof}

\subsection{The function $-\log\psi_1$}

Now we are going to prove that $\psi_1=\VP_1$ satisfies Proposition~\ref{p_r=1_Prop5.1}.
Let $A=A(\nu_\infty)$ and $B=B(\nu_\infty)>0$ be as in \eqref{eq_AAA} and \eqref{eq_BBB}.
Take $\kappa'_1>0$, $N_1\in\NN$ and $\rho'_1>0$ as in Proposition~\ref{p_r=1_Lem5.5}. By definition,
\begin{equation}\label{eq_r=1_logVP1}
-\log\psi_1(x,x')=-\log\VP_1(x,x') = -\log\VA_1(x,x')-\gamma_1\log d(x,x')
\end{equation}
for every $x \neq x'$ in $P$.
%For $g\in\supp\nu^{(n)}_\infty$, condition \eqref{eq_def_AAA} implies that
%\begin{equation}\label{eq_r=1_battery2}
%- \log d(gx,gx') \le - \log d(x,x') + A n
%\end{equation}
%and we have seen in \eqref{eq_r=1_VA1_def_linear} that
%\begin{equation}\label{eq_r=1_VA8new}
%- \log\VA_1(gx,gx') \le - \log\VA_1(x,x') + Bn.
%\end{equation}
Define
\begin{equation}\label{eq_r=1_C1linha}
C'_1 = B+\gamma_1 A.
\end{equation}
Substituting \eqref{eq_r=1_VA1_def_linear} and \eqref{eq_AAA} in \eqref{eq_r=1_logVP1}, we find that
\begin{equation}\label{eq_r=1_VP1_linear}
\begin{aligned}
-\log\VP_1(gx,gx')
& \le - \log\VA_1(x,x') + Bn - \gamma_1 \log(dx,dx') + \gamma_1 A n\\
& \le - \log\VP_1(x,x') + C'_1n
\end{aligned}
\end{equation}
for every $x \neq x'$ in $P$ and $g\in\supp\nu^{(n)}_\infty$. Integrating \eqref{eq_r=1_VPstat} over $\cE'_1$
and \eqref{eq_r=1_VP1_linear} over the complement, and using the fact that $\nu_\infty^{(n)}((\cE'_1)^c)<\delta$, we get that
$$
\int_G -\log\VP_1(gx,gx') \, d\nu_\infty^{(n)}(g) \le -\log\VP_1(x,x') - \kappa'_1 n + C'_1\delta n.
$$
for every $n\ge N_1$ and $x \neq x'$ in $E_1(\rho'_1)$. This completes the proof of Proposition~\ref{p_r=1_Prop5.1}.

\medskip

Another relevant feature is that $-\log\psi_1(x,x')$ goes to infinity when $x$ and $x'$ approach the equator $E$:

\begin{lemma}\label{l_r=1_Lem5.8}
Given any $R>0$, there exists $\tilde\rho_1=\tilde\rho_1(\nu_\infty,R)>0$ such that $-\log\psi_1(x,x')>R$
for any $x\neq x'$ in $E_1(\tilde\rho_1)$.
\end{lemma}

\begin{proof}
Consider any $\tilde\rho_1>0$ and let $v$ and $v'$ be unit vectors in the direction of $x$ and $x'$,
respectively. Observe that
$$
\frac{\|v^\perp\|}{\|v\|} = d(x,E) \le \tilde\rho_1
\quand
\frac{\|(v')^\perp\|}{\|v'\|} = d(x',E) \le \tilde\rho_1.
$$
If $d(x,x') \le \tilde\rho_1^{1/2}$ then
\begin{equation}\label{eq_r=1_VPtoinfty1}
\VP_1(x,x') = \VA_1(x,x') d(x,x')^{\gamma_1} \le \tilde\rho_1^{\gamma_1/2}.
\end{equation}
Now suppose that $d(x,x')=|\sin\angle(v,v')|$ is greater than $\tilde\rho_1^{1/2}$.
Then every $u$ in the great circle $y$ generated by $x$ and $x'$ may be written as
$u=\alpha v + \alpha'v'$ with $\alpha, \alpha'\in \RR$ such that
$\|u\| \ge \tilde\rho_1^{1/2} \|\alpha v\|$ and $\|u\| \ge \tilde\rho_1^{1/2} \|\alpha' v'\|$.
Thus,
$$
\frac{\|u^\perp\|}{\|u\|}
\le \frac{\|\alpha v^\perp\|}{\|u\|} +  \frac{\|\alpha' (v')^\perp\|}{\|u\|}
\le \frac{1}{\tilde\rho_1^{1/2}} \left(\frac{\|v^\perp\|}{\|v\|} + \frac{\|(v')^\perp\|}{\|v'\|}\right)
\le 2\tilde\rho_1^{1/2}.
$$
Since $u$ is arbitrary, this proves that $\VA_1(x,x')\le 2\tilde\rho_1^{1/2}$, and so
\begin{equation}\label{eq_r=1_VPtoinfty2}
\VP_1(x,x') = \VA_1(x,x') d(x,x')^{\gamma_1} \le 2\tilde\rho_1^{1/2} \le 2 \tilde\rho_1^{\gamma_1/2}
\end{equation}
(because $\gamma_1\le 1$).
The inequalities \eqref{eq_r=1_VPtoinfty1} and \eqref{eq_r=1_VPtoinfty2} imply that
$$
-\log\psi_1(x,x')=-\log\VP_1(x,x') \ge - \frac{\gamma_1}{2}\log\tilde\rho_1 - \log 2
$$
for any $x\neq x'$ in $E_1(\tilde\rho_1)$, which yields the claim.
%Let $x\neq x'$ in $E_1(\rho)$ for some $\rho>0$.
%If $\VA_1(x,x')\le\sqrt\rho$ then \eqref{eq_r=1_logVP1} gives that (keep in mind that $d(x,x')\le\pi/2$)
%$$
%-\log\psi_1(x,x') \ge - \log \sqrt\rho - \gamma_1\log\frac{\pi}{2}
%$$
%and the right-hand side is greater than $R$ if $\rho$ is small enough.
%Form now on, suppose that $\VA_1(x,x')\ge\sqrt\rho$. Define $w$ as in \eqref{eq_r=1_w} and assume
%that $\rho<1/25$. Then
%$$
%\VA_1(x,x') > 5 \rho \ge 5 d(x,E)
%$$
%and so we may apply Lemma~\ref{l_r=1_Lem5.6} to conclude that
%$$
%\|w^\perp\| > \frac 15 \|w\| \VA_1(x,x').
%$$
%The condition that $x, x' \in E_1(\rho)$ implies that $\|w^\perp\|\le 2\rho$. Using also \eqref{eq_r=1_dista6}, we
%get that $d(x,x') \le 10\pi \sqrt\rho/\sqrt 2 < 25 \sqrt\rho$. Then, keeping in mind that $\VA_1(x,x')\le\pi/2$,
%the relation \eqref{eq_r=1_logVP1} gives that
%$$
%-\log\psi_1(x,x') \ge \log\frac{\pi}{2} - \gamma_1 \log 25\sqrt\rho.
%$$
%Since $\gamma_1>0$ is fixed, depending only on $\nu_\infty$, the right-hand side is greater than $R$ if $\rho$
%is small enough.
\end{proof}

\section{Stabilization and cut-off}\label{s_step1_stabilization_and_cut-off}

Previously, we constructed a function $-\log\psi_1=-\log\VP_1$ that exhibits some of the features
of a Margulis function. In this section we modify this function to correct two important defects.

One problem with $-\log\psi_1$, that originates from $\VA_1$, is that the definition refers explicitly
to $E$. The reason why this is a problem is that for $k\in\NN$ the equator is not necessarily a
$\nu_k$-invariant set. Thus, even for $k$ large, for $g\in\supp\nu_k$ there is no uniform upper
bound on the absolute value of
$$
-\log \psi_1(gx,gx')+\log\psi_1(x,x')
$$
because $-\log\VA_1(gx,gx')$ may be wildly different from $-\log\VA_1(x,x')$.
To rectify this, in Section~\ref{ss_step1_stabilization} we define ``stabilized'' versions of the functions
$\VA_1$ and $\VP_1$. This will come at a price: the analogue of Proposition~\ref{p_r=1_Prop5.1}
will hold only outside a ``stabilization region'' near $E$.
Still, as we will explain in Section~\ref{s_step1_turning_the_perturbation_on},
stabilization does allow us to by-pass this first difficulty in a satisfactory way.

Another problem is that, because of the way the Markov operators $\cT_{k,1}$
will be constructed, we have little control over the border region.
The simplest way to by-pass this is to take the Margulis function to be bounded when either $x$
or $x'$ is in the border region.
The function $-\log\psi_1$ does not satisfy this: for instance, it can get arbitrarily large when $x$
and $x'$ are close to each other and, of course, that may occur even if they are both in the
border region. Thus, it is necessary to cut that function off in (a neighborhood of) the border region.

However, doing a cut-off creates a discontinuity that translates into a drastic failure of
\eqref{eq_sec7_Margulis22} near the discontinuity. The main idea to handle this, which we call
\emph{recoupling}, will be explained in Section~\ref{ss_step1_recoupling}.
The recoupling technique is far from universal, it can only handle certain types of discontinuities.
Thus, it makes sense to make the discontinuity as mild as possible.
In Section~\ref{ss_step1_cut-off} we explain just how to do this.
In particular, we only cut $-\log\psi_1$ off if $x$ and $x'$ are both in (a neighborhood of) the
border region. In the present ($r=1$) situation this is fine because if one of the points, $x$ or $x'$,
is close to the border region and the other one is far outside then $-\log\psi_1(x,x')$ is bounded
anyway. The steps $r>1$ involve additional issues, which we will discuss in Sections~\ref{s_stabilization_and_cut-off}
and~\ref{s_spreading_out}.

We use the following elementary inequalities, whose proof we leave to the reader:
given any $C>0$, $c>0$ and $\psi>0$,
\begin{align}\label{eq_r=1_elementary1}
& \log(\Omega+c\psi^{-1}) \le \log(\Omega+\psi^{-1}) + \log c
\quad\text{if}\quad c>1\\ \label{eq_r=1_elementary2}
& \log(\Omega+c\psi^{-1}) \le \log(\Omega+\psi^{-1})
\quad\text{if}\quad c<1\\ \label{eq_r=1_elementary3}
&\log(\Omega+c\psi^{-1}) \le \log(\Omega+\psi^{-1}) + \log \sqrt c
\quad\text{if}\quad c<1 \quand \psi^{-1}\ge \Omega/\sqrt{c}.
\end{align}
%\item There exists $K=K(\Omega)>0$ such that if $\psi^{-1}>K$ and $c<1$ then
%$$
%\log(\Omega+c\psi^{-1}) - \log(\Omega+\psi^{-1}) \le 0.5 \log c.
%$$
%\item If $c\psi^{-1} < \Omega/10$ and $\psi^{-1}<\Omega/10$ then
%$$
%\log(\Omega+c\psi^{-1}) - \log(\Omega+\psi^{-1}) \le 1.
%$$

Let $\kappa'_1=\kappa'_1(\nu_\infty)>0$, $C'_1=C'_1(\nu_\infty)>0$,
$N_1=N_1(\nu_\infty,\delta)\in\NN$ be as in Propositions~\ref{p_r=1_Prop5.1} and~\ref{p_r=1_Lem5.5}.
Keep in mind that $\vep_1<\rho'_1\le\hat\rho_1$ and $C'_1=B+\gamma_1A$,
according to \eqref{eq_r=1_constants}, \eqref{eq_r=1_C1linha} and \eqref{eq_r=1_vep1small}.

\subsection{Stabilization}\label{ss_step1_stabilization}

Let $B=B(\nu_\infty)>0$ and $\gamma_1=\gamma_1(\nu_\infty)>0$ be as in \eqref{eq_BBB} and
\eqref{eq_r=1_constants}, respectively. For each $\omega>0$, define the
\emph{stabilized vertical angle}
\begin{equation}\label{eq_r=1_SVA1_def}
\SVA_1(x,x';\omega) = \max\big\{\VA_1(x,x'),\omega e^{-Bn}\big\}
\end{equation}
and the \emph{stabilized vertical projection} by
\begin{equation}\label{eq_r=1_SVP1_def}
\psi_1(x,x';\omega) = \SVP_1(x,x';\omega) = \SVA_1(x,x';\omega) d(x,x')^{\gamma_1}
\end{equation}
for every $x \neq x'$ in $E_1(\vep_1)$.
The estimate \eqref{eq_r=1_VP1_linear} remains valid for the stabilized vertical projection:

\begin{lemma}\label{l_r=1_stabilizedbound}
For every $x \neq x'$ in $P$, $g \in \supp\nu_\infty^{(n)}$, and $\omega >0$,
\begin{equation}\label{eq_r=1_SVP1_linear}
-\log\psi_1(gx,gx';\omega) \le - \log\psi_1(x,x';\omega) + C'_1 n.
\end{equation}
\end{lemma}

\begin{proof}
We begin by claiming that
\begin{equation}\label{eq_r=1_SVA1_linear}
-\log\SVA_1(gx,gx';\omega) \le - \log\SVA_1(x,x';\omega) + B n.
\end{equation}
The proof can be split into two cases. First, suppose that $\VA_1(x,x') < \omega$.
Then, by the definition \eqref{eq_r=1_SVA1_def},
\begin{align*}
- \log\SVA_1(x,x';\omega)  & \ge - \log\omega \quand\\
-\log\SVA_1(gx,gx';\omega) & \le -\log\omega + B n \le - \log\SVA_1(x,x';\omega)+Bn,
\end{align*}
as claimed. Now suppose that $\VA_1(x,x') \ge \omega$. Then, again by \eqref{eq_r=1_SVA1_def},
\begin{align*}
-\log\SVA_1(x,x';\omega)
& = - \log\VA_1(x,x') \quand\\
-\log\SVA_1(gx,gx';\omega)
& \le  - \log\VA_1(gx,gx').
\end{align*}
Then \eqref{eq_r=1_SVA1_linear} is a direct consequence of \eqref{eq_r=1_VA1_def_linear}.
This completes the proof of \eqref{eq_r=1_SVA1_linear}.
%Now observe that \eqref{eq_def_AAA} implies that
%\begin{equation}\label{eq_r=1_SVA1_linear_pre}
%- \log d(gx,gx') \le - \log d(x,x') + An
%\end{equation}
%for any $x\neq x'$ and $g\in\supp\nu_\infty^{(n)}$.
Now \eqref{eq_r=1_SVP1_linear} follows easily from  \eqref{eq_r=1_SVA1_linear} and
\eqref{eq_AAA}: recalling the definition of $C'_1$ in \eqref{eq_r=1_C1linha},
\begin{equation}\label{eq_r=1_VP8stab}
\begin{aligned}
-\log\SVP_1(gx,gx';\omega)
& = - \log \SVA_1(gx,gx';\omega) - \gamma_1 \log d(gx,gx')\\
& \le- \log \SVA_1(x,x';\omega) +Bn - \gamma_1 \log d(x,x') - \gamma_1 A n \\
& = - \log \SVP_1(x,x') + (B+\gamma_1A) n.
\end{aligned}
\end{equation}
This completes the argument.
\end{proof}

The \emph{stabilization region} is the set of pairs $(x,x')$ such that $\VA_1(x,x') < \omega$.
We have seen in \eqref{eq_r=1_VA1_def_linear} that
$$
\VA_1(gx,gx') \ge \VA_1(x,x') e^{-Bn}
$$
for any $x \neq x'$ in $P$, $g\in\supp\nu_\infty^{(n)}$ and $n\in\NN$.
Consequently, if $(x,x')$ is not in the stabilization region then
\begin{equation}\label{eq_r=1_SVP1}
\begin{aligned}
\SVA_1(x,x';\omega) = \VA_1(x,x')  \quand
\SVA_1(gx,gx';\omega) = \VA_1(gx,gx') \\
\SVP_1(x,x';\omega) = \VP_1(x,x')  \quand
\SVP_1(gx,gx';\omega) = \VP_1(gx,gx')
\end{aligned}
\end{equation}
for any $g \in \supp\nu_\infty^{(n)}$ and $n\in\NN$.

Propositions~\ref{p_r=1_Prop5.1} and~\ref{p_r=1_Lem5.5} immediately yield the following
analogues for stabilized vertical angles and vertical projections:

%%% The original statement has an additional constant $\theta''_1$
\begin{proposition}\label{p_r=1_Prop6.1Lem6.2}
For every $\delta>0$, $n \ge N_1$, $x \neq x'$ in $E_1(\rho'_1)$,
and $\omega>0$ with $\VA_1(x,x')\ge\omega$,
\begin{equation}\label{eq_r=1_SVP3}
-\log\psi_1(gx,gx';\omega) \le - \log\psi_1(x,x';\omega) - \kappa'_1 n
\text{ for every $g\in\cE'_1$,}
\end{equation}
and
\begin{equation}\label{eq_r=1_SVP2}
\int_G - \log\psi_1(gx,gx';\omega) \, d\nu_\infty^{(n)}(g)
\le -\log\psi_1(x,x';\omega) - (\kappa'_1-C'_1\delta)n.
\end{equation}
\end{proposition}

\begin{proof}
Since it is assumed that $(x,x')$ is not in the stabilization region,
\eqref{eq_r=1_SVP3} is a restatement of Proposition~\ref{p_r=1_Lem5.5} and
\eqref{eq_r=1_SVP2} is a restatement of Proposition~\ref{p_r=1_Prop5.1}.
\end{proof}

\subsection{Cutoff}\label{ss_step1_cut-off}

Recall that $\vep_1=\vep_1(\nu_\infty,\delta,n)>0$ was chosen in Section~\ref{ss_step1_outline},
in the context of \eqref{eq_sec7_um01}. As observed then, it may be taken to be arbitrarily small.
In particular, it is no restriction to suppose that
\begin{equation}\label{eq_r=1_vep1small}
\vep_1 < \min\{\rho_0, \rho'_1\}
\end{equation}
where $\rho_0=\rho_0(\nu_\infty,n)>0$ is as in Corollaries~\ref{c_sec5_distance_contraction1} and~\ref{c_sec5_distance_contraction2}
and $\rho'_1=\rho'_1(\nu_\infty,\delta,n)>0$ is as in Propositions~\ref{p_r=1_Prop5.1} and~\ref{p_r=1_Lem5.5}.

Arguing twice as in Remark~\ref{r_sec7_repeatedly}, we find constants
$\vep'_1=\vep'_1(\nu_\infty,\delta,n)>0$ and $\tilde\vep_1=\tilde\vep_1(\nu_\infty,\delta,n)>0$ with
$0 < \vep'_1 < \tilde\vep_1 < \vep_1$, and a compact neighborhood
$\cW_1=\cW_1(\nu_\infty,\delta,n)$ of $\supp\nu_\infty^{(n)}$, such that
\begin{align}
\label{eq_r=1_I}
gx \in E_1(\vep_1/2) &\text{ for every } x\in E_1(2\tilde\vep_1) \quand g\in \cW_1 \text{ and}\\
\label{eq_r=1_II}
g^{-1}x   \in E_1(\tilde\vep_1/2) & \text{ for every } x\in E_1(2\vep'_1)  \quand g \in \cW_1.
\end{align}

Let $\kappa_0=\kappa_0(\nu_\infty)>0$ be as in Proposition~\ref{p_sec5_derivative_estimate},
and define $\vep''_1=\vep''_1(\nu_\infty,\delta,n)$ by
\begin{equation}\label{eq_r=1_epsilon}
\vep''_1 = 3 \vep'_1e^{-\kappa_0 n/2}.
\end{equation}
Taking $\rho=\vep''_1$ in Corollary~\ref{c_sec5_distance_contraction2},
and recalling that $\vep_1$ was chosen smaller than $\rho_0$,
we get that there are $\tilde k_1=\tilde k_1(\nu_\infty,\delta,n)\in\NN$ and
$\cD_k(x) =\cD_k(\nu_\infty,\delta,n,x)\subset\supp\nu_k^{(n)}$
such that $\nu_k^{(n)}(\cD_k(x)^c) < \delta$ and
$$
d(gx,E) > e^{\kappa_0 n/2} d(x,E) > e^{\kappa_0 n/2} \vep''_1 > 2\vep'_1
$$
for any $g\in\cD_k(x)$, $x\in E(\vep_1,\vep''_1)$ and $k\ge \tilde k_1$.
In other words, for $k \ge \tilde k_1$,
\begin{equation}\label{eq_r=1_III}
x \in E(\vep_1, \vep''_1) \Rightarrow gx \notin E(2\vep'_1) \text{ for every } g\in\cD_k(x).
\end{equation}

Increasing $\tilde k_1$ if necessary, we may suppose that $\supp\nu_k^{(n)} \subset \cW_1$ for every $k\ge \tilde k_1$.
Then \eqref{eq_r=1_I} and \eqref{eq_r=1_II} imply
\begin{align}
\label{eq_r=1_IV}
E_1(2\tilde\vep_1) \subset \Chi_{\nu_k^{(n)}} E_1(\vep_1)
\quand
E_1(2\vep'_1) \subset \Chi^\#_{\nu_k^{(n)}} E_1(\vep_1)\\
\label{eq_r=1_V}
x \notin E_1(\tilde\vep_1/2) \Rightarrow g x \notin E_1(2\vep'_1)
\text{ for every } g \in \supp\nu_k^{(n)}.
\end{align}

Finally, fix $\Omega_1=\Omega_1(\nu_\infty,\delta,n)>1$ large enough that
\begin{equation}\label{eq_r=1_defOmega}
\Omega_1 \, \vep'_1 (\vep'_1-2\vep''_1)^{\gamma_1} \ge 1
\end{equation}
and define
\begin{equation}\label{eq_r=1_Psi1_def}
\Psi_1(x,x';\omega) = \left\{\begin{array}{ll}\log\left(\Omega_1+\psi_1(x,x';\omega)^{-1}\right)
                                               & \text{if $x\in E_1(2\vep''_1)$ or $x'\in E_1(2\vep''_1)$} \\
                                               \log\Omega_1 & \text{otherwise.}\end{array}\right.
\end{equation}
We will refer to the set $E_1(2\vep''_1)^c \times E_1(2\vep''_1)^c$ as the \emph{cut-off region}.
See Figure~\ref{f_cut-off}.

\begin{figure}[ht]
\begin{center}
\psfrag{E0}{{core region}}
\psfrag{E1}{{$E_1(\vep''_1)^2$}}
\psfrag{E2}{{$E_1(2\vep''_1)^2$}}
%\psfrag{E3}{{$E_1(\tilde\vep_1)^2$}}
%\psfrag{E4}{{$E_1(2\tilde\vep_1)^2$}}
\psfrag{E5}{{$E_1(\vep_1)^2$}}
\includegraphics[height=2in]{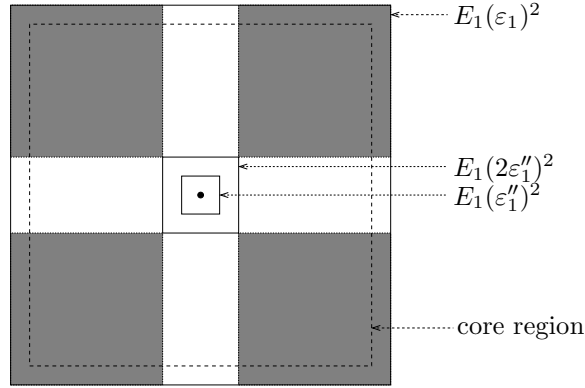}
\caption{\label{f_cut-off}
Illustrating the cut-off in the definition of the Margulis function.
The black dot at the center marks the point $(E,E)$.
The dashed lined represents the boundary between the $\nu_k^{(n)}$-core and the
$\nu_k^{(n)}$-border of $E_1(\vep_1)$.
The shaded area is the cut-off region, where $\Psi_1\equiv\log\Omega_1$.}
\end{center}
\end{figure}

%%% The original statement has an additional constant $\omega_1^+$
\begin{proposition}\label{p_r=1_Lem7.3}
There exist $\kappa'''_1=\kappa'''_1(\nu_\infty)>0$ and for each $\delta>0$ and $n\ge N_1$
there exists $\vep'''_1=\vep'''_1(\nu_\infty,\delta,n)>0$ such that
\begin{itemize}
\item[(i)] For any $x \neq x'$ in $E_1(\vep_1)$ with $\Psi_1(x,x';\omega)>\log\Omega_1$,
$$
\int_G \Psi_1(gx,gx';\omega) \, d\nu_\infty^{(n)}(g) \le \Psi_1(x,x';\omega) + C'_1 n.
$$
\item[(ii)] For any $x \neq x'$ in $E_1(\vep_1)$ with $\Psi_1(x,x';\omega)>\log\Omega_1$ and
$\VA_1(x,x')\ge\omega$,
$$
\int_G \Psi_1(gx,gx';\omega) \, d\nu_\infty^{(n)}(g) \le \Psi_1(x,x';\omega) + C'_1\delta n.
$$
\item[(iii)] For any $x \neq x'$ in $E_1(\vep'''_1)$ with $\VA_1(x,x')\ge\omega$,
$$
\int_G \Psi_1(gx,gx';\omega) \, d\nu_\infty^{(n)}(g) \le \Psi_1(x,x';\omega) - (\kappa'''_1 - C'_1\delta) n.
$$
\end{itemize}
\end{proposition}

\begin{proof}
Define $\kappa'''_1=\kappa'_1/2$. Let $n \ge N_1$.
Part (i) of the proposition is a consequence of the following lemma:

\begin{lemma}\label{l_r=1_eq7.5}
If $x \neq x'$ in $P$ are such that $\Psi_1(x,x';\omega)>\log\Omega_1$ then
\begin{equation}\label{eq_r=1_Psi1_linear}
\Psi_1(gx,gx';\omega) \le \Psi_1(x,x';\omega) + C'_1 n
\text{ for any $g\in\supp\nu_\infty^{(n)}$.}
\end{equation}
\end{lemma}

\begin{proof}
%By the definition \eqref{eq_r=1_SVP1_def},
%\begin{equation}\label{eq_r=1_logpsistab}
%-\log\Psi_1(x,x';\omega) = - \log \SVA_1(x,x';\omega) - \gamma_1\log d(x,x').
%\end{equation}
By \eqref{eq_r=1_SVP1_linear} and \eqref{eq_r=1_elementary1},
\begin{equation*}%\label{eq_r=1_SVP1_linearbis}
\begin{aligned}
\Psi_1(gx,gx';\omega)
& \le \log\left(\Omega_1+\psi_1(gx,gx';\omega)^{-1}\right)
\le \log\left(\Omega_1+e^{C'_1n}\psi_1(x,x';\omega)^{-1}\right)\\
& \le \log\left(\Omega_1+ \psi_1(x,x';\omega)^{-1}\right) + C'_1 n
 = \Psi_1(x,x';\omega) + C'_1 n.
\end{aligned}
\end{equation*}
This gives the claim.
\end{proof}

To prove part (ii) we use Proposition~\ref{p_r=1_Prop6.1Lem6.2}:
given any $\omega>0$ and $x \neq x'$ in $E_1(\vep_1)$,
\begin{equation}\label{eq_r=1_Psi5}
-\log\psi_1(gx,gx';\omega)\le-\log\psi_1(x,x';\omega)-\kappa'_1n\le-\log\psi_1(x,x';\omega)
\end{equation}
for every $g\in\cE'_1$. Then,
\begin{equation}\label{eq_r=1_Psi6}
\begin{aligned}
\Psi_1(gx,gx';\omega)
& \le \log\left(\Omega_1+\psi_1(gx,gx';\omega)^{-1}\right)\\
& \le \log\left(\Omega_1+ \psi_1(x,x';\omega)^{-1}\right)
 = \Psi_1(gx,gx';\omega),
\end{aligned}
\end{equation}
for every $g \in \cE'_1$. Integrating \eqref{eq_r=1_Psi6} over $\cE'_1$ and \eqref{eq_r=1_Psi1_linear}
over the complement, we obtain the estimate in part (ii).

Now take $c=e^{-2\kappa'''_1n}$ in the relation \eqref{eq_r=1_elementary3}.
By Lemma~\ref{l_r=1_Lem5.8}, there exists $\vep'''_1>0$ depending only on $\nu_\infty$, $\delta$ and
$n$ (through $c$ and $\Omega_1$) such that
\begin{equation}\label{eq_r=1_3lines}
-\log \psi_1(x,x';\omega) = -\log \psi_1(x,x') \ge \log \left(\Omega_1/\sqrt{c}\right)
\end{equation}
for any $x \neq x'$ in $E_1(\vep'''_1)$.
Then, using  \eqref{eq_r=1_Psi5}, \eqref{eq_r=1_3lines}, and \eqref{eq_r=1_elementary3},
\begin{equation}\label{eq_r=1_Psi7}
\begin{aligned}
\Psi_1(gx,gx';\omega)
& \le \log\left(\Omega_1+\psi_1(gx,gx';\omega)^{-1}\right)\\
& \le \log\left(\Omega_1+e^{-2\kappa'''_1n} \psi_1(x,x';\omega)^{-1}\right)\\
& \le \log\left(\Omega_1+\psi_1(x,x';\omega)^{-1}\right) -\kappa'''_1 n = \Psi_1(x,x';\omega)-\kappa'''_1 n
\end{aligned}
\end{equation}
for every $g \in \cE'_1$. Integrating \eqref{eq_r=1_Psi7} over $\cE'_1$ and \eqref{eq_r=1_Psi1_linear}
over the complement, we obtain the estimate in part (iii) of the proposition.
\end{proof}

\begin{lemma}\label{l_r=1_Prop8.1}
For any $\omega > 0$ and $x\neq x'$ in $E_1(\vep_1)$ such that $x\notin E_1(\vep'_1)$ or $x'\notin E_1(\vep'_1)$,
$$
\log \Omega_1 \le \Psi_1(x,x';\omega) \le \log \Omega_1 + \log 2.
$$
\end{lemma}

\begin{proof}
The inequality on the left is an immediate consequence of the definition \eqref{eq_r=1_Psi1_def},
and the same is true for the one on the right if both points $x$ and $x'$ are outside $E_1(2\vep''_1)$.
Let us suppose that $x\notin E_1(\vep'_1)$ but $x'\in E_1(2\vep''_1)$; the case when
$x\in E_1(2\vep''_1)$ but $x'\notin E_1(\vep'_1)$ is analogous.
Let $y=y(x,x')$ be the great circle associated to $x$ and $x'$. Then
$$
d(x,x')\ge \vep'_1 - 2\vep''_1
\quand
\VA_1(x,x') = d(y,E) \ge d(x,E) > \vep'_1
$$
and so
$$
\psi_1(x,x';\omega) \ge \psi_1(x,x') \ge \vep'_1 (\vep'_1-2\vep''_1)^{\gamma_1}.
$$
In view of the definition of $\Omega_1$ in \eqref{eq_r=1_defOmega}, it follows that
$$
\Psi_1(x,x';\omega)
= \log(\Omega_1 + \psi_1(x,x';\omega)^{-1})
\le \log 2\Omega_1,
$$
as claimed.
\end{proof}

\section{Turning the perturbation on}\label{s_step1_turning_the_perturbation_on}

Now we show that the conclusions of Proposition~\ref{p_r=1_Lem7.3} hold for $\nu_k^{(n)}$ instead
of $\nu_\infty^{(n)}$, as long as $k\in\NN$ is sufficiently large. More precisely, we prove:

\begin{proposition}\label{p_r=1_Lem9.1}
Given $\delta>0$, $n\ge N_1$, and $\omega>0$ there is
$k_1=k_1(\nu_\infty, \delta, n,\omega)\in\NN$ such that the following holds for every $k\ge k_1$:
\begin{itemize}
\item[(i)] For any $x \neq x'$ in $E_1(\vep_1)$ with $\Psi_1(x,x';\omega)>\log\Omega_1$,
$$
\int_G \Psi_1(gx,gx';\omega) \, d\nu_k^{(n)}(g) \le \Psi_1(x,x';\omega) + C'_1 n.
$$
\item[(ii)] For any $x \neq x'$ in $E_1(\vep_1)$ with $\Psi_1(x,x';\omega)>\log\Omega_1$ and
$\VA_1(x,x')\ge\omega$,
$$
\int_G \Psi_1(gx,gx';\omega) \, d\nu_k^{(n)}(g) \le \Psi_1(x,x';\omega) + C'_1\delta n.
$$
\item[(iii)] For any $x \neq x'$ in $E_1(\vep'''_1)$ with $\VA_1(x,x')\ge\omega$,
$$
\int_G \Psi_1(gx,gx';\omega) \, d\nu_k^{(n)}(g) \le \Psi_1(x,x';n) - (\kappa'_1 - C'_1\delta) n.
$$
\end{itemize}
\end{proposition}

Keep in mind that we have chosen $\vep_1<\rho'_1\le\hat\rho_1$ and
$C'_1=B+\gamma_1A$. Recall also that $\tilde k_1=\tilde k_1(\nu_\infty,\delta,n)\in\NN$
was chosen so that the relations \eqref{eq_r=1_III} through \eqref{eq_r=1_V} hold for every
$k\ge \tilde k_1$. Moreover, $N_1\in\NN$ is given by Propositions~\ref{p_r=1_Lem5.5}
and~\ref{p_r=1_Prop6.1Lem6.2}.

\begin{proof}
We are going to extend to large $k\in\NN$ several estimates in the proof of Proposition~\ref{p_r=1_Lem7.3}.
This will require a number of conditions on $k$, depending on $\nu_\infty$, $\delta$, $n$
and $\omega$, that we state along the way.
We begin with the following extension of Lemma~\ref{l_r=1_eq7.5}:

\begin{lemma}\label{l_r=1_eq7.5bis}
Given $n\ge N_1$ and $\omega>0$, there is $\hat{k}_1=\hat{k}_1(\nu_\infty,n,\omega)\in\NN$
such that if $x \neq x'$ in $P$ are such that $\Psi_1(x,x';\omega)>\log\Omega_1$ then
\begin{equation}\label{eq_r=1_Psi1_linear_finite}
\Psi_1(gx,gx';\omega) \le \Psi_1(x,x';\omega) + C'_1 n
\end{equation}
for any $g\in\supp\nu_k^{(n)}$ and any $k\ge \hat{k}_1$.
\end{lemma}

\begin{proof}
Let $n\ge N_1$ and $\omega>0$.
We claim that there is $\hat{k}_1=\hat{k}_1(\nu_\infty,n,\omega)$ such that
\begin{equation}\label{eq_r=1_SVA1_linear_finite}
-\log\SVA_1(gx,gx';\omega) \le - \log\SVA_1(x,x';\omega) + B n
\end{equation}
for any $x\neq x'$, $g\in\supp\nu_k^{(n)}$, and $k\ge \hat{k}_1$. This can be seen as follows.
If $\VA_1(x,x') < \omega$ then, by the definition \eqref{eq_r=1_SVA1_def},
\begin{align*}
- \log\SVA_1(x,x';\omega)  & \ge - \log\omega \quand\\
-\log\SVA_1(gx,gx';\omega) & \le -\log\omega + B n \le - \log\SVA_1(x,x';\omega)+Bn,
\end{align*}
as claimed. Now suppose that $\VA_1(x,x') \ge \omega$. The relation \eqref{eq_r=1_VA1_def_linear}
does not apply here. Instead, from \eqref{eq_r=1_VA23again} and \eqref{eq_r=1_SVA1_def} we get that
\begin{equation}\label{eq_r=1_VA1_def_new3}
-\log\SVA_1(gx,gx';\omega)
\le - \log\SVA_1(x,x';\omega)  - \log \frac{\|(g u)^\perp\|}{\|gu\|}\frac{\|u\|}{\|u^\perp\|}
\end{equation}
for every $g\in\supp\nu_k^{(n)}$, with $u=u(x,x')$ realizing the supremum in \eqref{eq_r=1_VA1_def}.
The assumption $\VA_1(x,x') \ge \omega$ means that $\|u^\perp\|\ge\omega\|u\|$.
Since $\supp\nu_k$ converges to $\supp\nu_\infty$ in the Hausdorff topology, we may find
$r_k=r_k(\nu_\infty,n) \to 0$ such that every $g\in\supp\nu_k^{(n)}$ is in the $r_k$-neighborhood of some $f\in\supp\nu_\infty^{(n)}$. Then
$$
\|(gu)^\perp-(fu)^\perp\| \le \|gu-fu\| \le r_k \|u\| \le \frac{r_k}{\omega}\|u^\perp\|.
$$
Fix $\hat{k}_1\in\NN$ large enough that, for every $g\in\supp\nu_k^{(n)}$ and $k \ge \hat{k}_1$,
$$
\begin{aligned}
- \log \frac{\|(g u)^\perp\|}{\|gu\|}\frac{\|u\|}{\|u^\perp\|}
& \le - \log \frac{\|(f u)^\perp\|}{\|fu\|}\frac{\|u\|}{\|u^\perp\|} + \log 2\\
& = - \log \frac{\|f^\perp u^\perp\|}{\|fu\|}\frac{\|u\|}{\|u^\perp\|} + \log 2\\
& \le \log\|(f^\perp)^{-1}\| + \log \|f\| + \log 2 \le B n.
\end{aligned}
$$
Substituting this in \eqref{eq_r=1_VA1_def_new3} completes the proof of
\eqref{eq_r=1_SVA1_linear_finite}.

Next, substituting \eqref{eq_r=1_SVA1_linear_finite} and \eqref{eq_AAA} in the definition
\eqref{eq_r=1_SVP1_def}
%\begin{equation*}%\label{eq_r=1_logpsistab_finite}
%-\log\Psi_1(x,x';\omega) = - \log \SVA_1(x,x';\omega) - \gamma_1\log d(x,x').
%\end{equation*}
we find that
\begin{equation}\label{eq_r=1_SVP1_linear_finite}
\begin{aligned}
-\log\psi_1(gx,gx';\omega)
& \le -\log\SVA_1(x,x';\omega) + Bn - \log d(x,x') + \gamma_1 A n \\
& = - \log\psi_1(x,x';\omega) + C'_1 n
\end{aligned}
\end{equation}
for any $x \neq x'$, $g\in\supp\nu_k^{(n)}$, and $k \ge \hat{k}_1$.
Using \eqref{eq_r=1_elementary1}, it follows that
\begin{equation*}%\label{eq_r=1_SVP1_linearbis_finite}
\begin{aligned}
\Psi_1(gx,gx';\omega)
& \le \log\left(\Omega_1+\psi_1(gx,gx';\omega)^{-1}\right)\\
& \le \log\left(\Omega_1+e^{C'_1n}\psi_1(x,x';\omega)^{-1}\right)\\
& \le \log\left(\Omega_1+ \psi_1(x,x';\omega)^{-1}\right) + C'_1 n
 = \Psi_1(x,x';\omega) + C'_1 n,
\end{aligned}
\end{equation*}
as stated.
\end{proof}

Next, we prove the following extension of Lemma~\ref{l_r=1_Lem5.4}:

\begin{lemma}\label{l_r=1_Lem5.4bis}
Given $\delta>0$, $n\ge N_1$ and $\omega>0$, there is
$k'_1 = k'_1(\nu_\infty,\delta,n,\omega)\in\NN$ and for any $x \neq x'$ in $P$ with $\VA_1(x,x') \ge \omega$
there is $\cE'_{k,1}=\cE'_{k,1}(\nu_\infty,\delta,n,x,x',\omega)\subset\supp\nu_k^{(n)}$ with $\nu_k^{(n)}((\cE'_{k,1})^c)<\delta$
and
\begin{equation}\label{eq_r=1_VA4bis}
-\log\SVA_1(gx,gx';\omega) \le \max\{-\log\SVA_1(x,x';\omega) - \tilde\kappa_1 n, \tilde\theta_1\}
\end{equation}
for every $g\in\cE'_{k,1}$ and $k\ge k'_1$.
\end{lemma}

\begin{proof}
Fix $\delta>0$ and $n\ge N_1$ and $\omega>0$.
Let $P_\omega$ denote the (compact) subset of all $v\in P$ such that $\|v^\perp\|/\|v\|\ge\omega/2$.
For $v\in P_\omega$ and $g$ in some compact neighborhood $V_\omega$ of $\supp\nu_\infty^{(n)}$,
consider
\begin{equation}\label{eq_r=1_stab1}
(v,g) \mapsto - \log\frac{\|(gv)^\perp\|}{\|gv\|}.
\end{equation}
As long as $V_\omega$ is sufficiently small, depending on $\nu_\infty$, $n$ and $\omega$,
the map \eqref{eq_r=1_stab1} is well defined and (uniformly) continuous. So, there exists
$\alpha=\alpha(\nu_\infty,n,\omega)>0$ such that
\begin{equation}\label{eq_r=1_stab2}
- \log\frac{\|(gu)^\perp\|}{\|gu\|}
\le - \log\frac{\|(f v)^\perp\|}{\|f v\|} + \log 2
\end{equation}
whenever $d(u,v) <\alpha$ and $d(g,f)<\alpha$.
Reducing $\alpha$ if necessary, depending only on $\omega$, we may also assume that
\begin{equation}\label{eq_r=1_stab2bis}
d(u,v) <\alpha
\quad\Rightarrow\quad
- \log\frac{\|u^\perp\|}{\|u\|}
\ge - \log\frac{\|v^\perp\|}{\|v\|} - \log 2.
\end{equation}
Fix $v_1, \dots, v_l \in P_\omega$ such that $P_\omega \subset B(v_1,\alpha) \cup \cdots \cup B(v_l,\alpha)$.
For each $v \in P_\omega$ choose $j\in\{1, \dots, l\}$ such that $v \in B(v_j,\alpha)$ and then define
$\cE_{k,0}=\cE_{k,0}(\nu_\infty,\delta,n,v,\omega)\subset\supp\nu_k^{(n)}$ by
\begin{equation}\label{eq_r=1_stab3}
\cE_{k,0} = \big[\text{$\alpha$-neighborhood of } \cE_0(\nu_\infty,\delta,n,v_j^\perp)\big] \cap\supp\nu_k^{(n)},
\end{equation}
where $\cE_0(\nu_\infty,\delta,n,v_j^\perp)\subset\supp\nu_\infty^{(n)}$ is as defined in Proposition~\ref{p_sec5_derivative_estimate}.
Since $\nu_k^{(n)}$ converges to $\nu_\infty^{(n)}$ in the weak$^*$ topology, the limit inferior of the
$\nu_k^{(n)}$-measure of \eqref{eq_r=1_stab3} as $k\to\infty$ is greater than or equal to
$$
\nu^{(n)}_\infty\left(\cE_0(\nu_\infty,\delta,n,v_j^\perp)\right) > 1 - \delta
$$
for every $j=1, \dots, l$. In particular, there is $k'_1=k'_1(\nu_\infty,\delta, n,\omega)\in\NN$
such that
\begin{equation}\label{eq_r=1_stab4}
\nu_k^{(n)}(\cE_{k,0}) > 1 - \delta
\text{ for every $k\ge k'_1$ and $v\in P_\omega$.}
\end{equation}

Given $x \neq x'$ with $\VA_1(x,x')\ge\omega$, take $u=u^E+u^\perp$ to be a non-zero vector that realizes
the supremum in \eqref{eq_r=1_VA1_def}. Then,
\begin{align}
\label{eq_r=1_stab5}
& \omega \le \VA_1(x,x')  = \frac{\|u^\perp\|}{\|u\|} \quad(\text{in particular, } u\in P_\omega) \quand\\
\label{eq_r=1_stab6}
& \VA_1(gx,gx') \ge \frac{\|(g u)^\perp\|}{\|gu\|}
\ge \VA_1(x,x') \frac{\|(g u)^\perp\|}{\|gu\|}\frac{\|u\|}{\|u^\perp\|}
\end{align}
for any $g\in G$. Then define
\begin{equation}\label{eq_r=1_stab6.5}
\cE'_{k,1}=\cE_{k,0}(\nu_\infty,\delta,n,u,\omega).
\end{equation}
It follows from \eqref{eq_r=1_stab4} that $\nu_k^{(n)}((\cE'_{k,1})^c)<\delta$ for every $k\ge k'_1$.

Let $g\in\cE'_{k,1}$ and $k \ge k'_1$.
Then, by definition, there exist $v=v^E+v^\perp$ in $P_\omega$
(take $v=v_j$ as in \eqref{eq_r=1_stab3})
and $f\in\cE_0(\nu_\infty,\delta,n,v^\perp)\subset \supp\nu_\infty^{(n)}$ such that $d(u,v)<\alpha$
and $d(g,f)<\alpha$. %Write $v=v^E+v^\perp$ with $v^E\in E$ and $v^\perp\in E^\perp$.
Thus, substituting \eqref{eq_r=1_stab2} and \eqref{eq_r=1_stab2bis} in \eqref{eq_r=1_stab6}, we find that
\begin{equation}\label{eq_r=1_stab7}
\VA_1(gx,gx')
\ge \frac 12 \frac{\|f^\perp v^\perp\|}{\|f v\|}
\ge \frac{1}{4} \VA_1(x,x') \frac{\|f^\perp v^\perp\|}{\|f v\|}\frac{\|v\|}{\|v^\perp\|}.
\end{equation}
Let $\tau_0=\tau_0(\nu_\infty,\delta)>0$ be as in Proposition~\ref{p_sec5_derivative_estimate}.
If  ${\|f^\perp v^\perp\|}/{\|f v\|} \ge {\tau_0}/{2}$ then the first part of \eqref{eq_r=1_stab7} gives that
(recall \eqref{eq_r=1_VA5.5} also)
\begin{equation}\label{eq_r=1_stab8}
- \log\VA_1(gx,gx') \le - \log\frac{\tau_0}{4} \le \tilde\theta_1.
\end{equation}
Now suppose that ${\|f^\perp v^\perp\|}/{\|f v\|} < {\tau_0}/{2}$.
Then part (2) of Proposition~\ref{p_sec5_derivative_estimate} gives that
\begin{equation}\label{eq_r=1_fv}
\frac{\|f v^\perp\|}{\|f v\|}
< \frac 12
\text{ and so }
\frac{\|f v^E\|}{\|f v\|}
> \frac 12.
\end{equation}
Substituting \eqref{eq_r=1_fv} and $\|v\| \ge \|v^E\|$ in \eqref{eq_r=1_stab7}, we find that
\begin{equation}\label{eq_r=1_stab8.5}
\begin{aligned}
\VA_1(gx,gx')
& \ge \frac 1{8} \VA_1(x,x') \frac{\|f^\perp v^\perp\|}{\|v^\perp\|}\frac{\|v^E\|}{\|f v^E\|}\\
& = \frac 1{8} \VA_1(x,x') \frac{\|Df^\perp_{v^E} v^\perp\|}{\|v^\perp\|}.
\end{aligned}
\end{equation}
By part (1) of Proposition~\ref{p_sec5_derivative_estimate} and the choices of $\tilde\kappa_1>0$
and $\widetilde{N}_1\in\NN$ in \eqref{eq_r=1_VA5.5}, this implies
\begin{equation}\label{eq_r=1_stab9}
\begin{aligned}
-\log\VA_1(gx,gx')
& \le - \log \VA_1(x,x') + \log 8 - 2 \tilde\kappa_1 n\\
& \le - \log \VA_1(x,x') - \tilde\kappa_1 n.
\end{aligned}
\end{equation}
The conclusion of the lemma is contained in \eqref{eq_r=1_stab8} and \eqref{eq_r=1_stab9}.
\end{proof}

Next, let us prove the following extension of Lemma~\ref{l_r=1_Lem5.7}:

\begin{lemma}\label{l_r=1_Lem5.7bis}
Given $\delta>0$ and $n\ge N_1$
there exists $k''_1=k''_1(\nu_\infty,\delta,n)\in\NN$ such that for any $x\neq x'$ in $E_1(\vep_1)$
with $-\log\VA_1(x,x') \le \tilde\theta_1+\tilde\kappa_1 n$ there exists
$\cE''_{k,1}=\cE''_{k,1}(\nu_\infty,\delta,n,x,x')\subset\supp\nu_k^{(n)}$ with
$\nu_k^{(n)}((\cE''_{k,1})^c)<\delta$ and
\begin{equation}\label{eq_r=1_hatE1}
-\log d(gx,gx') \le -\log d(x,x')-\log\VA_1(x,x')-\hat\kappa_1 n
\end{equation}
for every $g\in\cE''_{k,1}$ and $k\ge k''_1$.
\end{lemma}

\begin{proof}
Let $\hat P$ denote the (compact) subset of pairs $(v,w) \in P\times P$ such that
\begin{equation}\label{eq_r=1_vw}
\frac{\|v^\perp\|}{\|v\|} \le \hat\rho_1 < 2\hat\rho_1 \le \frac{\|w^\perp\|}{\|w\|}
\end{equation}
(as before, $v^\perp$ and $w^\perp$ denote the vertical components of $v$ and $w$).
Let $n\ge N_1$ and $\hat V$ be a compact neighborhood of the support of $\nu_\infty^{(n)}$.
Since, $\hat\rho_1=\hat\rho_1(\nu_\infty,n)$, $\hat P$ and $\hat V$ depend only on $\nu_\infty$ and $n$.
Condition \eqref{eq_r=1_vw} ensures that the angle between $v$ and $w$ is bounded away from zero and,
consequently, so is the angle between $gv$ and $gw$ for any $g\in\hat V$;
both bounds depend only on $\nu_\infty$ and $n$. Thus, the map
\begin{equation}\label{eq_r=1_hatE2}
(v,w,g) \mapsto - \log \frac{\|\Pi_{gv} gw\|}{\|gv\|}\frac{\|v\|}{\|w\|}
\end{equation}
is well-defined and (uniformly) continuous on the domain $(v,w)\in\hat P$ and $g\in\hat V$.
In particular, there exists $\hat\alpha=\hat\alpha(\nu_\infty,n)>0$ such that
\begin{equation}\label{eq_r=1_hatE3}
- \log \frac{\|\Pi_{gv} gw\|}{\|gv\|}\frac{\|v\|}{\|w\|}
\le - \log \frac{\|\Pi_{fu} f z\|}{\|f u\|}\frac{\|u\|}{\|z\|} + \log 2
\end{equation}
whenever $d(v,u)<\hat\alpha$ and $d(z,w)<\hat\alpha$ and $d(g,f) < \hat\alpha$.
Reducing $\hat\alpha$ if necessary, depending only on $\nu_\infty$ and $n$,
we may also suppose that
\begin{equation}\label{eq_r=1_hatE3.5}
d(z,w)<\hat\alpha
\quad\Rightarrow\quad
-\log\frac{\|z^\perp\|}{\|z\|} \le -\log\frac{\|w^\perp\|}{\|w\|} + \log 2.
\end{equation}

Fix points $(v_1,w_1), \dots, (v_l,w_l)\in \hat P$ such that the balls of radius $\hat\rho$
around these points cover $\hat P$.
For each $(v,w)\in\hat P$ choose $j\in\{1, \dots, l\}$ such that $v\in B(v_j,\hat\alpha)$ and
$w\in B(w_j,\hat\alpha)$ and then define
$\hat\cE_{k,0}=\hat\cE_{k,0}(\nu_\infty,\delta,n,v,w)\subset G$ by
\begin{equation}\label{eq_r=1_hatE4}
\hat\cE_{k,0} = \big[\text{$\hat\alpha$-neighborhood of $\cE_0(\nu_\infty,\delta,n,w_j^\perp)\big] \cap \supp\nu_k^{(n)}$},
\end{equation}
where $\cE_0(\nu_\infty,\delta,n,w_j^\perp)$ is given by Proposition~\ref{p_sec5_derivative_estimate}.
Since $\nu_k^{(n)}\to\nu_\infty^{(n)}$ in the weak$^*$ topology, the limit inferior of the
$\nu_k^{(n)}$-measure of \eqref{eq_r=1_hatE4} as $k\to\infty$ is greater than or equal to
$$
\nu^{(n)}_\infty\left(\cE_0(\nu_\infty,\delta,n,w_j^\perp)\right) > 1 - \delta
$$
for every $j=1, \dots, l$. In particular, there is $k''_1=k''_1(\nu_\infty,\delta, n)\in\NN$
such that
\begin{equation}\label{eq_r=1_hatE5}
\nu_k^{(n)}(\hat\cE_{k,0}) > 1- \delta
\text{ for every $k\ge k''_1$ and $(v,w) \in \hat P$.}
\end{equation}

Given $x \neq x'$ take $w=v - v'$ to be the difference between unit vectors $v$ and $v'$ in the
directions of $x$ and $x'$, respectively. Just as in
%\eqref{eq_r=1_dista6} and \eqref{eq_r=1_dista7},
%$$
%d(x,x') \le \|w\| = \frac{\|w\|}{\|v\|}
%\quand
%d(gx,gx') \ge \frac{\|\Pi_{gv} gw\|}{\|gv\|}
%$$
%for any $g\in G$ such that $\|gv\|\ge\|gv'\|$
%(the case $\|gv\|\le\|gv'\|$ is analogous, reversing the roles of $x$ and $x'$).
%Thus,
\eqref{eq_r=1_dg1},
\begin{equation}\label{eq_r=1_dg4}
-\log d(gx,gx')
\le - \log d(x,x') - \log \frac{\|\Pi_{gv}gw\|}{\|gv\|}\frac{\|v\|}{\|w\|}
\end{equation}
for any $g\in G$ such that $\|gv\|\ge\|gv'\|$
(the case $\|gv\|\le\|gv'\|$ is analogous, reversing the roles of $x$ and $x'$).
The assumption $x\in E_1(\vep_1)$ implies that
\begin{equation}\label{eq_r=1_hatE6a}
\frac{\|v^\perp\|}{\|v\|}
\le d(x,E) \le \vep_1 < \hat\rho_1.
\end{equation}
Now assume that $-\log\VA_1(x,x')<\tilde\theta_1+\tilde\kappa_1 n$.
Then, using \eqref{eq_r=1_dista22} and \eqref{eq_r=1_dista75},
\begin{equation}\label{eq_r=1_hatE6b}
\frac{\|w^\perp\|}{\|w\|}
> \frac{1}{4} \VA_1(x,x')
\ge \frac {1}{4} e^{-\tilde\theta_1-\tilde\kappa_1 n}
> 2\hat\rho_1.
\end{equation}
Thus, $(v,w)\in\hat P$. Then define
\begin{equation}\label{eq_r=1_hatE6c}
\cE''_{k,1}=\hat\cE_{k,0}(\nu_\infty,\delta,n,v,w).
\end{equation}
It follows from \eqref{eq_r=1_hatE5} that $\nu_k^{(n)}((\cE''_{k,1})^c)<\delta$ for every $k\ge k''_1$.

Take $u=v_j$ and $z=w_j$ as in \eqref{eq_r=1_hatE4}. By definition, $(u,z)\in\hat P$ and $d(u,v)<\hat\alpha$
and $d(z,w)<\hat\alpha$. Let $g\in\cE''_{k,1}$ and $k \ge \hat{k}_1$.
The definitions \eqref{eq_r=1_hatE4} and
\eqref{eq_r=1_hatE6c} imply that there exists $f\in\cE_0(\nu_\infty,\delta,n,z^\perp)\subset\supp\nu_\infty^{(n)}$
such that $d(g,f)<\hat\alpha$. Replacing $z$, $w$ and $g$ with $u$, $z$ and $f$, respectively, in
\eqref{eq_r=1_dista9fresh} and \eqref{eq_r=1_dista10}, we find that
\begin{equation*}
\frac{\|fu\|}{\|u\|} \ge \frac 12 \frac{\|fu^E\|}{\|u^E\|}
\quand
\|\Pi_{fu} fz\| \ge \frac 12\|f^\perp z^\perp\|.
\end{equation*}
Combining this with \eqref{eq_r=1_dg4} and \eqref{eq_r=1_hatE3}, we obtain
\begin{equation}\label{eq_r=1_dg5}
\begin{aligned}
-\log d(gx,gx')
&\le - \log d(x,x') + \log 2 - \log \frac{\|\Pi_{fu} fz\|}{\|fu\|}\frac{\|u\|}{\|z\|}\\
&\le - \log d(x,x') + \log 8 - \log \frac{\|f^\perp z^\perp\|}{\|z\|}\frac{\|u^E\|}{\|fu^E\|}.
\end{aligned}
\end{equation}
Conditions \eqref{eq_r=1_hatE3.5} and \eqref{eq_r=1_hatE6b} give that
$$
-\log\frac{\|z^\perp\|}{\|z\|}
\le -\log\frac{\|w^\perp\|}{\|w\|} + \log 2
\le -\log\VA_1(x,x') + \log 8.
$$
Substituting this in \eqref{eq_r=1_dg5}, we find that
$$
\begin{aligned}
-\log d(gx,gx')
& \le - \log d(x,x') -\log\VA_1(x,x') + \log 64 - \log \frac{\|f^\perp z^\perp\|}{\|z^\perp\|}\frac{\|u^E\|}{\|f u^E\|}\\
& = - \log d(x,x') - \log \VA_1(x,x') + \log 64 - \log\frac{\|Df^\perp_{u^E} z^\perp\|}{\|z^\perp\|}.
\end{aligned}
$$
By part (1) of Proposition~\ref{p_sec5_derivative_estimate} and the choice of $\widehat{N}_1\in\NN$
in \eqref{eq_r=1_dista22}, this implies that
$$
\begin{aligned}
-\log d(gx,gx')
& \le - \log d(x,x') - \log \VA_1(x,x') + 5 - 2\hat\kappa_1 n \\
& \le -\log d(x,x') - \log \VA_1(x,x') - \hat\kappa_1 n,
\end{aligned}
$$
as claimed.
\end{proof}

Now we deduce the following extension of Proposition~\ref{p_r=1_Lem5.5}:

\begin{lemma}\label{l_r=1_Prop7.1bis}
For $\delta>0$ and $n\ge N_1$ and $\omega>0$ there is $k'''_1=k'''_1(\nu_\infty,\delta,n,\omega)\in\NN$
and for $x\neq x'$ in $E_1(\vep_1)$ with $\VA_1(x,x')\ge\omega$ there is
$\cE'''_{k,1}=\cE'''_{k,1}(\nu_\infty,\delta,n,x,x',\omega)\subset\supp\nu_k^{(n)}$
such that $\nu_k^{(n)}((\cE'''_{k,1})^c)<\delta$ and
\begin{equation}\label{eq_r=1_Prop7.1bis}
-\log\psi_1(gx,gx';\omega) \le - \log\psi_1(x,x';\omega) - \kappa'_1n.
\end{equation}
for any $g\in\cE'''_{k,1}$ and $k\ge k'''_1$.
\end{lemma}

\begin{proof}
Fix $\delta>0$ and $n\ge N_1$ and $\omega>0$. Recall that
$$
N_1 > \max\big\{\widetilde{N}_1(\nu_\infty,{\delta}/{2}), \widehat{N}_1(\nu_\infty,{\delta}/{2}), {2\tilde\theta_1}/{(\gamma\hat\kappa_1)}\big\},
$$
by \eqref{eq_r=1_constants}. Define
\begin{equation}\label{eq_r=1_kkEE}
\begin{aligned}
&k'''_1=\max\{k'_1(\nu_\infty,\delta/2,n,\omega),k''_1(\nu_\infty,\delta/2,n)\}\\
&\hspace*{1cm}\quand\cE'''_{k,1}=\cE'_{k,1}(\nu_\infty,\delta/2,n,x,x',\omega) \cap \cE''_{k,1}(\nu_\infty,\delta/2,x,x'n).
\end{aligned}
\end{equation}
By construction, $\cE'''_{k,1}$ is contained in the support of $\nu_k^{(n)}$ and $\nu_k^{(n)}((\cE'''_{k,1})^c)<\delta$.

Consider $x \neq x'$ in $E_1(\vep_1)$ with $\VA_1(x,x')\ge\omega$.
As observed in \eqref{eq_r=1_SVP1}, the latter implies that
$\SVA_1(x,x';\omega)=\VA_1(x,x')$ and $\SVA_1(gx,gx';\omega)=\VA_1(gx,gx')$.
Thus, the claim \eqref{eq_r=1_Prop7.1bis} may be rewritten as
\begin{equation}\label{eq_r=1_Prop7.1tris}
-\log\VP_1(gx,gx';\omega) \le - \log\VP_1(x,x';\omega) - \kappa'_1n.
\end{equation}
Let $g\in\cE'''_{k,1}$ and $k \ge k'''_1$. Suppose first that $-\log\VA_1(x,x')\ge\tilde\theta_1+ \tilde\kappa_1 n$.
Then, by Lemma~\ref{l_r=1_Lem5.4bis},
\begin{equation}\label{eq_r=1_VP2bis}
-\log\VA_1(gx,gx') \le - \log\VA_1(x,x') - \tilde\kappa_1 n.
 \end{equation}
 Substituting \eqref{eq_r=1_VP2bis} and \eqref{eq_AAA} in the definition \eqref{eq_r=1_VP1_def},
 we find that
 \begin{equation}\label{eq_r=1_VP3bis}
 \begin{aligned}
-\log\VP_1(gx,gx')
& \le  - \log \VA_1(x,x') - \tilde\kappa_1 n - \gamma_1 \log d(x,x') + \gamma_1 A n \\
& \le - \log \VP_1(x,x') - \frac{\tilde\kappa_1}{2} n % \le - \log \psi_1(x,x') - \kappa'_1 n.
\end{aligned}
\end{equation}
Now suppose that $-\log\VA_1(x,x') \le \tilde\theta_1+\tilde\kappa_1 n$.  In this case, Lemma~\ref{l_r=1_Lem5.4bis} yields
\begin{equation}\label{eq_r=1_VP4bis}
-\log\VA_1(gx,gx') \le \tilde\theta_1,
\end{equation}
whereas Lemma~\ref{l_r=1_Lem5.7bis} asserts that
\begin{equation}\label{eq_r=1_VP5bis}
- \log d(gx,gx') \le -\log d(x,x') - \log\VA_1(x,x') - \hat\kappa_1 n.
\end{equation}
Substituting \eqref{eq_r=1_VP4bis} and \eqref{eq_r=1_VP5bis}
in the definition \eqref{eq_r=1_VP1_def}, we obtain
$$
\begin{aligned}
-\log\VP_1(gx,gx')
& \le \tilde\theta_1 - \gamma_1 \log d(x,x') - \gamma_1 \log\VA_1(x,x') - \gamma_1\hat\kappa_1n \\
& \le -\log\VP_1(x,x')  + \tilde\theta_1 + (1-\gamma_1) \log \VA_1(x,x') - \gamma_1\hat\kappa_1n.
\end{aligned}
$$
Since $\VA_1(x,x') \le 1$, $\gamma_1\le 1$, and $n \ge N_1 \ge 2\tilde\theta_1/(\gamma_1\hat\kappa_1)$,
it follows that
\begin{equation}\label{eq_r=1_VP6bis}
\begin{aligned}
-\log\VP_1(gx,gx')
& \le -\log\VP_1(x,x') + \tilde\theta_1 - \gamma_1\hat\kappa_1n\\
& \le -\log\VP_1(x,x') - \frac{\gamma_1\hat\kappa_1}{2} n
% \le  - \log\VP_1(x,x') - \kappa'_1 n.
\end{aligned}
\end{equation}
In view of choice of the constant $\kappa'_1$ in \eqref{eq_r=1_constants},
the claim  \eqref{eq_r=1_Prop7.1tris} is contained in \eqref{eq_r=1_VP3bis} and
\eqref{eq_r=1_VP6bis}.
\end{proof}

Let us go back to proving Proposition~\ref{p_r=1_Lem9.1}. Define
\begin{equation}\label{eq_r=1_k1}
\begin{aligned}
k_1=\max\{\tilde k_1, \hat{k}_1, k'_1, k''_1, k'''_1\}.
\end{aligned}
\end{equation}
By construction, $k_1$ depends only on $\nu_\infty$, $\delta$, $n$ and $\omega$.
Part (i) of the proposition is an immediate consequence of Lemma~\ref{l_r=1_eq7.5bis}.
To prove part (ii), consider the set $\cE'''_{k,1}=\cE'''_{k,1}(\nu_\infty,\delta,n,x,x',\omega)$ given by
Lemma~\ref{l_r=1_Prop7.1bis}. Then $\nu_k^{(n)}((\cE'''_{k,1})^c)<\delta$ and
\begin{equation}\label{eq_r=1_Psi5.8}
-\log\psi_1(gx,gx';\omega)\le-\log\psi_1(x,x';\omega)-\kappa'_1n\le-\log\psi_1(x,x';\omega)
\end{equation}
for every $g\in\cE'''_{k,1}$ and $k\ge k'''_1$. By \eqref{eq_r=1_elementary2}, this implies that
\begin{equation}\label{eq_r=1_Psi6bis}
\Psi_1(gx,gx';\omega)\le\Psi(x,x';\omega)
\text{ for every $g \in \cE'''_{k,1}$.}
\end{equation}
Integrating \eqref{eq_r=1_Psi6bis} over $\cE'''_{k,1}$ and \eqref{eq_r=1_Psi1_linear_finite} over the complement,
we obtain the estimate in part (ii).

Next, recall that we took $\kappa'''_1=\kappa'_1/2$, $c=e^{-2\kappa'''_1n}$,
and $\vep'''_1>0$ such that \eqref{eq_r=1_3lines} holds:
$$
-\log \psi_1(x,x';\omega) = -\log \psi_1(x,x') \ge \Omega_1/\sqrt{c}
$$
for any $x \neq x'$ in $E_1(\vep'''_1)$. Then, by \eqref{eq_r=1_elementary3} and the first
inequality in \eqref{eq_r=1_Psi5.8},
\begin{equation}\label{eq_r=1_Psi7bis}
\Psi_1(gx,gx';\omega)\le\Psi_1(x,x';\omega)+\log\sqrt{c}=\Psi_1(x,x';\omega)-\kappa'''_1 n
\end{equation}
for every $g \in \cE'''_{k,1}$. Integrating \eqref{eq_r=1_Psi7bis} over $\cE'''_{k,1}$ and \eqref{eq_r=1_Psi1_linear_finite}
over the complement, we obtain the estimate in part (iii) of the proposition.
This completes the proof of Proposition~\ref{p_r=1_Lem9.1}.
\end{proof}

It is clear from the statements of Lemma~\ref{l_r=1_eq7.5bis}, Lemma~\ref{l_r=1_Lem5.4bis},
and Lemma~\ref{l_r=1_Prop7.1bis} that $\hat{k}_1$, $k'_1$, and $k'''_1$ may be taken to increase
to $\infty$ when $\omega$ decreases to zero and $\nu_\infty$, $\delta$, $n$ remain fixed.
Then the same is true about the map $\omega \mapsto k_1(\nu_\infty,\delta,n,\omega)$ in \eqref{eq_r=1_k1}.
Hence, we may find $\omega_{k,1}=\omega_{k,1}(\nu_\infty,\delta,n)>0$ such that
\begin{equation}\label{eq_r=1_large_k}
(\omega_{k,1})_k \text{ decreases to $0$ and }
k\ge k_1(\nu_\infty,\delta,n,\omega_{k,1})
\end{equation}
for every large $k\in\NN$. For instance, $\omega_{k,1} = 2\inf\{\omega>0: k_1(\nu_\infty,\delta,n,\omega)\le k\}$.
Fix $\check{k}_1=\check{k}_1(\nu_\infty, \delta, n)\in\NN$ such that \eqref{eq_r=1_large_k}
holds for all $k \ge \check{k}_1$. Define
\begin{equation}\label{eq_r=1_psiPsi_def}
\psi_{k,1}(x,x') = \psi_1(x,x';\omega_{k,1})
\quand
\Psi_{k,1}(x,x') = \Psi_1(x,x';\omega_{k,1}).
\end{equation}
Then the following statement is contained in Proposition~\ref{p_r=1_Lem9.1}:

\begin{corollary}\label{c_r=1_Lem9.1}
For any $\delta>0$, $n\ge N_1$, and $k\ge \check{k}_1$:
\begin{itemize}
\item[(i)] For any $x \neq x'$ in $E_1(\vep_1)$ with $\Psi_{k,1}(x,x')>\log\Omega_1$,
$$
\int_G \Psi_{k,1}(gx,gx') \, d\nu_k^{(n)}(g) \le \Psi_{k,1}(x,x') + C'_1 n.
$$
\item[(ii)] For any $x \neq x'$ in $E_1(\vep_1)$ with $\Psi_{k,1}(x,x')>\log\Omega_1$ and
$\VA_1(x,x')\ge\omega_{k,1}$,
$$
\int_G \Psi_{k,1}(gx,gx') \, d\nu_k^{(n)}(g) \le \Psi_{k,1}(x,x') + C'_1\delta n.
$$
\item[(iii)] For any $x \neq x'$ in $E_1(\vep'''_1)$ with $\VA_1(x,x')\ge\omega_{k,1}$,
$$
\int_G \Psi_{k,1}(gx,gx') \, d\nu_k^{(n)}(g) \le \Psi_{k,1}(x,x') - (\kappa'_1 - C'_1\delta) n.
$$
\end{itemize}
\end{corollary}

\section{Recoupling and conclusion}\label{s_step_1_recoupling_and_conclusion}

Here we resolve the difficulty arising out of the discontinuity of the Margulis function $\Psi_{k,1}$.
The main issue is that the inequality in Lemma~\ref{l_r=1_eq7.5bis} may not hold when trajectories leave
the cut-off region, that is, when $\Psi_{k,1}(x,x';n)=\log\Omega_1$ but $\Psi(gx,gx';n)\neq\log\Omega_1$
for some $g$ in the support of $\nu_k^{(n)}$. That may cause the analogue of \eqref{eq_sec6_Margulis3}
to fail at such points, which is catastrophic for our proof.

From the form of the cut-off, that problem can only happen if both $x$ and $x'$ are outside
$E_1(2\vep''_1)$ and at least one of $gx$ or $gx'$ is inside $E_1(2\vep''_1)$.
This is a relatively rare occurrence: for instance, Proposition~\ref{p_sec5_derivative_estimate}
shows that, for every $x$ not too far from $E$, the image $gx$ is further away from $E$
for the majority of $g$ in the support of $\nu_\infty^{(n)}$. However, we still need to handle those
rare cases where at least one of the points $gx$ or $gx'$ is in $E_1(2\vep''_1)$.

The idea is to modify the dynamics on the space of pairs, more precisely the self-coupling
$\tT_{k,1}$ of the Markov operator $\cT_{k,1}$, to allow for the points $x$ and $x'$ to move in a
more independent way: instead of the diagonal embedding, we will consider couplings supported on
pairs $(u,u')$ such that if one of the components is in $E_1(2\vep''_1)$ then the other lies outside
$E_1(\vep'_1)$.
Then $\Psi_{k,1}(u,u')$ remains bounded above by a constant, which avoids the catastrophe.

\subsection{Recoupling}\label{ss_step1_recoupling}

As in Section~\ref{ss_step1_outline}, for each $k \in \NN$ let
$$
\cT_{k,1}:\Bd(E_1(\vep_1))\to \Bd(E_1(\vep_1)), \quad
\cT_{k,1}\varphi(x) = \int_{E_1(\vep_1)} \varphi(y) \, d\sigma_{k,1,x}(y)
$$
be a continuous Markov operator adapted to $(\nu_k^{(n)},E_1(\vep_1))$, and let
$\eta_{k,1}$ be a $\cT_{k,1}$-invariant probability measure converging, as $k\to\infty$,
to a probability measure $\eta_{\infty,1}$ such that $\eta_{\infty,1}(E)>0$.
Recall that $\vep_1>0$ was chosen small enough that
\begin{equation}\label{eq_r=1_um01_again}
\eta_{\infty,1}(E_1(\vep_1)\setminus E_1) < \frac{1}{10} \eta_{\infty,1}(E_1).
\end{equation}
Then, for every $k$ sufficiently large,
\begin{equation}\label{eq_r=1_um01_cons}
\eta_{k,1}(E_1(\vep_1, \vep'''_1)) < \frac{2}{10} \eta_{k,1}(E_1(\vep_1)).
\end{equation}

Consider $X=X'=E_1(\vep_1)$ and $Y=Y'=E_1(\vep_1)$, and $\eta_y=\eta'_y=\sigma_{k,1,y}$ for every $y \in Y$.
Moreover, let $K=\Diag_1$ be the diagonal of $\prodspa1$. Clearly, $K(x)=K'(x)=\{x\}$ for every $x\in X$.
Since the $\sigma_{k,1,y}$ are non-atomic measures, it follows that
$\sigma_{k,1,y}(K(x'))=\sigma_{k,1,y}(K'(x))=\{0\}$ for every $x, x'$ and $y$.
Thus \eqref{eq_nonatomicKY} holds in this case, and so we may use
Proposition~\ref{p_sec6_coupling_Kset_parametrized} to find a continuous family
$$
\{\ttheta_{k,1,x,x'}: (x,x') \in \prodspa1\}
$$
of generic probability measures on $\prodspa1$ such that each $\ttheta_{k,1,x,x'}$ is a coupling
of $\sigma_{k,1,x}$ and $\sigma_{k,1,x'}$ vanishing on a uniform neighborhood of the diagonal.

Let $\nu^{(n)}_{k,1,x}$ and $\nu^{(n)}_{k,1,x,x'}$ denote the push-forwards of $\nu^{(n)}_k$ under
the maps $G \to P$, $g \mapsto gx$ and $G \to P \times P$, $g \mapsto (gx,gx')$, respectively.
Let $\tomega:\prodspa1 \to [0,1]$ be a continuous function such that $\tomega(x,x')=0$ if $x$
and $x'$ are both in $E_1(\tilde\vep_1)$,
and $\tomega(x,x')=1$ if either point is outside $E_1(2\tilde\vep_1)$. Then
\begin{equation}\label{eq_def_adjusted2}
\tsigma_{k,1,x,x'} = \left(1-\tomega(x,x')\right) \nu^{(n)}_{k,1,x,x'} + \tomega(x,x') \ttheta_{k,1,x,x'}
\end{equation}
is a coupling of $\sigma_{k,1,x}$ and $\sigma_{k,1,x'}$ depending continuously on $(x, x')$,
and so
\begin{equation*}%\label{eq_r=1_tT1again}
\tT_{k,1}:\Bd(\prodspa1)\to \Bd(\prodspa1),\quad
\tT_{k,1}\varphi(x,x') = \int_{\prodspa1} \varphi(y,y') \, d\tsigma_{k,1,x,x'}(y,y'),
\end{equation*}
is a continuous self-coupling of $\cT_{k,1}$.
Since $\nu^{(n)}_k$ and $\ttheta_{k,1,x,x'}$ are generic measures, so is $\tsigma_{k,1,x,x'} $.
We are going to modify these Markov operators on the \emph{recoupling region}
$E_1(2\tilde\vep_1, \vep''_1)^2$ as follows. See Figure~\ref{f_coupling}.

For $x \in E_1(2\tilde\vep_1, \vep''_1)$, it follows from \eqref{eq_r=1_III} that the subset of
$g\in\supp\nu^{(n)}_k$ such that $gx \in E_1(2\vep'_1)$ is disjoint from the set $\cD_k(x)$
given by Corollary~\ref{c_sec5_distance_contraction2}. Hence,
\begin{equation}\label{eq_r=1_Dk}
\nu^{(n)}_{k,1,x}\left(E_1(2\vep'_1)\right) \le \nu^{(n)}_k\left(\cD_k(x)^c\right)<\delta
\text{ for every $x \in E_1(2\tilde\vep_1, \vep''_1)$.}
%\quand
%\nu^{(n)}_{k,1,x'}\left(E_1(2\vep'_1)\right) \le \nu^{(n)}_k\left(\cD_k(x')^c\right)
\end{equation}
Keep in mind that $\sigma_{k,1,x}=\nu^{(n)}_{k,1,x}$ if $x$ is in the $\nu_k^{(n)}$-core of $E_1(\vep_1)$.

\begin{figure}[ht]
\begin{center}
\psfrag{E0}{core region}
\psfrag{E1}{{$E_1(\vep''_1)^2$}}
\psfrag{E2}{{$E_1(2\vep''_1)^2$}}
\psfrag{E3}{{$E_1(\tilde\vep_1)^2$}}
\psfrag{E4}{{$E_1(2\tilde\vep_1)^2$}}
\psfrag{E5}{{$E_1(\vep_1)^2$}}
\includegraphics[height=2in]{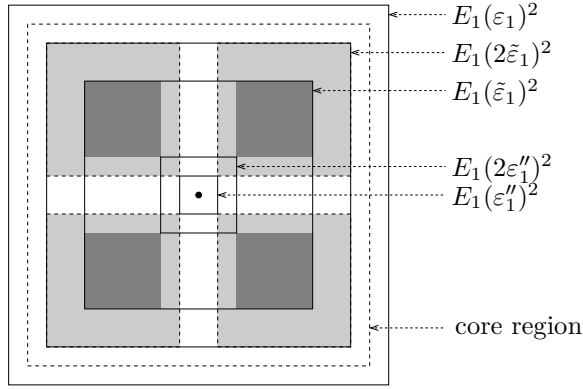}
\caption{\label{f_coupling}
Illustrating the recoupling of the Markov operators.
The black dot at the center marks the point $(E,E)$.
The dashed lined corresponds to the boundary between the $\nu_k^{(n)}$-core and the
$\nu_k^{(n)}$-border of $E_1(\vep_1)$. The shaded area is the recoupling region.
On the complement, marked in white, $\hsigma_{k,1,x,x'}=\tsigma_{k,1,x,x'}$.
On the dark gray area $\tsigma_{k,1,x,x'}$ is replaced with $\zeta_{k,1,x,x'}$, and
on the light gray area we interpolate between the two.}
\end{center}
\end{figure}

Consider $X=X'=E_1(\vep_r)$, $Y=Y'=E_1(2\tilde\vep_1, \vep''_1)^2$, and $\eta_y=\eta'_y=\sigma_{k,1,y}$ for every $y \in Y$. Moreover, let $K = E_1(\vep'_1)^2 \cup \Diag_1$.
Then
$$
K(x') = E_1(\vep'_1) \cup \{x'\} \text{ and } K'(x) = E_1(\vep'_1) \cup \{x\}
$$
for every $(x,x')\in\prodspa1$.
Since the  $\sigma_{k,1,y}$ are non-atomic measures, it follows from \eqref{eq_r=1_Dk} that
$\sigma_{k,1,y}(K(x'))$ and $\sigma_{k,1,y}(K'(x))$ are less than $\delta < 1/2$
for every $x, x'$ and $y$.
This ensures that \eqref{eq_nonatomicKY} holds in this case, and so we may use
Proposition~\ref{p_sec6_coupling_Kset_parametrized} to find a continuous family
$$
\{\zeta_{k,1,x,x'}: (x, x') \in E_1(2\tilde\vep_1, \vep''_1)^2\}
$$
of generic probability measures on $\prodspa1$ such that every $\zeta_{k,1,x,x'}$ is a coupling
of $\sigma_{k,1,x}$ and $\sigma_{k,1,x'}$ which vanishes on a uniform neighborhood of the diagonal
and satisfies
\begin{equation}\label{eq_r=1_construction_zeta}
\zeta_{k,1,x,x'}\left(E_1(\vep'_1)^2\right) = 0.
\end{equation}

Fix a continuous function $\tau:\prodspa1 \to [0,1]$ such that $\tau \equiv 1$
on $E_1(\tilde\vep_1,2\vep''_1)^2$ and $\tau \equiv 0$ on the complement of
$E_1(2\tilde\vep_1,\vep''_1)^2$. Check Figure~\ref{f_coupling}. Then define
\begin{equation}\label{eq_r=1_interpolate1}
\hsigma_{k,1,x,x'}= \left(1-\tau(x,x')\right)\tsigma_{k,1,x,x'} + \tau(x,x') \zeta_{k,1,x,x'}
\end{equation}
for every $(x,x') \in \prodspa1$.
It is clear that $\hsigma_{k,1,x,x'}$ is a coupling of $\sigma_{k,1,x}$ and $\sigma_{k,1,x'}$
depending continuously on $(x,x')$. Thus
$$
\hT_{k,1}:\Bd(\prodspa1) \to \Bd(\prodspa1), \quad
\hT_{k,1}\varphi(x,x') = \int_{\prodspa1} \varphi(y,y') \, d\hsigma_{k,1,x,x'}(y,y')
$$
is also a continuous self-coupling of $\cT_{k,1}$.
Moreover, $\hsigma_{k,1,x,x'}$ is a generic measure, since $\tsigma_{k,1,x,x'}$ and
$\zeta_{k,1,x,x'}$ are generic, and it coincides with $\tsigma_{k,1,x,x'}$ outside the
\emph{recoupling region} $E_1(2\tilde\vep_1,\vep''_1)^2$.

\begin{lemma}\label{l_r=1_recoupling_operation}
Let $(x,x') \in \prodspa1$ be such that
\begin{itemize}
\item[(a)] either at least one of the points $x$ or $x'$ is in the $\nu_k^{(n)}$-border of $E_1(\vep_1)$,
\item[(b)] or both $x$ and $x'$ are in the $\nu_k^{(n)}$-core of $E_1(\vep_1)$ but outside $E_1(2\vep''_1)$.
\end{itemize}
Then $\hsigma_{k,1,x,x'}(E_1(\vep'_1)^2)=0$ and so
$\hT_{k,1}\Psi_{k,1}(x,x') \le \log \Omega_1+ \log 2$.
\end{lemma}

\begin{proof}
Let us begin by proving the claim that $\hsigma_{k,1,x,x'}$ vanishes on $E_1(\vep'_1)^2$.
If $x$ is in the $\nu_k^{(n)}$-border of $E_1(\vep_1)$ then, using \eqref{eq_r=1_IV},
$$
\hsigma_{k,1,x,x'}(E_1(\vep'_1)^2)
\le \sigma_{k,1,x}(E_1(\vep'_1))
\le \sigma_{k,1,x}\left(\Chi^\#_{\nu_k^{(n)}}E_1(\vep_1)\right)=0.
$$
The same argument applies when $x'$ is in the $\nu_k^{(n)}$-border of $E_1(\vep_1)$.
This settles the claim in case (a). Now let $x$ and $x'$ be as in (b).
Keep in mind that $\sigma_{k,1,x}=\nu_{k,1,x}^{(n)}$ and $\sigma_{k,1,x'}=\nu_{k,1,x'}^{(n)}$.
By \eqref{eq_r=1_construction_zeta}, $\zeta_{k,1,x,x'}$ vanishes on $E_1(\vep'_1)^2$,
and so  \eqref{eq_r=1_interpolate1} gives that
$$
\hsigma_{k,1,x,x'}\left(E_1(\vep'_1)^2\right) = (1-\tau(x,x'))\tsigma_{k,1,x,x'}\left(E_1(\vep'_1)^2\right).
$$
If $x$ and $x'$ are both in $E_1(\tilde\vep_1)$ then $\tau(x,x')=1$, and the claim follows.
When $x \notin E_1(\tilde\vep_1)$ we get from \eqref{eq_r=1_V} that $\nu^{(n)}_{k,1,x}(E_1(\vep'_1))=0$.
Then
$$
\tsigma_{k,1,x,x'}(E_1(\vep'_1)^2)
\le \sigma_{k,1,x}(E_1(\vep'_1))
= \nu^{(n)}_{k,1,x}(E_1(\vep'_1))=0.
$$
The case when $x'\notin E_1(\tilde\vep_1)$ is analogous.
We have shown that $\hsigma_{k,1,x,x'}(E_1(\vep'_1)^2)=0$ also in case (b).

By Lemma~\ref{l_r=1_Prop8.1}, it follows that $\Psi_{k,1}(u,u') \le \log\Omega_1 + \log 2$ for
$\hsigma_{k,1,x,x'}$-almost every $(u,u')\in\prodspa1$. Integrating with respect to $\hsigma_{k,1,x,x'}$
we immediately get that $\hT_{k,1}\Psi_{k,1}(x,x') \le \log \Omega_1+ \log 2$.
\end{proof}

\begin{proposition}\label{p_r=1_Prop11.3}
There exist $\kappa'''_1=\kappa'''_1(\nu_\infty)>0$ and $C'''_1=C'''_1(\nu_\infty)>0$
such that given any $\delta>0$ and $n \ge N_1$ the following holds for every $k\ge \check{k}_1$:
\begin{itemize}
\item[(i)] For any $x \neq x'$ in $E_1(\vep_1)$
$$
\hT_{k,1} \Psi_{k,1}(x,x') \le \Psi_{k,1}(x,x') + C'''_1 n.
$$
\item[(ii)] For any $x \neq x'$ in $E_1(\vep_1)$ with $\VA_1(x,x')\ge\omega_{k,1}$,
$$
\hT_{k,1} \Psi_{k,1}(x,x') \le \Psi_{k,1}(x,x') + C'''_1(1+\delta n).
$$
\item[(iii)] For any $x \neq x'$ in $E_1(\vep'''_1)$ with $\VA_1(x,x')\ge\omega_{k,1}$,
$$
\hT_{k,1} \Psi_{k,1}(x,x') \le \Psi_{k,1}(x,x') - (\kappa'''_1 - C'''_1\delta) n.
$$
\end{itemize}
\end{proposition}

\begin{proof}
Take $\kappa'''_1=\kappa'_1$ and $C'''_1=\max\{C'_1,\log 2\}$, and let $k \ge \check{k}_1$.
We split the argument into four cases (check Figure~\ref{f_coupling}).

First, suppose that both $x$ and $x'$ are in the $\nu_k^{(n)}$-core of $E_1(\vep_1)$, and at least
one of them is in $E_1(\vep''_1)$. This is necessarily the case in the setting of (iii).
In particular $(x,x')$ is outside the cut-off region,
$$
\Psi_{k,1}(x,x') = \log(\Omega_1 + \psi_{k,1}(x,x')^{-1})>\log\Omega_1,
$$
and there is no recoupling either:
$$
\hsigma_{k,1,x,x'}=\nu^{(n)}_{k,1,x,x'},
\quad
\hT_{k,1} \Psi_{k,1}(x,x') = \int_G \Psi_{k,1}(gx,gx') \, d\nu_k^{(n)}(g).
$$
Hence the claims in (i), (ii) and (iii) are contained in Corollary~\ref{c_r=1_Lem9.1}.

Now suppose that both $x$ and $x'$ are in the $\nu_k^{(n)}$-core of $E_1(\vep_1)$ but outside
$E_1(\vep''_1)$, and at least one of them is in $E_1(2\vep''_1)$.
It is still true that $(x,x')$ is outside the cut-off region, and so $\Psi_{k,1}(x,x')>\log\Omega_1$.
Thus the estimates in Corollary~\ref{c_r=1_Lem9.1} remain valid for
\begin{equation}\label{eq_r=1_linear_combination1}
\int \Psi_{k,1} \, d\tsigma_{k,1,x,x}  = \int_G \Psi_{k,1}(gx,gx') \, d\nu_k^{(n)}(g).
\end{equation}
However, this time $(x,x')$ may be in the recoupling region. That is dealt with as follows.
By \eqref{eq_r=1_construction_zeta}, the measure $\zeta_{k,1,x,x'}$ vanishes on $E_1(\vep'_1)^2$.
Thus, by Lemma~\ref{l_r=1_Prop8.1},
\begin{equation}\label{eq_r=1_linear_combination2}
\int_{\prodspa1} \Psi_{k,1} \, d\zeta_{k,1,x,x}
\le \log\Omega_1+\log 2
\le \Psi_{k,1}(x,x') + \log 2.
\end{equation}
The claims (i) and (ii) follow because, by \eqref{eq_r=1_interpolate1}, $\hT_{k,1} \Psi_{k,1}(x,x')$
is a convex combination of the integrals in \eqref{eq_r=1_linear_combination1} and
\eqref{eq_r=1_linear_combination2}.

Next suppose that both $x$ and $x'$ are in the $\nu_k^{(n)}$-core of $E_1(\vep_1)$ but outside $E_1(2\vep''_1)$.
This corresponds to case (b) of Lemma~\ref{l_r=1_recoupling_operation}: claims (i) and (ii) are
contained in the conclusion of that lemma. Finally, suppose that at lest one of the points $x$ and $x'$
is in the $\nu_k^{(n)}$-border of $E_1(\vep_1)$. This is precisely the situation in case (a) of
Lemma~\ref{l_r=1_recoupling_operation}, and so claims (i) and (ii) are again contained in the conclusion
of that lemma.
\end{proof}

%Applying Lemma~\ref{l_sec6_coupling_invariant_exists} with $\cT=\cT_{k,1}$,
%$\eta=\eta_{k,1}$ and $\tT=\hT_{k,1}$, we find a $\hT_{k,1}$-invariant self-coupling
%$\heta_{k,1}$ of $\eta_{k,1}$.

\subsection{Contradicting $\dim E=1$}\label{ss_step1_contradiction}

We are going to use the following refinement of Lemma~\ref{l_sec6_Margulis_at_work1}:

\begin{lemma}\label{l_sec10_Lem12.1}
Let $\cT:\Bd(X) \to \Bd(X)$ be a Markov operator and $\psi:X\to[0,\infty]$ be a measurable function.
Suppose that there exist constants $\kappa_A$, $\kappa'_B$, $\kappa''_B \in \RR$ and pairwise
disjoint sets $A$, $B'$, $B''$ such that $X=A \cup B' \cup B''$ and
\begin{itemize}
\item[(i)] $\cT\psi(x) \le \psi(x) - \kappa_A$ for $x\in A$,
\item[(ii)] $\cT\psi(x) \le \psi(x) + \kappa'_B$ for $x\in B'$,
\item[(iii)] $\cT\psi(x) \le \psi(x) + \kappa''_B$ for $x\in B''$.
\end{itemize}
Let $\heta$ be a measure on $X$ with $\int\psi \, d\heta < \infty$ and
$\int_X \cT\psi(x) \, d\heta(x) \ge \int_X \psi(x) \, d\heta(x)$. Then
\begin{equation}\label{eq_r=1_Margulis9}
\heta(B'') \ge \frac{\kappa_A\heta(X) - (\kappa_A+\kappa'_B)\heta(B')}{\kappa_A+\kappa''_B}.
\end{equation}
\end{lemma}

\begin{proof}
Conditions (i) - (iii) imply
$$
\begin{aligned}
\int_X \psi(x) \, d\heta(x)
& \le \int_X \cT\psi(x) \, d\heta(x)\\
& \le \int_X \psi(x) \, d\heta(x) -\kappa_A\theta(A) + \kappa'_B\theta(B')+\kappa''_B\theta(B'').
\end{aligned}
$$
Thus, $-\kappa_A\heta(A) + \kappa'_B\heta(B') + \kappa''_B \heta(B'') \ge 0$,
which implies \eqref{eq_r=1_Margulis9}.
\end{proof}

Take $X = \prodspa1$, $\cT=\hT_{k,1}$, $\psi=\Psi_{k,1}$, $\heta=\heta_{k,1}$, and
\begin{align}
A_k & = \{(x,x') \in E_1(\vep'''_1)^2: \VA_1(x,x') > \omega_{k,1}\},
\label{eq_r=1_A_bis}\\
B' & = \{(x,x') \in \prodspa1: d(x, E) > \vep'''_1 \text{ or } d(x',E) > \vep'''_1\},
\label{eq_r=1_Bprime_bis}\\
B''_k & = \{(x,x')\in \prodspa1: \VA_1(x,x') \le \omega_{k,1}\}.
\label{eq_r=1_Bsecond_bis}
\end{align}
It is clear that $A=A_k$ is disjoint from $B=B' \cup B''_k$, and their union is the whole $\prodspa1$.
The sets $B'$ and $B''=B''_k$ are also disjoint if $k$ is sufficiently large, because
$$
\VA_1(x,x') = d(x+x',E) \ge \max\{d(x,E),d(x',E)\}
$$
is greater than $\vep'''_1$ whereas  $(\omega_{k,1})_k$ converges to zero when $k\to\infty$.
Also, $A \cup B' \cup B''$ is the whole $\prodspa1$.
Moreover, \eqref{eq_sec7_Theta2} implies that $B''=\emptyset$ when $\dim E =1$.

Proposition~\ref{p_r=1_Prop11.3} shows that, assuming that $k$ is sufficiently large,
the hypotheses of Lemma~\ref{l_sec10_Lem12.1} are satisfied for these choices, with
$$%\begin{equation}\label{eq_r=1_kappas}
\kappa_A = (\kappa'''_1-C'''_1\delta) n, \quad
\kappa'_B = C'''_1(1+\delta n), \quand
\kappa''_B = C'''_1 n.
$$%\end{equation}
Take $\delta>0$ to be sufficiently small, depending on $\nu_\infty$, and $n\in\NN$ to be
sufficiently large, depending on $\nu_\infty$ and $\delta$, that
\begin{equation}\label{eq_r=1_defdeltan}
\kappa_A > 9 \kappa'_B.
\end{equation}

As in Proposition~\ref{p_sec6_Margulis_finite_energy}, we find a sequence $(\heta_{k,1,j})_j$ of
probability measures on $\prodspa1$ converging to a $\hT_{k,1}$-invariant self-coupling $\heta_{k,1}$
of $\eta_{k,1}$ such that every $\heta_{k,1,j}$ satisfies $\int_{\prodspa1} \Psi_{k,1} \, d\heta_{k,1,j} <\infty$
and
$$
\int_{\prodspa1} \hT_{k,1}\Psi_{k,1}(x,x') \, d\heta_{k,1,j}(x,x') \ge  \int_{\prodspa1} \Psi_{k,1}(x,x') \, d\heta_{k,1,j}(x,x').
$$
Applying Lemma~\ref{l_sec10_Lem12.1} with $\cT=\hT_{k,1}$ and $\heta=\heta_{k,1,j}$
we get that
\begin{equation}\label{eq_r=1_conclusion1}
\heta_{k,1,j}(B'')
\ge \frac{\kappa_A \heta_{k,1,j}(\prodspa1) - (\kappa_A + \kappa'_B)\heta_{k,1,j}(B')}{\kappa_A + \kappa''_B}
\end{equation}
for every $j$. Passing to the limit as $j\to\infty$, we conclude that
\begin{equation}\label{eq_r=1_conclusion2}
\heta_{k,1}(B'')
\ge \frac{\kappa_A \heta_{k,1}(\prodspa1) - (\kappa_A + \kappa'_B)\heta_{k,1}(B')}{\kappa_A + \kappa''_B}.
\end{equation}
Observe that $\heta_{k,1}(\prodspa1)=\eta_{k,1}(E_1(\vep_1))$ and, using \eqref{eq_r=1_um01_cons},
\begin{equation}\label{eq_r=1_conclusion3}
\heta_{k,1}(B') \le 2 \eta_{k,1}(E_1(\vep_1,\vep'''_1)) < \frac{4}{10} \eta_{k,1}(E_1(\vep_1)).
\end{equation}
Thus, \eqref{eq_r=1_defdeltan} and  \eqref{eq_r=1_conclusion2} imply
\begin{equation}\label{eq_r=1_conclusion4}
\heta_{k,1}(B'')
\ge \frac{\kappa_A - \frac{4}{10}(\kappa_A + \kappa'_B)}{\kappa_A + \kappa''_B}\eta_{k,1}(E_1(\vep_1))
\ge \frac{5\kappa'_B}{\kappa_A + \kappa''_B}\eta_{k,1}(E_1(\vep_1))>0.
\end{equation}
When $\dim E=1$ this yields a contradiction, because $B''$ is empty in that case. Thus $\dim E \ge 2$.

\subsection{Completing the first step}\label{ss_step1_closure}

Let us consider the map
$$
\Sigma:\prodspa1\setminus\Diag_1 \to \grass(2,d), \quad \Sigma(x,x') = x+x'.
$$
We would like to define $\eta_{k,2}=\Sigma_*(\teta_{k,1})$ but there is a problem in that
$\Sigma(x,x')$ is not defined on $\Diag_1$ and we cannot exclude the possibility that $\teta_{k,1}$
is positive on the diagonal.

To by-pass this difficulty, we introduce the compact topological spaces
$$
\begin{aligned}
\cY_1 & = \{(x,x',y) \in \grass(1,d)^2 \times \grass(2,d): x \subset y \quand x '  \subset y\} \\
\cY_1(\vep) & = \{(x,x',y) \in \cY_1: x, x' \in E_1(\vep)\} \text{ for } \vep>0,
\end{aligned}
$$
together with the canonical projections
$$
\begin{aligned}
& p_1: \cY_1 \to \grass(1,d)^2, \quad (x,x',y) \mapsto (x,x')\\
& p_2: \cY_1 \to \grass(2,d), \quad (x,x',y) \mapsto y.
\end{aligned}
$$
For $(x,x',y)\in\cY_1$, $k \in \NN$, and $n\in\NN$, denote by $\nu^{(n)}_{k,1,x,x',y}$ the image
of $\nu^{(n)}_k$ under the diagonal action
$$
G \to \cY_1,  \quad (g \mapsto (gx,gx',gy).
$$
Clearly, each $\nu^{(n)}_{k,1,x,x',y}$ is a lift of $\nu^{(n)}_{k,1,x,x'}$ relative to $p_1:\cY_1 \to \grass(1,d)^2$.
The complement of the diagonal in $\grass(1,d)^2$ embeds in $\cY_1$ through
$$(x,x') \mapsto (x,x',x+x').$$ In particular, every measure $\xi$ on $\grass(1,d)^2$ that vanishes on
the diagonal has a (unique) lift $\check\xi$ to $\cY_1$.

From the relations \eqref{eq_def_adjusted2} and \eqref{eq_r=1_interpolate1}, we see that
\begin{equation}\label{eq_r=1_interpolate2}
\hsigma_{k,1,x,x'}= \left(1-\chomega(x,x')\right)\nu^{(n)}_{k,1,x,x'} + \chomega(x,x') \htheta_{k,1,x,x'}
\end{equation}
where $\chomega:\prodspa1 \to [0,1]$ is a continuous function that vanishes identically on
$E_1(\vep''_1)^2$, and each $\htheta_{k,1,x,x'}$ is a coupling of $\sigma_{k,1,x}$ and
$\sigma_{k,1,x'}$ vanishing on a uniform neighborhood of the diagonal.
In view of the previous remarks, it follows that the $\hsigma_{k,1,x,x'}$ lift to probability measures
\begin{equation}\label{eq_r=1_interpolate3}
\chsigma_{k,1,x,x',y}= \left(1-\chomega(x,x')\right)\nu^{(n)}_{k,1,x,x',y} + \chomega(x,x') \chtheta_{k,1,x,x',y}
\end{equation}
on $\cY_1(\vep_1)$, where $\chtheta_{k,1,x,x',y}$ is the unique lift of $\htheta_{k,1,x,x'}$.
Since $\nu^{(n)}$ and $\htheta_{k,1,x,x'}$ are generic measures, so is $\chsigma_{k,1,x,x',y}$
for every $(x,x',y)\in\cY_1(\vep_1)$.

It is clear that $\chnu^{(n)}_{k,x,x',y}$ varies continuously on $\cY_1(\vep_1)$ and,
by uniqueness, so does $\chtheta_{k,1,x,x',y}$. Thus,
$$
\chT_{k,1}:\Bd(\cY_1(\vep_1)) \to \Bd(\cY_1(\vep_1)), \quad
\chT_{k,1}\Psi(x,x',y) = \int_{\cY_1(\vep_1)} \Psi \, d\chsigma_{k,1,x,x',y}.
$$
defines a continuous Markov operator. From the definition we see that $\chT_{k,1}$ is a lift of
$\hT_{k,1}$, in the sense that
$$
\chT_{k,1}(\psi \circ p_1) = \left(\hT_{k,1}\psi\right) \circ p_1
\text{ for every } \psi \in \Bd(\cY_1(\vep_1)).
$$
Thus, the construction in Proposition~\ref{p_sec6_Margulis_finite_energy} can be applied
simultaneously to the two operators, to yield a sequence of probability measures $\cheta_{k,1,j}$
converging to a $\chT_{k,1}$-invariant measure $\cheta_{k,1}$ on $\cY_1(\vep_1)$,
and whose projections down to $\prodspa1$ are self-couplings $\heta_{k,1,j}$ of the $\eta_{k,1}$
vanishing on neighborhoods of the diagonal of $\prodspa1$,
and converging to the $\hT_{k,1}$-invariant measure $\heta_{k,1}$.

Next, define $\eta_{k,2} = p_{2*} \cheta_{k,1}$ and let $\{d\cheta_{k,1,v}: v \in p_2 \cY_1(\vep_1)\}$
be a disintegration of $\cheta_{k,1}$ with respect to the partition $\{p_2^{-1}(v): v \in p_2 \cY_1(\vep_1)\}$.
Then define
\begin{equation*}%\label{eq_r=1_cTk2a}
\begin{aligned}
& \cT_{k,2}:\Bd(p_2 \cY_1(\vep_1)) \to \Bd(p_2 \cY_1(\vep_1)), \\
& \cT_{k,2}\Phi(y) = \int_{p_2^{-1}(y)} \chT_{k,1} (\Phi \circ p_2) (x,x',y) \, d\cheta_{k,1,y}(x,x').
\end{aligned}
\end{equation*}
Equivalently, $\cT_{k,2}\Phi(y) = \int_{p_2\cY_1(\vep_1)} \Phi \, d\sigma_{k,2,y}$ with
\begin{equation}\label{eq_r=1_cTk2b}
\sigma_{k,2,y} = \int_{p_2^{-1}(y)} p_{2*} \chsigma_{k,1,x,x',y} \, d\cheta_{k,1,y}(x,x').
\end{equation}

Let $\cB_k=p_2p_1^{-1}(B''_k)=\{x+x': (x, x') \in B''_k\}$, where $B''_k$ is as in the
previous section. Define also $\eta_{k,2,j} = p_{2*} \cheta_{k,1,j}$ for $j\in\NN$. Then
$$
\eta_{k,2,j} (\cB_k) \ge \cheta_{k,1,j} \left(p_1^{-1}(B''_k)\right) = \heta_{k,1,j}(B''_k).
$$
Passing to the limit as $j \to \infty$ and arguing as in
\eqref{eq_r=1_conclusion1}--\eqref{eq_r=1_conclusion4} we find that
$$
\eta_{k,2}(\cB_k) \ge \frac{5\kappa'_B}{\kappa_A + \kappa''_B}\eta_{k,1}(E_1(\vep_1)).
$$
Now, the definition \eqref{eq_r=1_Bsecond_bis} implies that $\cB_k$ converges to $E_2$
as  $k \to \infty$, because $\omega_{k,1}\to 0$.
Thus, any accumulation point $\eta_{\infty,2}$ of $\eta_{k,2}$ must satisfy
\begin{equation}\label{eq_r=1_limit_weight_is_positive}
\eta_{\infty,2}(E_2) \ge \frac{5\kappa'_B}{\kappa_A + \kappa''_B}\eta_{\infty,1}(E_1) >0.
\end{equation}

Take $n_2=n$ and $\vep_2 = \vep''_1$. Let $\nu^{(n_2)}_{k,2,y}$ denote the push-forward
of $\nu^{(n_2)}_k$ under the map $G \to \grass(2,d)$, $g \mapsto gy$.

\begin{lemma}\label{l_r=1_downloading1}~
\begin{enumerate}
\item[(i)] $\sigma_{k,2,y} = \nu^{(n_2)}_{k,2,y}$ for every $y \in E_2(\vep_2)$.
\item[(ii)] $\sigma_{k,2,y} \left(\Chi_{\nu^{(n)}_k}^{\#} E_2(\vep_2)\right) = 0$ for every $y \notin E_2(\vep_2)$
\item[(iii)] The measure $\eta_{k,2}$ is $\cT_{k,2}$-invariant.
\end{enumerate}
\end{lemma}

\begin{proof}
It is clear that $\nu^{(n_2)}_{k,2,y}$ coincides with the push-forward of $\nu^{(n_2)}_{k,1,x,x',y}$
under the projection $p_2$. Thus \eqref{eq_r=1_interpolate3} gives that
\begin{equation*}%\label{eq_r=1_interpolate4}
p_{2*} \chsigma_{k,1,x,x',y}= \left(1-\chomega(x,x')\right)\nu^{(n_2)}_{k,2,y} +
\chomega(x,x') p_{2*}\chtheta_{k,1,x,x',y},
\end{equation*}
and so,
\begin{equation*}%\label{eq_r=1_interpolate5}
\begin{aligned}
\chsigma_{k,2,y}
& = \left(1-\int_{p_2^{-1}(y)} \chomega(x,x')  \, d\cheta_{k,1,y}(x,x')\right)\nu^{(n_2)}_{k,2,y} \\
& \hspace{3cm} + \int_{p_2^{-1}(y)} \chomega(x,x') p_{2*}\chtheta_{k,1,x,x',y} \, d\cheta_{k,1,y}(x,x'),
\end{aligned}
\end{equation*}
If $y \in E_2(\vep_2)$ then both $x$ and $x'$ are necessarily in $E_1(\vep_2)$,
by \eqref{eq_sec7_Theta1}, in which case $\chomega(x,x')=0$.
Then $\sigma_{k,2,y} = \nu^{(n_2)}_{k,2,y}$, as claimed in (i).

In view of the expression \eqref{eq_r=1_cTk2b}, to prove part (ii) it suffices to show that if
$y \notin E_2(\vep_2)$ then
\begin{equation}\label{eq_r=1_pi2vanish}
p_{2*} \chsigma_{k,1,x,x',y}\left(\Chi_{\nu^{(n)}_k}^{\#} E_2(\vep_2)\right)=0
\end{equation}
for any $x, x' \subset y$. If $x$ and $x'$ are both in $E_1(\vep_2)$ then
$$
p_{2*} \chsigma_{k,1,x,x',y}
=p_{2*} \nu^{(n_2)}_{k,1,x,x',y}
=\nu^{(n_2)}_{k,2,y}
$$
and then the claim follows from Remark~\ref{r_sec6_notdeep}. From now on, we
assume that one of the points, $x$ say, is not in $E_1(\vep_2)$.
It follows from the definitions that
%It is clear that if $v \in E_2(\vep_2)$ and $u, u' \subset v$ then $u, u' \in E_1(\vep_2)$.
%Thus $p_2^{-1}(E_2(\vep_2))$ is contained in $E_1(\vep_2)^2 \times E_2(\vep_2)$.
%Analogously,
$$
p_2^{-1}\left(\Chi_{\nu^{(n)}_k}^{\#} E_2(\vep_2)\right)
\subset \Chi_{\nu^{(n)}_k}^{\#} E_1(\vep_2)^2 \times \Chi_{\nu^{(n)}_k}^{\#} E_2(\vep_2),
$$
and so
\begin{equation}\label{eq_r=1_vanishingmeasure}
\begin{aligned}
p_{2*} \chsigma_{k,1,x,x',y}\left(\Chi_{\nu^{(n)}_k}^{\#} E_2(\vep_2)\right)
& \le \chsigma_{k,1,x,x',y}\left(\Chi_{\nu^{(n)}_k}^{\#} E_1(\vep_2)^2 \times \Chi_{\nu^{(n)}_k}^{\#} E_2(\vep_2)\right) \\
& \le \hsigma_{k,1,x,x'}\left(\Chi_{\nu^{(n)}_k}^{\#} E_1(\vep_2)^2\right)\\
& \le \sigma_{k,1,x}\left(\Chi_{\nu^{(n)}_k}^{\#} E_1(\vep_2)\right).
\end{aligned}
\end{equation}
If $x$ is in the $\nu^{(n)}_k$-border of $E_1(\vep_1)$ then
$$
\sigma_{k,1,x}\left(\Chi_{\nu^{(n)}_k}^{\#} E_1(\vep_2)\right)
\le \sigma_{k,1,x}\left(\Chi_{\nu^{(n)}_k}^{\#} E_1(\vep_1)\right) = 0,
$$
because the operator $\cT_{k,1}$ is adapted to $(\nu_k^{(n)},E_1(\vep_1))$.
If $x$ is in the $\nu^{(n)}_k$-core of $E_1(\vep_1)$ then Remark~\ref{r_sec6_notdeep} gives that
$$
\sigma_{k,1,x}\left(\Chi_{\nu^{(n)}_k}^{\#} E_1(\vep_2)\right)
= \nu^{(n)}_{k,1,x}\left(\Chi_{\nu^{(n)}_k}^{\#} E_1(\vep_2)\right) = 0.
$$
Thus the right-hand side of \eqref{eq_r=1_vanishingmeasure} vanishes in either case.
That completes the proof of \eqref{eq_r=1_pi2vanish} and of part (ii) of the lemma.

Finally, by definition,
$$
\begin{aligned}
\int_{p_2\cY_1(\vep_1)} \left(\cT_{k,2}\Phi\right) & \, d\eta_{k,2}\\
& =  \int_{p_2\cY_1(\vep_1)} \int_{p_2^{-1}(y)} \chT_{k,1} \left(\Phi \circ p_2\right) (x,x',y) \, d\cheta_{k,1,y}(x,x')\, d\eta_{k,2}(y)\\
& = \int_{\cY_1(\vep_1)}  \chT_{k,1}(\Phi \circ p_2) (x,x',y) \, d\cheta_{k,1}(x, x', y)
\end{aligned}
$$
for any $\Phi\in \Bd(p_2\cY_1(\vep_1))$.
Since $\cheta_{k,1}$ is $\chT_{k,1}$-invariant, this gives
$$
\begin{aligned}
\int_{p_2\cY_1(\vep_1)} \left(\cT_{k,2}\Phi\right)(y) \, d\eta_{k,2}(y)
& = \int_{\cY_1(\vep_1)}  (\Phi \circ p_2) (x,x',y) \, d\cheta_{k,1}(x, x', y)\\
& = \int_{p_2\cY_1(\vep_1)} \Phi \, d\heta_{k,2},
\end{aligned}
$$
which proves claim (iii).
\end{proof}

Since the $\chsigma_{k,1,x,x',y}$ are generic measures and the projection $p_2$ is algebraic,
it follows readily from \eqref{eq_r=1_cTk2b} and Remark~\ref{r_sec5_generic_elementary}
that every $\sigma_{k,2,y}$ is a generic measure.
Then, conclusions (i) and (ii) in Lemma~\ref{l_r=1_downloading1} allow us to apply Propositions~\ref{p_sec6_localized2_strong} and~\ref{p_sec6_adapted_strong} to
$X=p_2\cY_1(\vep_1)$ and $U=E_2(\vep_2)$.
In this way we get a continuous Markov operator adapted to $(\nu_k^{(n_2)},E_2(\vep_2))$
which leaves invariant the restriction of $\eta_{k,2} \mid E_2(\vep_2)$.
Replace $\cT_{k,2}$ and $\eta_{k,2}$ with this new Markov operator and invariant measure,
respectively. This finishes the first step of the induction.

\part{General step of the induction}\label{p_four}

\section{Preparing a Margulis function}\label{s_preparing_a_Margulis_function}

Let $r>1$ be fixed. Here we extend the construction in Section~\ref{s_step1_preparing_a_Margulis_function},
to find a positive function $\psi_r$ such that $-\log \psi_r$ has some of the features of a Margulis function.
The main result is Proposition~\ref{p_Prop5.1Lem5.5}, an extension of Proposition~\ref{p_r=1_Prop5.1}.
Throughout, $\delta>0$ and $n\in\NN$ should be seen as free parameters, whose values are fixed at the end of the construction.

Recall that $\cF(r,d)$ denotes the set of flags $F_1 \subset F_2 \subset \cdots \subset F_{r-1} \subset F_r \subset \RR^d$,
where each $F_i$ has dimension $i$. Moreover,
$$
\begin{aligned}
\fE_r & = \{(F_1, \dots, F_r) \in\cF(r,d): F_r \in E_r\} \text{ and},\\
\fE_r(\vep) & = \big\{(F_1, \dots, F_r) \in\cF(r,d): F_r \in E_r(\vep)\big\}
\text{ for each $\vep >0$.}
\end{aligned}
$$
We use $x=(F_1, \dots, F_r)$ and $x'=(F'_1, \dots, F'_r)$ to denote generic elements of $\cF(r,d)$.

Most steps towards Proposition~\ref{p_Prop5.1Lem5.5} are rather straightforward translations
of the arguments in Section~\ref{s_step1_preparing_a_Margulis_function}.
One significant difference is that Lemma~\ref{l_r=1_Lem5.8} no longer holds:
for $r>1$ it is possible to have $x$ and $x'$ with $F_r$ and $F'_r$ arbitrarily close to the equator $E$
without $\VP_r(x,x)'$ getting close to zero. For that reason, we cannot take $\psi_r=\VP_r$.
This is dealt with in Section~\ref{s_function_psir}: instead, we define $\psi_r$ inductively in terms of
both $\VP_r$ and $\psi_{r-1}$.

\subsection{Vertical angle function}\label{ss_vertical_angle}

Assume $x, x' \in \cF(r,d)$ to be such that $F'_1\not\subset F_r$.
By definition, the \emph{great circle} through $x$ and $x'$ is the subset $y=y(x,x')$ of $\grass(r,d)$
defined by
$$
y = \{\xi \in \grass(r,d): F_{r-1} \subset \xi \subset F'_1+F_r\}.
$$
This is consistent with the case $r=1$, as long as we follow the convention that $F_0=\{0\}$.
On the other hand, the great circle $y$ depends on $x$ through $F_{r-1}$ and $F_r$, whereas it depends on $x'$
through $F'_1$ only. In particular, $y(x,x')$ need not coincide with $y(x',x)$ when $r>1$. Related to this,
the analogue of \eqref{eq_r=1_commute} is usually false when $r>1$: the values of $\VA_r(x,x')$ and $\VP_r(x,x')$
that we define in the following may change when the roles of $x$ and $x'$ are exchanged.

The \emph{vertical angle function} $\VA_r$ is defined by
\begin{equation}\label{eq_VAR_def}
\VA_r(x,x')
= \sup_{\xi \in y} d(\xi,E)
= \sup_{\xi \in y} \sup_{u \in \xi} d(u,E).
\end{equation}
Note that this is consistent with the case $r=1$. Following the intuition from $r=1$, we think of $\VA_r(x,x')$ as
the angle between the great circle $y$ and $E_r$. Indeed,
\begin{equation}\label{eq_VAR_defbis}
\VA_r(x,x')
= \sup_{u \in F'_1+F_r} d(u,E)
= d(F'_1+F_r,E),
\end{equation}
and so $\VA_r(x,x')=0$ if and only if $F'_1+F_r\subset E$, that is, if and only if $y\subset E_r$.

We are going to prove the following extension of Proposition~\ref{p_r=1_Prop5.1} for $r>1$:

\begin{proposition}\label{p_Prop5.1}
There exist $\kappa'_r=\kappa'_r(\nu_\infty)>0$ and $C'_r=C'_r(\nu_\infty)>0$ and for any $\delta>0$
there exists $N_r=N_r(\nu_\infty,\delta)$ such that for every $n \ge N_r$ there exists
$\rho'_r=\rho'_r(\nu_\infty,\delta,n)>0$ such that
$$
\int_G - \log\VP_r(gx,gx') \, d\nu_\infty^{(n)}(g) \le -\log\VP_r(x,x') - (\kappa'_r-C'_r\delta)n
$$
for every $x, x' \in \fE_r(\rho'_r)$ with $F'_1 \not\subset F_r$.
\end{proposition}

Begin by noting that the properties \eqref{eq_r=1_VA23again} through
\eqref{eq_r=1_VA1_def_linear} extend to $r>1$.
More precisely, it follows from \eqref{eq_VAR_defbis} that
\begin{equation}\label{eq_VA2again}
\VA_r(x,x') \ge d(u,E) = \frac{\|u^\perp\|}{\|u\|}
\end{equation}
for any $u \in F'_1+F_r$ and $x, x' \in\cF(r,d)$ with $F'_1\not\subset F_r$.
Moreover, if $u \in F'_1+F_r$ realizes the supremum in the definition \eqref{eq_VAR_defbis} then
\begin{equation}\label{eq_VA3again}
\VA_r(x,x') = d(u,E) = \frac{\|u^\perp\|}{\|u\|}.
\end{equation}
Then, combining \eqref{eq_VA2again} with \eqref{eq_VA3again},
\begin{equation}\label{eq_VA23again}
\VA_r(gx,gx')
\ge \frac{\|(g u)^\perp\|}{\|gu\|}
= \VA_r(x,x') \frac{\|(g u)^\perp\|}{\|gu\|}\frac{\|u\|}{\|u^\perp\|}
\end{equation}
for any $g\in G$.  By Remark~\ref{r_sec3_perp}, when $g\in\supp\nu_\infty^{(n)}$ this means that
\begin{equation}\label{eq_VA5again}
\VA_r(gx,gx')
\ge \frac{\|g^\perp u^\perp\|}{\|gu\|}
\ge \VA_r(x,x') \frac{\|g^\perp u^\perp\|}{\|gu\|}\frac{\|u\|}{\|u^\perp\|}.
\end{equation}
Furthermore, just as for \eqref{eq_r=1_VA1_def_linear},
\begin{equation}\label{eq_VAr_linear}
- \log\VA_r(gx,gx') \le - \log\VA_r(x,x') + B n
\end{equation}
for any $x, x' \in \cF(r,d)$ with $F'_1 \not\subset F_r$ and $g\in\supp\nu^{(n)}_\infty$.

The proof of Proposition~\ref{p_Prop5.1} is analogous to that of Proposition~\ref{p_r=1_Prop5.1},
replacing $d(x,x')$, $d(gx,gx')$ and $d(x,E)$ with $d(F_r,F_{r-1}+F'_1)$, $d(gF_r,g(F_{r-1}+F'_1))$
and $d(F_r,E)$, respectively, and substituting $r$ for $1$ in the subscript.
The details follow, but the reader may choose to skip them and proceed directly
to Section~\ref{s_function_psir}.

\begin{lemma}\label{l_Lem5.4}
There exists $\tilde\kappa_r=\tilde\kappa_r(\nu_\infty)>0$ and for each $\delta>0$ there exist
$\tilde\theta_r=\tilde\theta_r(\nu_\infty,\delta)>0$ and $\widetilde{N}_r=\widetilde{N}_r(\nu_\infty,\delta)\in\NN$ such that for every $n \ge \widetilde{N}_r$
and $x, x' \in\cF(r,d)$ with $F'_1\not\subset F_r$ there exists
$\tE_r=\tE_r(\nu_\infty,\delta,n,x,x')\subset\supp\nu_\infty^{(n)}$
with $\nu_\infty^{(n)}(\tE_r^c)<\delta$ and
\begin{equation}\label{eq_VA4}
-\log\VA_r(gx,gx') \le \max\{-\log\VA_r(x,x') - \tilde\kappa_r n, \tilde\theta_r\}
\text{ for every $g\in\tE_r$.}
\end{equation}
\end{lemma}

\begin{proof}
Analogous to Lemma~\ref{l_r=1_Lem5.4}.
Let $\kappa_0=\kappa_0(\nu_\infty)>0$, $N_0=N_0(\nu_\infty,\delta)\in\NN$, and
$\cE_0=\cE_0(\nu_\infty,\delta,n,u^\perp)\subset \supp\nu_\infty^{(n)}$ and $\tau_0=\tau_0(\nu_\infty,\delta)>0$
be as in Proposition~\ref{p_sec5_derivative_estimate}.
Given $x, x' \in\cF(r,d)$ with $F'_1\not\subset F_r$, take $u$ to be a non-zero vector that
realizes the supremum in the definition \eqref{eq_VAR_defbis}.
Write $u=u^E+u^\perp$ with $u^E\in E$ and $u^\perp\in E^\perp$. Take
\begin{equation}\label{eq_VA5.5}
\begin{aligned}
\tilde\kappa_r=\kappa_0/2,
& \quad \tilde\theta_r=-\log(\tau_0/4), \\
\widetilde{N}_r > \max\{N_0, 4/\tilde\kappa_r\},
& \quad \text{and }
\tE_r=\cE_0(\nu_\infty,\delta,n,u^\perp).
\end{aligned}
\end{equation}
Let $n\ge \widetilde{N}_r$ and $g\in\tE_r$; then $g \in \supp \nu_\infty^{(n)}$.
If  ${\|g^\perp u^\perp\|}/{\|gu\|} \ge {\tau_0}/{2}$ then the first inequality in \eqref{eq_VA5again} implies that
\begin{equation}\label{eq_VA6}
- \log\VA_r(gx,gx') \le - \log\frac{\tau_0}{2} \le \tilde\theta_r.
\end{equation}
If ${\|g^\perp u^\perp\|}/{\|gu\|} < {\tau_0}/{2}$ then part (2) of Proposition~\ref{p_sec5_derivative_estimate} gives that
$$
\frac{\|g u^\perp\|}{\|gu\|}
%= \frac{\|g u^\perp\|}{\|g^\perp u^\perp\|} \frac{\|g^\perp u^\perp\|}{\|gu\|} < \frac{1}{\tau_0}\frac{\tau_0}{2}
< \frac 12,
\text{ and so }
\frac{\|g u^E\|}{\|gu\|}
> \frac 12.
$$
Substituting the latter inequality and $\|u\| \ge \|u^E\|$ in \eqref{eq_VA5again}, we find that
\begin{equation}\label{eq_VA6.25}
\VA_r(gx,gx')
\ge \frac {1}{2} \VA_r(x,x') \frac{\|g^\perp u^\perp\|}{\|u^\perp\|}\frac{\|u^E\|}{\|g u^E\|}.
\end{equation}
Thus, recalling the definition \eqref{eq_sec5_derivative2} and part (1) of
Proposition~\ref{p_sec5_derivative_estimate},
$$
\VA_r(gx,gx')
\ge \frac {1}{2} \VA_r(x,x') \frac{\|Dg_{u^E}^\perp u^\perp\|}{\|u^\perp\|}
\ge \frac {1}{2} \VA_r(x,x') e^{\kappa_0 n}.
$$
By the choices of $\tilde\kappa_r$ and $\widetilde{N}_r$ in \eqref{eq_VA5.5}, this implies
\begin{equation}\label{eq_VA7}
\begin{aligned}
-\log\VA_r(gx,gx')
& \le - \log \VA_r(x,x') + \log 2 - 2 \tilde\kappa_r n \\
& \le - \log \VA_r(x,x') - \tilde\kappa_r n.
\end{aligned}
\end{equation}
The conclusion of the lemma is contained in \eqref{eq_VA6} and \eqref{eq_VA7}.
\end{proof}

\subsection{Vertical projection function}\label{ss_vertical_projection}

The \emph{vertical projection function} $\VP_r$ is defined by
\begin{equation}\label{eq_VPr_def}
\VP_r(x,x') = \VA_r(x,x')d(F_r,F_{r-1}+F'_1)^{\gamma_r}
\end{equation}
where $\gamma_r=\gamma_r(\nu_\infty)$ is a small positive constant chosen through the following
result, which extends Proposition~\ref{p_r=1_Lem5.5} to $r>1$:

\begin{proposition}\label{p_Lem5.5}
There exist $\gamma_r=\gamma_r(\nu_\infty)>0$ and $\kappa'_r=\kappa'_r(\nu_\infty)>0$ and for each
$\delta>0$ there exists $N_r=N_r(\nu_\infty,\delta)\in\NN$ such that for every $n\ge N_r$ there exists
$\rho'_r=\rho'_r(\nu_\infty,\delta,n)>0$ such that for any $x, x' \in \fE_r(\rho'_r)$ with $F'_1\not\subset F_r$
there exists $\cE'_r=\cE'_r(\nu_\infty,\delta,n,x,x')\subset\supp\nu_\infty^{(n)}$ with
$\nu_\infty^{(n)}((\cE'_r)^c)<\delta$ and
\begin{equation}\label{eq_VPstat}
- \log \VP_r(gx,gx') \le -\log\VP_r(x,x')-\kappa'_r n \text{ for every $g\in\cE'_r$.}
\end{equation}
\end{proposition}

For the proof of Proposition~\ref{p_Lem5.5} we need to extend Lemmas~\ref{l_r=1_Lem5.6} and~\ref{l_r=1_Lem5.7}
to $r>1$, which we do in the couple of statements that follow.

Let $v\in F_r$ and $v'\in F_{r-1} + F'_1$ be unit vectors orthogonal to $F_{r-1}$ such that
\begin{equation}\label{eq_vvprime}
\begin{aligned}
F_r & = F_{r-1} + \RR v, \\
F_{r-1}+F'_1 & = F_{r-1} + \RR v', \text{ and}\\
d(F_r,F_{r-1}+F'_1)& = |\sin \angle(v,v')|.
\end{aligned}
\end{equation}
It is no restriction to take the angle between $v$ and $v'$ to be non-obtuse.
Then the vector $w=v'-v$ satisfies (compare Figure~\ref{f.dist1})
\begin{equation}\label{eq_dista12}
d(F_r,F_{r-1}+F'_1) \le \|w\|
\quand
\angle(w,v) \ge \frac{\pi}{4}.
\end{equation}

\begin{lemma}\label{l_Lem5.6}
Given $x, x' \in \cF(r,d)$ with $F'_1\not\subset F_r$, let $w=v'-v$ be as in \eqref{eq_dista12}.
Then
\begin{equation}\label{eq_swc}
d(F_r,E) < \frac {1}{4} \VA_r(x,x')
\text{ implies }
\frac{\|w^\perp\|}{\|w\|} > \frac {1}{4} \VA_r(x,x').
\end{equation}
\end{lemma}

\begin{proof}
Analogous to Lemma~\ref{l_r=1_Lem5.6}.
Let $u\in F'_1+F_r$ be any non-zero vector that realizes the supremum in the definition \eqref{eq_VAR_defbis}.
Since $F'_1+F_r=\gene\{F_{r-1},v,w\}$, we may write $u=u_0 + av + bw$
with $u_0 \in F_{r-1}$ and $a, b\in\RR$. Recall that $v$ and $w$ are orthogonal to $F_{r-1}$.
Moreover, by \eqref{eq_dista12} the angle between them is no less than $\pi/4$.
This implies that $\|u_0+av\|$ and $\|bw\|$ are both less than $2\|u\|$. Thus,
$$
\begin{aligned}
\VA_r(x,x')
& = \frac{\|u^\perp\|}{\|u\|}
\le \frac{\|(u_0+av)^\perp\|}{\|u\|} + \frac{\|bw^\perp\|}{\|u\|} \\
& < 2 \frac{\|(u_0+av)^\perp\|}{\|u_0+av\|} + 2 \frac{\|w^\perp\|}{\|w\|}
\le 2 d(F_r,E) + 2 \frac{\|w^\perp\|}{\|w\|}.
\end{aligned}
$$
Thus, $d(F_r,E)$ and $\|w^\perp\|/\|w\|$ cannot be both less than $\VA_r(x,x')/4$.
\end{proof}

Take $\tilde\kappa_r=\tilde\kappa_r(\nu_\infty)>0$ and $\tilde\theta_r=\tilde\theta_r(\nu_\infty,\delta)>0$ to be as
in Lemma~\ref{l_Lem5.4}.

\begin{lemma}\label{l_Lem5.7}
There exists $\hat\kappa_r=\hat\kappa_r(\nu_\infty)>0$ and for each $\delta>0$ there exists
$\widehat{N}_r=\widehat{N}_r(\nu_\infty, \delta)\in\NN$ such that for each $n\ge \widehat{N}_r$ there exists
$\hat\rho_r=\hat\rho_r(\nu_\infty,\delta,n)>0$ such that for any $x, x' \in \fE_r(\hat\rho_r)$
with $F'_1\not\subset F_r$ and $-\log\VA_r(x,x') \le \tilde\theta_r+\tilde\kappa_r n$ there exists
$\hat\cE_r=\hat\cE_r(\nu_\infty,\delta,n,x,x')\subset\supp\nu_\infty^{(n)}$ with
$\nu_\infty^{(n)}(\hat\cE_r^c)<\delta$ and
\begin{equation}\label{eq_5.17}
-\log d(gF_r,g(F_{r-1}+F'_1)) \le -\log d(F_r,F_{r-1}+F'_1)-\log\VA_r(x,x')-\hat\kappa_rn
\end{equation}
for every $g\in\hat\cE_r$.
\end{lemma}

\begin{proof}
Analogous to Lemma~\ref{l_r=1_Lem5.7}.
Let $\kappa_0=\kappa_0(\nu_\infty)>0$, $N_0=N_0(\nu_\infty,\delta)\in\NN$, and
$\cE_0=\cE_0(\nu_\infty,\delta,n,v^\perp)\subset\supp\nu_\infty^{(n)}$ be as in Proposition~\ref{p_sec5_derivative_estimate}.
Given $x$ and $x'$ with $F'_1\not\subset F_r$, let $w=v'-v$ be as in \eqref{eq_dista12}. Take
\begin{equation}\label{eq_dista22}
\begin{aligned}
& \hat\kappa_r= \tilde\kappa_r/2, \quad
\widehat{N}_r = \max\left\{N_0, {5}/{\hat\kappa_r}\right\}, \\
& \hat\rho_r < e^{- \tilde\theta_r - \tilde\kappa_r n}/10,
\quand
\hat\cE_r=\cE_0(\nu_\infty,\delta,n,w^\perp).
\end{aligned}
\end{equation}
Let $n\ge \widehat{N}_r$ and $g\in\hat\cE_r \subset \supp\nu_\infty^{(n)}$.
As observed in \eqref{eq_dista12},
\begin{equation}\label{eq_dista6}
d(F_r,F_{r-1}+F'_1) \le \|w\|.
\end{equation}
Let us suppose that $\|gv\|\ge\|gv'\|$; the case $\|gv\|\le\|gv'\|$ is analogous,
reversing the roles of $F_r$ and $F_{r-1}+F'_1$.
%As in \eqref{eq_r=1_dista7},
%\begin{equation}\label{eq_dista7}
%d(gF_r,g(F_{r-1}+F'_1))  \ge \Big\|\Pi_{gv} \frac{gw}{\|gv\|}\Big\|
%\text{ for any $g\in G$,}
%\end{equation}
%and thus (recall that $v$ and $v'$ are unit vectors),
Just as in \eqref{eq_r=1_dg1},
\begin{equation}\label{eq_dg1}
\begin{aligned}
-\log d(gF_r,g(F_{r-1}+F'_1))
\le - \log d(F_r,F_{r-1}+F'_1) - \log \frac{\|\Pi_{gv}gw\|}{\|gv\|}\frac{\|v\|}{\|w\|}.
\end{aligned}
\end{equation}
for any $g\in G$.
By the condition on $\hat\rho_r$ in \eqref{eq_dista22}, if $d(F_r,E)\le\hat\rho_r$ then
$$
4d(F_r,E) < e^{- \tilde\theta_r - \tilde\kappa_r n} < \VA_r(x,x'),
$$
and then Lemma~\ref{l_Lem5.6} gives that
\begin{equation}\label{eq_dista75}
\frac{\|w^\perp\|}{\|w\|} > \frac 14 \VA_r(x,x') > \frac 14 e^{- \tilde\theta_r - \tilde\kappa_r n}.
\end{equation}
Up to further reducing $\hat\rho_r$, we may also assume that
\begin{equation}\label{eq_dista9fresh}
\frac{\|hz\|}{\|z\|} \ge \frac 12 \frac{\|hz^E\|}{\|z^E\|}
\quand
\|\Pi_{hz} hw\| \ge \frac 12 \|\Pi_{hz^E} gw\|
\end{equation}
for any non-zero $z=z^E+z^\perp$ in $E\oplus E^\perp$ with $\|z^\perp\|/\|z\| \le \hat\rho_r$
and any $h\in\supp\nu_\infty^{(n)}$. The second part of \eqref{eq_dista9fresh} implies that
\begin{equation}\label{eq_dista10}
\|\Pi_{hz} hw\| \ge \frac 12\|\Pi_{hz^E} hw\| \ge \frac 12\|\Pi_E hw\|
= \frac 12\|(hw)^\perp\| = \frac 12\|h^\perp w^\perp\|.
\end{equation}

Noting that $\|v^\perp\|/\|v\| \le d(F_r,E)\le\hat\rho_r$, take $z=v$ and $h=g$ in the previous
two relations. Thus, substituting \eqref{eq_dista9fresh} and \eqref{eq_dista10} in \eqref{eq_dg1},
\begin{equation*}
-\log d(gF_r,g(F_{r-1}+F'_1)) \\
\le - \log d(F_r,F_{r-1}+F'_1) + \log 4 - \log \frac{\|g^\perp w^\perp\|}{\|w\|}\frac{\|v^E\|}{\|gv^E\|}.
\end{equation*}
Then, using also \eqref{eq_dista75} and \eqref{eq_sec5_derivative2},
\begin{equation*}
\begin{aligned}
& -\log d(gF_r,g(F_{r-1}+F'_1)) \\
& \hspace{1cm} \le - \log d(F_r,F_{r-1}+F'_1) - \log\VA_r(x,x') + \log 16 - \log \frac{\|g^\perp w^\perp\|}{\|w^\perp\|}\frac{\|v^E\|}{\|gv^E\|}\\
& \hspace{1cm} = - \log d(F_r,F_{r-1}+F'_1) - \log\VA_r(x,x') + \log 16 -  \log\frac{\|Dg^\perp_{v^E} w^\perp\|}{\|w^\perp\|}.
\end{aligned}
\end{equation*}
By part (1) of Proposition~\ref{p_sec5_derivative_estimate}
and the choice of $\hat\kappa_r$ and $\widehat{N}_r$ in \eqref{eq_dista22},
it follows that
$$
\begin{aligned}
-\log d(gF_r,g(F_{r-1}+F'_1))
& \le - \log d(F_r,F_{r-1}+F'_1) - \log \VA_r(x,x') + 5 - 2\hat\kappa_r n \\
& \le -\log d(F_r,F_{r-1}+F'_1) - \log \VA_r(x,x') - \hat\kappa_r n,
\end{aligned}
$$
as claimed.
\end{proof}

\begin{proof}[Proof of Proposition~\ref{p_Lem5.5}]
Analogous to Proposition~\ref{p_r=1_Lem5.5}. Take $A=A(\nu_\infty)$ as in \eqref{eq_AAA}
 and then define
\begin{equation}\label{eq_constants}
\begin{aligned}
& \gamma_r = \min\big\{1,{\tilde\kappa_r}/{(2A)}\big\},\\
& \kappa'_r=\min\big\{{\tilde\kappa_r}/{2}, {\gamma_r\hat\kappa_r}/{2}\big\}, \\
& \rho'_r=\min\{\hat\rho_r(\nu_\infty,\delta),\hat\rho_r(\nu_\infty,{\delta}/{2})\}, \\
& N_r=\max\big\{\widetilde{N}_r(\nu_\infty,{\delta}/{2}), \widehat{N}_r(\nu_\infty,{\delta}/{2}),
                         2\tilde\theta_r/(\gamma\hat\kappa_r)\big\}, \text{ and} \\
& \cE'_r=\tE_r(\nu_\infty,{\delta}/{2},n,x,x') \cap \hat\cE_r(\nu_\infty,{\delta}/{2},n,x,x').
\end{aligned}
\end{equation}
By construction, $\nu_\infty^{(n)}((\cE'_r)^c)<\delta$.
The definition \eqref{eq_VPr_def} gives that
\begin{equation}\label{eq_VP1}
-\log\VP_r(gx,gx')=-\log\VA_r(gx,gx') - \gamma_r \log d(gF_r,g(F_{r-1}+F'_1)).
\end{equation}
%The relation  \eqref{eq_def_AAA} implies that
%\begin{equation}\label{eq_AAAaplied}
%- \log d(gF_r,g(F_{r-1}+F'_1)) \le - \log d(F_r, F_{r-1}+F'_1) + A n
%\end{equation}
%for any $g \in \supp\nu_\infty^{(n)}$.

Consider $x, x' \in \fE_r(\rho'_r)$ with
$F'_1 \not\subset F$, and let $n\ge N_r$ and $g\in\cE'_r$.
First, suppose that $-\log\VA_r(x,x')\ge\tilde\theta_r+ \tilde\kappa_r n$. Then, by Lemma~\ref{l_Lem5.4},
\begin{equation}\label{eq_VP2}
-\log\VA_r(gx,gx') \le - \log\VA_r(x,x') - \tilde\kappa_r n
\end{equation}
Substituting \eqref{eq_VP2} and \eqref{eq_AAA} in \eqref{eq_VP1} we find that
$$
-\log\VP_r(gx,gx')
\le  - \log \VA_r(x,x') - \tilde\kappa_r n - \gamma_r \log d(F_r,F_{r-1}+F'_1) + \gamma_r A n.
$$
By \eqref{eq_VP1} and the choice of $\gamma_r$ and $\kappa'_r$ in \eqref{eq_constants}, this yields
\begin{equation}\label{eq_VP3}
-\log\VP_r(gx,gx')
\le - \log \VP_r(x,x') - \frac{\tilde\kappa_r}{2} n \le - \log \VP_r(x,x') - \kappa'_r n.
\end{equation}

Now assume that $-\log\VA_r(x,x') \le \tilde\theta_r+\tilde\kappa_r n$.  In this case, Lemma~\ref{l_Lem5.4} yields
\begin{equation}\label{eq_VP4}
-\log\VA_r(gx,gx') \le \tilde\theta_r
\end{equation}
and Lemma~\ref{l_Lem5.7} gives that
\begin{equation}\label{eq_VP5}
- \log d(gF_r,g(F_{r-1}+F'_1)) \le -\log d(F_r,F_{r-1}+F'_1) - \log\VA_r(x,x') - \hat\kappa_r n.
\end{equation}
Substituting \eqref{eq_VP4} and \eqref{eq_VP5} in \eqref{eq_VP1} we obtain
$$
\begin{aligned}
-\log\VP_r(gx,gx')
& \le \tilde\theta_r - \gamma_r \log d(F_r,F_{r-1}+F'_1) - \gamma_r \log\VA_r(x,x') - \gamma_r\hat\kappa_rn \\
& \le -\log\VP_r(x,x')  + \tilde\theta_r + (1-\gamma_r) \log \VA_r(x,x') - \gamma_r\hat\kappa_rn.
\end{aligned}
$$
Since $\VA_r(x,x') \le 1$, $\gamma_r \le 1$, and $n \ge N_r \ge 2\tilde\theta_r/(\gamma_r\hat\kappa_r)$,
this yields
\begin{equation}\label{eq_VP6}
\begin{aligned}
-\log\VP_r(gx,gx')
& \le -\log\VP_r(x,x') + \tilde\theta_r - \gamma_r\hat\kappa_rn\\
& \le -\log\VP_r(x,x') - \frac{\gamma_r\hat\kappa_r}{2} n
\le  - \log\VP_r(x,x') - \kappa'_r n.
\end{aligned}
\end{equation}
The relations \eqref{eq_VP3} and \eqref{eq_VP6} contain the conclusion of Proposition~\ref{p_Lem5.5}.
\end{proof}

\subsection{The function $-\log\psi_r$}\label{s_function_psir}
At this point, the proof of Proposition~\ref{p_Prop5.1} is analogous to that of Proposition~\ref{p_r=1_Prop5.1}.
Take $\kappa'_r>0$, $N_r\in\NN$ and $\rho'_r>0$ as in Proposition~\ref{p_Lem5.5}.
%Let $g\in\supp\nu^{(n)}_\infty$ and $x, x' \in \cF(r,d)$ be such that $F'_1 \not\subset F_r$.
%It follows from \eqref{eq_VA5again} and \eqref{eq_def_BBB} that
%\begin{equation}\label{eq_VA8a}
%- \log\VA_r(gx,gx') \le - \log\VA_r(x,x') + B n.
%\end{equation}
%Similarly, \eqref{eq_def_AAA} gives that
%\begin{equation}\label{eq_VA8b}
%- \log d(gF_r,g(F_{r-1}+F'_1)) \le - \log d(F_r, F_{r-1}+F'_1) + A n.
%\end{equation}
Let $C'_r=C'_r(\nu_\infty)>0$ be given by
\begin{equation}\label{eq_Crlinha}
C'_r = B+\gamma_r A.
\end{equation}
Substituting \eqref{eq_VAr_linear} and \eqref{eq_AAA} in the definition \eqref{eq_VPr_def},
we find that
%\begin{equation*}%\label{eq_VPr_def_again}
%-\log \VP_r(x,x') = -\log \VA_r(x,x') + \gamma_r \log d(F_r,F_{r-1}+F'_1).
%\end{equation*}
\begin{equation}\label{eq_VP7}
\begin{aligned}
-\log\VP_r(gx,&gx')\\
& = - \log\VA_r(gx,gx') - \gamma_r \log d(gF_r,g(F_{r-1}+F'_1)) \\
& \le - \log\VA_r(x,x') + Bn - \gamma_r \log d(F_r,F_{r-1}+F'_1) + \gamma_r A n \\
& = - \log \VP_r(x,x') + C'_r n.
\end{aligned}
\end{equation}
for any $g\in\supp\nu^{(n)}_\infty$ and $x, x' \in \cF(r,d)$ with $F'_1 \not\subset F_r$.
Integrating \eqref{eq_VPstat} over $\cE'_r$ and \eqref{eq_VP7}
over the complement, and using the fact that $\nu_\infty^{(n)}((\cE'_r)^c)<\delta$, we get that
$$
\int_G -\log\VP_r(gx,gx') \, d\nu_\infty^{(n)}(g) \le -\log\VP_r(x,x') - \kappa'_r n + C'_r\delta n.
$$
for every $n\ge N_r$ and $x, x' \in \fE_r(\rho'_r)$ with $F'_1\not\subset F_r$.
This completes the proof of Proposition~\ref{p_Prop5.1}.

However, as mentioned before, for $r>1$ it is possible to have $x$ and $x'$ with $F_r$ and $F'_r$
arbitrarily close to the equator $E$ without $\VP_r(x,x')$ getting close to zero.
Here is a simple example (see also part (2) of Lemma~\ref{l_Lem5.8}):

\begin{example}
Denote by $(x_1,x_2,x_3,x_4)$ the elements of $\RR^4$ and let $E = \{x_4=0\}$,
$F_1=\{x_2=x_3=x_4=0\}$, $F_2=\{x_3=x_4=0\}$,
$F'_1=\{x_2=x_3=0, x_4=\vep x_1\}$ and $F'_2=\{x_3=0,x_4=\vep x_1\}$.
It is clear that $F_2\subset E$ and $F'_2 \to E$ when $\vep \to 0$. However,
\begin{itemize}
\item $\VA_2(x,x') = d(F_2+F'_1,E) = 1$, since $(0,0,0,1) \in (F_2 + F'_1) \cap E^\perp$;
\item $d(F_2,F_1+F'_1) = 1$, since $(0,1,0,0) \in F_2 \cap (F_1 + F'_1)^\perp$.
\end{itemize}
It follows that $\VP_2(x,x')$ remains bounded from zero when $\vep \to 0$.
\end{example}

Thus, $-\log\VP_r(x,x')$ cannot be used as a Margulis function for the equator $\fE_r$ in the flag space.
To rectify this problem, we define the function $\psi_r$ inductively in $r$ as follows.
Let $x_-$, $x'_- \in \cF_{r-1}(\RR^d)$ be the \emph{truncated flags} obtained by dropping the
$r$-dimensional subspaces from $x$ and $x'$, respectively. The assumption $F'_1\not\subset F_r$ implies that
$F'_1\not\subset F_{r-1}$, and so we may assume that $\psi_{r-1}(x_-,x'_-)$ has already been defined.
Then define
\begin{equation}\label{eq_psir_def}
\psi_r(x,x') = \psi_{r-1}(x_-,x'_-)^{\beta_{r-1}}\VP_r(x,x'),
\end{equation}
where the exponent $\beta_{r-1}=\beta_{r-1}(\nu_\infty)$ is a small constant to be chosen as follows.

It follows from this definition and \eqref{eq_VP7} that
$$
\begin{aligned}
-\log\psi_j(gx,gx')&+\log\psi_j(x,x')\\
& \le
\big[-\log\psi_{j-1}(gx,gx')+\log\psi_{j-1}(x,x')\big]\beta_{j-1}+C_j'n
\end{aligned}
$$
for every $j=2, \dots, r$. Thus, recalling also \eqref{eq_r=1_VP1_linear},
\begin{equation}\label{eq_psir_linear}
\begin{aligned}
-\log\psi_r&(gx,gx') \\
& \le -\log\psi_r(x,x') + \big[C'_1 \beta_1 \cdots \beta_{r-1} + \cdots + C'_{r-1}\beta_{r-1}+C'_r\big]n \\
& \le -\log\psi_r(x,x') + C''_r n,
\end{aligned}
\end{equation}
where $C''_r=C''_r(\nu_\infty)$ is defined by
\begin{equation}\label{eq_Crdoislinha}
C''_r = C'_1 + \cdots + C'_{r-1}+C'_r.
\end{equation}
Take  the exponent $\beta_{r-1}$ in \eqref{eq_psir_def} small enough that
\begin{equation}\label{eq_bound_betar}
\beta_{r-1} C''_{r-1}\le \frac 12 \kappa'_r
\end{equation}
where $\kappa'_r=\kappa'_r(\nu_\infty)$ is as in Proposition~\ref{p_Prop5.1}.

The function $-\log\psi_r(x,x')$ thus defined does go to infinity when the flags $x$ and $x'$
approach the equator $\fE_r$:

\begin{lemma}\label{l_Lem5.8}
Given any $R>0$, there exists $\tilde\rho_r=\tilde\rho_r(\nu_\infty,R)>0$ such that for any $x, x' \in \fE_r(\tilde\rho_r)$
with $F'_1 \not\subset F_r$,
\begin{enumerate}
\item  $-\log\psi_r(x,x')>R$ and
\item $-\log\VP_r(x,x')>R$ unless $\VA_{r-1}(x_-,x'_-)>\tilde\rho_r$.
\end{enumerate}
\end{lemma}

\begin{proof}
It is clear from the definitions \eqref{eq_VAR_defbis} and \eqref{eq_VPr_def} that $\VP_j(x,x') \le 1$
for every $1 \le j \le r$ and any $x, x'\in\cF(j,d)$ with $F'_1\not\subset F_j$.
So, the definition \eqref{eq_psir_def} implies that
\begin{equation}\label{eq_betas}
-\log \psi_r(x,x')  \ge - \log\psi_1(F_1,F'_1)\beta_1 \cdots \beta_{r-1}.
\end{equation}
By definition, the $\beta_j$ depend only on $\nu_\infty$. Then Lemma~\ref{l_r=1_Lem5.8} gives
that for any $R>0$ there exists $\rho>0$ depending only on $\nu_\infty$ and $R$ such that the
right-hand side of \eqref{eq_betas} is greater than $R$ for any $F_1 \neq F'_1$ in $E_1(\rho)$.
Since $d(F_1,E) \le d(F_r,E)$ and $d(F'_1,E) \le d(F'_r,E)$,
because $F_1\subset F_r$ and $F'_1\subset F'_r$, we get that $-\log\psi_r(x,x') > R$ for any
$x, x'\in \fE_r(\rho)$ with $F'_1 \not\subset F_r$. This proves part (1).

To prove part (2), consider $x, x' \in \cF(r,d)$ with $d(F_r, E) \le \rho$, $d(F'_r, E) \le \rho$, and
\begin{equation}\label{eq_DFF}
\VA_{r-1}(x_-,x'_-) = d(F_{r-1}+F'_1,E) \le \rho.
\end{equation}
If $d(F_r,F'_{r-1}+F'_1) \le \sqrt\rho$, then
\begin{equation}\label{eq_VPtoinfty1}
\VP_r(x,x')
\le d(F_r,F_{r-1}+F'_1)^{\gamma_r}
\le \rho^{\gamma_r/2}.
\end{equation}
Now suppose that $d(F_r,F'_{r-1}+F'_1) \ge \sqrt\rho$.
Let $v$ and $v'$ be unit vectors orthogonal to $F_{r-1}$ as in \eqref{eq_vvprime}.
Take $u\in F'_1+F_r$ realizing the supremum in \eqref{eq_VAR_defbis}, and write $u=u_0+av+bv'$
with $u_0\in F_{r-1}$ and $a, b\in \RR$. Since $|\sin\angle(v',F_r)| = d(F_r,F'_1+F'_{r-1})$
is taken to be greater than $\sqrt\rho$,
$$
\|u_0+av\| \le \frac{2}{\sqrt\rho} \|u\| \quand \|b v'\| \le \frac{2}{\sqrt\rho} \|u\|.
$$
Then,
$$
\begin{aligned}
\VA_r(x,x') = \frac{\|u^\perp\|}{\|u\|}
& \le \frac{\|(u_0+av)^\perp\|}{\|u\|} +  \frac{\|(bv')^\perp\|}{\|u\|}\\
& \le \frac{2}{\sqrt\rho}\frac{\|(u_0+av)^\perp\|}{\|u_0+av\|} + \frac{2}{\sqrt\rho}\frac{\|(v')^\perp\|}{\|v'\|}.
\end{aligned}
$$
Recalling also \eqref{eq_DFF}, we get that
\begin{equation}\label{eq_VPtoinfty2}
\VP_r(x,x')
\le \VA_r(x,x')
\le \frac{2}{\sqrt\rho} d(F_r,E) + \frac{2}{\sqrt\rho}d(F'_1+F'_{r-1},E))
\le 4\sqrt\rho.
\end{equation}
Since $\gamma_r$ and $\rho$ may be taken to be smaller than $1$, both inequalities \eqref{eq_VPtoinfty1} and
\eqref{eq_VPtoinfty2} imply that
$$
-\log\VP_r(x,x') \ge - \frac{\gamma_r}{2}\log\rho - \log 4
$$
for any $x, x'$ in $\fE_r(\rho)$ with $F'_1 \not\subset F_r$.
The right-hand side is larger than $R$, as long as $\rho$ is chosen to be sufficiently small.
\end{proof}

For completeness, we include the following version of Propositions~\ref{p_Prop5.1}
and~\ref{p_Lem5.5} for the function $\psi_r$, although it will not be needed in what
follows (the related Proposition~\ref{p_Prop6.1Lem6.2} will be used instead):

\begin{proposition}\label{p_Prop5.1Lem5.5}
There exists $\kappa''_r=\kappa''_r(\nu_\infty)>0$ such that for each $\delta>0$ and $n \ge N_r$
there exists $\rho''_r=\rho''_r(\nu_\infty,\delta,n)>0$ such that for any $x, x' \in \fE_r(\rho''_r)$ with
$F'_1\not\subset F_r$ there exists $\cE''_r=\cE''_r(\nu_\infty,\delta,n,x,x')\subset\supp\nu_\infty^{(n)}$
with $\nu_\infty^{(n)}((\cE''_r)^c)<\delta$ and
\begin{equation}\label{eq_psir_A0}
- \log \psi_r(gx,gx') \le -\log\psi_r(x,x')-\kappa''_r n \text{ for every $g\in\cE''_r$}
\end{equation}
and
\begin{equation}\label{eq_psir_B0}
\int_G - \log\psi_r(gx,gx') \, d\nu_\infty^{(n)}(g) \le -\log\psi_r(x,x') - (\kappa''_r-C''_r\delta)n.
\end{equation}
\end{proposition}

\begin{proof}
The case $r=1$ consists of the inequalities \eqref{eq_r=1_VPstat} and \eqref{eq_r=1_VPaver}, respectively.
with $\kappa''_1=\kappa'_1$, $C''_1=C'_1$, $\rho''_1=\rho'_1$, and $\cE''_1=\cE'_1$.
Now suppose that $r>1$. Recall that $C''_r=C''_r(\nu_\infty)$ was defined in \eqref{eq_Crdoislinha}.
Define also
\begin{equation}\label{eq_psir_A4}
\begin{aligned}
\kappa''_r  = \kappa'_r(\nu_\infty)/2, \quad
\rho''_r = \rho'_r(\nu_\infty,\delta,n) \quand
\cE''_r = \cE'_r(\nu_\infty,\delta,n,x,x').
\end{aligned}
\end{equation}
Consider $n \ge N_r$, and $x, x' \in \fE_r(\rho''_r)$
with $F'_1 \not\subset F_r$.
By \eqref{eq_psir_linear},
\begin{equation}\label{eq_psir_A2}
- \log \psi_{r-1}(gx_-,gx'_-) \le -\log\psi_{r-1}(x_-,x'_-)+C''_{r-1}n
\end{equation}
for every $g\in\supp\nu_\infty^{(n)}$. By Proposition~\ref{p_Lem5.5},
\begin{equation}\label{eq_psir_A3}
- \log \VP_r(gx,gx') \le -\log\VP_r(x,x')-\kappa'_r n
\end{equation}
for every $g\in\cE''_r$.
Substituting \eqref{eq_psir_A2} and \eqref{eq_psir_A3} in the definition \eqref{eq_psir_def},
and recalling the choice of $\beta_{r-1}$ in \eqref{eq_bound_betar},
$$
\begin{aligned}
-\log\psi_r(gx,gx')
& - \log\psi_{r-1}(gx_-,gx'_-)\beta_{r-1} - \log\VP_r(gx,gx') \\
& \le -\log\psi_{r-1}(x_-,x'_-)\beta_{r-1}+C''_{r-1}n\beta_{r-1} -\log\VP_r(x,x')-\kappa'_r n \\
& \le -\log\psi_r(x,x') - (\kappa'_r/2)n
\end{aligned}
$$
for every $g\in\cE''_r$. This gives claim \eqref{eq_psir_A0}.
Moreover, integrating this inequality on $\cE''_r$ and \eqref{eq_psir_linear} on the complement,
we obtain
$$
\int_G - \log\psi_r(gx,gx') \, d\nu_\infty^{(n)}(g)
\le -\log\psi_r(x,x') - (\kappa'_r/2 - C''_r \delta) n,
$$
which gives claim \eqref{eq_psir_B0}.
\end{proof}

\section{Stabilization and cut-off}\label{s_stabilization_and_cut-off}

Next we present the analogues for $r>1$ of the two constructions in Section~\ref{s_step1_stabilization_and_cut-off}.
Both are fairly straightforward extensions of the case $r=1$, but it turns out that for $r>1$ they do not suffice to
deal with the questions discussed at the beginning of Section~\ref{s_step1_stabilization_and_cut-off}.
This difficulty will be handled later, in Section~\ref{s_spreading_out}.
Another issue is that the function $\psi_r$ we construct in the following is not symmetric when $r>1$.
Thus we will use instead the function $\hpsi_r$ defined by
$$
\hpsi_r(x,x') = \max\{\psi_r(x,x'),\psi_r(x',x)\}.
$$

Let $\kappa''_r=\kappa''_r(\nu_\infty)>0$, $C''_r=C''_r(\nu_\infty)>0$ and
$N_r=N_r(\nu_\infty,\delta)\in\NN$ be as in Proposition~\ref{p_Prop5.1Lem5.5}.
Keep in mind that $\vep_r<\rho''_r=\rho'_r \le \hat\rho_r$, $C'_r=B+\gamma_r A$,
and $C''_r=C'_1 + \cdots + C'_r$, according to \eqref{eq_constants},
\eqref{eq_Crlinha}, \eqref{eq_Crdoislinha},  \eqref{eq_psir_A4}, and \eqref{eq_veprsmall}.

\subsection{Stabilization}\label{ss_stabilization}

Let $\gamma_r=\gamma_r(\nu_\infty)>0$ and $B=B(\nu_\infty)>0$ be as in \eqref{eq_constants}
and \eqref{eq_BBB}, respectively. For each $\omega_r>0$ and $n\in\NN$, define the
\emph{stabilized vertical angle}
\begin{equation}\label{eq_SVAr_def}
\SVA_r(x,x';\omega_r) = \max\big\{\VA_r(x,x'),\omega_r e^{-Bn}\big\}
\end{equation}
and the \emph{stabilized vertical projection} by
\begin{equation}\label{eq_SVPr_def}
\SVP_r(x,x';\omega_r) = \SVA_r(x,x';\omega_r) d(F_r,F_{r-1}+F'_1)^{\gamma_r}
\end{equation}
for every $x, x'$ in $\fE_r(\vep_r)$ with $F'_1\not\subset F_r$.

Given $\uomega=(\omega_1, \dots, \omega_r)\in\RR_+^d$,
let $\uomega_-= (\omega_1, \dots, \omega_{r-1})$ and define the
\emph{stabilized function $\psi_r$} by
\begin{equation}\label{eq_Spsir_def}
\psi_r(x,x';\uomega) = \psi_{r-1}(x_-,x'_-;\uomega_-)^{\beta_{r-1}} \SVP_r(x,x';\omega_r).
\end{equation}
The following extension of Lemma~\ref{l_r=1_stabilizedbound} asserts that the
estimate in \eqref{eq_psir_linear} remains valid for these stabilized functions:

\begin{lemma}\label{l_stabilizedbound}
For every $g\in\supp\nu_\infty^{(n)}$, any $x, x' \in \fE_r(\rho''_r)$ with $F'_1 \not\subset F_r$,
and any $\uomega=(\omega_1, \dots,\omega_r)\in\RR_+^d$,
\begin{equation}\label{eq_Spsir_linear}
-\log \psi_r(gx,gx';\uomega) \le -\log\psi_r(x,x';\uomega)+C''_r n.
\end{equation}
\end{lemma}

\begin{proof}
The case $r=1$ was done in Lemma~\ref{l_r=1_stabilizedbound}, so let us consider $r>1$.
By induction,
\begin{equation}\label{eq_stabilizebound1}
-\log \psi_r(gx_-,gx'_-;\uomega_-) \le -\log\psi_r(x_-,x'_-;\uomega_-) + C''_{r-1}n.
\end{equation}
We claim that
\begin{equation}\label{eq_SVAr_linear}
-\log\SVA_r(gx,gx';\omega_r) \le - \log\SVA_r(x,x';\omega_r) + B n.
\end{equation}
and
\begin{equation}\label{eq_SVPr_linear}
-\log \SVP_r(gx,gx';\omega_r) \le - \log\SVP_r(x,x';\omega_r) + C'_r n.
\end{equation}
The inequality \eqref{eq_Spsir_linear} follows directly from combining \eqref{eq_stabilizebound1}
and \eqref{eq_SVPr_linear}, and recalling the definition  of $C''_r$ in \eqref{eq_Crdoislinha}:
\begin{equation}\label{eq_stabilizedbound4}
\begin{aligned}
-\log \psi_r(gx,gx';\uomega)
& = -\log \psi_r(gx_-,gx'_-;\uomega_-)\beta_{r-1} \\
& \hspace{3cm} -\log \SVP_r(gx,gx';\omega_r) \\
& \le -\log\psi_r(x_-,x'_-;\uomega_-)\beta_{r-1} + C''_{r-1}n \beta_{r-1}\\
& \hspace{3cm} - \log\SVP_r(x,x';\omega_r) + C'_{r-1} n\\
& \le -\log\psi_r(x,x';\uomega) + C''_r n.
\end{aligned}
\end{equation}
Recall also that $\beta_{r-1} \le 1$.

We split the proof of \eqref{eq_SVAr_linear} into two cases.
Suppose first that $\VA_r(x,x') < \omega_r$. Then, by the definition \eqref{eq_SVAr_def},
\begin{align*}
- \log\SVA_r(x,x';\omega_r)  & \ge - \log\omega_r \quand\\
-\log\SVA_r(gx,gx';\omega_r) & \le -\log\omega_r + B n \le - \log\SVA_r(x,x';\omega_r)+Bn
\end{align*}
as claimed. Now suppose that $\VA_r(x,x') \ge \omega_r$.
Then, again by the definition \eqref{eq_SVAr_def},
\begin{align*}
-\log\SVA_r(x,x';\omega_r)
& = - \log\VA_r(x,x') \quand\\
-\log\SVA_r(gx,gx';\omega_r)
& \le  - \log\VA_r(gx,gx').
\end{align*}
Together with \eqref{eq_VAr_linear}, this yields
\begin{equation*}%\label{eq_VA1_new1}
-\log\SVA_r(gx,gx';\omega_r)
\le - \log\SVA_r(x,x';\omega_r) + Bn,
\end{equation*}
which completes the proof of \eqref{eq_SVAr_linear}.

%Next, observe that \eqref{eq_def_AAA} implies
%\begin{equation}\label{eq_VA8stab_pre}
%- \log d(gF_r,g(F_{r-1}+F'_1)) \le - \log d(F_r,F_{r-1}+F'_1) + An.
%\end{equation}
Finally, substituting \eqref{eq_SVAr_linear} and \eqref{eq_AAA} in the definition \eqref{eq_VPr_def},
and recalling the definition of $C'_r$ in \eqref{eq_Crlinha},
\begin{equation}\label{eq_VP8stab}
\begin{aligned}
-\log\SVP_r(gx,gx';\omega_r)
& = - \log \SVA_r(gx,gx';\omega_r) \\
& \hspace{2cm} - \gamma_r \log d(gF_r,g(F_{r-1}+F'_1))\\
& \le- \log \SVA_r(x,x';\omega_r) +Bn \\
& \hspace{2cm} - \gamma_r \log d(gF_r,g(F_{r-1}+F'_1)) - \gamma_r A n \\
& = - \log \SVP_r(x,x';\omega_r) + (B+\gamma_rA) n
\end{aligned}
\end{equation}
This proves \eqref{eq_SVPr_linear}, and thus completes the proof of the lemma.
\end{proof}

We say that $(x,x')$ is in the \emph{stabilization region} if $\VA_r(x,x') < \omega_r$.
As we have seen in \eqref{eq_VAr_linear},
\begin{equation*}%\label{eq_VAstable}
\VA_r(gx,gx') \ge \VA_r(x,x') e^{-B n}
\end{equation*}
for any $x, x' \in \cF(r,d)$ with $F'_1 \not\subset F_r$ and $g\in\supp\nu^{(n)}_\infty$.
So, if $(x,x')$ is outside the stabilization region then
\begin{equation}\label{eq_SVP1}
\begin{aligned}
\SVA_r(x,x';\omega_r) = \VA_r(x,x')  \quand
\SVA_r(gx,gx';\omega_r) = \VA_r(gx,gx') \\
\SVP_r(x,x';\omega_r) = \VP_r(x,x')  \quand
\SVP_r(gx,gx';\omega_r) = \VP_r(gx,gx')
\end{aligned}
\end{equation}
for any $g \in \supp\nu_\infty^{(n)}$ and $n\in\NN$.

The following analogue of Proposition~\ref{p_Prop5.1Lem5.5} for stabilized vertical angles
and the stabilized vertical projections extends Proposition~\ref{p_r=1_Prop6.1Lem6.2} to $r>1$:

%%% The original statement has an additional constant $\theta''_r$
%%% and parameters \omega_j^+
\begin{proposition}\label{p_Prop6.1Lem6.2}
For every $\delta>0$, $n \ge N_r$, $x, x'$ in $\fE_r(\rho''_r)$ with $F'_1\not\subset F_r$,
and $\uomega=(\omega_1, \dots, \omega_r)\in\RR_+^d$ with $\VA_r(x,x')\ge\omega_r$,
\begin{equation}\label{eq_SVP3}
-\log\psi_r(gx,gx';\uomega) \le - \log\psi_r(x,x';\uomega) - \kappa''_r n
\text{ for every $g\in\cE''_r$,}
\end{equation}
and
\begin{equation}\label{eq_SVP2}
\int_G - \log\psi_r(gx,gx';\uomega) \, d\nu_\infty^{(n)}(g)
\le -\log\psi_r(x,x';\uomega) - (\kappa''_r-C''_r\delta)n
\end{equation}
for any
\end{proposition}

\begin{proof}
%The proof is by induction on $r$.
The case $r=1$ is given by Proposition~\ref{p_r=1_Prop6.1Lem6.2} with
$\kappa''_1=\kappa'_1$, $C''_1=C'_1$, $\rho''_1=\rho'_1$, and $\cE''_1=\cE'_1$.
Now let us consider $r>1$. Let $\delta>0$, $n \ge N_r$, $x, x' \in \fE_r(\rho''_r)$,
and $\uomega\in\RR_+^d$ be as in the statement. By  \eqref{eq_Spsir_def} and \eqref{eq_SVP1},
\begin{equation}\label{eq_psir_stable}
\begin{aligned}
\psi_r(x,x';\uomega) & = \psi_{r-1}(x_-,x_-';\uomega_-)^{\beta_{r-1}}\VP_r(x,x') \quand\\
\psi_r(gx,gx';\uomega) & = \psi_{r-1}(gx_-,gx_-';\uomega_-)^{\beta_{r-1}}\VP_r(gx,gx').
\end{aligned}
\end{equation}
By \eqref{eq_Spsir_linear},
\begin{equation}\label{eq_psir_A2bis}
- \log \psi_{r-1}(gx_-,gx'_-;\uomega_-) \le -\log\psi_{r-1}(x_-,x'_-;\uomega_-)+C''_{r-1}n
\end{equation}
for every $g\in\supp\nu_\infty^{(n)}$. By Proposition~\ref{p_Lem5.5},
\begin{equation}\label{eq_psir_A3bis}
- \log \VP_r(gx,gx') \le -\log\VP_r(x,x')-\kappa'_r n
\end{equation}
for every $g\in\cE''_r$.
Substituting \eqref{eq_psir_A2bis} and \eqref{eq_psir_A3bis} in the second part of
\eqref{eq_psir_stable}, and recalling the choice of $\beta_{r-1}$ in \eqref{eq_bound_betar},
$$
\begin{aligned}
-\log\psi_r(gx,gx';\uomega)
& = - \log\psi_{r-1}(gx_-,gx'_-;\uomega_-)\beta_{r-1} - \log\VP_r(gx,gx') \\
& \le -\log\psi_{r-1}(x_-,x'_-;\uomega_-)\beta_{r-1}+ C''_{r-1}n\beta_{r-1}\\
& \hspace{4.8cm} -\log\VP_r(x,x')-\kappa'_r n \\
& \le -\log\psi_r(x,x';\uomega) - (\kappa'_r/2)n
= -\log\psi_r(x,x';\uomega) - \kappa''_r n
\end{aligned}
$$
for every $g\in\cE''_r$. This gives claim \eqref{eq_SVP3}.
Moreover, integrating \eqref{eq_SVP3} on $\cE''_r$ and \eqref{eq_Spsir_linear} on the complement,
we obtain that
$$
\int_G - \log\psi_r(gx,gx';\uomega) \, d\nu_\infty^{(n)}(g)
\le -\log\psi_r(x,x';\uomega) - (\kappa''_r - C''_r \delta) n,
$$
as claimed in \eqref{eq_SVP2}.
\end{proof}

\subsection{Cutoff}\label{ss_cut-off}

Recall also that the constant $\vep_2>0$ was chosen at the end of the initial step of the induction,
in Section~\ref{ss_step1_closure}. Recall also that it may be taken to be as small as we want.

For any $r>1$, assume that $\vep_r=\vep_r(\nu_\infty,\delta,n)>0$ has been chosen, satisfying
\begin{equation}\label{eq_um01}
\eta_{\infty,r}(E_r(\vep_r)\setminus E) < \frac{1}{10}\eta_{\infty,r}(E).
\end{equation}
and
\begin{equation}\label{eq_veprsmall}
\vep_r < \min\{\rho_0, \rho''_r\},
\end{equation}
where $\rho_0=\rho_0(\nu_\infty,n)>0$ is as in Corollaries~\ref{c_sec5_distance_contraction1} and~\ref{c_sec5_distance_contraction2}
and $\rho''_r=\rho''_r(\nu_\infty,\delta,n)>0$ is as in Proposition~\ref{p_Prop6.1Lem6.2}.

Using Remark~\ref{r_sec7_repeatedly} twice, we find $\vep'_r=\vep'_r(\nu_\infty,\delta,n)>0$ and
$\tilde\vep_r=\tilde\vep_r(\nu_\infty,\delta,n)>0$ with $0 < \vep'_r < \tilde\vep_r < \vep_r$,
and a compact neighborhood $\cW_r=\cW_r(\nu_\infty,\delta,n)$ of $\supp\nu_\infty^{(n)}$
such that
\begin{align}
\label{eq_I}
gF_r \in E_r(\vep_r/2) &\text{ for every } F_r\in E_r(2\tilde\vep_r) \quand g\in \cW_r \text{ and}\\
\label{eq_II}
g^{-1}F_r   \in E_r(\tilde\vep_r/2) & \text{ for every } F_r\in E_r(2\vep'_r)  \quand g \in \cW_r.
\end{align}

Let $\kappa_0=\kappa_0(\nu_\infty)>0$ be as in Proposition~\ref{p_sec5_derivative_estimate}, and define
$\vep''_r=\vep''_r(\nu_\infty,\delta,n)$ by
\begin{equation}\label{eq_epsilon}
\vep''_r=3\vep'_re^{-\kappa_0 n/2}.
\end{equation}
Taking $\rho=\vep''_r$ in Corollary~\ref{c_sec5_distance_contraction2},
and keeping in mind that $\vep_r < \rho_0$,
we get that there are $\tilde k_r=\tilde k_r(\nu_\infty,\delta,n)\in\NN$ and
$\cD_k(F_r) =\cD_k(\nu_\infty,\delta,n,F_r)\subset\supp\nu_k^{(n)}$
such that $\nu_k^{(n)}(\cD_k(F_r)^c) < \delta$ and
$$
d(gF_r,E) > e^{\kappa_0 n/2} d(F_r,E) > e^{\kappa_0 n/2} \vep''_r > 2\vep'_r
$$
for any $g\in\cD_k(F_r)$, $F_r \in E(\vep_r,\vep''_r)$ and $k\ge \tilde k_r$.
In other words, for $k \ge \tilde k_r$,
\begin{equation}\label{eq_III}
F_r \in E(\vep_r, \vep''_r) \Rightarrow gF_r \notin E(2\vep'_r) \text{ for every } g\in\cD_k(F_r).
\end{equation}

Increasing $\tilde k_r$ if necessary, we may suppose that $\supp\nu_k^{(n)} \subset \cW_r$ for every $k\ge \tilde k_r$. Then \eqref{eq_I} and \eqref{eq_II} imply
\begin{align}
\label{eq_IV}
E_r(2\tilde\vep_r) \subset \Chi_{\nu_k^{(n)}} E_r(\vep_r)
\quand
E_r(2\vep'_r) \subset \Chi^\#_{\nu_k^{(n)}} E_r(\vep_r)\\
\label{eq_V}
F_r \notin E_r(\tilde\vep_r/2) \Rightarrow g F_r \notin E_r(2\vep'_r)
\text{ for every } g \in \supp\nu_k^{(n)}.
\end{align}

We say that $x, x' \in \cF(r,d)$ are \emph{in general position} if $F'_1 \not\subset F_r$ and
$F_1 \not\subset F'_r$.  For any $x, x' \in \fE_r(\vep_r)$ in general position and
$\uomega\in\RR_+^d$, define
\begin{equation}\label{eq_Psir_def}
\Psi_r(x,x';\uomega) = \left\{\begin{array}{ll}\log\left(\Omega_r+\hat\psi_r(x,x';\uomega)^{-1}\right)
                                               & \text{if $x\in \fE_r(2\vep''_r)$ or $x'\in \fE_r(2\vep''_r)$} \\
                                               \log\Omega_r & \text{otherwise,}\end{array}\right.
\end{equation}
where $\Omega_r=\Omega_r(\nu_\infty,\delta,n)>1$ is a large constant to be chosen
in Proposition~\ref{p_Prop8.1}, and
\begin{equation}\label{eq_hatpsi_def}
\hpsi(x,x';\uomega) = \max\{\psi(x,x';\uomega),\psi(x',x;\uomega)\}.
\end{equation}
It is clear from the definition that $\Psi_r(\cdot,\cdot;\uomega)$ is a symmetric function:
$$
\Psi_r(x,x';\uomega) = \Psi_r(x',x;\uomega) \text{ for all $x, x'\in\fE_r(\vep_r)$.}
$$
The set $\fE_r(2\vep''_r)^c \times \fE_r(2\vep''_r)^c$ is the \emph{cut-off region} at the stage $r>1$.
Compare Figure~\ref{f_cut-off}.

%%% The original statement has an additional constant $\omega_1^+$
\begin{proposition}\label{p_Lem7.3}
There exists $\kappa'''_r=\kappa'''_r(\nu_\infty)>0$ such that for each $\delta>0$ and
$n\ge N_r$  there exists $\vep'''_r=\vep'''(\nu_\infty,\delta,n)>0$ such that given any
$\uomega=(\omega_1, \dots, \omega_r) \in \RR_+^d$:
\begin{itemize}
\item[(i)] For any $x, x' \in \fE_r(\vep_r)$ in general position with
$\Psi_r(x,x';\uomega)>\log\Omega_r$,
$$
\int_G \Psi_r(gx,gx';\uomega) \, d\nu_\infty^{(n)}(g) \le \Psi_r(x,x';\uomega) + C''_r n.
$$
\item[(ii)] For any $x, x' \in \fE_r(\vep_r)$ in general position with
$\Psi_r(x,x';\uomega)>\log\Omega_r$, $\VA_r(x,x')\ge\omega_r$ and $\VA_r(x',x)\ge\omega_r$,
$$
\int_G \Psi_r(gx,gx';\uomega) \, d\nu_\infty^{(n)}(g) \le \Psi_r(x,x';\uomega) + C''_r\delta n.
$$
\item[(iii)] For any $x, x' \in \fE_r(\vep'''_r)$ in general position satisfying $\VA_r(x,x')\ge\omega_r$
and $\VA_r(x',x)\ge\omega_r$,
$$
\int_G \Psi_r(gx,gx';\uomega) \, d\nu_\infty^{(n)}(g) \le \Psi_r(x,x';\uomega) - (\kappa'''_r - C''_r\delta) n.
$$
\end{itemize}
\end{proposition}

\begin{proof} Define $\kappa'''_r=\kappa''_r/2$. Let $n \ge N_r$.
Part (i) of the proposition is a consequence of the following lemma:

\begin{lemma}\label{l_eq7.5}
If $x, x' \in\cF(r,d)$ are in general position and $\Psi_r(x,x';\uomega)>\log\Omega_r$ then
\begin{equation}\label{eq_Psi3}
\Psi_r(gx,gx';\uomega) \le \Psi_r(x,x';\uomega) + C''_r n
\text{ for any $g\in\supp\nu_\infty^{(n)}$.}
\end{equation}
\end{lemma}

\begin{proof}
It follows immediately from \eqref{eq_Spsir_linear} that
\begin{equation}\label{eq_Shatpsir_linear}
-\log \hpsi_r(gx,gx';\uomega) \le -\log\hpsi_r(x,x';\uomega)+C''_r n
\end{equation}
for any $g\in\supp\nu_\infty^{(n)}$. Then, using \eqref{eq_r=1_elementary1},
\begin{equation*}%\label{eq_Spsir_linearbis}
\begin{aligned}
\Psi_r(gx,gx';\uomega)
& \le \log\left(\Omega_r+\hpsi_r(gx,gx';\uomega)^{-1}\right)
\le \log\left(\Omega_r+e^{C''_r n}\hpsi_r(x,x';\uomega)^{-1}\right)\\
& \le \log\left(\Omega_r + \hpsi_r(x,x';\uomega)^{-1}\right) + C''_r n
 = \Psi_r(x,x';\uomega) + C''_r n.
\end{aligned}
\end{equation*}
This proves the claim.
\end{proof}

To prove part (ii) we use Proposition~\ref{p_Prop6.1Lem6.2}: given any $\uomega\in\RR_+^d$
and $x, x' \in \fE_r(\vep_r)$ with $F'_1 \not\subset F_r$,
\begin{equation}\label{eq_Psi5}
-\log\psi_r(gx,gx';\uomega) \le -\log\psi_r(x,x';\uomega)-\kappa''_r n \le -\log\psi_r(x,x';\uomega)
\end{equation}
for every $g\in\cE''_r$. This remains true if we exchange the roles of $x$ and $x'$, of course.
Thus, $-\log\hpsi_r(gx,gx';\uomega) \le -\log\hpsi_r(x,x';\uomega)$, and so
\begin{equation}\label{eq_Psi6}
\begin{aligned}
\Psi_r(gx,gx';\uomega)
& \le \log\left(\Omega_r+\hpsi_r(gx,gx';\uomega)^{-1}\right)\\
& \le \log\left(\Omega_r+ \hpsi_r(x,x';\uomega)^{-1}\right)
 = \Psi_r(x,x';\uomega)
\end{aligned}
\end{equation}
for every $g \in \cE''_r$. Integrating \eqref{eq_Psi6} over $\cE''_r$ and \eqref{eq_Psi3}
over the complement, we obtain the estimate in part (ii).

Now take $c=e^{-2\kappa'''_rn}$ in the relation \eqref{eq_r=1_elementary3}.
By Lemma~\ref{l_Lem5.8}, there exists $\vep'''_r>0$ depending only on $\nu_\infty$, $\delta$ and
$n$ (through $c$ and $\Omega_r$) such that
\begin{equation}\label{eq_3lines}
-\log \hpsi_r(x,x';\uomega) = -\log \hpsi_r(x,x') \ge \log \left(\Omega_r/\sqrt{c}\right)
\end{equation}
for any $x, x' \in E_r(\vep'''_r)$ in general position.
Then, using \eqref{eq_3lines}, \eqref{eq_Psi5} and \eqref{eq_r=1_elementary3},
\begin{equation}\label{eq_Psi7}
\begin{aligned}
\Psi_r(gx,gx';\uomega)
& \le \log\left(\Omega_r+\hpsi_r(gx,gx';\uomega)^{-1}\right)\\
& \le \log\left(\Omega_r+e^{-2\kappa'''_rn} \hpsi_r(x,x';\uomega)^{-1}\right)\\
& \le \log\left(\Omega_r+\hpsi_r(x,x';\uomega)^{-1}\right) -\kappa'''_r n
= \Psi_r(x,x';\uomega)-\kappa'''_r n
\end{aligned}
\end{equation}
for every $g \in \cE''_r$. Integrating \eqref{eq_Psi7} over $\cE''_r$ and \eqref{eq_Psi3}
over the complement, we obtain the estimate in part (iii) of the proposition.
\end{proof}

\section{Turning the perturbation on}\label{s_turning_the_perturbation_on}

We show that the conclusions of Proposition~\ref{p_Lem7.3} hold for
$\nu_k^{(n)}$ instead of $\nu_\infty^{(n)}$, as long as $k\in\NN$ is sufficiently large.
The arguments are close to those in Section~\ref{s_step1_turning_the_perturbation_on},
but we have to deal with the fact that $\psi_r$ and $\Psi_r$ are not entirely
straightforward generalizations of $\psi_1$ and $\Psi_1$.

\begin{proposition}\label{p_Lem9.1}
Given $\delta>0$, $n\ge N_r$, and $\uomega=(\omega_1, \dots, \omega_r) \in\RR_+^d$ there is
$k_r=k_r(\nu_\infty, \delta, n,\omega_r)\in\NN$ such that the following holds for every $k\ge k_r$:
\begin{itemize}
\item[(i)] For any $x,  x' \in \fE_r(\vep_r)$ in general position with
$\Psi_r(x,x';\uomega)>\log\Omega_r$,
$$
\int_G \Psi_r(gx,gx';\uomega) \, d\nu_k^{(n)}(g) \le \Psi_r(x,x';\uomega) + C''_r n.
$$
\item[(ii)] For any $x,  x' \in \fE_r(\vep_r)$ in general position with
$\Psi_r(x,x';\uomega)>\log\Omega_r$, $\VA_r(x,x')\ge\omega_r$, and $\VA_r(x',x)\ge\omega_r$,
$$
\int_G \Psi_r(gx,gx';\uomega) \, d\nu_k^{(n)}(g) \le \Psi_r(x,x';\uomega) + C''_r\delta n.
$$
\item[(iii)] For any $x, x' \in \fE_r(\vep'''_r)$ in general position satisfying $\VA_r(x,x')\ge\omega_r$
and $\VA_r(x',x)\ge\omega_r$,
$$
\int_G \Psi_r(gx,gx';\uomega) \, d\nu_k^{(n)}(g) \le \Psi_r(x,x';n) - (\kappa''_r - C''_r\delta) n.
$$
\end{itemize}
\end{proposition}

Keep in mind that we have chosen $\vep_r<\rho''_r=\rho'_r\le\hat\rho_r$ and
$C'_r=B+\gamma_rA$. Recall also that $\tilde k_r=\tilde k_r(\nu_\infty,\delta,n)\in\NN$
was chosen so that the relations \eqref{eq_III} through \eqref{eq_V} hold for every $k\ge \tilde k_r$.

\begin{proof}
We are going to extend to large $k\in\NN$ several estimates in the proof of Proposition~\ref{p_Lem7.3}.
This will require a number of conditions on $k$, depending on $\nu_\infty$, $\delta$, $n$
and $\omega$, that we state along the way.
We begin with the following extension of Lemma~\ref{l_eq7.5}:

\begin{lemma}\label{l_eq7.5bis}
Given any $n \ge N_r$ and $\uomega=(\omega_1, \dots, \omega_r) \in\RR^d_+$,
there exists $\hat{k}_r=\hat{k}_r(\nu_\infty,n,\omega_r)$ such that if $x, x' \in \fE_r(\vep_r)$
are in general position and $\uomega\in\RR_+^d$ is such that $\Psi_r(x,x';\uomega)>\log\Omega_r$ then
\begin{equation}\label{eq_Psi3bis}
\Psi_r(gx,gx';\uomega) \le \Psi_r(x,x';\uomega) + C''_r n
\end{equation}
for any $g\in\supp\nu_k^{(n)}$ and $k\ge \hat{k}_r$.
\end{lemma}

\begin{proof}
Consider $\delta>0$, $n \ge N_r$, and $\uomega =(\omega_1, \dots, \omega_r) \in\RR_+^d$.
We begin by claiming that there exists
$\hat{k}_r=\hat{k}_r(\nu_\infty,n, \uomega)\in\NN$ such that
\begin{equation}\label{eq_SVAr_linear_finite}
-\log\SVA_r(gx,gx';\omega_r) \le - \log\SVA_r(x,x';\omega_r) + B n
\end{equation}
for any $x, x' \in \fE_r(\vep_r)$ in general position, $g \in\supp\nu_k^{(n)}$, and $k \ge \hat{k}_r$.
This can be seen as follows. If $\VA_r(x,x') < \omega_r$ then, by the definition \eqref{eq_SVAr_def},
\begin{align*}
- \log\SVA_r(x,x';\omega_r)  & \ge - \log\omega_r \text{ and}\\
-\log\SVA_r(gx,gx';\omega_r) & \le -\log\omega_r + B n \le - \log\SVA_r(x,x';\omega_r)+Bn,
\end{align*}
as claimed. Now suppose that $\VA_r(x,x') \ge \omega_r$. The relation \eqref{eq_VAr_linear}
does not apply here. Instead, from \eqref{eq_VA23again} and \eqref{eq_SVAr_def} we get that
\begin{equation}\label{eq_VA1_new3}
-\log\SVA_r(gx,gx';\omega_r)
\le - \log\SVA_r(x,x';\omega_r)  - \log \frac{\|(g u)^\perp\|}{\|gu\|}\frac{\|u\|}{\|u^\perp\|}
\end{equation}
for every $g\in\supp\nu_k^{(n)}$, with $u=u(x,x')$ realizing the supremum in \eqref{eq_VAR_defbis}.
The assumption $\VA_r(x,x') \ge \omega_r$ means that $\|u^\perp\|\ge\omega_r\|u\|$.
Since $\supp\nu_k$ converges to $\supp\nu_\infty$ in the Hausdorff topology, we may find
$r_k=r_k(\nu_\infty,n) \to 0$ such that every $g\in\supp\nu_k^{(n)}$ is in the $r_k$-neighborhood of some $f\in\supp\nu_\infty^{(n)}$. Then
$$
\|(gu)^\perp-(fu)^\perp\| \le \|gu-fu\| \le r_k \|u\| \le \frac{r_k}{\omega_r}\|u^\perp\|.
$$
Then we may take $\hat{k}_r=\hat{k}_r(\nu_\infty, n, \omega_r)\in\NN$ large enough so that this
inequality implies
$$
\begin{aligned}
- \log \frac{\|(g u)^\perp\|}{\|gu\|}\frac{\|u\|}{\|u^\perp\|}
& \le - \log \frac{\|(f u)^\perp\|}{\|fu\|}\frac{\|u\|}{\|u^\perp\|} + \log 2\\
& = - \log \frac{\|f^\perp u^\perp\|}{\|fu\|}\frac{\|u\|}{\|u^\perp\|} + \log 2\\
& \le \log\|(f^\perp)^{-1}\| + \log \|f\| + \log 2 \le B n
\end{aligned}
$$
for every $g\in\supp\nu_k^{(n)}$ and $k \ge \hat{k}_r$.
This proves \eqref{eq_SVAr_linear_finite}.

Substituting \eqref{eq_SVAr_linear_finite} and \eqref{eq_AAA} in the definition \eqref{eq_SVPr_def},
and recalling the definition of $C'_r$ in \eqref{eq_Crlinha}, we find that
\begin{equation}\label{eq_SVPr_linear_finite}
\begin{aligned}
-\log\SVP_r(gx,gx';\omega_r)
& = - \log \SVA_r(gx,gx';\omega_r) \\
& \hspace{2cm} - \gamma_r\log d(gF_r,g(F_{r-1}+F'_1)\\
& \le -\log\SVA_1(x,x';\omega_r) + Bn\\
& \hspace{2cm} - \gamma_r\log d(F_r,F_{r-1}+F'_1) + \gamma_r An \\
& \le - \log\SVP_r(x,x';\omega_r) + C'_r n
\end{aligned}
\end{equation}
for any $x, x' \in \fE_r(\vep_r)$ in general position, $g \in\supp\nu_k^{(n)}$, and $k \ge \hat{k}_r$.

Next we claim that
\begin{equation}\label{eq_Spsir_linear_finite}
- \log\psi_r(gx,gx';\uomega) \le - \log\psi_r(x,x';\uomega) + C''_r n
\end{equation}
for any $x, x' \in \fE_r(\vep_r)$ in general position, $g \in\supp\nu_k^{(n)}$, and $k \ge \hat{k}_r$.
The case $r=1$ of \eqref{eq_Spsir_linear_finite} was done in \eqref{eq_r=1_SVP1_linear_finite},
so let us consider $r>1$. By induction,
\begin{equation}\label{eq_Spsir_induction}
- \log\psi_{r-1}(gx_-,gx'_-;\uomega_-) \le - \log\psi_{r-1}(x_-,x'_-;\uomega_-) + C''_{r-1} n.
\end{equation}
Replacing this and \eqref{eq_SVPr_linear_finite} in the definition \eqref{eq_Spsir_def},
we find that
\begin{equation*}%\label{eq_Spsir_defbis}
\begin{aligned}
-\log\psi_r(x,x';\uomega)
& = -\beta_{r-1} \log \psi_{r-1}(gx_-,gx'_-;\uomega_-) \\
& \hspace{3cm}  - \log \SVP_r(gx,gx';\omega_r) \\
& \le -\beta_{r-1} \log \psi_{r-1}(x_-,x'_-;\uomega_-) + \beta_{r-1} C''_{r-1} n \\
& \hspace{3cm} - \log \SVP_r(x,x';\omega_r) + C'_r n \\
& = \psi_r(x,x';\uomega) + (\beta_{r-1}C''_{r-1}+C'_r) n.
\end{aligned}
\end{equation*}
Since $C''_r = C''_{r-1}+C'_r$ and $\beta_{r-1}<1$, this proves \eqref{eq_Spsir_linear_finite}.
The estimate remains valid if we exchange the roles of $x$ and $x'$, obviously.
Thus, we have also shown that
\begin{equation}\label{eq_hatpsir_linear}
- \log\hpsi_r(gx,gx';\omega) \le - \log\hpsi_r(x,x';\uomega) + C''_r n
\end{equation}
for any $x, x' \in \fE_r(\vep_r)$ in general position, $g \in\supp\nu_k^{(n)}$, and $k \ge \hat{k}_r$.

Combining \eqref{eq_hatpsir_linear} with \eqref{eq_r=1_elementary1} in the definition \eqref{eq_Psir_def}, we get that
\begin{equation*}%\label{eq_Spsir_linearbis_finite}
\begin{aligned}
\Psi_r(gx,gx';\uomega)
& \le \log\left(\Omega_r+\hpsi_r(gx,gx';\uomega)^{-1}\right)
\le \log\left(\Omega_r+e^{C''_rn}\hpsi_r(x,x';\uomega)^{-1}\right)\\
& \le \log\left(\Omega_r+ \hpsi_r(x,x';\uomega)^{-1}\right) + C''_r n
 = \Psi_r(x,x';\uomega) + C''_r n,
\end{aligned}
\end{equation*}
as stated.
\end{proof}

Next, we prove the following extension of Lemma~\ref{l_Lem5.4}:

\begin{lemma}\label{l_Lem5.4bis}
Given $\delta>0$, $n \ge N_r$ and $\omega_r>0$, there is
$k'_r = k'_r(\nu_\infty,\delta,n,\omega_r)\in\NN$ and for any $x, x'\in\cF(r,d)$ with
$F'_1 \not\subset F_r$ and $\VA_r(x,x') \ge \omega_r$ there is
$\cE''_{k,r}=\cE''_{k,r}(\nu_\infty,\delta,n,x,x',\omega_r)\subset\supp\nu_k^{(n)}$
with $\nu_k^{(n)}((\cE''_{k,r})^c)<\delta$ and
\begin{equation}\label{eq_VA4bis}
-\log\SVA_r(gx,gx';\omega_r) \le \max\{-\log\SVA_r(x,x';\omega_r) - \tilde\kappa_r n, \tilde\theta_r\}
\end{equation}
for every $g\in\cE''_{k,r}$ and $k\ge k'_r$.
\end{lemma}

\begin{proof}
Fix $\delta>0$ and $n \ge N_r$ and $\omega_r>0$.
Let $P_{\omega_r}$ denote the (compact) subset of all $v\in P$ such that $\|v^\perp\|/\|v\|\ge\omega_r/2$.
For $v\in P_{\omega_r}$ and $g$ in some compact neighborhood $V_{\omega_r}$ of $\supp\nu_\infty^{(n)}$,
consider
\begin{equation}\label{eq_stab1}
(v,g) \mapsto - \log\frac{\|(gv)^\perp\|}{\|gv\|}.
\end{equation}
As long as $V_{\omega_r}$ is sufficiently small, depending on $\nu_\infty$, $n$ and $\omega_r$,
the map \eqref{eq_stab1} is well defined and (uniformly) continuous. So, there exists
$\alpha=\alpha(\nu_\infty,n,\omega_r)>0$ such that
\begin{equation}\label{eq_stab2}
- \log\frac{\|(gu)^\perp\|}{\|gu\|}
\le - \log\frac{\|(f v)^\perp\|}{\|f v\|} + \log 2
\end{equation}
whenever $d(u,v) <\alpha$ and $d(g,f)<\alpha$.
Reducing $\alpha$ if necessary, depending only on $\omega_r$, we may also assume that
\begin{equation}\label{eq_stab2bis}
d(u,v) <\alpha
\quad\Rightarrow\quad
- \log\frac{\|u^\perp\|}{\|u\|}
\ge - \log\frac{\|v^\perp\|}{\|v\|} - \log 2.
\end{equation}
Fix $v_1, \dots, v_l \in P_{\omega_r}$ such that
$P_{\omega_r} \subset B(v_1,\alpha) \cup \cdots \cup B(v_l,\alpha)$.
For each $v \in P_{\omega_r}$ choose $j\in\{1, \dots, l\}$ such that $v \in B(v_j,\alpha)$ and define
$\cE_{k,0}=\cE_{k,0}(\nu_\infty,\delta,n,v,\omega_r)\subset\supp\nu_k^{(n)}$ by
\begin{equation}\label{eq_stab3}
\cE_{k,0} = \big[\text{$\alpha$-neighborhood of } \cE_0(\nu_\infty,\delta,n,v_j^\perp)\big] \cap\supp\nu_k^{(n)},
\end{equation}
where $\cE_0(\nu_\infty,\delta,n,v_j^\perp)\subset\supp\nu_\infty^{(n)}$ is as defined in Proposition~\ref{p_sec5_derivative_estimate}.
Since $\nu_k^{(n)}$ converges to $\nu_\infty^{(n)}$ in the weak$^*$ topology, the limit inferior of the
$\nu_k^{(n)}$-measure of \eqref{eq_stab3} as $k\to\infty$ is greater than or equal to
$$
\nu^{(n)}_\infty\left(\cE_0(\nu_\infty,\delta,n,v_j^\perp)\right) > 1 - \delta
$$
for every $j=1, \dots, l$. In particular, there is $k'_r=k'_r(\nu_\infty,\delta, n,\omega_r)\in\NN$
such that
\begin{equation}\label{eq_stab4}
\nu_k^{(n)}(\cE_{k,0}) > 1 - \delta
\text{ for every $k\ge k'_r$ and $v\in P_{\omega_r}$.}
\end{equation}

Given $x, x' \in \cF(r,d)$ with $F'_1 \not\subset F_r$ and $\VA_r(x,x')\ge\omega_r$,
take $u=u^E+u^\perp$ to be a non-zero vector that realizes the supremum in the definition \eqref{eq_VAR_defbis}.
Then,
\begin{align}
\label{eq_stab5}
& \omega_r \le \VA_r(x,x')  = \frac{\|u^\perp\|}{\|u\|} \quad(\text{in particular, } u\in P_{\omega_r}) \text{ and}\\
\label{eq_stab6}
& \VA_r(gx,gx') \ge \frac{\|(g u)^\perp\|}{\|gu\|}
\ge \VA_r(x,x') \frac{\|(g u)^\perp\|}{\|gu\|}\frac{\|u\|}{\|u^\perp\|}
\end{align}
for any $g\in G$. Then define
\begin{equation}\label{eq_stab6.5}
\cE''_{k,r}=\cE_{k,0}(\nu_\infty,\delta,n,u,\omega_r).
\end{equation}
It follows from \eqref{eq_stab4} that $\nu_k^{(n)}((\cE''_{k,r})^c)<\delta$ for every $k\ge k'_r$.

Let $g\in\cE''_{k,r}$ and $k \ge k'_r$.
Then, by definition, there exist $v=v^E+v^\perp$ in $P_{\omega_r}$ (take $v=v_j$ as in \eqref{eq_stab3})
and $f\in\cE_0(\nu_\infty,\delta,n,v^\perp)\subset \supp\nu_\infty^{(n)}$ such that $d(u,v)<\alpha$
and $d(g,f)<\alpha$. %Write $v=v^E+v^\perp$ with $v^E\in E$ and $v^\perp\in E^\perp$.
Thus, substituting \eqref{eq_stab2} and \eqref{eq_stab2bis} in \eqref{eq_stab6}, we find that
\begin{equation}\label{eq_stab7}
\VA_r(gx,gx')
\ge \frac 12 \frac{\|f^\perp v^\perp\|}{\|f v\|}
\ge \frac{1}{4} \VA_r(x,x') \frac{\|f^\perp v^\perp\|}{\|f v\|}\frac{\|v\|}{\|v^\perp\|}.
\end{equation}
Let $\tau_0=\tau_0(\nu_\infty,\delta)>0$ be as in Proposition~\ref{p_sec5_derivative_estimate}.
If  ${\|f^\perp v^\perp\|}/{\|f v\|} \ge {\tau_0}/{2}$ then the first part of \eqref{eq_stab7} gives that
(recall \eqref{eq_VA5.5} also)
\begin{equation}\label{eq_stab8}
- \log\VA_r(gx,gx') \le - \log\frac{\tau_0}{4} \le \tilde\theta_r.
\end{equation}
Now suppose that ${\|f^\perp v^\perp\|}/{\|f v\|} < {\tau_0}/{2}$.
Then part (2) of Proposition~\ref{p_sec5_derivative_estimate} gives that
\begin{equation}\label{eq_fv}
\frac{\|f v^\perp\|}{\|f v\|}
< \frac 12
\text{ and so }
\frac{\|f v^E\|}{\|f v\|}
> \frac 12.
\end{equation}
Substituting \eqref{eq_fv} and $\|v\| \ge \|v^E\|$ in \eqref{eq_stab7}, we find that
\begin{equation}\label{eq_stab8.5}
\begin{aligned}
\VA_r(gx,gx')
& \ge \frac 1{8} \VA_r(x,x') \frac{\|f^\perp v^\perp\|}{\|v^\perp\|}\frac{\|v^E\|}{\|f v^E\|}\\
& = \frac 1{8} \VA_r(x,x') \frac{\|Df^\perp_{v^E} v^\perp\|}{\|v^\perp\|}.
\end{aligned}
\end{equation}
By part (1) of Proposition~\ref{p_sec5_derivative_estimate} and the choices of $\tilde\kappa_r>0$ and $\widetilde{N}_r\in\NN$ in \eqref{eq_VA5.5},
this implies
\begin{equation}\label{eq_stab9}
\begin{aligned}
-\log\VA_r(gx,gx')
& \le - \log \VA_r(x,x') + \log 8 - 2 \tilde\kappa_r n\\
& \le - \log \VA_r(x,x') - \tilde\kappa_r n.
\end{aligned}
\end{equation}
The conclusion of the lemma is contained in \eqref{eq_stab8} and \eqref{eq_stab9}.
\end{proof}

Next, let us prove the following extension of Lemma~\ref{l_Lem5.7}:

\begin{lemma}\label{l_Lem5.7bis}
Given $\delta>0$ and $n \ge N_r$
there exists $\hat k_r=\hat k_r(\nu_\infty,\delta,n)\in\NN$ such that for any $x, x' \in \fE_r(\vep_r)$
with $F'_1 \not\subset F_r$ and $-\log\VA_r(x,x') \le \tilde\theta_r+\tilde\kappa_r n$ there exists
$\cE''_{k,r}=\cE''_{k,r}(\nu_\infty,\delta,n,x,x')\subset\supp\nu_k^{(n)}$ with
$\nu_k^{(n)}((\cE''_{k,r})^c)<\delta$ and
\begin{equation}\label{eq_hatE1}
-\log d(gF_r,g(F_{r-1}+F'_1)) \le -\log d(F_r, F_{r-1}+F'_1)-\log\VA_r(x,x')-\hat\kappa_r n
\end{equation}
for every $g\in\cE''_{k,r}$ and $k\ge \hat k_r$.
\end{lemma}

\begin{proof}
Let $\hat P$ denote the (compact) subset of pairs $(v,w) \in P\times P$ such that
\begin{equation}\label{eq_vw}
\frac{\|v^\perp\|}{\|v\|} \le \hat\rho_1 < 2\hat\rho_1 \le \frac{\|w^\perp\|}{\|w\|}
\end{equation}
($v^\perp$ and $w^\perp$ denote the components of $v$ and $w$ orthogonal to the equator). Let $\hat V$ be some compact neighborhood of the support of $\nu_\infty^{(n)}$.
Since, $\hat\rho_r=\hat\rho_r(\nu_\infty,n)$, both $\hat P$ and $\hat V$ depend only on $\nu_\infty$ and $n$.
Condition \eqref{eq_vw} ensures that the angle between $v$ and $w$ is bounded away from zero and,
consequently, so is the angle between $gv$ and $gw$ for any $g\in\hat V$;
both bounds depend only on $\nu_\infty$ and $n$. Thus, the map
\begin{equation}\label{eq_hatE2}
(v,w,g) \mapsto - \log \frac{\|\Pi_{gv} gw\|}{\|gv\|}\frac{\|v\|}{\|w\|}
\end{equation}
is well-defined and (uniformly) continuous on the domain $(v,w)\in\hat P$ and $g\in\hat V$.
In particular, there exists $\hat\alpha=\hat\alpha(\nu_\infty,n)>0$ such that
\begin{equation}\label{eq_hatE3}
- \log \frac{\|\Pi_{gv} gw\|}{\|gv\|}\frac{\|v\|}{\|w\|}
\le - \log \frac{\|\Pi_{fu} f z\|}{\|f u\|}\frac{\|u\|}{\|z\|} + \log 2
\end{equation}
whenever $d(v,u)<\hat\alpha$ and $d(z,w)<\hat\alpha$ and $d(g,f) < \hat\alpha$.
Reducing $\hat\alpha$ if necessary, depending only on $\nu_\infty$ and $n$,
we may also suppose that
\begin{equation}\label{eq_hatE3.5}
d(z,w)<\hat\alpha
\quad\Rightarrow\quad
-\log\frac{\|z^\perp\|}{\|z\|} \le -\log\frac{\|w^\perp\|}{\|w\|} + \log 2.
\end{equation}

Fix points $(v_1,w_1), \dots, (v_l,w_l)\in \hat P$ such that the balls of radius $\hat\rho$
around these points cover $\hat P$.
For each $(v,w)\in\hat P$ choose $j\in\{1, \dots, l\}$ such that $v\in B(v_j,\hat\alpha)$ and
$w\in B(w_j,\hat\alpha)$ and then define
$\hat\cE_{k,0}=\hat\cE_{k,0}(\nu_\infty,\delta,n,v,w)\subset G$ by
\begin{equation}\label{eq_hatE4}
\hat\cE_{k,0} = \big[\text{$\hat\alpha$-neighborhood of $\cE_0(\nu_\infty,\delta,n,w_j^\perp)\big] \cap \supp\nu_k^{(n)}$},
\end{equation}
where $\cE_0(\nu_\infty,\delta,n,w_j^\perp)$ is given by Proposition~\ref{p_sec5_derivative_estimate}.
Since $\nu_k^{(n)}\to\nu_\infty^{(n)}$ in the weak$^*$ topology, the limit inferior of the
$\nu_k^{(n)}$-measure of \eqref{eq_hatE4} as $k\to\infty$ is greater than or equal to
$$
\nu^{(n)}_\infty\left(\cE_0(\nu_\infty,\delta,n,w_j^\perp)\right) > 1 - \delta
$$
for every $j=1, \dots, l$. In particular, there is $\hat k_r=\hat k_r(\nu_\infty,\delta, n)\in\NN$
such that
\begin{equation}\label{eq_hatE5}
\nu_k^{(n)}(\hat\cE_{k,0}) > 1- \delta
\text{ for every $k\ge \hat k_r$ and $(v,w) \in \hat P$.}
\end{equation}

Given $x, x' \in \fE_r(\vep_r)$ with $F'_1 \not\subset F_r$,
let $v\in F_r$, $v'\in F_{r-1} + F'_1$ and $w=v-v'$ be as in \eqref{eq_vvprime} and \eqref{eq_dista12}.
%Just as in \eqref{eq_dista6} and \eqref{eq_dista7},
%$$
%d(F_r,F_{r-1}+F'_1) \le \|w\| = \frac{\|w\|}{\|v\|}
%\quand
%d(gF_r,g(F_{r-1}+F'_1)) \ge \frac{\|\Pi_{gv} gw\|}{\|gv\|}
%$$
%for any $g\in G$ such that $\|gv\|\ge\|gv'\|$ (the case $\|gv\|\le\|gv'\|$ is analogous,
%reversing the roles of $F_r$ and $F_{r-1}+F'1$). Thus,
Just as in \eqref{eq_dg1},
\begin{equation}\label{eq_dg4}
-\log d(gF_r,g(F_{r-1}+F'_1))
\le - \log d(F_r,F_{r-1}+F'_1) - \log \frac{\|\Pi_{gv}gw\|}{\|gv\|}\frac{\|v\|}{\|w\|}
\end{equation}
for any $g\in G$ such that $\|gv\|\ge\|gv'\|$ (the case $\|gv\|\le\|gv'\|$ is analogous,
reversing the roles of $F_r$ and $F_{r-1}+F'_1$). The assumption $x\in E_r(\vep_r)$ implies that
\begin{equation}\label{eq_hatE6a}
\frac{\|v^\perp\|}{\|v\|}
\le d(x,E) \le \vep_r < \hat\rho_r.
\end{equation}
Now assume that $-\log\VA_r(x,x') \le \tilde\theta_r+\tilde\kappa_r n$.
Then, using \eqref{eq_dista22} and \eqref{eq_dista75},
\begin{equation}\label{eq_hatE6b}
\frac{\|w^\perp\|}{\|w\|}
> \frac{1}{4} \VA_r(x,x')
\ge \frac {1}{4} e^{-\tilde\theta_r-\tilde\kappa_r n}
> 2\hat\rho_r.
\end{equation}
Thus, $(v,w)\in\hat P$. Then define
\begin{equation}\label{eq_hatE6c}
\cE''_{k,r}=\hat\cE_{k,0}(\nu_\infty,\delta,n,v,w).
\end{equation}
It follows from \eqref{eq_hatE5} that $\nu_k^{(n)}((\cE''_{k,r})^c)<\delta$ for every $k\ge \hat k_r$.

Take $u=v_j$ and $z=w_j$ as in \eqref{eq_hatE4}. By definition, $(u,z)\in\hat P$ and $d(u,v)<\hat\alpha$
and $d(z,w)<\hat\alpha$. Let $g\in\cE''_{k,r}$ and $k\ge \hat k_r$. The definitions \eqref{eq_hatE4} and
\eqref{eq_hatE6c} imply that there exists $f\in\cE_0(\nu_\infty,\delta,n,z^\perp)\subset\supp\nu_\infty^{(n)}$
such that $d(g,f)<\theta$. Combining \eqref{eq_dg4} with \eqref{eq_hatE3},
$$
-\log d(gF_r,g(F_{r-1}+F'_1))
\le - \log d(F_r,F_{r-1}+F'_1) + \log 2 - \log \frac{\|\Pi_{fu} fz\|}{\|fu\|}\frac{\|u\|}{\|z\|}.
$$
Replacing $z$, $w$ and $g$ with $u$, $z$ and $f$ in
\eqref{eq_dista9fresh} and \eqref{eq_dista10}, we get that
\begin{equation*}
\frac{\|fu\|}{\|u\|} \ge \frac 12 \frac{\|fu^E\|}{\|u^E\|}
\text{ and}
\|\Pi_{fu} fz\| \ge \frac 12\|f^\perp z^\perp\|.
\end{equation*}
Substituting this in the previous inequality,
\begin{equation*}%\label{eq_dg5}
-\log d(gF_r,g(F_{r-1}+F'_1))
\le - \log d(F_r,F_{r-1}+F'_1) + \log 8 - \log \frac{\|f^\perp z^\perp\|}{\|z\|}\frac{\|u^E\|}{\|fu^E\|}.
\end{equation*}
Conditions \eqref{eq_hatE3.5} and \eqref{eq_hatE6b} give that
$$
-\log\frac{\|z^\perp\|}{\|z\|}
\le -\log\frac{\|w^\perp\|}{\|w\|} + \log 2
\le -\log\VA_r(x,x') + \log 8.
$$
Combining this with the previous inequality, we find that
$$
\begin{aligned}
-\log d(gF_r,g(F_{r-1}+F'_1))
& \le - \log d(F_r,F_{r-1}+F'_1) -\log\VA_r(x,x') \\
& \hspace{3cm} + \log 64 - \log \frac{\|f^\perp z^\perp\|}{\|z^\perp\|}\frac{\|u^E\|}{\|f u^E\|}\\
& = - \log d(F_r,F_{r-1}+F'_1) - \log \VA_r(x,x') \\
& \hspace{3cm}+ \log 64 - \log\frac{\|Df^\perp_{u^E} z^\perp\|}{\|z^\perp\|}.
\end{aligned}
$$
By part (1) of Proposition~\ref{p_sec5_derivative_estimate},
and the choice of $\widehat{N}_r\in\NN$ in \eqref{eq_dista22}, this implies that
$$
\begin{aligned}
-\log d(gF_r,g(F_{r-1}+F'_1))
& \le - \log d(F_r,F_{r-1}+F'_1) - \log \VA_r(x,x') + 5 - 2\hat\kappa_r n \\
& \le -\log d(F_r,F_{r-1}+F'_1) - \log \VA_r(x,x') - \hat\kappa_r n,
\end{aligned}
$$
as claimed.
\end{proof}

We deduce the following extension of Proposition~\ref{p_Lem5.5}:

\begin{lemma}\label{l_Prop7.1bis}
For $\delta>0$, $n\ge N_r$ and $\uomega=(\omega_1, \dots, \omega_r)\in\RR^d_+$ there is
$k'''_r=k'''_r(\nu_\infty,\delta,n,\omega_r)\in\NN$
and for $x, x' \in \fE_r(\vep_r)$ with $F'_1 \not\subset F_r$ and $\VA_r(x,x')\ge\omega_r$ there is
$\cE'''_{k,r}=\cE'''_{k,r}(\nu_\infty,\delta,n,x,x',\uomega)\subset\supp\nu_k^{(n)}$
such that $\nu_k^{(n)}((\cE'''_{k,r})^c)<\delta$ and
\begin{equation}\label{eq_Prop7.1bis}
-\log\psi_r(gx,gx';\uomega) \le - \log\psi_r(x,x';\uomega) - \kappa''_rn.
\end{equation}
for any $g\in\cE'''_{k,r}$ and $k\ge k'''_r$.
\end{lemma}

\begin{proof}
Fix $\delta>0$ and $n\ge N_r$ and $\uomega\in\RR_+^d$. Recall that
$$
N_r > \max\big\{\widetilde{N}_r(\nu_\infty,{\delta}/{2}), \widehat{N}_r(\nu_\infty,{\delta}/{2}), {2\tilde\theta_r}/{(\gamma\hat\kappa_r)}\big\},
$$
by \eqref{eq_constants}. Define
\begin{equation}\label{eq_kkEE}
\begin{aligned}
&k'''_r=\max\{k'_r(\nu_\infty,\delta/2,n,\omega),\hat k_r(\nu_\infty,\delta/2,n)\}\\
&\hspace*{1cm}\text{and }\cE'''_{k,r}=\cE''_{k,r}(\nu_\infty,\delta/2,n,x,x',\omega) \cap \cE''_{k,r}(\nu_\infty,\delta/2,x,x'n).
\end{aligned}
\end{equation}
By construction, $\cE'''_{k,r}$ is contained in the support of $\nu_k^{(n)}$ and $\nu_k^{(n)}((\cE'''_{k,r})^c)<\delta$.

We claim that given any $x, x' \in \fE_r(\vep_r)$ with $F'_1 \not\subset F_r$
and $\VA_r(x,x')\ge\omega_r$,
\begin{equation}\label{eq_SVP_bounding_bis}
- \log \SVP_r(gx,gx';\omega_r) \le  - \log\SVP_r(x,x';\omega_r) - \kappa'_r n
\end{equation}
for every $g\in\cE'''_{k,r}$ and $k\ge k'''_r$. As observed in \eqref{eq_SVP1}, the assumption on
$(x,x')$ implies that $\SVA_r(x,x';\omega_r)=\VA_r(x,x')$ and $\SVA_r(x,x';\omega_r)=\VA_r(gx,gx')$,
and so \eqref{eq_SVP_bounding_bis} may be rewritten as
\begin{equation}\label{eq_SVP_bounding_tris}
- \log \VP_r(gx,gx') \le  - \log\VP_r(x,x') - \kappa'_r n.
\end{equation}

Let $g\in\cE'''_{k,r}$ and $k\ge k'''_r$. Suppose first that $-\log\VA_r(x,x')\ge\tilde\theta_r+ \tilde\kappa_r n$.
Then, by Lemma~\ref{l_Lem5.4bis},
\begin{equation}\label{eq_VP2bis}
-\log\VA_r(gx,gx') \le - \log\VA_r(x,x') - \tilde\kappa_r n.
 \end{equation}
 Substituting \eqref{eq_VP2bis} and \eqref{eq_AAA} in the definition \eqref{eq_VPr_def} we find that
 \begin{equation}\label{eq_VP3bis}
 \begin{aligned}
-\log\VP_r(gx,gx')
& \le  - \log \VA_r(x,x') - \tilde\kappa_r n \\
& \hspace{2cm} - \gamma_r \log d(F_r,F_{r-1}+F'_1) + \gamma_r A n \\
& \le - \log \VP_r(x,x') - \frac{\tilde\kappa_r}{2} n
%\le - \log \SVP_r(x,x';\omega_r) - \kappa'_r n.
\end{aligned}
\end{equation}
(we chose $\gamma_r \le \tilde\kappa_r/(2A)$ in \eqref{eq_constants}).
Now suppose that $-\log\VA_r(x,x') \le \tilde\theta_r+\tilde\kappa_r n$.  In this case, Lemma~\ref{l_Lem5.4bis} yields
\begin{equation}\label{eq_VP4bis}
-\log\VA_r(gx,gx') \le \tilde\theta_r
\end{equation}
whereas Lemma~\ref{l_Lem5.7bis} yields
\begin{equation}\label{eq_VP5bis}
- \log d(gx,gx') \le -\log d(x,x') - \log\VA_r(x,x') - \hat\kappa_r n.
\end{equation}
Substituting \eqref{eq_VP4bis} and \eqref{eq_VP5bis} in the definition \eqref{eq_VPr_def},
we obtain
$$
\begin{aligned}
-\log\VP_r(gx,gx')
& \le - \gamma_r \log d(F_r,F_{r-1}+F'_1) - \gamma_r \log\VA_r(x,x') - \gamma_r\hat\kappa_rn \\
& \le -\log\VP_r(x,x')  + \tilde\theta_r + (1-\gamma_r) \log \VA_r(x,x') - \gamma_r\hat\kappa_rn.
\end{aligned}
$$
Since $\VA_r(x,x') \le 1$, $\gamma_r \le 1$, and $ n \ge N_r \ge 2\tilde\theta_r/(\gamma_r\hat\kappa_r)$,
it follows that
\begin{equation}\label{eq_VP6bis}
\begin{aligned}
-\log\VP_r(gx,gx')
& \le -\log\VP_r(x,x') + \tilde\theta_r - \gamma_r\hat\kappa_rn\\
& \le -\log\VP_r(x,x') - \frac{\gamma_r\hat\kappa_r}{2} n.
%\le  - \log\VP_r(x,x') - \kappa'_r n.
\end{aligned}
\end{equation}
Because of the way we chose $\kappa'_r$ in \eqref{eq_constants},
the relations \eqref{eq_VP3bis} and \eqref{eq_VP6bis} contain the claim \eqref{eq_SVP_bounding_bis}.

Now we prove the claim \eqref{eq_Prop7.1bis}.
The case $r=1$ was done in \eqref{eq_r=1_Prop7.1bis}, so let us suppose $r>1$.
By \eqref{eq_Spsir_induction},
\begin{equation*}
- \log\psi_{r-1}(gx_-,gx'_-;\omega) \le - \log\psi_{r-1}(x_-,x'_-;\uomega) + C''_{r-1} n.
\end{equation*}
Substituting this and \eqref{eq_SVP_bounding_bis} in the definition \eqref{eq_psir_def},
$$
\begin{aligned}
- \log\psi_r(x,x';\uomega)
& = - \beta_{r-1} \log\psi_{r-1}(gx_-,gx'_-;\omega)
- \log \SVP_r(gx,gx';\omega_r) \\
& \le  - \beta_{r-1} \log\psi_{r-1}(x_-,x'_-;\uomega) + \beta_{r-1} C''_{r-1} n\\
& \hspace{3cm} - \log\SVP_r(x,x';\omega_r) - \kappa'_r n \\
& \le - \log\psi_r(x,x';\uomega) - \kappa''_r,
\end{aligned}
$$
where the last step uses our choices of $\beta_{r-1}$, $C''_r$, and $\kappa''_r$
in \eqref{eq_bound_betar} and \eqref{eq_psir_A4}.
\end{proof}

Let us go back to proving Proposition~\ref{p_Lem9.1}. Define
\begin{equation}\label{eq_k1}
\begin{aligned}
k_r=\max\{\tilde k_r(\nu_\infty,\delta,n), \hat{k}_r(\nu_\infty,n,\omega_r), & k'_r(\nu_\infty,\delta,n\,\omega_r),\\
&  \hat k_r(\nu_\infty,\delta,n), k'''_r(\nu_\infty,\delta/2,n,\omega_r)\}.
\end{aligned}
\end{equation}
Then, $k_r$ depends only on $\nu_\infty$, $\delta$, $n$ and $\omega_r$.
Part (i) of the proposition is an immediate consequence of Lemma~\ref{l_eq7.5bis}.
To prove part (ii), consider
$$
\cE = \cE'''_{k,r}(\nu_\infty,\delta/2,n,x,x',\uomega) \cap \cE'''_{k,r}(\nu_\infty,\delta/2,n,x',x,\uomega)
$$
where $\cE'''_{k,r}$ is as given by Lemma~\ref{l_Prop7.1bis}. Then $\nu_k^{(n)}(\cE^c)<\delta$ and
\begin{equation}\label{eq_Psi5.8}
-\log\hpsi_r(gx,gx';\uomega) \le -\log\hpsi_r(x,x';\uomega)-\kappa'_rn \le -\log\hpsi_r(x,x';\uomega)
\end{equation}
for every $g\in\cE$ and $k\ge k'''_r$. By \eqref{eq_r=1_elementary2}, this implies that
\begin{equation}\label{eq_Psi6bis}
\Psi_r(gx,gx';\omega)\le\Psi_r(x,x';\omega)
\text{ for every $g \in \cE$.}
\end{equation}
Integrating \eqref{eq_Psi6bis} over $\cE$ and \eqref{eq_Psi3bis} over the complement,
we obtain part (ii).

Next, recall that we took $\kappa'''_r=\kappa''_r/2=\kappa'_r/2$ and $c=e^{-2\kappa'''_rn}$ and $\vep'''_r>0$
such that \eqref{eq_3lines} holds:
$$
-\log \hpsi_r(x,x';\uomega) = -\log \hpsi(x,x') \ge \Omega_r/\sqrt{c}
$$
for any $x, x' \in \fE_r(\vep'''_r)$ in general position.
Then, by \eqref{eq_r=1_elementary3} and the first inequality in \eqref{eq_Psi5.8},
\begin{equation}\label{eq_Psi7bis}
\Psi_r(gx,gx';\uomega)\le\Psi_r(x,x';\uomega)+\log\sqrt{c}=\Psi_r(x,x';\uomega)-\kappa'''_r n
\end{equation}
for every $g \in \cE$. Integrating \eqref{eq_Psi7bis} over $\cE$ and \eqref{eq_Psi3bis} over the complement,
we obtain part (iii) of the proposition.
\end{proof}

It is clear from the statements of Lemmas~\ref{l_eq7.5bis},~\ref{l_Lem5.4bis} and~\ref{l_Prop7.1bis}
that one may take $\hat{k}_r$, $k'_r$ and $k'''_r$ to increase to $\infty$ when $\omega_r$
decreases to zero and $\nu_\infty$, $\delta$, $n$ remain fixed.
Then the same is true about the map $\omega \mapsto k_r(\nu_\infty,\delta,n,\omega_r)$ defined in \eqref{eq_k1}.
Hence, we may find $\check{k}_r=\check{k}_r(\nu_\infty, \delta, n)\in\NN$ and
$\omega_{k,r}=\omega_{k,r}(\nu_\infty,\delta,n)>0$ such that
\begin{itemize}
\item the sequence $(\omega_{k,r})_k$ decreases to $0$, and
\item $k\ge k_r(\nu_\infty,\delta,n,\omega_{k,r})$ for every $k \ge \check{k}_r$.
\end{itemize}
Denote $\uomega_{k,r}=(\omega_{k,1}, \dots, \omega_{k,r})$ and then define
\begin{equation}\label{eq_psis_def}
\hpsi_{k,r}(x,x') = \hpsi_r(x,x';\uomega_{k,r})
\text{ and }
\Psi_{k,r}(x,x') = \Psi_r(x,x';\uomega_{k,r}) .
\end{equation}
%\begin{equation}\label{eq_psis_def}
%\psi_{k,r}(x,x') = \psi_r(x,x';\omega_{k,r})
%\quand
%\Psi_{k,r}(x,x') = \Psi_r(x,x';\omega_{k,r}).
%\end{equation}
%Then the following statement is contained in Proposition~\ref{p_Lem9.1}:
%
%\begin{corollary}\label{c_Lem9.1}
%For any $\delta>0$, $n\ge N_r$, the following holds for any $k\ge\check{k}_r$:
%\begin{itemize}
%\item[(i)] For any $x, x' \in \fE_r(\vep_r)$ in general position with $\Psi_{k,r}(x,x')>\log\Omega_r$,
%$$
%\int_G \Psi_{k,r}(gx,gx') \, d\nu_k^{(n)}(g) \le \Psi_{k,r}(x,x') + C''_r n.
%$$
%\item[(ii)] For any $x, x' \in \fE_r(\vep_r)$ in general position satisfying $\Psi_{k,r}(x,x')>\log\Omega_r$,
%$\VA_r(x,x')\ge\omega_{k,r}$, and $\VA_r(x',x)\ge\omega_{k,r}$,
%$$
%\int_G \Psi_{k,r}(gx,gx') \, d\nu_k^{(n)}(g) \le \Psi_{k,r}(x,x') + C''_r\delta n.
%$$
%\item[(iii)] For any $x, x' \in \fE_r(\vep'''_r)$ in general position satisfying $\VA_r(x,x')\ge\omega_{k,r}$
%and $\VA_r(x',x)\ge\omega_{k,r}$,
%$$
%\int_G \Psi_{k,r}(gx,gx') \, d\nu_k^{(n)}(g) \le \Psi_{k,r}(x,x') - (\kappa''_r - C''_r\delta) n.
%$$
%\end{itemize}
%\end{corollary}

\section{Spreading out}\label{s_spreading_out}

For $r>1$ the cut-off in Section~\ref{ss_cut-off} is insufficient to ensure that $\Psi_r(x,x';\uomega)$
is bounded near the border region of $\fE_r(\vep_r)$.
To explain why, let us consider $x=(F_1, \dots, F_r)$ and $x=(F'_1, \dots, F'_r)$ such that $F_r$ is near
the border and $F'_r$ is far from the border of $E_r(\vep_r)$.
Then $-\log\VP_r(x,x')$ is bounded, but the problem is that the term $-\log\psi_{r-1}(x_-,x'_-)$ may be
arbitrarily large, which forces $\Psi_r(x,x';\uomega)$ to be arbitrarily large as well.
For instance, $d(F_1,F'_1)$ may be very small, in which case $-\log\psi_1(F_1,F'_1)$ is very large.
Thus, Lemma~\ref{l_r=1_Prop8.1} as stated does not extend to $r>1$.

To fix this difficulty, we introduce a ``spreading out" Markov operator $\tQ_r$ which leaves
$F_r$ and $F'_r$ fixed, but averages the function out over the pairs of flags whose $r$-dimensional
components are $F_r$ and $F'_r$. This is done only on certain domains far from the equator:
elsewhere we just take $\tQ_r=\id$. The main properties of this operator are stated
in Proposition~\ref{p_Prop8.1}: roughly speaking, $\tQ_r\Psi_r$ is never much bigger than $\Psi_r$
itself, with equality close to the equator, and it is bounded near the border region.
That replaces Lemma~\ref{l_r=1_Prop8.1} when $r>1$.

The details of the spreading out construction follow. In Section~\ref{ss_Markov_on_flags}
we will incorporate $\tQ_r$ into the definition of our main Markov operators, $\qhT_{k,r}$,
defined on the space of pairs of flags, and $\qchT_{k,r}$, its lift to the blow-up space
$\cY_r(\vep_r)$.

\subsection{Homogeneous measures on flag varieties}\label{ss_homogeneous_measures}

The orthogonal group $\OO(d)$ acts transitively on the flag space $\cF(r,d)$.
We denote by $\mu$ the corresponding \emph{homogeneous measure} on $\cF(r,d)$.
This may be described as the image of the Haar probability measure of the orthogonal group
under
$$
\OO(d) \to \cF(r,d), \quad g \mapsto gz=(gH_1, \dots, gH_r),
$$
for any choice of $z=(H_1, \dots, H_r)$, and it is invariant under the $\OO(d)$-action.

Analogously, for any $F_r \in\grass(r,d)$, denote by $\mu_{F_r}$ the homogeneous measure on
\begin{equation}\label{eq_flag_Fr}
\cF(F_r) = \{(G_1, \dots, G_{r-1}, G_r)\in\cF(r,d): G_r=F_r\}
\end{equation}
corresponding to the natural action of the orthogonal group $\OO(F_r)$ on $\cF(F_r)$.
We also consider
\begin{equation}\label{eq_flag_tFr}
\cF^*(F_r) = \left\{(G_1, \dots, G_{r-1}, F_r)\in\cF(F_r): d(G_1,E) \ge \frac 12 d(F_r,E)\right\}.
\end{equation}
It is clear that there exists $a_r>0$, depending only on $r$, such that
$\mu_{F_r}(\cF^*(F_r))\ge a_r$. See Figure~\ref{f_homogeneous1}.
We denote by $\mu^*_{F_r}$ the normalized restriction of $\mu_{F_r}$ to $\cF^*(F_r)$.

\begin{figure}[ht]
\begin{center}
\psfrag{E}{{$E$}}
\psfrag{y}{{$F_r$}}
\includegraphics[height=1.8in]{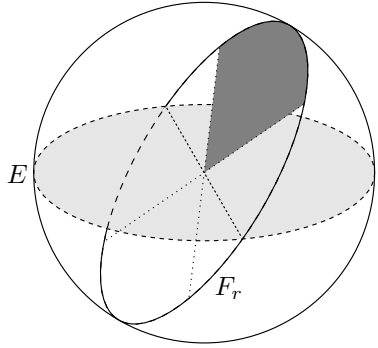}
\caption{\label{f_homogeneous1}
The subset $\cF^*(F_r)$ of the flags $(G_1, \dots, G_{r-1},F_r)$ such that $d(G_1,E)\ge d(F_r,E)/2$
is represented (for $r=2$ and $d=3$) by the dark gray region.
It corresponds to a definite fraction of all the flags in $\cF(F_r)$, relative to the homogeneous
measure $\mu_{F_r}$.
}\end{center}
\end{figure}

More generally, there exist $c_r>0$ and $b_r>0$, depending only on $r$, such that
\begin{equation}\label{eq_homogeneous2}
\mu^*_{F_r}(\{(G_1, \dots, G_{r-1}, F_r) \in \cF^*(F_r): d(G_1,F) < \rho d(F_r,F)\}) \le c_r \rho^{b_r}
\end{equation}
for any $F \neq F_r$ in $\grass(r,d)$ and $\rho>0$.
Also (increasing $c_r$ and decreasing $b_r$ if necessary),
\begin{equation}\label{eq_homogeneous3}
\mu^*_{F_r}(\{(G_1, \dots, G_{r-1}, F_r) \in \cF^*(F_r): d(G_1,F_-) < \rho\}) \le c_r \rho^{b_r}
\end{equation}
for any $F_- \in \grass(r-1,d)$ and $\rho>0$ (because $G_1$ is allowed to vary in a domain
whose dimension is strictly greater than $\dim F_-$).

\subsection{Spreading out operators}\label{ss_spreading_operators}

Let $\delta_{(x,x')}$ denote the Dirac mass at a point $(x,x')\in\prodspar$,
and $\tau:E_r(\vep_r)^2 \to [0,1]$ be a continuous symmetric function such that
\begin{equation}\label{eq_q_def1}
\begin{aligned}
\tau \equiv 1 &\text{ on } E_r(\vep_r,\vep'_r) \times E_r(2\vep''_r) \cup E_r(2\vep''_r) \times E_r(\vep_r,\vep'_r) \\
\tau \equiv 0 &\text{ outside } E_r(\vep_r,\vep'_r/2) \times E_r(3\vep''_r) \cup E_r(3\vep''_r) \times E_r(\vep_r,\vep'_r/2).
\end{aligned}
\end{equation}
See Figure~\ref{f_spreading}. Then
\begin{equation}\label{eq_q_def2}
\tq_{r,x,x'} = (1-\tau(F_r,F'_r)) \delta_{(x,x')} + \tau(F_r,F'_r)\mu^*_{F_r} \times \mu^*_{F'_r},
\end{equation}
defines a probability measure on $\cF(F_r) \times \cF(F'_r) \subset \prodspar$ depending
continuously on $(x,x') \in \prodspar$.
The \emph{spreading out operator} is the corresponding continuous Markov operator
\begin{equation}\label{eq_spreading3}
\begin{aligned}
& \tQ_r:\Bd(\prodspar) \to \Bd (\prodspar), \
 \tQ_r\tpsi(x,x') =
\int_{\cF(F_r) \times \cF(F'_r)} \tpsi \, d\tq_{r,x,x'}.
\end{aligned}
\end{equation}

\begin{figure}[ht]
\begin{center}
\psfrag{E0}{core region}
\psfrag{E1}{{$\fE_r(\vep''_r)^2$}}
\psfrag{E2}{{$\fE_r(2\vep''_r)^2$}}
\psfrag{E3}{{$\fE_r(\vep'_r/2)^2$}}
\psfrag{E4}{{$\fE_r(\vep'_r)^2$}}
\psfrag{E5}{{$\fE_r(\vep_r)^2$}}
\includegraphics[height=2in]{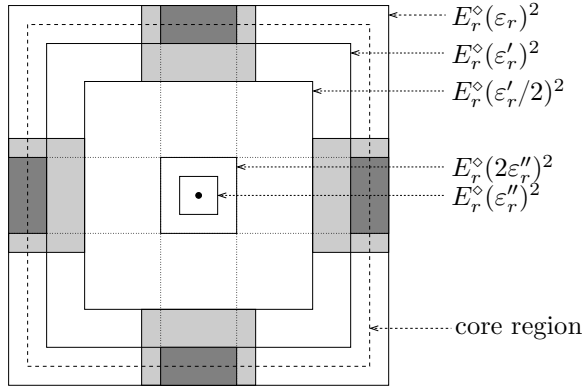}
\caption{\label{f_spreading}
Illustrating the spreading out construction. The black dot at the center marks the point $(E,E)$.
The dashed lined represents the boundary between the $\nu_k^{(n)}$-core and the $\nu_k^{(n)}$-border
of $E_1(\vep_1)$. On the dark gray area we do full averaging.
No averaging is needed on the white area. On the light gray area we interpolate between the two.
In the cut-off region $\fE(\vep_r,2\vep''_r)^2$, where $\Psi_{k,r}\equiv\log\Omega_r$,
averaging is innocuous.
 }
\end{center}
\end{figure}
It is clear from the definition that $\tQ_r$ is a lift of the identity relative to $(f,f)$, meaning that
\begin{equation}\label{eq_lift_of_identity}
\tQ_r\big(\tpsi \circ (f,f)\big) = \tpsi \circ (f,f) \text{ for every } \psi\in\Bd(E_r(\vep_r)^2),
\end{equation}
where $f:\cF(r,d) \to \grass(r,d)$ denotes the \emph{forgetfulness map}
\begin{equation}\label{eq_forgetfulness}
(F_1, \dots, F_r) \mapsto F_r.
\end{equation}
Since $\tau$ is assumed to be symmetric, we also have that $\tQ_r$ preserves the space
of symmetric functions.

For each fixed $\uomega$, denote by $\tQ_r\Psi_r(\cdot,\cdot;\uomega)$ the image of
$\Psi_r(\cdot,\cdot;\uomega)$ under the operator $\tQ_r$.
Since $\Psi_r(\cdot,\cdot;\uomega)$ is symmetric, by ~\eqref{eq_Psir_def} and
\eqref{eq_hatpsi_def}, the function $\tQ_r\Psi_r(\cdot,\cdot;\uomega)$ is also symmetric.

\begin{proposition}\label{p_Prop8.1}
There exist $K_r=K_r(\nu_\infty)>0$ and $\Omega_r=\Omega_r(\nu_\infty,\delta,n)>1$ such that
for every $(x, x') \in \prodspar$ and $\uomega=(\omega_1, \dots, \omega_r) \in \RR_+^d$
with $\omega_r < \vep'''_r$,
\begin{enumerate}
\item[(i)] $\tQ_r\Psi_r(x,x';\uomega) = \Psi_r(x,x';\uomega)$ if $(x,x') \in \fE_r(\vep'_r/2)^2$;
\item[(ii)] $\tQ_r\Psi_r(x,x';\uomega)
\le \log \Omega_r + K_r \le \Psi_r(x,x';\uomega) + K_r$ if $(x,x') \notin \fE_r(\vep'_r)^2$.
\item[(iii)] $\tQ_r\Psi_r(x,x';\uomega) \le \Psi_r(x,x';\uomega) + K_r$ if $(x,x') \notin \fE_r(\vep'_r/2)^2$.
\end{enumerate}
\end{proposition}

The constants $K_r$ and $\Omega_r$ are determined in \eqref{eq_Omegar_def} below.
The proof of this proposition occupies the remainder of the present section.
The first step is the following elementary lemma:

\begin{lemma}\label{l_Lem8.2}
Let $a$, $b >0$ and $(Z,\theta)$ be a probability space.
Let $f:Z\to(0,\infty)$ be a measurable function such that $\theta(Z \setminus Z_\tau) \le \tau$
for every $\tau\in(0,1]$, where $Z_\tau$ denotes the subset of points $z\in Z$ such that
\begin{equation}\label{eq_tau}
f(z) \ge a \tau^b.
\end{equation}
Then
$$
\int_Z \log(\Omega + f(z)^{-1}) \, d\theta(z)
\le \log(\Omega + a^{-1}) + 10 b
$$
\end{lemma}

\begin{proof}
Define $q(z) = f(z) a^{-1}$. The assumption means that $q(z) \ge \tau^b$ for every $z \in Z_\tau$.
Then, by \eqref{eq_r=1_elementary1} and \eqref{eq_r=1_elementary2},
$$
\begin{aligned}
\log (\Omega+f(z)^{-1})
& \le \log(\Omega+a^{-1}) + \max\{0,\log q(z)^{-1}\} \\
& \le \log(\Omega+a^{-1}) - b \log\tau
\end{aligned}
$$
for every $z \in Z_\tau$. In particular,
$$
\int_{Z_\tau\setminus Z_{e\tau}} \left(\log(\Omega+f(z)^{-1})- \log(\Omega+a^{-1})\right) \, d\theta(x)
\le - e \tau b \log\tau.
$$
Considering $\tau=e^{-j}$ and summing over all $j$, we get
$$
\int_{Z} \left(\log(\Omega+f(z)^{-1}) - \log(\Omega+a^{-1})\right) \, d\theta(x)
\le eb \sum_{j=1}^\infty je^{-j},
$$
which implies the claim.
\end{proof}

Keep in mind that we write $x=(F_1, \dots, F_{r-1}, F_r)$ and $x' = (F'_1, \dots, F'_{r-1}, F'_r)$.
We use $y=(G_1, \dots, G_{r-1}, F_r)$ and $y' = (G'_1, \dots, G'_{r-1},F'_r)$ to denote the
generic elements of $\cF(F_r)$ and $\cF(F'_r)$, respectively.

\begin{lemma}\label{l_Lem8.3}
There exist $c'_r=c'_r(\mu_\infty,\delta,n)>0$, $c''_r=c''_r(\mu_\infty,\delta,n)>0$,
$\alpha'_r=\alpha'_r(\nu_\infty)>0$, and $\alpha''_r=\alpha''_r(\nu_\infty)>0$ such that
for any $\tau\in(0,1]$, $x \in \cF(r,d)$, $x' \in \fE_r(\vep_r,\vep'_r/2)$,
and $\uomega=(\omega_1, \dots, \omega_r)$ with $\omega_j \in(0,\vep'''_j)$ for $j=1, \dots, r$:
\begin{enumerate}
\item[(i)] For every $y'$ in a set $X'_\tau \subset \cF^*(F'_r)$ with
$\mu^*_{F'_r}(\cF^*(F'_r) \setminus X'_\tau) < \tau$,
\begin{equation}\label{eq_spreading4b}
-\log\psi_r(x,y';\uomega) \le -\gamma_r \log d(G'_1,F_r) + c'_r - \alpha'_r \log \tau.
\end{equation}
\item[(ii)] For every $y'$ in a set $X''_\tau \subset \cF^*(F'_r)$
with $\mu^*_{F'_r}(\cF^*(F'_r) \setminus X''_\tau) < \tau$,
\begin{equation}\label{eq_spreading5}
-\log\psi_r(x,y';\uomega) \le -\gamma_r \log d(F'_r,F_r) + c''_r - \alpha''_r \log \tau.
\end{equation}
\end{enumerate}
\end{lemma}

\begin{proof}
By assumption, $x' \notin \fE_r(\vep'_r/2)$, that is, $F'_r \notin E_r(\vep'_r/2)$. Thus,
\begin{equation*}
d(G'_1,E) \ge \frac 12 d(F'_r,E)> \frac 14 \vep'_r
\end{equation*}
for any $y' \in \cF^*(F'_r)$. Consequently, recalling \eqref{eq_VAR_defbis},
\begin{equation*}
\SVA_r(x,y';\omega_r) \ge \VA_r(x,y') = d(F_r+G'_1,E) \ge d(G'_1,E) > \frac 14 \vep'_r
\end{equation*}
for any $y'\in\cF(F'_r)$. Therefore, using \eqref{eq_VPr_def},
\begin{equation}\label{eq_spreading8a}
\begin{aligned}
- \log \SVP_r(x,y';\omega_r)
& \le - \log \VP_r(x,y') \\
& \le - \gamma_r \log d(F_{r-1}+G'_1,F_r) - \log (\vep'_r/4) \\
& \le - \gamma_r \log d(G'_1,F_r) - \log (\vep'_r/4)
\end{aligned}
\end{equation}
for any $y'\in\cF(F'_r)$ (keep in mind that $d(\cdot,\cdot)$ is a distance restricted to $\grass(r,d)$).
In particular, \eqref{eq_spreading8a} contains the case $r=1$ of the claim (i):
\begin{equation}\label{eq_spreading9}
\begin{aligned}
- \log\psi_1(x,x';\omega_1)
= - \log \SVP_1(x,x';\omega_1)
\le - \gamma_1 \log d(y',x) + c'_1
\end{aligned}
\end{equation}
with $c'_1 = - \log(\vep'_1/4)$ and $\alpha'_1=0$ (and $X'_\tau=\cF^*(x')=\{x'\}$).

Now let $r>1$.
By induction, given $y'_-=(G'_1, \dots, G'_{r-1})$ in $\fE_{r-1}(\vep_{r-1},\vep'_{r-1}/2)$,
one has
\begin{equation}\label{eq_spreading9a}
-\log\psi_{r-1}(x_-,z';\uomega_-) \le -\gamma_{r-1} \log d(H'_1,F_{r-1}) + c'_{r-1} - \alpha'_{r-1} \log(\tau/2)
\end{equation}
for every $z'=(H'_1, \dots, H'_{r-2}, G'_{r-1})$ in a set $Z' \subset \cF^*(G'_{r-1})$
with
$$
\mu^*_{G'_{r-1}}(\cF^*(G'_{r-1}) \setminus Z') < \tau/2.
$$
Integrating this estimate over all admissible $G'_{r-1}$, we find that
\begin{equation}\label{eq_spreading9b}
-\log\psi_{r-1}(x_-,y'_-;\uomega_-) \le -\gamma_{r-1} \log d(G'_1,F_{r-1}) + c'_{r-1} - \alpha'_{r-1} \log(\tau/2)
\end{equation}
for every $y'$ in a set $Y' \subset \cF^*(F'_r)$
with
$$
\mu^*_{F'_r}(\cF^*(F'_r) \setminus Y') < \tau/2.
$$
On the other hand, taking $\tau = 2c_r \rho^{b_r}$ and $F_-=F_{r-1}$ in \eqref{eq_homogeneous3},
we find that
\begin{equation}\label{eq_spreading9c}
- \log d(G'_1,F_{r-1}) \le \frac{1}{b_r} \left( - \log \tau + \log c_r \right)
\end{equation}
for every $y'$ in a set $Y'' \subset \cF^*(F'_r)$
with
$$
\mu^*_{F'_r}(\cF^*(F'_r) \setminus Y'') < \tau/2.
$$
Define $X'_\tau = Y' \cap Y''$. Then $\mu^*_{F'_r}(\cF^*(F'_r) \setminus X'_\tau) < \tau$
and, by \eqref{eq_spreading9b} and \eqref{eq_spreading9c},
\begin{equation}\label{eq_spreading9d}
\begin{aligned}
-\log\psi_{r-1}(x_-,y'_-;\uomega_-)
%& \le \frac{\gamma_{r-1}}{b_r}  \left( - \log \tau + \log c_r \right) + c'_{r-1} - \alpha'_{r-1} \log(\tau/2) \\
%& \le \left(\frac{\gamma_{r-1}}{b_r}\log c_r + c'_{r-1} + \alpha'_r \log 2 \right)
%     - \left(\frac{\gamma_{r-1}}{b_r}  + \alpha'_{r-1}\right) \log \tau \\
& \le \tilde{c}_r - \tilde\alpha_r \log \tau
\end{aligned}
\end{equation}
for every $y'\in X'_\tau$, with
\begin{equation}\label{eq_spreading7a}
\tilde{c}_r = \frac{\gamma_{r-1}}{b_r}\log c_r + c'_{r-1} + \alpha'_r \log 2
\quand
\tilde\alpha_r = \frac{\gamma_{r-1}}{b_r}  + \alpha'_{r-1}.
\end{equation}
Combining \eqref{eq_spreading8a} and \eqref{eq_spreading9d}, we find that
\begin{equation}\label{eq_spreading9e}
\begin{aligned}
- \log \psi_r(x,y';\uomega)
& = - \beta_{r-1} \log \psi_{r-1}(x_-,y'_-;\uomega_-) - \log \SVP_r(x,y';\omega_r) \\
& \le \beta_{r-1} \left( \tilde{c}_r - \tilde\alpha_r \log \tau \right)
  - \gamma_r \log d(G'_1,F_r) - \log(\vep'_r/4)
\end{aligned}
\end{equation}
for every $y'$ in $X''$. This proves the claim (ii), with
\begin{equation}\label{eq_spreading_7b}
c'_r = \beta_{r-1} \tilde{c}_r - \log(\vep'_r/4) \quand \alpha'_r =  \beta_{r-1}\tilde\alpha_r.
\end{equation}

Now we deduce the claim (ii). By part (i), there exists $X'\subset\cF(F'_r)$ such that
$\mu^*_{F'_r}(\cF^*(F'_r) \setminus X')< \tau/2$ and
\begin{equation}\label{eq_spreading6a}
-\log\psi_r(x,y';\uomega) \le -\gamma_r \log(G'_1,F_r) + c'_r - \alpha'_r \log (\tau/2).
\end{equation}
for every $y'\in X'$. Taking $\tau = 2c_r\rho^{b_r}$ and $F=F'_r$ in \eqref{eq_homogeneous2},
we get that
\begin{equation}\label{eq_spreading6b}
\begin{aligned}
-\log d(G'_1,F_r)
& \le -\log d(F'_r,F_r) - \log \rho \\
& = - \log d(F'_r,F_r) + \frac{1}{b_r} \left(\log 2c_r - \log \tau\right)
\end{aligned}
\end{equation}
for every $y'$ in a set $Y \subset \cF(F'_r)$ with $\mu^*_r(\cF^*(F'_r) \setminus Y) < \tau/2$.
Define $X''_\tau=X' \cap Y$. Then  $\mu^*_{F'_r}(\cF^*(F'_r) \setminus X''_\tau) < \tau$ and
the relations \eqref{eq_spreading6a} and \eqref{eq_spreading6b} imply that \eqref{eq_spreading5}
holds for every $y'\in X''_\tau$, with
%\begin{equation}\label{eq_spreading8b}
%\begin{aligned}
%-\log\psi_r(x,y';\uomega)
%& \le \gamma_r \Big[- \log d(F'_r,F_r) + \frac{1}{b_r} \left(\log 2c_r - \log \tau\right) \Big] + c'_r - \alpha_r \log(\tau/2)\\
%& \le - \gamma_r \log d(F'_r,F_r) + \frac{\gamma_r}{b_r} \log 2c_r + c'_r + \alpha_r \log 2 - \left(\frac{\gamma_r}{b_r}+\alpha_r\right) - \log \tau
%\end{aligned}
%\end{equation}
\begin{equation}\label{eq_spreading8b}
 c''_r = c'_r + \frac{\gamma_r}{b_r} \log 2c_r + \alpha'_r \log 2
 \quand
 \alpha''_r=\alpha_r + \frac{\gamma_r}{b_r}.
\end{equation}
This completes the proof of the lemma.
\end{proof}

\begin{corollary}\label{c_Prop8.1cor}
There exist $K_r=K_r(\nu_\infty)>0$ and $\Omega_r=\Omega_r(\nu_\infty,\delta,n)>1$
such that for all $(x, x') \in \prodspar$ such that $x \notin \fE_r(\vep'_r/2)$
or $x' \notin \fE_r(\vep'_r/2)$, and for all
$\uomega=(\omega_1, \dots, \omega_r) \in \RR_+^d$ with $\omega_r < \vep'''_r$,
\begin{equation}\label{eq_twin_integrals}
\int_{\cF(F'_r)} \Psi_r(x,y';\uomega) \, d\mu^*_{F'_r}(y')
\quand
\int_{\cF(F_r)} \Psi_r(y,x';\uomega) \, d\mu^*_{F_r}(y)
\end{equation}
are both bounded above by $\log \Omega_r + K_r$.
\end{corollary}

\begin{proof}
Since the function $\Psi_r(\cdot,\cdot;\uomega)$ is symmetric, it suffices to consider the
first integral in \eqref{eq_twin_integrals}.
Initially, suppose that $x\notin \fE_r(2\vep''_r)$ and $x' \notin \fE_r(2\vep''_r)$.
Then $\Psi_r(x,x';\uomega)=\log\Omega_r$ and, in fact,
$\Psi_r(\cdot,\cdot;\uomega) \equiv \log\Omega_r$ on $\cF(F_r) \times \cF(F'_r)$.
Hence the integral is equal to $\log \Omega_r$, and so the claim is trivial in this case.

From now on, let us assume that some of the points $x$ and $x'$ is in $\fE_r(2\vep''_r)$.
Since the other is necessarily outside $\fE_r(\vep'/2)$, by hypothesis, it follows that
\begin{equation}\label{eq_distance_FF}
d(F_r,F'_r) \ge \vep'_r/2 - 2\vep''_r.
\end{equation}
We are going to apply Lemma~\ref{l_Lem8.2} to
$$
\begin{aligned}
Z=\cF(F'_r), \quad
\theta&=\mu^*_{F'_r}, \quad
\Omega = \Omega_r, \quad \\
f(y')=\psi_r(x,y';\uomega), \quad
a &= d(F_r,F'_r)^{\gamma_r} e^{-c''_r},
\quand b = \alpha''_r.
\end{aligned}
$$
The assumption \eqref{eq_tau} of Lemma~\ref{l_Lem8.2} corresponds precisely to part (ii)
of the conclusion of Lemma~\ref{l_Lem8.3}: given any $\tau\in(0,1]$,
\begin{equation*}
-\log\psi_r(x,y';\uomega) \le -\gamma_r \log d(F'_r,F_r) + c''_r - \alpha''_r \log \tau
\end{equation*}
for every $y'$ in a set $Z''_\tau \subset Z$ with $q_{r,x'}(Z \setminus Z''_\tau)< \tau$.
The conclusion of Lemma~\ref{l_Lem8.2} asserts that
\begin{equation*}
\begin{aligned}
\int_{\cF(F'_r)} \Psi_r(x,y';\uomega) \, d\mu^*_{F'_r}(y')
& = \int_{\cF(F'_r)} \log(\Omega_r + \psi_r(x,y';\uomega)^{-1}) \, d\mu^*_{F'_r}(y')\\
& \le \log(\Omega_r + e^{c''_r}d(F_r,F'_r)^{-\gamma_r}) + 10 \alpha''_r
\end{aligned}
\end{equation*}
Define
\begin{equation}\label{eq_Omegar_def}
\Omega_r = e^{c''_r} (\vep'_r/2-2\vep''_r)^{-\gamma_r}
\quand
K_r = 10\alpha''_r + \log 2.
\end{equation}
Replacing \eqref{eq_distance_FF} in the previous inequality we find that
\begin{equation*}
\begin{aligned}
\int_{\cF(F'_r)} \Psi_r(x,y';\uomega) \, d\mu^*_{F'_r}(y')
& \le \log(\Omega_r + e^{c''_r}(\vep'_r/2 - 2\vep''_r)^{-\gamma_r}) + 10 \alpha''_r \\
& \le \log\Omega_r + K_r
\end{aligned}
\end{equation*}
as claimed.
\end{proof}

\begin{proof}[Proof of Proposition~\ref{p_Prop8.1}]
It follows immediately from the definitions \eqref{eq_q_def1} and \eqref{eq_q_def2} that
$\tQ_r\Psi_r=\Psi_r$ outside $\fE_r(\vep_r,\vep'_r/2) \times \fE_r(3\vep''_r) \cup \fE_r(3\vep''_r) \times \fE_r(\vep_r,\vep'_r/2)$
This contains part (i) of the proposition.
We also have that $\tQ_r\Psi_r=\Psi_r = \log\Omega_r$ on the cut-off region $\fE_r(\vep_r,2\vep''_r)^2$.
Thus (check Figure~\ref{f_spreading}), to complete the proof of part (ii) we only have to consider the case when
$(x,x') \in \fE_r(\vep_r,\vep'_1) \times \fE_r(2\vep''_r) \cup \fE_r(2\vep''_r) \times \fE_r(\vep_r,\vep'_r)$.
In this case $\tq_{r,r,x'} = \mu^*_{F_r} \times \mu^*_{F'_r}$, and so Corollary~\ref{c_Prop8.1cor}
gives that
$$
\begin{aligned}
\tQ_r\Psi_r(x,x';\omega)
& = \int_{\cF(F_r)} \int_{\cF(F'_r)} \Psi_r(u,u';\uomega) \, d\mu^*_{F'_r}(u') \, d\mu^*_{F_r}(u) \\
& \le \int_{\cF(F_r)} (\log\Omega_r + K_r)  \, d\mu^*_{F_r}(u)
= \log\Omega_r + K_r.
\end{aligned}
$$
This proves the first inequality in part (ii). The second one is an immediate consequence of the definition of $\Psi_r$.

We are left to proving part (iii) of the proposition.
Combining the definition \eqref{eq_q_def2}  with Corollary~\ref{c_Prop8.1cor},
we see that
\begin{equation}\label{eq_estimate_Kr}
\begin{aligned}
\int_{\cF(F_r) \times \cF(F'_r)} & \Psi_r(u,u';\uomega) \, d\tq_{r,x,x'}(u,u') \\
& \le \left(1-\tau(F_r,F'_r)\right) \Psi_r(x,x';\uomega) + \tau(F_r,F'_r) \left(\log \Omega_r + K_r\right)\\
& \le \Psi_r(x,x';\uomega) + \tau(F_r, F'_r) K_r
 \le \Psi_r(x,x';\uomega) + K_r,
\end{aligned}
\end{equation}
as claimed.
\end{proof}

\begin{corollary}\label{c_Lem9.1bis}
For any $\delta>0$, $n\ge N_r$, the following holds for any $k\ge\check{k}_r$:
\begin{itemize}
\item[(i)] For any $x, x' \in \fE_r(\vep_r)$ in general position with $\Psi_{k,r}(x,x')>\log\Omega_r$,
$$
\int_G \tQ_r\Psi_{k,r}(gx,gx') \, d\nu_k^{(n)}(g) \le \Psi_{k,r}(x,x') + C''_r n + K_r.
$$
\item[(ii)] For any $x, x' \in \fE_r(\vep_r)$ in general position satisfying $\Psi_{k,r}(x,x')>\log\Omega_r$,
$\VA_r(x,x')\ge\omega_{k,r}$, and $\VA_r(x',x)\ge\omega_{k,r}$,
$$
\int_G \tQ_r\Psi_{k,r}(gx,gx') \, d\nu_k^{(n)}(g) \le \Psi_{k,r}(x,x') + C''_r\delta n + K_r.
$$
\item[(iii)] For any $x, x' \in \fE_r(\vep'''_r)$ in general position satisfying $\VA_r(x,x')\ge\omega_{k,r}$
and $\VA_r(x',x)\ge\omega_{k,r}$,
$$
\int_G \tQ_r\Psi_{k,r}(gx,gx') \, d\nu_k^{(n)}(g) \le \Psi_{k,r}(x,x') - (\kappa''_r - C''_r\delta) n.
$$
\end{itemize}
\end{corollary}

\begin{proof}
Parts (i) and (iii) of Proposition~\ref{p_Prop8.1} imply that
\begin{equation*}%\label{eq_upper_bound}
\tQ_r\Psi_r(x,x') \le \Psi_r(x,x') + K_r \text{ for every } (x,x') \in \prodspar,
\end{equation*}
Then the claims in parts (i) and (ii) of the corollary follow immediately from the corresponding
statements in Proposition~\ref{p_Lem9.1}.
In the context of part (iii) of the corollary, we even have that $\tQ_r\Psi_r(x,x')=\Psi_r(x,x')$,
and so the claim corresponds exactly to part (iii) of  Proposition~\ref{p_Lem9.1}.
\end{proof}

\section{Recoupling and conclusion}\label{s_recoupling_conclusion}

By induction, there are constants $\vep_r>0$ and $n_r \in \NN$, continuous Markov operators
\begin{equation}\label{eq_flag_operator1}
\cT_{k,r}: \Bd(E_r(\vep_r)) \to \Bd(E_r(\vep_r)),
\quad
\cT_{k,r}\varphi (F_r) = \int_{E_r(\vep_r)} \psi \, d\sigma_{k,r,F_r}
\end{equation}
adapted to $(\nu_k,E_r(\vep_r))$, and $\cT_{k,r}$-invariant probability measures
$\eta_{k,r}$ on $E_r(\vep_r)$ such that the sequence $\eta_{\infty,r} = \lim_k \eta_{k,r}$
exists and satisfies $\eta_{\infty,r}\left(E_r\right) > 0$.
Up to "localizing" the Markov operators as described in Section~\ref{ss_adapted_operators},
if necessary, we may assume that $\vep_r>0$ is small enough that
\begin{equation}\label{eq_um00}
\eta_{\infty,r}(E_r(\vep_r) \setminus E_r) < \frac 1{10} \eta_{k,r}(E_r).
\end{equation}
Then, for every $k$ sufficiently large,
\begin{equation}\label{eq_um01_cons}
\eta_{k,r}(E_r(\vep_r, \vep'''_r)) < \frac{2}{10} \eta_{k,r}(E_r(\vep_r)).
\end{equation}
We are going to show that this leads to a contradiction when $\dim E=r$, and to recover all this
information for $r+1$ when $\dim E > r$.

\subsection{Markov operators on flag varieties}\label{ss_Markov_on_flags}

Let $k$ and $r$ be fixed.
We are going to extend $\cT_{k,r}$ to a suitable Markov operator $\fcT_{k,r}$ in the space of flags, as follows.

The first step is to find a suitable lift of $\sigma_{k,r,x}$ to a probability measure on the group $G=\GL(\RR^d)$.
For each $F_r \in\grass(r,d)$ define
$$
\ev_{F_r} : G \to \grass(r,d), \quad g \mapsto gF_r.
$$

\begin{lemma}\label{l_lifting}
There exists a continuous family $\{\mu_{k,r,F_r}: F_r \in E_r(\vep_r)\}$  of probability measures
on $G$ such that $(\ev_{F_r})_*\mu_{k,r,F_r} =  \sigma_{k,r,F_r}$ for every  $F_r \in E_r(\vep_r)$.
\end{lemma}

\begin{proof}
Write $\RR^d=E\oplus E^\perp$. Every $F \in E_r(\vep_r)$ is the graph of a linear map
$u_{F}:E \to E^\perp$. Define
$$
h_{F}:\RR^d \to \RR^d, \quad v^E+v^\perp \mapsto v^E + \left(v^\perp + u_{F}(v^E)\right).
$$
Then $F \mapsto h_{F}$ is a continuous injective map from $E_r(\vep_r)$ to $G=$
with $h_{F}(E) = F$. For each fixed $F_r \in E_r(\vep_r)$, define
$$
\li_{F_r}: E_r(\vep_r) \to G, \quad F \mapsto h_F \circ h^{-1}_{F_r}.
$$
Then $\li_{F_r}$ is a continuous injection and a right-inverse of $\ev_{F_r}$:
\begin{equation}\label{eq_flag_operator4}
\ev_{F_r}\left(\li_{F_r}(F)\right)
= \li_{F_r}(F) F_r
= \left(h_F \circ h^{-1}_{F_r}\right)(F_r)
= h_F(E)
= F
\end{equation}
for every $F\in E_r(\vep_r)$. Define
$
\mu_{k,r,F_r} = (\li_{F_r})_*\sigma_{k,r,F_r}.
$
It is clear that this varies continuously with $F_r$.
The claim in the lemma follows directly from \eqref{eq_flag_operator4}.
\end{proof}

Since $\cT_{k,r}$ is adapted to $(\nu_k^{(n)},E_r(\vep_r))$, there exists a neighborhood  $V$ of
the $\nu_k^{(n)}$-core of $E_r(\vep_r)$ such that $\sigma_{k,r,F_r}$ coincides with
$\nu^{(n)}_{k,F_r}=(\ev_{F_r})_*\nu_k^{(n)}$  for every $F_r \in V$.
Let $\tau:E_r(\vep_r) \to [0,1]$ be a continuous function vanishing on a neighborhood
$U \subset V$ of the $\nu_k^{(n)}$-core of $E_r(\vep_r)$, and constant equal to $1$
outside $U$. Define
\begin{equation*}
\nu_{k,r,F_r} = (1- \tau(F_r)) \nu_k^{(n)} + \tau(F_r) \mu_{k,r,F_r}.
\end{equation*}
Observe that $(\ev_{F_r})_*\nu_{k,r,F_r} = \sigma_{k,r,F_r}$ for all $F_r \in E_r(\vep_r)$,
and so the operator $\cT_{k,r}:\Bd(E_r(\vep_r)) \to \Bd(E_r(\vep_r))$ may be rewritten as
$$
\cT_{k,r}\psi (F_r) = \int_{G} \psi(gF_r) \, d\nu_{k,r,F_r}(g).
$$

We extend this to
\begin{equation}\label{eq_flag_operator5}
\fcT_{k,r}:\Bd(\fE_r(\vep_r)) \to \Bd(\fE_r(\vep_r)), \quad
\fcT_{k,r} \psi(x) = \int_{G} \psi (gx) \, d\nu_{k,r,F_r}(g),
\end{equation}
where $x=(F_1, \dots, F_r)$ and $gx=(gF_1, \dots, gF_r)$.
It is clear that $\fcT_{k,r}$ projects to $\cT_{k,r}$ under the forgetfulness
map \eqref{eq_forgetfulness}. Let $\feta_{k,r}$ be a $\fcT_{k,r}$-invariant
probability measure projecting to $\eta_{k,r}$.

\subsection{Recoupling}\label{ss_recoupling}

We move to construct suitable self-couplings for the Markov operators $\fcT_{k,r}$.
Begin by writing the definition in \eqref{eq_flag_operator5} as
\begin{equation}\label{eq_flag_revert1}
\fcT_{k,r} \psi(x) = \int_{\fE_r(\vep_r)} \psi \, d\fsigma_{k,r,x} \text{ for } \psi\in\Bd(\fE_r(\vep_r)),
\end{equation}
where $\fsigma_{k,r,x}$ is the push-forward of $\nu_{k,r,F_r}$ under the map
$G \to \cF(r,d)$, $g \mapsto gx$. By construction, $\fsigma_{k,r,x}$ coincides with the
push-forward $\nu^{(n)}_{k,r,x}$ of $\nu^{(n)}_k$ whenever $F_r$ is in the neighborhood
$U$ of $\nu^{(n)}_k$-core of $E_r(\vep_r)$.
We denote by $\nu^{(n)}_{k,r,x,x}$ the push-forward of $\nu^{(n)}_k$ under the diagonal
embedding $G \to \cF(r,d)^2$, $g \mapsto (gx,gx')$.

Consider $X=X'=\fE_r(\vep_r)$, $Y=Y'=\fE_r(\vep_r)$, and $\eta_y=\eta'_y=\fsigma_{k,r,y}$
for every $y \in Y$. Let $K$ be the (compact) subset $\Diag_r$ of pairs $(x,x')\in\fE_r(\vep_r)^2$
which are \emph{not} in general position, that is, such that either $F'_1 \subset F_r$ or
$F_1 \subset F'_r$. It is clear that $K(x')$ and $K'(x)$ are algebraic subvarieties of $\fE_r(\vep_r)$,
and so
$$
\fsigma_{k,r,y}(K(x'))=\fsigma_{k,r,y}(K'(x))=\{0\}
$$
for every $x, x'$ and $y$.
This means that \eqref{eq_nonatomicKY} holds in this setting, and so we may use
Proposition~\ref{p_sec6_coupling_Kset_parametrized} to find a continuous family
$$
\left\{\ttheta_{k,r,x,x'}: (x,x') \in \prodspar\right\}
$$
of generic probability measures on $\prodspar$ such that each $\theta_{k,r,x,x'}$ is a coupling
of $\fsigma_{k,r,x}$ and  $\fsigma_{k,r,x'}$ vanishing on a uniform neighborhood of $\Diag_r$.

Let $\tomega:\prodspar \to [0,1]$ be a continuous function such that $\tomega(x,x')=0$ if
$F_r$ and $F'_r$ are both in $\fE_r(\tilde\vep_r)$ and $\tomega(x,x')=1$ if either of them is
outside $\fE_r(2\tilde\vep_r))$. Then
\begin{equation}\label{eq_def_adjusted2r}
\ftsigma_{k,r,x,x'} = \left(1-\tomega(x,x')\right) \nu^{(n)}_{k,r,x,x'} + \tomega(x,x') \ttheta_{k,r,x,x'}
\end{equation}
is a coupling of $\fsigma_{k,r,x}$ and $\fsigma_{k,r,x'}$ depending continuously on $(x, x')$,
and so
\begin{equation*}
\ftT_{k,r}:\Bd(\prodspar)\to \Bd(\prodspar),\
\ftT_{k,r}\tpsi(x,x') = \int_{\prodspar} \tpsi \, d\ftsigma_{k,r,x,x'},
\end{equation*}
is a continuous self-coupling of $\fcT_{k,r}$. Just as we did for $r=1$, we must modify these
operators, by recoupling the measures $\fsigma_{k,r,x}$ and $\fsigma_{k,r,x'}$ in a suitable way
on the region $\fE_r(2\tilde\vep_r, \vep''_r)^2$.

For $x \in \fE_r(2\tilde\vep_r, \vep''_r)$, it follows from \eqref{eq_r=1_III} that the subset of
$g\in\supp\nu^{(n)}_k$ such that $gx \in \fE_r(2\vep'_r)$ is disjoint from the set $\cD_k(F_r)$
given by Corollary~\ref{c_sec5_distance_contraction2}.
Hence,
\begin{equation}\label{eq_Dk}
\nu^{(n)}_{k,r,x}\left(\fE_r(2\vep'_r)\right) \le \nu^{(n)}_k\left(\cD_k(F_r)^c\right)<\delta
%\quand
%\nu^{(n)}_{k,r,x'}\left(\fE_r(2\vep'_r)\right) \le \nu^{(n)}_k\left(\cD_k(F'_r)^c\right)
\text{ for every $x$ in $\fE_r(2\tilde\vep_r, \vep''_r)$.}
\end{equation}
Note that $\fsigma_{k,r,x}=\nu^{(n)}_{k,r,x}$ if $F_r$ is in the $\nu_k^{(n)}$-core of $E_r(\vep_r)$.

Take $X=X'=\fE_r(\vep_r)$, $Y=Y'=\fE_r(2\tilde\vep_r, \vep''_r)^2$,
and $\eta_y=\eta'_y=\fsigma_{k,r,y}$ for every $y \in Y$.
Moreover, let $K = \fE_r(\vep'_r)^2 \cup \Diag_r$.
On the one hand, \eqref{eq_Dk} implies that $\fsigma_{k,r,y} (\fE_r(\vep'_r))$ and
$\fsigma_{k,r,y} (\fE_r(\vep'_r))$ are less than $\delta<1/2$ for every $y \in Y$.
On the other hand,
$$
\begin{aligned}
\Diag_r(x') & = \{x \in \fE_r(\vep_r): F'_1 \subset F_r \text{ or } F_1 \subset F'_r\}
\text{ and } \\
\Diag'_r(x) & = \{x' \in \fE_r(\vep_r): F'_1 \subset F_r \text{ or } F_1 \subset F'_r\}
\end{aligned}
$$
are algebraic subvarieties of $\fE_r(\vep_r)$, and so they have zero $\fsigma_{k,r,y}$-measure
for every $y \in Y$. These two observations show that
$$
K(x') = \fE_r(\vep'_r) \cup \Diag_r(x') \text{ and } K'(x) = \fE_r(\vep'_r) \cup \Diag'_r(x)
$$
satisfy \eqref{eq_nonatomicKY}. So we may use Proposition~\ref{p_sec6_coupling_Kset_parametrized}
to find a continuous family
$$
\{\zeta_{k,r,x,x'}: (x, x') \in \fE_r(2\tilde\vep_r, \vep''_r)^2\}
$$
of generic probability measures on $\prodspar$ such that every $\zeta_{k,r,x,x'}$ is a
coupling of $\fsigma_{k,r,x}$ and $\fsigma_{k,r,x'}$ vanishing on a uniform neighborhood
of $\Diag_r$ and such that
\begin{equation}\label{eq_construction_zeta}
\zeta_{k,r,x,x'}\left(\fE_r(\vep'_r)^2\right)=0.
\end{equation}

Fix a continuous function $\tau:\prodspar \to [0,1]$ such that $\tau\equiv 1$ on
$\fE_r(\tilde\vep_r,2\vep''_r)^2$ and $\tau\equiv 0$ on the complement of $\fE_r(2\tilde\vep_r,\vep''_r)^2$,
and then define
\begin{equation}\label{eq_interpolate1}
\fhsigma_{k,r,x,x'}= \left(1-\tau(x,x')\right)\ftsigma_{k,r,x,x'} + \tau(x,x') \zeta_{k,r,x,x'}
\end{equation}
for every $(x,x') \in \prodspar$. Then $\fhsigma_{k,r,x,x'}$ is a coupling of $\fsigma_{k,r,x}$
and $\fsigma_{k,r,x'}$ depending continuously on $(x,x')$, and so
%and
$$
\fhT_{k,r}:\Bd(\prodspar) \to \Bd(\prodspar), \quad
\fhT_{k,r}\tpsi(x,x') = \int_{\prodspar} \tpsi \, d\fhsigma_{k,1,x,x'}
$$
is another continuous self-coupling of $\fcT_{k,r}$, coinciding with $\ftT_{k,r}$ outside the
\emph{recoupling region} $\fE_r(2\tilde\vep_r,\vep''_r)^2$.

Finally, define $\qhT_{k,r}:\Bd(\prodspar) \to \Bd(\prodspar)$ by $\qhT_{k,r} = \fhT_{k,r} \circ \tQ_r$,
that is,
\begin{equation}\label{eq_qT}
\qhT_{k,r}\tvarphi(x,x') = \int_{\prodspar} (\tQ_r\tvarphi) \, d\fhsigma_{k,1,x,x'}.
\end{equation}
Let $\pi_i:E_r(\vep_r)^2 \to E_r(\vep_r)$ and $\fpi_i:\fE_r(\vep_r)^2 \to \fE_r(\vep_r)$
denote the projections to the $i$th factor, $i=1,2$, and $f$ be the forgetfulness
map \eqref{eq_forgetfulness}. By \eqref{eq_qT} and \eqref{eq_lift_of_identity},
$$
\qhT_{k,r}\big(\tpsi \circ (f,f) \big)
= \fhT_{k,r}\big(\tQ_r\big(\tpsi \circ (f,f)\big)\big)
= \fhT_{k,r}\big(\tpsi \circ (f,f)\big)
$$
for any $\tpsi\in\Bd(E_r(\vep_r)^2)$. Take $\tpsi = \psi \circ \pi_i$ for any $\psi\in\Bd(E_r(\vep_r))$.
Observing that $\pi_i \circ (f,f) = f \circ \fpi_i$, and keeping in mind that $\fhT_{k,r}$ is a self-coupling
of $\fcT_{k,r}$ and the latter projects to $\cT_{k,r}$ under the forgetfulness map, we get that
\begin{equation}\label{eq_partial_lift}
\qhT_{k,r}\big(\psi \circ f \circ \fpi_i \big)
= \fhT_{k,r}\big(\psi \circ f \circ \fpi_i\big)
= \fcT_{k,r}\big(\psi \circ f\big) \circ \fpi_i
= \big(\cT_{k,r}\psi) \circ f \circ \fpi_i.
\end{equation}
In other words, $\qhT_{k,r}$ projects to $\cT_{k,r}$ under $f \circ \fpi_i$ for any $i=1, 2$.

\begin{remark}
Unlike $\fhT_{k,r}$, this $\qhT_{k,r}$ is not a coupling of operators on $\Bd(\fE_r(\vep_r))$
because $\tQ_r$ itself is not a coupling. That could be remedied by choosing $\tQ_r$ differently.
However, our choice is convenient for lifting these operators to the blow-up space $\cY_r(\vep_r)$,
as we will see in Section~\ref{ss_closure}.
\end{remark}

We will need the following extension of Lemma~\ref{l_r=1_recoupling_operation} to $r>1$:

\begin{lemma}\label{l_recoupling_operation}
Let $(x,x') \in \prodspar$ be such that
\begin{itemize}
\item[(a)] either at least one of the points $x$ or $x'$ is in the $\nu_k^{(n)}$-border of $\fE_r(\vep_r)$,
\item[(b)] or both $x$ and $x'$ are in the $\nu_k^{(n)}$-core of $\fE_r(\vep_r)$ but outside $\fE_r(2\vep''_r)$.
\end{itemize}
Then $\fhsigma_{k,r,x,x'}(\fE_r(\vep'_r)^2)=0$ and $\qhT_{k,r}\Psi_{k,r}(x,x') \le \log \Omega_r + K_r$.
\end{lemma}

\begin{proof}
Let us begin by proving the claim that $\fhsigma_{k,r,x,x'}$ vanishes on $\fE_r(\vep'_r)^2$.
If $x$ is in the $\nu_k^{(n)}$-border of $\fE_r(\vep_r)$ then, using \eqref{eq_IV},
$$
\fhsigma_{k,r,x,x'}(\fE_r(\vep'_r)^2)
\le \fsigma_{k,r,x}(\fE_r(\vep'_r))
\le \fsigma_{k,r,x}\left(\Chi^\#_{\nu_k^{(n)}}\fE_r(\vep_r)\right)=0.
$$
The same argument applies when $x'$ is in the $\nu_k^{(n)}$-border of $\fE_r(\vep_r)$.
This settles the claim in case (a). Now let $x$ and $x'$ be as in (b).
Keep in mind that $\fsigma_{k,r,x}=\nu_{k,r,x}^{(n)}$ and $\fsigma_{k,r,x'}=\nu_{k,r,x'}^{(n)}$.
By \eqref{eq_construction_zeta}, $\zeta_{k,r,x,x'}$ vanishes on $\fE_r(\vep'_r)^2$,
and so  \eqref{eq_interpolate1} gives that
$$
\fhsigma_{k,r,x,x'}\left(\fE_r(\vep'_r)^2\right) = (1-\tau(x,x'))\ftsigma_{k,r,x,x'}\left(\fE_r(\vep'_r)^2\right).
$$
If $x$ and $x'$ are both in $\fE_r(\tilde\vep_r)$ then $\tau(x,x')=1$, and the claim follows.
When $x \notin \fE_r(\tilde\vep_r)$ we get from \eqref{eq_V} that $\nu^{(n)}_{k,r,x}(\fE_r(\vep'_r))=0$.
Then
$$
\ftsigma_{k,r,x,x'}(\fE_r(\vep'_r)^2)
\le \fsigma_{k,r,x}(\fE_r(\vep'_r))
= \nu^{(n)}_{k,r,x}(\fE_r(\vep'_r))=0.
$$
The case when $x'\notin \fE_r(\tilde\vep_r)$ is analogous.
Thus $\fhsigma_{k,r,x,x'}(\fE_r(\vep'_r)^2)=0$ also in case (b).

By part (ii) of Proposition~\ref{p_Prop8.1},
it follows that $\tQ_r\Psi_{k,r}(u,u') \le \log\Omega_r + K_r$
for $\fhsigma_{k,r,x,x'}$-almost every $(u,u')\in\prodspar$.
Integrating with respect to $\fhsigma_{k,r,x,x'}$ we get that
$\qhT_{k,r}\Psi_{k,r}(x,x') \le \log \Omega_r + K_r$ as claimed.
\end{proof}

\begin{proposition}\label{p_Prop11.3}
There exist $\kappa'''_r=\kappa'''_r(\nu_\infty)>0$ and $C'''_r=C'''_r(\nu_\infty)>0$
such that given any $\delta>0$ and $n \ge N_r$ the following holds for every $k\ge \check{k}_r$:
\begin{itemize}
\item[(i)] For any $(x, x') \in \prodspar\setminus\Diag_r$,
$$
\qhT_{k,r} \Psi_{k,r}(x,x') \le \Psi_{k,r}(x,x') + C'''_r n.
$$
\item[(ii)] For any $(x, x') \in \prodspar\setminus\Diag_r$ with $\VA_r(x,x')\ge\omega_{k,r}$,
$$
\qhT_{k,r} \Psi_{k,r}(x,x') \le \Psi_{k,r}(x,x') + C'''_r(1+\delta n).
$$
\item[(iii)] For any $(x,x')\in \fE_r(\vep'''_r)\setminus\Diag_r$ with $\VA_r(x,x')\ge\omega_{k,r}$,
$$
\qhT_{k,r} \Psi_{k,r}(x,x') \le \Psi_{k,r}(x,x') - (\kappa'''_r - C'''_r\delta) n.
$$
\end{itemize}
\end{proposition}

\begin{proof}
Take $\kappa'''_r=\kappa''_r$ and $C'''_r=C''_r + K_r$, and let $k \ge \check{k}_r$.
We split the argument into four cases (check Figure~\ref{f_coupling}).

First, suppose that both $x$ and $x'$ are in the $\nu_k^{(n)}$-core of $\fE_r(\vep_r)$,
and at least one of them is in $\fE_r(\vep''_r)$. This is necessarily the case in the
setting of (iii). In particular $(x,x')$ is outside the cut-off region, which means that
$\Psi_{k,r}(x,x') > \log\Omega_r$, and there is no recoupling either:
$$
\fhsigma_{k,r,x,x'}=\nu^{(n)}_{k,r,x,x'},
\quad
\qhT_{k,r} \Psi_{k,r}(x,x') = \int_G \tQ_r \Psi_{k,r}(gx,gx') \, d\nu_k^{(n)}(g).
$$
Hence the claims in (i), (ii) and (iii) are contained in Corollary~\ref{c_Lem9.1bis}.

Now suppose that both $x$ and $x'$ are in the $\nu_k^{(n)}$-core of $\fE_r(\vep_r)$ but outside
$\fE_r(\vep''_r)$, and at least one of them is in $\fE_r(2\vep''_r)$.
It is still true that $(x,x')$ is outside the cut-off region, and so $\Psi_{k,r}(x,x')>\log\Omega_r$.
Thus the estimates in Corollary~\ref{c_Lem9.1bis} remain valid for
\begin{equation}\label{eq_linear_combination1}
\int_{\prodspar} \tQ_r \Psi_{k,r} \, d\ftsigma_{k,r,x,x}  = \int_G \tQ_r \Psi_{k,r}(gx,gx') \, d\nu_k^{(n)}(g).
\end{equation}
By \eqref{eq_construction_zeta}, the measure $\zeta_{k,r,x,x'}$ vanishes on $\fE_r(\vep'_r)^2$.
So, part (ii) of Proposition~\ref{p_Prop8.1} gives that
\begin{equation}\label{eq_linear_combination2}
\int_{\prodspar} \tQ_r\Psi_{k,r} \, d\zeta_{k,r,x,x} \le \log\Omega_r+K_r \le \Psi_{k,r}(x,x') + K_r.
\end{equation}
By the definition \eqref{eq_interpolate1}, $\qhT_{k,r} \Psi_{k,r}(x,x')$ is a convex combination of
the integrals
in \eqref{eq_linear_combination1} and \eqref{eq_linear_combination2}.
Thus the claims (i) and (ii) follow in this case.

Next suppose that both $x$ and $x'$ are in the $\nu_k^{(n)}$-core of $\fE_r(\vep_r)$ but outside $\fE_r(2\vep''_r)$.
This corresponds to case (b) of Lemma~\ref{l_recoupling_operation}: claims (i) and (ii) are
contained in the conclusion of that lemma. Finally, suppose that at lest one of the points $x$ and $x'$
is in the $\nu_k^{(n)}$-border of $\fE_r(\vep_r)$. This is precisely the situation in case (a) of
Lemma~\ref{l_recoupling_operation}, and so claims (i) and (ii) are again contained in the conclusion
of that lemma.
\end{proof}

\subsection{Contradicting $\dim E=r$}\label{ss_contradiction}

Now we are going to apply Lemma~\ref{l_sec10_Lem12.1} with $X = \prodspar$, $\cT=\qhT_{k,r}$,
$\psi=\Psi_{k,r}$, $\heta=\heta_{k,r}$,
\begin{align}
\label{eq_A_bis}
A_k & = \big\{(x, x') \in \prodspar: \VA_r(x,x') > \omega_{k,r}, \ \VA_r(x',x) > \omega_{k,r},\\
\nonumber
& \hspace{4.5cm} d(F_r,E) \le \vep'''_r, \text{ and } d(F'_r,E) \le \vep'''_r\},\\
\label{eq_Bprime_bis}
B'_k & = \big\{(x, x') \in \prodspar: \VA_r(x,x') > \omega_{k,r}, \ \VA_r(x',x) > \omega_{k,r},\\
\nonumber
& \hspace{4.2cm} \text{and }d(F_r,E) > \vep'''_r \text{ or } d(F'_r,E) > \vep'''_r\},\\
\label{eq_Bone_bis}
B^1_k & = \big\{(x,x') \in \prodspar: \VA_r(x,x')  \le \omega_{k,r}\big\},\\
\label{eq_Btwo_bis}
B^2_k & =\big\{(x,x') \in \prodspar: \VA_r(x',x) \le \omega_{k,r}\big\}.
\end{align}
The sets $A=A_k$, $B'=B'_k$, and $B'' = B^1_k \cup B^2_k$ are pairwise disjoint,
and their union is the whole $\prodspar$.
Moreover, \eqref{eq_sec7_Theta2} implies that $B''=\emptyset$ when $\dim E =r$.

Proposition~\ref{p_Prop11.3} shows that, as long as $k$ is sufficiently large,
the hypotheses of Lemma~\ref{l_sec10_Lem12.1} are satisfied for these choices, with
$$%\begin{equation}\label{eq_kappas}
\kappa_A = (\kappa'''_r-C'''_r\delta) n, \quad
\kappa'_B = C'''_r(1+\delta n), \quand
\kappa''_B = C'''_r n.
$$%\end{equation}
Take $\delta>0$ to be sufficiently small, depending on $\nu_\infty$, and $n\in\NN$ to be
sufficiently large, depending on $\nu_\infty$ and $\delta$, that
\begin{equation}\label{eq_defdeltan}
\kappa_A > 9 \kappa'_B.
\end{equation}

Using Proposition~\ref{p_sec6_coupling_Kset}, we find a self-coupling $\fheta_0$ of $\feta_{k,r}$
vanishing on a neighborhood of $\Diag_r \subset\prodspar$.
Observe that $\fheta_0$ projects to $\eta_{k,r}$ under $f \circ \fpi_i$ for any $i=1, 2$.
Then the same is true about every $\qhT_{k,r}$-iterate of $\feta_0$, by \eqref{eq_partial_lift}.
Starting from $\fheta_0$ and arguing as in the proof of Proposition~\ref{p_sec6_Margulis_finite_energy},
we find a sequence $(\qheta_{k,r,j})_j$ of probability measures on $\prodspar$ projecting to
$\eta_{k,r}$ under $f \circ \fpi_i$ for any $i=1, 2$, converging to a $\qhT_{k,r}$-invariant
measure $\qheta_{k,r}$, and satisfying $\int_{\prodspar} \Psi_{k,r} \, d\qheta_{k,r,j} <\infty$
and
$$
\int_{\prodspar} \qhT_{k,r}\Psi_{k,r}(x,x') \, d\qheta_{k,r,j}(x,x')
\ge  \int_{\prodspar} \Psi_{k,r}(x,x') \, d\qheta_{k,r,j}(x,x')
$$
for every $j$.
Applying Lemma~\ref{l_sec10_Lem12.1} with $\cT=\qhT_{k,r}$ and $\heta=\qheta_{k,r,j}$,
we get that
\begin{equation}\label{eq_conclusion1}
\qheta_{k,r,j}(B'')
\ge \frac{\kappa_A \qheta_{k,r,j}(\prodspar) - (\kappa_A + \kappa'_B)\qheta_{k,r,j}(B')}{\kappa_A + \kappa''_B}
\end{equation}
for every $j$. Passing to the limit as $j\to\infty$, we conclude that
\begin{equation}\label{eq_conclusion2}
\qheta_{k,r}(B'')
\ge \frac{\kappa_A \qheta_{k,r}(\prodspar) - (\kappa_A + \kappa'_B)\qheta_{k,r}(B')}{\kappa_A + \kappa''_B}.
\end{equation}
By definition, $B'$ is contained in the union of the pre-images
$\left(f\circ\fpi_i\right)^{-1}(E_r(\vep_r,\vep'''_r))$, $i=1, 2$. Thus, using \eqref{eq_um01_cons},
\begin{equation}\label{eq_conclusion3}
\qheta_{k,r}(B') \le 2 \eta_{k,r}(E_r(\vep_r,\vep'''_r)) < \frac{4}{10} \eta_{k,r}(E_r(\vep_r)).
\end{equation}
It is clear that $\qheta_{k,r}(\prodspar)=\eta_{k,r}(E_r(\vep_r))$.
Substituting these relations in \eqref{eq_conclusion2} and using \eqref{eq_defdeltan}, we find that
\begin{equation}\label{eq_conclusion4}
\qheta_{k,r}(B'')
\ge \frac{\kappa_A - \frac{4}{10}(\kappa_A + \kappa'_B)}{\kappa_A + \kappa''_B}\eta_{k,r}(E_r(\vep_r))
\ge \frac{5\kappa'_B}{\kappa_A + \kappa''_B}\eta_{k,r}(E_r(\vep_r)) > 0.
\end{equation}
When $\dim E=r$ this is a contradiction, because $B''$ is empty in that case.
Thus $\dim E\ge r+1$.

\subsection{Completing step $r$}\label{ss_closure}

By \eqref{eq_conclusion4}, there exists $i=1, 2$ such that
\begin{equation}\label{eq_conclusion5}
\qheta_{k,r}(B^i_k)
\ge \frac{2\kappa'_B}{\kappa_A + \kappa''_B}\eta_{k,r}(E_r(\vep_r)).
\end{equation}
It is no restriction to assume that $i=1$, as the other case can be deduced just by
exchanging the roles of $x$ and $x'$.

Consider the map
$$
\Sigma:\prodspar\setminus\Diag_r \to \grass(r+1,d),
\quad \Sigma(x,x') = F'_1+F_r
$$
and the compact topological spaces
$$
\begin{aligned}
\cY_r & = \{(x,x',y) \in \cF(r,d)^2 \times \grass(r+1,d): F'_1 \subset y \text{ and } F_r  \subset y\} \\
\cY_r(\vep) & = \{(x,x',y) \in \cY_r: x, x' \in \fE_r(\vep)\} \text{ for } \vep>0,
\end{aligned}
$$
together with the canonical projections
$$
\begin{aligned}
& p_1: \cY_r \to \cF(r,d)^2, \quad  (x,x',y) \mapsto (x,x')\\
& p_2: \cY_r \to \grass(r+1,d), \quad (x,x',y) \mapsto y.
\end{aligned}
$$
For $(x,x',y)\in\cY_r$ and $n\in\NN$, denote by $\nu^{(n)}_{k,r,x,x',y}$ the image of
$\nu^{(n)}_k$ under the diagonal action
$$
G \to \cY_r,  \quad (g \mapsto (gx,gx',gy).
$$
Clearly, each $\nu^{(n)}_{k,r,x,x',y}$ is a lift of $\nu^{(n)}_{k,r,x,x'}$ relative to $p_1:\cY_r \to \cF(r,d)^2$.
The complement of $\Diag_r$ in $\cF(r,d)^2$ embeds in $\cY_r$ through
$$
(x,x') \mapsto (x,x',F'_1+F_r).
$$
In particular, every measure $\xi$ on $\cF(r,d)^2$ that vanishes on $\Diag_r$ has a (unique) lift $\check\xi$ to $\cY_r$.

From the relations \eqref{eq_def_adjusted2} and \eqref{eq_interpolate1}, we see that
\begin{equation}\label{eq_interpolate2}
\fhsigma_{k,r,x,x'}= \left(1-\chomega(x,x')\right)\nu^{(n)}_{k,r,x,x'} + \chomega(x,x') \htheta_{k,r,x,x'}
\end{equation}
where $\chomega:\prodspar \to [0,1]$ is a continuous function that vanishes identically on
$\fE_r(\vep''_r)^2$, and each $\htheta_{k,r,x,x'}$ is a coupling of $\fsigma_{k,1,x}$ and $\fsigma_{k,1,x'}$
vanishing on a uniform neighborhood of $\Diag_r$. In view of the previous remarks, it follows that the
$\fhsigma_{k,r,x,x'}$ lift to probability measures
\begin{equation}\label{eq_interpolate3}
\chsigma_{k,r,x,x',y} = \left(1-\chomega(x,x')\right)\nu^{(n)}_{k,r,x,x',y} + \chomega(x,x') \chtheta_{k,r,x,x',y}
\end{equation}
on $\cY_r(\vep_r)$, where $\chtheta_{k,r,x,x',y}$ is the unique lift of $\htheta_{k,r,x,x',y}$.
This lift is continuous: it is clear that $\chnu^{(n)}_{k,r,x,x',y}$ varies continuously and,
by uniqueness, so does $\chtheta_{k,r,x,x',y}$.

We claim that the spreading out measures $\tq_{r,x,x'}$ in \eqref{eq_q_def2} also lift continuously to measures
$\chq_{r,x,x',y}$ on $\cY_r(\vep_r)$. Indeed, it is clear that $\{\delta_{(x,x',y)}: (x,x',)\in\cY_r(\vep_r)\}$
is a continuous lift of $\{\delta_{(x,x')}: (x,x')\in\prodspar\}$, and so it suffices to show that the family
$$
\{\mu^*_{F_r} \times \mu^*_{F'_r}: (x,x') \in \fE_r(\vep_r,\vep'_r/2) \times \fE_r(3\vep''_r)
\cup \fE_r(3\vep''_r) \times \fE_r(\vep_r,\vep'_r/2)\}
$$
lifts uniquely to $\cY_r(\vep_r)$. The latter is a direct consequence of the fact that
\begin{equation}\label{eq_the_other}
\left(\mu^*_{F_r} \times \mu^*_{F'_r}\right)(\Diag_r) = 0
\end{equation}
for any $(F_r,F'_r) \in E_r(\vep_r,\vep'_r/2) \times E_r(3\vep''_r) \cup E_r(3\vep''_r) \times E_r(\vep_r,\vep'_r/2)$,
To prove \eqref{eq_the_other}, let us write $u=(G_1, \dots, G_r)$ and $u'=(G'_1, \dots, G'_r)$.
By definition,
\begin{equation}\label{eq_the_other_vanishes}
\begin{aligned}
\left(\mu^*_{F_r} \times \mu^*_{F'_r}\right)(\Diag_r)
& \le \mu^*_{F'_r}(\{u'\in\cF(F'_r): G'_1 \subset F_r\}) \\
& \hspace{2cm} + \mu^*_{F_r}(\{u\in\cF(F_r): G_1 \subset F'_r\})
\end{aligned}
\end{equation}
The key point is that in this setting we always have $F_r \neq F'_r$.
Thus the set of $u'\in\cF(F'_r)$ such that $G'_1 \subset F_r$ is a subvariety of $\cF(F'_r)$ of
strictly smaller dimension, and so it has zero $\mu^*_{F'_r}$-measure.
Thus the first term on the right-hand side of \eqref{eq_the_other_vanishes} vanishes identically,
and then so does the second term, by symmetry. This proves the claim.

These observations ensure that the Markov operator $\qhT_{k,r}=\hT_{k,r} \circ \tQ_{r}$ admits a continuous
lift
$$
\qchT_{k,r}:\Bd(\cY_r(\vep_r)) \to \Bd(\cY_r(\vep_r)), \
\qchT_{k,r}\Phi(x,x',y) = \int_{\cY_r(\vep_r)} \Phi \, d\qchsigma_{k,r,x,x',y},
$$
given by
$$
\qchsigma_{k,r,x,x',y}=\int_{\cY_r(\vep_r)} \chq_{r,u,u',v} \, d\chsigma_{k,r,x,x',y}(u,u',v).
$$
Since the measure $\fheta_0$ was taken to vanish on a neighborhood of
$\Diag_r \subset\prodspar$, it also admits a (unique) lift $\fcheta_0$ to $\cY_r(\vep_r)$.
Applying the construction in Proposition~\ref{p_sec6_Margulis_finite_energy} simultaneously to the operators
$\qhT_{k,r}$ and $\qchT_{k,r}$, starting from $\fheta_0$ and $\fcheta_0$ respectively,
we find a sequence $(\qcheta_{k,r,j})_j$  of probability measures on $\cY_r(\vep_r)$ converging to a
$\qchT_{k,r}$-invariant probability measure $\qcheta_{k,r}$ and such that (up to restricting to a subsequence)
each $\qcheta_{k,r,j}$ projects to $\qheta_{k,r,j}$ under $p_1$.
Observe that $\qcheta_{k,r}$ and each $\qcheta_{k,r,j}$ project to $\eta_{k,r}$ under $f \circ \fpi_i \circ p_1$ for $i=1, 2$.

Next, define $\eta_{k,r+1} = p_{2*} \qcheta_{k,r}$ and let $\{d\cheta_{k,r,v}: v \in p_2 \cY_r(\vep_r)\}$
be a disintegration of $\qcheta_{k,r}$ with respect to the partition $\{p_2^{-1}(v): v \in p_2 \cY_r(\vep_r)\}$.
Then define
$$
\begin{aligned}
& \cT_{k,r+1}:\Bd(p_2 \cY_r(\vep_r)) \to \Bd(p_2 \cY_r(\vep_r)), \\
& \cT_{k,r+1}\Phi(y) = \int_{p_2^{-1}(y)} \qchT_{k,r} (\Phi \circ p_2) (x,x',y) \, d\cheta_{k,r,y}(x,x').
\end{aligned}
$$
Equivalently, $\cT_{k,r+1}\Phi(y) = \int_{p_2\cY_r(\vep_r)} \Phi \, d\sigma_{k,r+1,y}$ with
\begin{equation}\label{eq_cTk2b}
\sigma_{k,r+1,y} = \int_{p_2^{-1}(y)}  p_{2*} \qchsigma_{k,r,x,x',y} \, d\cheta_{k,r,y}(x,x').
\end{equation}

Let $\cB_k=p_2p_1^{-1}(B^1_k)=\{F'_1+F_r: (x, x') \in B^1_k\}$, where $B^1_k$ is as in the
previous section. Define also $\eta_{k,r+1,j} = p_{2*} \qcheta_{k,r,j}$ for $j\in\NN$. Then
$$
\eta_{k,r+1,j} (\cB_k) \ge \qcheta_{k,r,j} \left(p_1^{-1}(B^1_k)\right) = \qheta_{k,r,j}(B^1_k).
$$
Passing to the limit as $j \to \infty$, and arguing as in \eqref{eq_conclusion1}--\eqref{eq_conclusion4},
we find from \eqref{eq_conclusion5} that
$$
\eta_{k,r+1}(\cB_k) \ge \frac{5\kappa'_B}{\kappa_A + \kappa''_B}\eta_{k,r}(E_r(\vep_r)).
$$
Now, since $\omega_{k,r}\to 0$, the definition \eqref{eq_Bone_bis} implies that $\cB_k$ converges
to $E_r$ as $k \to \infty$.
Thus, any accumulation point $\eta_{\infty,r+1}$ of $\eta_{k,r+1}$ must satisfy
\begin{equation}\label{eq_limit_weight_is_positive}
\eta_{\infty,r+1}(E_{r+1})
\ge \frac{5\kappa'_B}{\kappa_A + \kappa''_B}\eta_{\infty,r}(E_r(\vep_r))
\ge \frac{5\kappa'_B}{\kappa_A + \kappa''_B}\eta_{\infty,r}(E_r) >0.
\end{equation}

Take $n_{r+1}=n$ and $\vep_{r+1} = \vep''_r$. Let $\nu^{(n_{r+1})}_{k,r+1,y}$ denote the push-forward
of $\nu^{(n_{r+1})}_k$ under the map $G \to \grass(r+1,d)$, $g \mapsto gy$.

\begin{lemma}\label{l_downloading1}~
\begin{enumerate}
\item[(i)] $\sigma_{k,r+1,y} = \nu^{(n_{r+1})}_{k,r+1,y}$ for every $y \in E_{r+1}(\vep_{r+1})$.
\item[(ii)] $\sigma_{k,r+1,y} \left(\Chi_{\nu^{(n)}_k}^{\#} E_{r+1}(\vep_{r+1})\right) = 0$ for every $y \notin E_{r+1}(\vep_{r+1})$
\item[(iii)] The measure $\eta_{k,r+1}$ is $\cT_{k,r+1}$-invariant.
\end{enumerate}
\end{lemma}

\begin{proof}
It is clear that $\nu^{(n_{r+1})}_{k,r+1,y}$ coincides with the push-forward of $\nu^{(n_{r+1})}_{k,r,x,x',y}$
under the projection $p_2$. Thus \eqref{eq_interpolate3} gives that
\begin{equation*}%\label{eq_interpolate4}
p_{2*} \chsigma_{k,r,x,x',y}= \left(1-\chomega(x,x')\right)\nu^{(n_{r+1})}_{k,r+1,y} +
\chomega(x,x') p_{2*}\chtheta_{k,r,x,x',y},
\end{equation*}
and so,
\begin{equation*}%\label{eq_interpolate5}
\begin{aligned}
\chsigma_{k,r+1,y}
& = \left(1-\int_{p_2^{-1}(y)} \chomega(x,x')  \, d\cheta_{k,r,y}(x,x')\right)\nu^{(n_{r+1})}_{k,r+1,y} \\
& \hspace{3cm} + \int_{p_2^{-1}(y)} \chomega(x,x') p_{2*}\chtheta_{k,r,x,x',y} \, d\cheta_{k,r,y}(x,x'),
\end{aligned}
\end{equation*}
If $y \in E_{r+1}(\vep_{r+1})$ then both $x$ and $x'$ are necessarily in $\fE_r(\vep_{r+1})$,
by \eqref{eq_sec7_Theta1}, in which case $\chomega(x,x')=0$.
Then $\sigma_{k,r+1,y} = \nu^{(n_{r+1})}_{k,r+1,y}$, as claimed in (i).

In view of \eqref{eq_cTk2b}, to prove part (ii) it suffices to show that if
$y \notin E_{r+1}(\vep_{r+1})$ then
\begin{equation}\label{eq_pi2vanish}
 p_{2*} \chsigma_{k,r,x,x',y}\left(\Chi_{\nu^{(n)}_k}^{\#} E_{r+1}(\vep_{r+1})\right)=0
\end{equation}
for any $x, x' \subset y$. If $x$ and $x'$ are both in $\fE_r(\vep_{r+1})$ then
$$
p_{2*} \chsigma_{k,r,x,x',y}
=p_{2*} \nu^{(n_{r+1})}_{k,r,x,x',y}
=\nu^{(n_{r+1})}_{k,r+1,y}
$$
and then the claim follows from Remark~\ref{r_sec6_notdeep}.
From now on, we assume that one of the points, $x$ say, is not in $\fE_r(\vep_{r+1})$.
It follows from the definitions that
$$
p_2^{-1}\left(\Chi_{\nu^{(n)}_k}^{\#} E_{r+1}(\vep_{r+1})\right)
\subset \Chi_{\nu^{(n)}_k}^{\#} \fE_r(\vep_{r+1})^2 \times \Chi_{\nu^{(n)}_k}^{\#} E_{r+1}(\vep_{r+1}),
$$
and so
\begin{equation}\label{eq_vanishingmeasure}
\begin{aligned}
p_{2*} \chsigma_{k,r,x,x',y}&\left(\Chi_{\nu^{(n)}_k}^{\#} E_{r+1}(\vep_{r+1})\right)\\
& \le \chsigma_{k,r,x,x',y}\left(\Chi_{\nu^{(n)}_k}^{\#} \fE_r(\vep_{r+1})^2 \times \Chi_{\nu^{(n)}_k}^{\#} E_{r+1}(\vep_{r+1})\right) \\
& \le \fhsigma_{k,r,x,x'}\left(\Chi_{\nu^{(n)}_k}^{\#} \fE_r(\vep_{r+1})^2\right)\\
& \le \fsigma_{k,r,x}\left(\Chi_{\nu^{(n)}_k}^{\#} \fE_r(\vep_{r+1})\right).
\end{aligned}
\end{equation}
If $x$ is in the $\nu^{(n)}_k$-border of $\fE_r(\vep_r)$ then
$$
\fsigma_{k,r,x}\left(\Chi_{\nu^{(n)}_k}^{\#} \fE_r(\vep_{r+1})\right)
\le \fsigma_{k,r,x}\left(\Chi_{\nu^{(n)}_k}^{\#} \fE_r(\vep_r)\right) = 0,
$$
because the operator $\cT_{k,1}$ is adapted to $(\nu_k^{(n)},\fE_r(\vep_r))$.
If $x$ is in the $\nu^{(n)}_k$-core of $\fE_r(\vep_r)$ then Remark~\ref{r_sec6_notdeep} gives that
$$
\fsigma_{k,r,x}\left(\Chi_{\nu^{(n)}_k}^{\#} \fE_r(\vep_{r+1})\right)
= \nu^{(n)}_{k,r,x}\left(\Chi_{\nu^{(n)}_k}^{\#} \fE_r(\vep_{r+1})\right) = 0.
$$
Thus the right-hand side of \eqref{eq_vanishingmeasure} vanishes in either case.
That completes the proof of \eqref{eq_pi2vanish} and of part (ii) of the lemma.

Finally, by definition,
$$
\begin{aligned}
\int_{p_2\cY_r(\vep_r)} \left(\cT_{k,r+1}\Phi\right) & \, d\eta_{k,r+1}\\
& =  \int_{p_2\cY_r(\vep_r)} \int_{p_2^{-1}(y)} \chT_{k,r} \left(\Phi \circ p_2\right) (x,x',y) \, d\cheta_{k,r,y}(x,x')\, d\eta_{k,r+1}(y)\\
& = \int_{\cY_r(\vep_r)}  \chT_{k,r}(\Phi \circ p_2) (x,x',y) \, d\cheta_{k,r}(x, x', y)
\end{aligned}
$$
for any $\Phi\in \Bd(p_2\cY_r(\vep_r))$.
Since $\cheta_{k,r}$ is $\chT_{k,r}$-invariant, this gives
$$
\begin{aligned}
\int_{p_2\cY_r(\vep_r)} \left(\cT_{k,r+1}\Phi\right)(y) \, d\eta_{k,r+1}(y)
& = \int_{\cY_r(\vep_r)}  (\Phi \circ p_2) (x,x',y) \, d\cheta_{k,r}(x, x', y)\\
& = \int_{p_2\cY_r(\vep_r)} \Phi \, d\heta_{k,r+1},
\end{aligned}
$$
which proves claim (iii).
\end{proof}

Since the $\chsigma_{k,r,x,x',y}$ are generic measures and the projection $p_2$ is algebraic,
it follows readily from \eqref{eq_cTk2b} and Remark~\ref{r_sec5_generic_elementary}
that every $\sigma_{k,r+1,y}$ is a generic measure.
Then, conclusions (i) and (ii) in Lemma~\ref{l_downloading1} allow us to apply
Propositions~\ref{p_sec6_localized2_strong} and~\ref{p_sec6_adapted_strong} with
$E_{r+1}(\vep_{r+1})$ and $p_2\cY_r(\vep_r)$ in the roles of $X$ and $U$, respectively.
Thus we get a continuous Markov operator adapted to $(\nu_k^{(n_{r+1})},E_{r+1}(\vep_{r+1}))$ and
which leaves the restriction of $\eta_{k,r+1} \mid E_{r+1}(\vep_{r+1})$ invariant.
Replace $\cT_{k,r+1}$ and $\eta_{k,r+1}$ with these new Markov operator and invariant measure,
respectively. This finishes step $r$ of the induction.

The proof of Theorem~\ref{theorem:inductive} is now complete.

%\bibliographystyle{plain}
%\bibliography{../bib}

\end{document}